\documentclass[reqno,12pt,letterpaper]{amsart}
\usepackage{amsmath,amssymb,amsthm,graphicx,mathrsfs,url,tikz-cd}
\definecolor{Red}{cmyk}{0,100,65,0}
\definecolor{Green}{cmyk}{34,0,55,0}
\usepackage[colorlinks=true,linkcolor=Red,citecolor=Green]{hyperref}

\def\?[#1]{\textbf{[#1]}\marginpar{\Large{\textbf{??}}}}

\pgfrealjobname{RuelleResonForHn}

\def\smallsection#1{\smallskip\noindent\textbf{#1}.}

\newtheorem{theo}{Theorem}
\newtheorem{prop}{Proposition}[section]
\newtheorem{defi}[prop]{Definition}

\newtheorem{lemm}[prop]{Lemma}
\newtheorem{corr}[prop]{Corollary}

\numberwithin{equation}{section}

\newcommand{\lra}{\leftrightarrow}
\newcommand{\mc}{\mathcal}
\newcommand{\rr}{\mathbb{R}}
\newcommand{\nn}{\mathbb{N}}
\newcommand{\cc}{\mathbb{C}}
\newcommand{\hh}{\mathbb{H}}
\newcommand{\zz}{\mathbb{Z}}
\newcommand{\sph}{\mathbb{S}}
\newcommand{\bb}{\mathbb{B}}

\newcommand{\la}{\lambda}
\newcommand{\eps}{\epsilon}

\newcommand{\pl}{\partial}
\newcommand{\x}{\times}

\newcommand{\bbar}{\overline}

\newcommand{\cjd}{\rangle}
\newcommand{\cjg}{\langle}

\newcommand{\demi}{\tfrac{1}{2}}
\newcommand{\ndemi}{\tfrac{n}{2}}

\DeclareMathOperator{\Bd}{Bd}
\DeclareMathOperator{\Eig}{Eig}
\DeclareMathOperator{\Res}{Res}
\DeclareMathOperator{\Id}{Id}
\DeclareMathOperator{\Spec}{Spec}

\DeclareMathOperator{\Div}{Div}

\let\Im=\Imag
\DeclareMathOperator{\Ker}{Ker}
\DeclareMathOperator{\Mult}{Mult}

\DeclareMathOperator{\Pol}{Pol}
\DeclareMathOperator{\PSL}{PSL}
\DeclareMathOperator{\PSO}{PSO}

\let\Re=\Real

\DeclareMathOperator{\Span}{span}
\DeclareMathOperator{\SL}{SL}
\DeclareMathOperator{\SO}{SO}

\DeclareMathOperator{\Vol}{Vol}
\DeclareMathOperator{\WF}{WF}

\def\CI{\mathcal{C}^\infty}

\title[Power spectrum of hyperbolic manifolds]
{Power spectrum of the geodesic flow\\ on hyperbolic manifolds}
\author{Semyon Dyatlov}
\email{dyatlov@math.mit.edu}
\address{Department of Mathematics, Massachusetts Institute of Technology,
Cambridge, MA 02139, USA}
\author{Fr\'ed\'eric Faure}
\email{frederic.faure@ujf-grenoble.fr}
\address{Universit\'e Joseph Fourier, 100, rue des Maths, BP74,
38402 St Martin d'Heres, France}
\author{Colin Guillarmou}
\email{cguillar@dma.ens.fr}
\address{DMA, U.M.R. 8553 CNRS, \'Ecole Normale Superieure, 45 rue d'Ulm,
75230 Paris cedex 05, France}

\begin{document}

\begin{abstract}
We describe the complex poles of the power spectrum of correlations for the geodesic flow on compact hyperbolic manifolds
 in terms of eigenvalues of the Laplacian acting on
certain natural tensor bundles.
These poles are a special case of Pollicott--Ruelle resonances, which can be defined
for general Anosov flows. In our case, resonances are stratified into bands by decay rates.
The proof also gives an explicit relation between resonant states
and eigenstates of the Laplacian.
\end{abstract}

\maketitle

%%%%%%%%%%%%%%%%%%%%%%%%%%%%%%%%%%%%%%%%%%%%%%%%%%%%%%%%%%%%%%%%%%%%%%%%%%%%%%%%
%                                 INTRODUCTION                                 %
%%%%%%%%%%%%%%%%%%%%%%%%%%%%%%%%%%%%%%%%%%%%%%%%%%%%%%%%%%%%%%%%%%%%%%%%%%%%%%%%
\addtocounter{section}{1}
\addcontentsline{toc}{section}{1. Introduction}

In this paper, we consider the characteristic frequencies of correlations,
\begin{equation}
  \label{e:correlation}
\rho_{f,g}(t)=\int_{SM} (f\circ\varphi_{-t}) \cdot \bar g \,d\mu,\quad
f,g\in \CI(SM),
\end{equation}
for the geodesic flow $\varphi_t$ on a compact hyperbolic manifold $M$ of dimension $n+1$
(that is, $M$ has constant sectional curvature $-1$). Here $\varphi_t$ acts on $SM$, the unit tangent
bundle of $M$, and $\mu$ is the natural smooth probability measure. Such $\varphi_t$ are classical examples of \emph{Anosov flows}; for this
family of examples, we are able to prove much
more precise results than in the general Anosov case.

An important question, expanding on the notion of mixing, is the behavior of $\rho_{f,g}(t)$ as $t\to +\infty$.
Following~\cite{ruelle}, we take the \emph{power spectrum}, which in our convention is the Laplace transform
$\hat\rho_{f,g}(\lambda)$ of $\rho_{f,g}$ restricted to $t>0$. The long time behavior of $\rho_{f,g}(t)$ is
related to the properties of the meromorphic extension of $\hat\rho_{f,g}(\lambda)$
to the entire complex plane. The poles of this extension, called \emph{Pollicott--Ruelle} resonances
(see~\cite{pollicott,ruelle,FaSj} and~\eqref{e:actual-def} below), are the complex characteristic frequencies of $\rho_{f,g}$, describing
its decay and oscillation and not depending on $f,g$.

For the case of dimension $n+1=2$, the following connection between resonances and
the spectrum of the Laplacian was announced in~\cite[Section~4]{FaTs2}
(see \cite{FlFo} for a related result and
the remarks below regarding the zeta function techniques).
%%%%%%%%%%%%%%%%%%%%%%%%%%%%%%%%%%%%%%%%%%%%%%%%%%%%%%%%%%%%%%%%%%%%%%%%%%%%%%%%
\begin{theo}\label{t:dim2}
Assume that $M$ is a compact hyperbolic surface ($n=1$) and the spectrum of the positive Laplacian on $M$ is
(see Figure~\ref{f:dim2})
$$
\Spec(\Delta)=\{s_j(1-s_j)\},\quad
s_j\in [0,1]\cup \Big({1\over 2}+i\mathbb R\Big).
$$
Then Pollicott--Ruelle resonances for the geodesic
flow on $SM$ in $\mathbb C\setminus (-1-{1\over 2}\mathbb N_0)$ are
\begin{equation}
  \label{e:dim2}
\lambda_{j,m}=-m-1+s_j,\quad m\in\mathbb N_0.
\end{equation}
\end{theo}
%%%%%%%%%%%%%%%%%%%%%%%%%%%%%%%%%%%%%%%%%%%%%%%%%%%%%%%%%%%%%%%%%%%%%%%%%%%%%%%%

%%%%%%%%%%%%%%%%%%%%%%%%%%%%%%%%%%%%%%%%%%%%%%%%%%%%%%%%%%%%%%%%%%%%%%%%%%%%%%%%
\begin{figure}
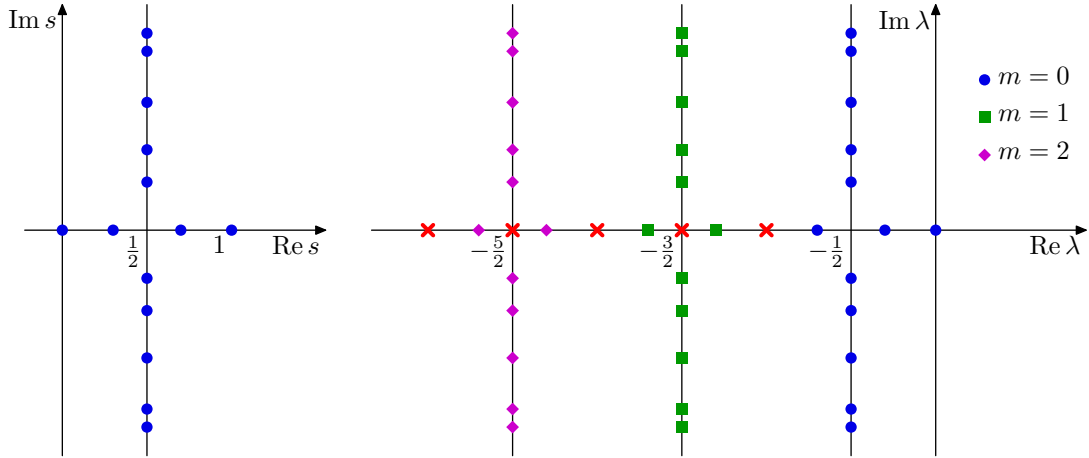

\includegraphics{rrh.8}
\quad
\includegraphics{rrh.9}
\caption{An illustration of Theorem~\ref{t:dim2}, with eigenvalues of the Laplacian
on the left and the resonances of geodesic flow, on the right. The red crosses
mark exceptional points where the theorem does not apply.}
\label{f:dim2}
\end{figure}
%%%%%%%%%%%%%%%%%%%%%%%%%%%%%%%%%%%%%%%%%%%%%%%%%%%%%%%%%%%%%%%%%%%%%%%%%%%%%%%%

\noindent\textbf{Remark}.
We use the Laplace transform (which has poles in the left half-plane) rather than the Fourier transform
as in~\cite{ruelle,FaSj}
to simplify the relation to the parameter $s$ used for Laplacians on hyperbolic manifolds.

Our main result concerns the case of higher dimensions $n+1>2$. The situation is considerably
more involved than in the case of Theorem~\ref{t:dim2}, featuring the
spectrum of the Laplacian on certain tensor bundles.
More precisely, for $\sigma\in\mathbb R$, denote
$$
\Mult_\Delta(\sigma,m):=\dim\Eig^m(\sigma),
$$
where $\Eig^m(\sigma)$, defined in~\eqref{e:eigdef}, is
the space of \emph{trace-free divergence-free symmetric sections} of $\otimes^m T^*M$
satisfying $\Delta f=\sigma f$.
Denote by
$\Mult_R(\lambda)$ the geometric multiplicity of $\lambda$
as a Pollicott--Ruelle resonance of the geodesic flow on $M$ (see Theorem~\ref{t:noalg} and the remarks preceding it for a definition).
%%%%%%%%%%%%%%%%%%%%%%%%%%%%%%%%%%%%%%%%%%%%%%%%%%%%%%%%%%%%%%%%%%%%%%%%%%%%%%%%
\begin{theo}
  \label{t:main}
Let $M$ be a compact hyperbolic manifold of dimension $n+1\geq 2$.
Assume that $\lambda\in\mathbb C\setminus \big(-{n\over 2}-{1\over 2}\mathbb N_0\big)$. Then
for $\lambda\not\in -2\mathbb N$, we have (see Figure~\ref{f:main})
\begin{equation}
  \label{e:main-1}
\Mult_R(\lambda)=\sum_{m\geq 0} \sum_{\ell=0}^{\lfloor m/2 \rfloor} \Mult_\Delta\Big(
-\Big(\lambda+m+{n\over 2}\Big)^2+{n^2\over 4}+m-2\ell, m-2\ell\Big)
\end{equation}
and for $\lambda\in -2\mathbb N$, we have
\begin{equation}
  \label{e:main-2}
\Mult_R(\lambda)= \sum_{m\geq 0\atop m\neq -\lambda} \sum_{\ell=0}^{\lfloor m/2 \rfloor} \Mult_\Delta\Big(
-\Big(\lambda+m+{n\over 2}\Big)^2+{n^2\over 4}+m-2\ell, m-2\ell\Big).
\end{equation}
\end{theo}
%%%%%%%%%%%%%%%%%%%%%%%%%%%%%%%%%%%%%%%%%%%%%%%%%%%%%%%%%%%%%%%%%%%%%%%%%%%%%%%%

%%%%%%%%%%%%%%%%%%%%%%%%%%%%%%%%%%%%%%%%%%%%%%%%%%%%%%%%%%%%%%%%%%%%%%%%%%%%%%%%
\begin{figure}
\includegraphics{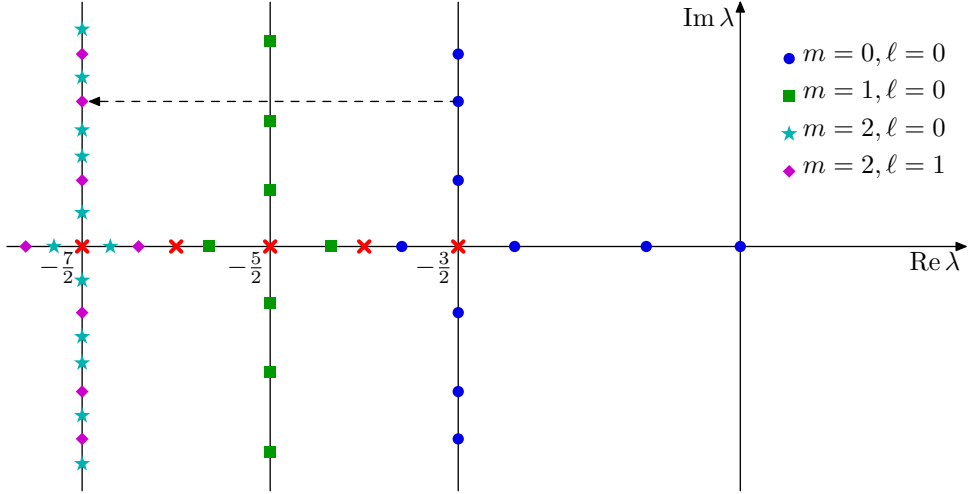}
\caption{An illustration of Theorem~\ref{t:main} for $n=3$. The red crosses
mark exceptional points where the theorem does not apply. Note that the points with
$m=2,\ell=1$ are simply the points with $m=0,\ell=0$ shifted by $-2$ (modulo exceptional points),
as illustrated by the arrow.}
\label{f:main}
\end{figure}
%%%%%%%%%%%%%%%%%%%%%%%%%%%%%%%%%%%%%%%%%%%%%%%%%%%%%%%%%%%%%%%%%%%%%%%%%%%%%%%%

\noindent\textbf{Remarks}. (i) If $\Mult_\Delta\big(
-\big(\lambda+m+{n\over 2}\big)^2+{n^2\over 4}+m-2\ell, m-2\ell\big)>0$, then Lemma~\ref{bottomsp}
and the fact that $\Delta\geq 0$ on functions imply that either $\lambda\in -m-{n\over 2}+i\mathbb R$
or
\begin{equation}
  \label{e:reparts}
\begin{aligned}
\lambda\in [-1-m,\,-m],&\quad\text{if }n=1,\ m>2\ell;\\
\lambda\in [1-n-m,\,-1-m],&\quad\text{if }n>1,\ m>2\ell;\\
\lambda\in [-n-m,\,-m],&\quad\text{if }m=2\ell.
\end{aligned}
\end{equation}
In particular, we confirm that resonances lie in $\{\Re\lambda\leq 0\}$ and
the only resonance on the imaginary axis is $\lambda=0$ with $\Mult_R(0)=1$, corresponding
to $m=\ell=0$. We call the set of resonances corresponding to some $m$
the \emph{$m$th band}. This is a special case of the band structure for general contact Anosov
flows established in the
work of Faure--Tsujii~\cite{FaTs1,FaTs2,FaTs3}.

\noindent (ii) The case $n=1$ fits into Theorem~\ref{t:main} as follows: for $m\geq 2$, the spaces
$\Eig^{m}(\sigma)$ are trivial unless $\sigma$ is an exceptional point
(since the corresponding spaces $\Bd^{m,0}(\lambda)$ of Lemma~\ref{l:reduced} would
have to be trace free sections of a one-dimensional vector bundle), and the spaces
$\Eig^{1}(\sigma+1)$ and $\Eig^0(\sigma)$ are isomorphic as shown
in Appendix~\ref{s:m1}.

\noindent (iii) The band with $m=0$ corresponds to the spectrum of the scalar Laplacian;
the band with $m=1$ corresponds to the spectrum of the Hodge Laplacian on coclosed
1-forms, see Appendix~\ref{s:m1}.

\noindent (iv) As seen from~\eqref{e:main-1}, \eqref{e:main-2}, for $m\geq 2$ the
$m$-th band of resonances contains shifted copies of bands $m-2,m-4,\ldots$
The special case~\eqref{e:main-2} means that the resonance $0$ of the $m=0$
band is not copied to other bands.

\noindent (v) A Weyl law holds for the spaces $\Eig^m(\sigma)$, see Appendix~\ref{s:weyl}.
It implies the following Weyl law for resonances in the $m$-th band:
\begin{equation}
  \label{e:the-weyl}
\sum_{\lambda\in -{n\over 2}-m+i\mathbb [-R,R]} \Mult_R(\lambda)={2^{-n}\pi^{-{n+1\over 2}}\over \Gamma({n+3\over 2})}
\cdot {(m+n-1)!\over m!(n-1)!}\Vol(M) R^{n+1}+\mathcal O(R^n).
\end{equation}
The power $R^{n+1}$ agrees with the Weyl law of~\cite[(5.3)]{FaTs2} and with the earlier
upper bound of~\cite{ddz}.
We also see that if $n>1$, then
each $m$ and $\ell\in [0, {m\over 2}]$ produce a nontrivial contribution to the set of resonances.
The factor $(m+n-1)!\over m!(n-1)!$ is the dimension of the space of homogeneous
polynomials of order $m$ in $n$ variables; it is natural
in light of~\cite[Proposition~5.11]{FaTs1}, which locally reduces resonances to such polynomials.

The proof of Theorem~\ref{t:main} is outlined in Section~\ref{s:overview}.
We use in particular the microlocal method of Faure--Sj\"ostrand~\cite{FaSj}, defining Pollicott--Ruelle resonances
as the points $\lambda\in\mathbb C$ for which the (unbounded nonselfadjoint) operator
\begin{equation}
  \label{e:actual-def}
X+\lambda:\mathcal H^r\to\mathcal H^r,\quad
r>-C_0\Re\lambda,
\end{equation}
is not invertible. Here $X$ is the vector field on $SM$ generating the geodesic flow, so that $\varphi_t=e^{tX}$,
$\mathcal H^r$ is a certain \emph{anisotropic Sobolev space}, and $C_0$ is a fixed constant independent of $r$,
see Section~\ref{s:rr-1} for details. Resonances do not depend on the choice of $r$.
The relation to correlations~\eqref{e:correlation} is given by the formula
$$
\hat\rho_{f,g}(\lambda)=\int_0^\infty e^{-\lambda t}\rho_{f,g}(t)\,dt
=\int_0^\infty e^{-\lambda t}\langle e^{-tX}f,g\rangle\,dt
=\langle (X+\lambda)^{-1}f,g\rangle_{L^2(SM)},
$$
valid for $\Re\lambda>0$ and $f,g\in\mathcal C^\infty(SM)$.
See also Theorem~\ref{t:resexp} below.

We stress that our method provides an \emph{explicit relation between classical and quantum states,\/}
that is between Pollicott--Ruelle resonant states
(elements of the kernel of~\eqref{e:actual-def}) and eigenstates of the Laplacian;
that is, in addition to the poles of $\hat\rho_{f,g}(\lambda)$, we describe its residues.
For instance for the $m=0$ band, if $u(x,\xi)$, $x\in M,\xi\in S_xM$, is a resonant state, then the
corresponding eigenstate of the Laplacian, $f(x)$, is obtained by integration of $u$ along the fibers $S_xM$,
see~\eqref{e:pushy}. On the other hand, to obtain $u$ from $f$ one needs to take the \emph{boundary distribution}
$w$ of $f$, which is a distribution on the conformal boundary $\mathbb S^n$
of the hyperbolic space $\mathbb H^{n+1}$ appearing as the leading coefficient of a weak
asymptotic expansion at $\mathbb S^n$ of the lift of $f$ to $\mathbb H^{n+1}$.
Then $u$ is described by $w$ via an explicit formula, see~\eqref{e:poppy};
this formula features the Poisson kernel $P$ and the map
$B_-:S\mathbb H^{n+1}\to \mathbb S^n$ mapping a tangent vector to the endpoint in negative infinite time of the corresponding
geodesic of $\hh^{n+1}$. The explicit relation can be schematically described as follows:
\begin{center}
\includegraphics{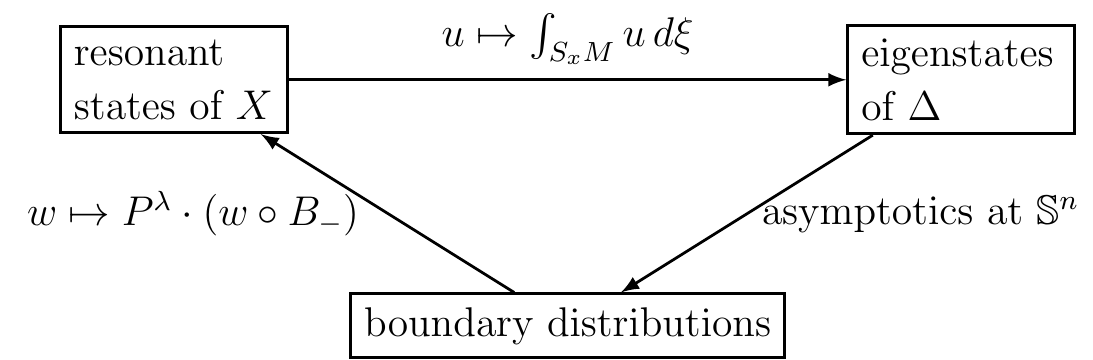}
%\beginpgfgraphicnamed{diagram1}
%\begin{tikzpicture}[arrows=-latex,style=thick]
%\node (res) at (0,0)[draw,text width=.8in]{resonant states of $X$};
%\node (eig) at (8,0)[draw,text width=.8in]{eigenstates of $\Delta$};
%\node (bd) at (4,-2.5)[draw]{boundary distributions};

%\draw (res) -- (eig) node[above,midway] {$u\mapsto \int_{S_xM} u\,d\xi$};
%\draw (bd) -- (res) node[left,midway] {$w\mapsto P^{\lambda} \cdot(w\circ B_-)\ $};
%\draw (eig) -- (bd) node[right,midway] {asymptotics at $\mathbb S^n$};
%\end{tikzpicture}
%\endpgfgraphicnamed
\end{center}
For $m>0$, one needs to also use \emph{horocyclic differential operators}, see Section~\ref{s:overview}.

Theorem~\ref{t:main} used the notion of \emph{geometric multiplicity} of a resonance $\lambda$,
that is, the dimension of the kernel of $X+\lambda$ on $\mathcal H^r$. For nonselfadjoint
problems, it is often more natural to consider the \emph{algebraic multiplicity},
that is, the dimension of the space of elements of $\mathcal H^r$ which are killed
by some power of $X+\lambda$.
%%%%%%%%%%%%%%%%%%%%%%%%%%%%%%%%%%%%%%%%%%%%%%%%%%%%%%%%%%%%%%%%%%%%%%%%%%%%%%%%
\begin{theo}
  \label{t:noalg}
If $\lambda\not \in-{n\over 2}-{1\over 2}\mathbb N_0$, then the algebraic and geometric multiplicities of $\lambda$ as a
Pollicott--Ruelle resonance coincide.
\end{theo}
%%%%%%%%%%%%%%%%%%%%%%%%%%%%%%%%%%%%%%%%%%%%%%%%%%%%%%%%%%%%%%%%%%%%%%%%%%%%%%%%
Theorem~\ref{t:noalg} relies on a \emph{pairing formula} (Lemma~\ref{l:the-pairing}), which states that
$$
\langle u,u^*\rangle_{L^2(SM)} = F_{m,\ell}(\la)\langle f,f^*\rangle_{L^2(M;\otimes^{m-2\ell} T^*M)},
$$
where $u$ is a resonant state at some resonance $\lambda$ corresponding to some $m,\ell$ in Theorem~\ref{t:main},
$u^*$ is a coresonant state (that is, an element of the kernel of the adjoint of $(X+\lambda)$),
$f,f^*$ are the corresponding eigenstates of the Laplacian, and $F_{m,\ell}(\la)$ is an explicit
function. Here $\langle u,u^*\rangle_{L^2}$ refers to the integral
$\int u\,\overline{u^*}$, which is well-defined despite the fact that
neither $u$ nor $u^*$ lie in $L^2$, see~\eqref{e:inner-product}.
This pairing formula is of independent interest as a step towards understanding
the high frequency behavior of resonant states and attempting to prove \emph{quantum ergodicity of resonant states}
in the present setting. Anantharaman--Zelditch~\cite{an-ze} obtained the pairing formula
in dimension 2 and studied concentration of Patterson--Sullivan distributions, which are directly related
to resonant states; see also~\cite{HHS}.

To motivate the study of Pollicott--Ruelle resonances, we also apply to our setting
the following \emph{resonance expansion} proved 
by Tsujii~\cite[Corollary~1.2]{Ts0} and Nonnenmacher--Zworski~\cite[Corollary~5]{NoZw2}:
%%%%%%%%%%%%%%%%%%%%%%%%%%%%%%%%%%%%%%%%%%%%%%%%%%%%%%%%%%%%%%%%%%%%%%%%%%%%%%%%
\begin{theo}
  \label{t:resexp} 
Fix $\varepsilon>0$. Then for $N$ large enough and $f,g$ in the Sobolev space $H^N(SM)$,
\begin{equation}
  \label{e:resexp}
\rho_{f,g}(t)=\int f\,d\mu\int g\,d\mu+\sum_{\lambda\in (-{n\over 2},0)} \sum_{k=1}^{\Mult_R(\lambda)} e^{\lambda t}
\langle f, u^*_{\lambda,k}\rangle_{L^2} \langle u_{\lambda,k}, g\rangle_{L^2}+\mathcal O_{f,g}(e^{-({n\over 2}-\varepsilon)t})
\end{equation}
where $u_{\lambda,k}$ is any basis of the space of resonant
states associated to $\lambda$ and $u^*_{\lambda,k}$ is the dual basis of the space of coresonant states
(so that $\sum_k u_{\lambda,k}\otimes_{L^2} u^*_{\lambda,k}$ is the spectral projector of $-X$
at $\lambda$).
\end{theo}
%%%%%%%%%%%%%%%%%%%%%%%%%%%%%%%%%%%%%%%%%%%%%%%%%%%%%%%%%%%%%%%%%%%%%%%%%%%%%%%%
Here we use Theorem~\ref{t:noalg} to see that there are no powers of $t$ in the expansion
and that there exists the dual basis of coresonant states to a basis of resonant states.

Combined with Theorem~\ref{t:main}, the expansion~\eqref{e:resexp}
in particular gives the optimal exponent in the decay of correlations
in terms of the small eigenvalues of the Laplacian; more precisely, the difference between
$\rho_{f,g}(t)$ and the product of the integrals of $f$ and $g$ is
$\mathcal O(e^{-\nu_0 t})$, where
$$
\nu_0=\min_{0\leq m<{n\over 2}}\min\{\nu+m \mid \nu\in (0,{\textstyle {n\over 2}-m}),\ \nu(n-\nu)+m\in \Spec^m(\Delta)\},
$$
or $\mathcal O(e^{-({n\over 2}-\varepsilon)t})$ for each $\varepsilon>0$ if the set above is empty.
Here $\Spec^m(\Delta)$ denotes the spectrum of the Laplacian
on trace-free divergence-free symmetric tensors of order $m$.
Using~\eqref{e:reparts}, we see that in fact one has $\nu \in [1,{n\over 2}-m)$
for $m>0$.

In order to go beyond the $\mathcal O(e^{-({n\over 2}-\varepsilon)t})$ remainder in~\eqref{e:resexp},
one would need to handle the infinitely many resonances in the $m=0$ band. This is thought
to be impossible in the general context of scattering theory, as the scattering resolvent can grow
exponentially near the bands; however, there exist cases such as Kerr--de Sitter black holes
where a resonance expansion with infinitely many terms holds, see~\cite{bo-ha,zeeman}.
The case of black holes is somewhat similar to the one considered here because in both cases
the trapped set is normally hyperbolic, see~\cite{kdsu} and~\cite{FaTs3}.
What is more, one can try to prove a resonance expansion with remainder
$\mathcal O(e^{-({n\over 2}+1-\varepsilon)t})$ where the sum over resonances in the first
band is replaced by $\langle(\Pi_0 f)\circ\varphi^{-t},g\rangle$ and $\Pi_0$ is the
projector onto the space of resonant states with $m=0$, having the microlocal structure
of a Fourier integral operator~-- see~\cite{kdsu} for a similar result in the context of
black holes.

%%%%%%%%%%%%%%%%%%%%%%%%%%%%%%%%%%%%%%%%%%%%%%%%%%%%%%%%%%%%%%%%%%%%%%%%%%%%%%%%
\smallsection{Previous results}
In the constant curvature setting in dimension $n+1=2$, the spectrum of the geodesic flow on $L^2$ was studied by Fomin--Gelfand using representation theory~\cite{FoGe}.
An exponential rate of mixing was proved by Ratner~\cite{Ra} and it was extended to higher dimensions by Moore~\cite{Mo}.
In variable negative curvature for surfaces
and more generally for Anosov flows with stable/unstable jointly non-integrable foliations, exponential decay of correlations was first shown by Dolgopyat~\cite{dolgopyat} and then by Liverani for contact flows~\cite{liverani}. The
work of Tsujii~\cite{Ts0,tsujii} established the asymptotic size of the resonance free strip and the
work of Nonnenmacher--Zworski~\cite{NoZw2} extended this result to general normally hyperbolic trapped sets.
Faure--Tsujii~\cite{FaTs1,FaTs2,FaTs3} established the band structure for general smooth contact Anosov flows
and proved an asymptotic for the number of resonances in the first band.

In dimension 2, the study of resonant states in the first band ($m=0$), that is distributions
which lie in the spectrum of $X$ and are annihilated by the horocylic vector field $U_-$
appears already in the works of Guillemin~\cite[Lecture~3]{guillemin}
and Zelditch~\cite{zelditch}, both using the representation theory of $\PSL(2;\mathbb R)$,
albeit without explicitly interpreting them as Pollicott--Ruelle resonant states. A more general
study of the elements in the kernel of $U_-$ was performed by Flaminio--Forni~\cite{FlFo}.

An alternative approach to resonances
involves the \emph{Selberg and Ruelle zeta functions}. 
The singularities (zeros and poles) of the Ruelle zeta function
correspond to
Pollicott--Ruelle resonances on differential forms
(see~\cite{fried,fried2}, \cite{GLP}, and~\cite{DyZw}), while the singularities of the Selberg zeta function correspond
to eigenvalues of the Laplacian.
The Ruelle and Selberg zeta functions
are closely related, see~\cite[Section 5.1, Figure~1]{leboeuf} and~\cite[(1.2)]{DyZw}
in dimension 2 and \cite{fried} and~\cite[Proposition~3.4]{BuOlBook} in arbitrary dimensions.
However, the Ruelle zeta function does not recover all resonances on functions, due to cancellations
with singularities coming from differential forms of different orders.
For example, \cite[Theorem~3.7]{Ju} describes the spectral singularities of the Ruelle zeta function for $n=3$ in terms
of the spectrum of the Laplacian on functions and 1-forms, which is much smaller than the
set obtained in Theorem~\ref{t:main}.

The book of Juhl~\cite{Ju}
and the works of Bunke--Olbrich~\cite{BuOlBook,BuOl1,BuOl2,BuOl3} study Ruelle and Selberg zeta functions
corresponding to various representations of the orthogonal group. They also consider general locally
symmetric spaces and address the question of what happens at the exceptional points (which in our case
are contained in $-{n\over 2}-{1\over 2}\mathbb N_0$), relating the behavior of the zeta
functions at these points to topological invariants. It is possible that the results~\cite{Ju,BuOlBook,BuOl1,BuOl2,BuOl3} together
with an appropriate representation theoretic calculation recover our description of resonances,
even though no explicit description featuring the spectrum
of the Laplacian on trace-free divergence-free symmetric tensors as in~\eqref{e:main-1}, \eqref{e:main-2}
seems to be available in the literature.
The direct spectral approach used in this paper,
unlike the zeta function techniques, gives an explicit relation between resonant
states and eigenstates of the Laplacian (see the remarks following~\eqref{e:actual-def})
and is a step towards a more quantitative understanding of decay of correlations.

An essential component of our work is the analysis of the correspondence
between eigenstates of the Laplacian on $\mathbb H^{n+1}$ and distributions on the conformal infinity $\mathbb S^n$.
In the scalar case, such a correspondence for hyperfunctions on $\mathbb S^n$ is due to Helgason~\cite{He1,He} (see also Minemura~\cite{Mi}); the correspondence between tempered eigenfunctions of $\Delta$ and distributions (instead of hyperfunctions) was shown by Oshima--Sekiguchi~\cite{OsSe} and Van Der Ban--Schlichtkrull~\cite{VdBSc} (see also Grellier--Otal~\cite{GrOt}).
Olbrich~\cite{olbrich} studied Poisson transforms on general homogeneous vector bundles,
which include the bundles of tensors used in the present paper.
The question of regularity of equivariant distributions on $\sph^n$ by certain Kleinian groups
of isometries of $\mathbb H^{n+1}$ (geometrically finite groups) is interesting since it tells the regularity of resonant states for the flow; precise regularity
was studied by Otal \cite{Ot} in the 2-dimensional co-compact case, Grellier--Otal \cite{GrOt} in higher dimensions, and Bunke--Olbrich \cite{BuOl2} for geometrically finite groups. In dimension $2$, the correspondence between the eigenfunctions of the Laplacian on the hyperbolic plane and distributions on the conformal boundary $\mathbb{S}^1$ appeared in Pollicott~\cite{Po} and Bunke--Olbrich \cite{BuOl1}, it is also an important tool in the theory developed by Bunke--Olbrich \cite{BuOl3} to study Selberg zeta functions on convex co-compact hyperbolic manifolds (see also the book of Juhl \cite{Ju} in the compact setting). These distributions on the conformal boundary $\sph^n$, of Patterson--Sullivan type, are also the central object of the recent work of Anantharaman--Zelditch \cite{an-ze, an-ze2} studying quantum ergodicity on hyperbolic compact surfaces;
a generalization to higher rank locally symmetric spaces was provided by Hansen--Hilgert--Schr\"oder~\cite{HHS}.

%%%%%%%%%%%%%%%%%%%%%%%%%%%%%%%%%%%%%%%%%%%%%%%%%%%%%%%%%%%%%%%%%%%%%%%%%%%%%%%%
%                                  SECTION 2                                   %
%%%%%%%%%%%%%%%%%%%%%%%%%%%%%%%%%%%%%%%%%%%%%%%%%%%%%%%%%%%%%%%%%%%%%%%%%%%%%%%%
\section{Outline and structure}
\label{s:overview}

In this section, we give the ideas of the proof of Theorem~\ref{t:main}, first in dimension 2 and then
in higher dimensions, and describe the structure of the paper. 

%%%%%%%%%%%%%%%%%%%%%%%%%%%%%%%%%%%%%%%%%%%%%%%%%%%%%%%%%%%%%%%%%%%%%%%%%%%%%%%%
\subsection{Dimension 2}
\label{s:o-dim2}

We start by using the following criterion (Lemma~\ref{l:criterion}):
$\lambda\in\mathbb C$ is a Pollicott--Ruelle resonance if and only if the
space
$$
\Res_X(\lambda):=\{u\in\mathcal D'(SM)\mid (X+\lambda) u=0,\
\WF(u)\subset E_u^*\}
$$
is nontrivial. Here $\mathcal D'(SM)$ denotes the space of \emph{distributions}
on $M$ (see~\cite{ho1}), $\WF(u)\subset T^*(SM)$ is the \emph{wavefront set}
of $u$ (see~\cite[Chapter 8]{ho1}), and $E_u^*\subset T^*(SM)$ is the dual
unstable foliation described in~\eqref{e:dual-decomposition}. It is more convenient
to use the condition $\WF(u)\subset E_u^*$ rather than $u\in\mathcal H^r$ because
this condition is invariant under differential operators of any order.

The key tools for the proof are the \emph{horocyclic vector fields} $U_\pm$
on $SM$, pictured on Figure~\ref{f:horo}(a) below. To define them,
we represent $M=\Gamma\backslash \mathbf H^2$, where $\mathbf H^2=\{z\in\mathbb C\mid\Im z>0\}$ is the hyperbolic
plane and $\Gamma\subset \PSL(2;\mathbb R)$ is a co-compact Fuchsian group of isometries acting
by M\"obius transformations.
(See Appendix~\ref{s:dim2} for the relation of the notation we use in dimension 2,
based on the half-plane model of the hyperbolic space, to the notation used elsewhere in the paper
which is based on the hyperboloid model.) Then $SM$ is covered by
$S\mathbf H^2$, which is isomorphic to the group $G:=\PSL(2;\mathbb R)$ 
by the map $\gamma\in G\mapsto (\gamma(i),d\gamma(i)\cdot i)$.
Consider the left invariant vector fields on $G$ corresponding to the following
elements of its Lie algebra:
\begin{equation}
  \label{e:u-pm-dim-2}
X=\begin{pmatrix} {1\over 2}&0\\0&-{1\over 2}\end{pmatrix},\quad
U_+=\begin{pmatrix} 0&1\\0&0\end{pmatrix},\quad
U_-=\begin{pmatrix} 0&0\\1&0\end{pmatrix},
\end{equation}
then $X,U_\pm$ descend to vector fields on $SM$, with $X$ becoming the generator
of the geodesic flow. We have the commutation relations
\begin{equation}
  \label{e:comm-rel-dim-2}
[X,U_\pm]=\pm U_\pm,\quad
[U_+,U_-]=2X.
\end{equation}
For each $\lambda$ and $m\in\mathbb N_0$, define the spaces
$$
V_m(\lambda):=\{u\in \mathcal D'(SM)\mid (X+\lambda) u=0,\
U_-^m u=0,\
\WF(u)\subset E_u^*\},
$$
and put
$$
\Res^0_X(\lambda):=V_1(\lambda).
$$
By~\eqref{e:comm-rel-dim-2},
$U_-^m(\Res_X(\lambda))\subset \Res_X(\lambda+m)$.
Since there are no Pollicott--Ruelle resonances in the right half-plane, we conclude
that
$$
\Res_X(\lambda)=V_m(\lambda)\quad\text{for }m>-\Re\lambda.
$$
We now use the diagram (writing $\Id=U_\pm ^0$, $U_\pm=U_\pm^1$
for uniformity of notation)
\begin{center}
\includegraphics{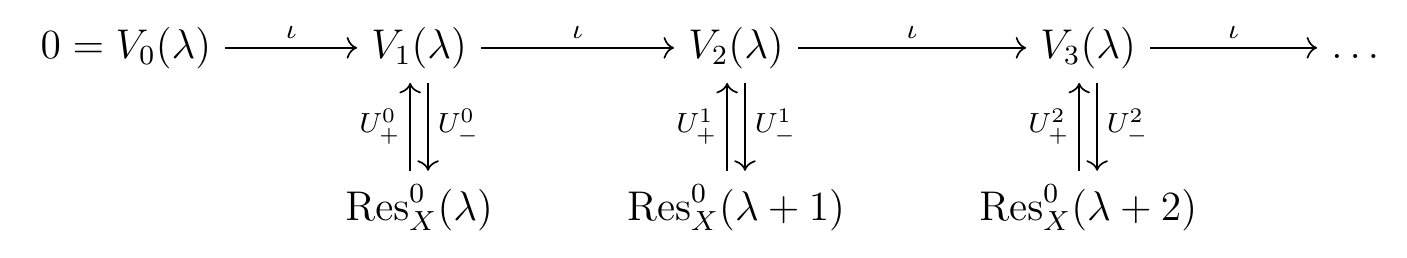}
%\beginpgfgraphicnamed{diagram2}
%\begin{tikzcd}
%0=V_0(\lambda)\arrow{r}{\iota}
%&V_1(\lambda)\arrow{r}{\iota}\arrow[xshift=0.5ex]{d}{U_-^0}
%&V_2(\lambda)\arrow{r}{\iota}\arrow[xshift=0.5ex]{d}{U_-^1}
%&V_3(\lambda)\arrow{r}{\iota}\arrow[xshift=0.5ex]{d}{U_-^2}
%&\dots\\
%&\Res^0_X(\lambda)\arrow[xshift=-0.5ex]{u}{U_+^0}
%&\Res^0_X(\lambda+1)\arrow[xshift=-0.5ex]{u}{U_+^1}
%&\Res^0_X(\lambda+2)\arrow[xshift=-0.5ex]{u}{U_+^2}
%&
%\end{tikzcd}
%\endpgfgraphicnamed
\end{center}
where $\iota$ denotes the inclusion maps and unless $\lambda\in -1-{1\over 2}\mathbb N_0$, we have
$$
V_{m+1}(\lambda)=V_m(\lambda)\oplus U_+^m(\Res^0_X(\lambda+m)),
$$
and $U_+^m$ is one-to-one on $\Res^0_X(\lambda+m)$; indeed, using~\eqref{e:comm-rel-dim-2} we calculate
$$
U_-^m U_+^m= m!\bigg(\prod_{j=1}^m (2\lambda+m+j)\bigg)\Id\quad\text{on }\Res^0_X(\lambda+m)
$$
and the coefficient above is nonzero when $\lambda\notin -1-{1\over 2}\mathbb N_0$. We then see that
$$
\Res_X(\lambda)=\bigoplus_{m\geq 0} U_+^m(\Res^0_X(\lambda+m)).
$$
It remains to describe the space of resonant states in the first band,
$$
\Res^0_X(\lambda)=\{u\in\mathcal D'(SM)\mid (X+\lambda) u=0,\ U_-u=0,\
\WF(u)\subset E_u^*\}.
$$
We can remove the condition $\WF(u)\subset E_u^*$ as it follows from the other two,
see the remark following Lemma~\ref{l:reduced}. We claim that the pushforward map
\begin{equation}
  \label{e:pushy}
u\in\Res^0_X(\lambda)\mapsto  f(x):=\int_{S_x M} u(x,\xi)\,dS(\xi)
\end{equation}
is an isomorphism from $\Res^0_X(\lambda)$ onto $\Eig(-\lambda(1+\lambda))$,
where $\Eig(\sigma)=\{u\in \CI(M)\mid \Delta u=\sigma u\}$; this would finish the proof.
In other words, the eigenstate of the Laplacian corresponding to $u$ is obtained by integrating
$u$ over the fibers of $SM$.

To show that~\eqref{e:pushy} is an isomorphism, we reduce
the elements of $\Res^0_X(\lambda)$ to the conformal boundary $\mathbb S^1$
of the ball model $\mathbb B^2$ of the hyperbolic space as follows:
\begin{equation}
  \label{e:poppy}
\Res^0_X(\lambda)=\{P(y,B_-(y,\xi))^\lambda w(B_-(y,\xi))\mid w\in\Bd(\lambda)\},
\end{equation}
where $P(y,\nu)$ is the Poisson kernel:
$P(y,\nu)={1-|y|^2\over |y-\nu|^2}$, $y\in\mathbb B^2$, $\nu\in\mathbb S^1$;
$B_-:S\mathbb B^2\to\mathbb S^1$ maps $(y,\xi)$ to the limiting point of the geodesic
$\varphi_t(y,\xi)$ as $t\to-\infty$, see Figure~\ref{f:horo}(a); and
$\Bd(\lambda)\subset\mathcal D'(\mathbb S^1)$ is the space of distributions
satisfying certain equivariance property with respect to $\Gamma$. Here we lifted
$\Res_X^0(\lambda)$ to distributions on $S\mathbb H^2$
and used the fact that the map $B_-$ is invariant under both $X$ and $U_-$;
see Lemma~\ref{l:reduced} for details.

It remains to show that the map $w\mapsto f$ defined via~\eqref{e:pushy} and~\eqref{e:poppy}
is an isomorphism from $\Bd(\lambda)$ to $\Eig(-\lambda(1+\lambda))$. This map is given by
(see Lemma~\ref{l:poisson-1})
\begin{equation}
  \label{e:peasy}
f(y)=\mathscr P_\lambda^- w(y):=\int_{\mathbb S^1}P(y,\nu)^{1+\lambda}w(\nu)\,dS(\nu)
\end{equation}
and is the Poisson operator for the (scalar) Laplacian
corresponding to the eigenvalue $s(1-s)$, $s=1+\lambda$. This Poisson operator
is known to be an isomorphism for $\lambda\notin -1-\mathbb N$, see the remark following Theorem~\ref{t:laplacian}
in Section~\ref{s:rr-2}, finishing the proof.

%%%%%%%%%%%%%%%%%%%%%%%%%%%%%%%%%%%%%%%%%%%%%%%%%%%%%%%%%%%%%%%%%%%%%%%%%%%%%%%%
\subsection{Higher dimensions}

In higher dimensions, the situation is made considerably more difficult by the fact
we can no longer define the vector fields $U_\pm$ on $SM$. To get around this problem,
we remark that in dimension $2$, $U_-u$ is the derivative of $u$ along a certain canonical
vector in the one-dimensional \emph{unstable foliation} $E_u\subset T(SM)$
and similarly $U_+u$ is the derivative along an element of the stable foliation $E_s$.
(See Section~\ref{s:horocycl}.) In dimension $n+1>2$, the foliations $E_u,E_s$
are $n$-dimensional and one cannot trivialize them. However, each of these foliations
is canonically parametrized by the following vector bundle $\mathcal E$ over $SM$:
$$
\mathcal E(x,\xi)=\{\eta\in T_x M\mid \eta\perp \xi\},\quad
(x,\xi)\in SM.
$$
This makes it possible to define \emph{horocyclic operators}
$$
\mathcal U_\pm^m:\mathcal D'(SM)\to\mathcal D'(SM;\otimes^m_S\mathcal E^*),
$$
where $\otimes^m_S$ stands for the $m$-th symmetric tensor power, and we have the diagram
\begin{center}
\includegraphics{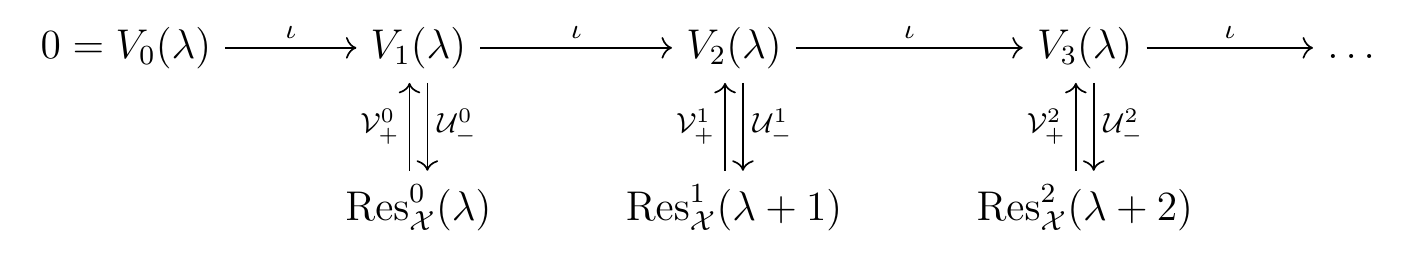}
%\beginpgfgraphicnamed{diagram3}
%\begin{tikzcd}
%0=V_0(\lambda)\arrow{r}{\iota}
%&V_1(\lambda)\arrow{r}{\iota}\arrow[xshift=0.5ex]{d}{\mathcal U_-^0}
%&V_2(\lambda)\arrow{r}{\iota}\arrow[xshift=0.5ex]{d}{\mathcal U_-^1}
%&V_3(\lambda)\arrow{r}{\iota}\arrow[xshift=0.5ex]{d}{\mathcal U_-^2}
%&\dots\\
%&\Res^0_{\mathcal X}(\lambda)\arrow[xshift=-0.5ex]{u}{\mathcal V_+^0}
%&\Res^1_{\mathcal X}(\lambda+1)\arrow[xshift=-0.5ex]{u}{\mathcal V_+^1}
%&\Res^2_{\mathcal X}(\lambda+2)\arrow[xshift=-0.5ex]{u}{\mathcal V_+^2}
%&
%\end{tikzcd}
%\endpgfgraphicnamed
\end{center}
where $\mathcal V_+^m=(-1)^m(\mathcal U_+^m)^*$ and
we put for a certain extension $\mathcal X$ of $X$ to $\otimes^m_S\mathcal E^*$
$$
\begin{gathered}
V_m(\lambda):=\{u\in\mathcal D'(SM)\mid (X+\lambda) u=0,\ \mathcal U_-^m u=0,\
\WF(u)\subset E_u^*\},\\
\Res^m_{\mathcal X}(\lambda):=\{v\in \mathcal D'(SM;\otimes^m_S\mathcal E^*)\mid (\mathcal X+\lambda)v=0,\
\mathcal U_- v=0,\
\WF(v)\subset E_u^*\}.
\end{gathered}
$$
Similarly to dimension 2, we reduce the problem to understanding the spaces
$\Res^m_{\mathcal X}(\lambda)$, and an operator similar to~\eqref{e:pushy} maps
these spaces to eigenspaces of the Laplacian on divergence-free symmetric tensors.
However, to make this statement precise, we have to further decompose
$\Res^m_{\mathcal X}(\lambda)$ into terms coming from traceless
tensors of degrees $m,m-2,m-4,\dots$, explaining the appearance of the parameter $\ell$
in the theorem. (Here the trace of a symmetric tensor of order $m$ is the result of contracting
two of its indices with the metric, yielding a tensor of order $m-2$.) The procedure of reducing
elements of $\Res^m_{\mathcal X}(\lambda)$ to the conformal boundary $\mathbb S^n$
is also made more difficult since the boundary distributions $w$
are now sections of $\otimes^m_S(T^*\mathbb S^n)$.

A significant part of the paper
is dedicated to proving that the higher-dimensional analog of~\eqref{e:peasy}
on symmetric tensors is indeed an isomorphism between appropriate spaces.
To show that the Poisson operator $\mathscr P^-_\lambda$ is injective,
we prove a weak
expansion of $f(y)\in \CI(\mathbb B^{n+1})$ in powers of $1-|y|$ as $y\in\mathbb B^{n+1}$
approaches the conformal boundary $\mathbb S^n$; since $w$ appears as the coefficient in one
of the terms of the expansion, $\mathscr P^-_\lambda w=0$ implies $w=0$.
To show the surjectivity of $\mathscr P^-_\lambda$, we prove that the lift to $\mathbb H^{n+1}$
of every trace-free divergence-free eigenstate $f$ of the Laplacian admits
a weak expansion at the conformal boundary (this requires
a fine analysis of the Laplacian and divergence operators on symmetric tensors);
putting $w$ to be the coefficient next to
one of the terms of this expansion, we can prove that $f=\mathscr P^-_\lambda w$.

%%%%%%%%%%%%%%%%%%%%%%%%%%%%%%%%%%%%%%%%%%%%%%%%%%%%%%%%%%%%%%%%%%%%%%%%%%%%%%%%
\subsection{Structure of the paper}

\begin{itemize}
\item In Section~\ref{s:geomhyp}, we study in detail the geometry of the
hyperbolic space $\mathbb H^{n+1}$, which is the covering space of $M$;
\item in Section~\ref{s:horror}, we introduce and study the horocyclic operators;
\item in Section~\ref{s:rr}, we prove Theorems~\ref{t:main} and~\ref{t:noalg}, modulo
properties of the Poisson operator;
\item in Sections~\ref{s:laplacian} and~\ref{s:poisson-is}, we show the injectivity
and the surjectivity of the Poisson operator;
\item Appendix~\ref{s:technical} contains several technical lemmas;
\item Appendix~\ref{s:dim2} shows how the discussion of Section~\ref{s:o-dim2}
fits into the framework of the remainder of the paper;
\item Appendix~\ref{s:weylie} shows a Weyl law for divergence free symmetric tensors and
relates the $m=1$ case to the Hodge Laplacian.
\end{itemize}

%%%%%%%%%%%%%%%%%%%%%%%%%%%%%%%%%%%%%%%%%%%%%%%%%%%%%%%%%%%%%%%%%%%%%%%%%%%%%%%%
%                                  SECTION 3                                   %
%%%%%%%%%%%%%%%%%%%%%%%%%%%%%%%%%%%%%%%%%%%%%%%%%%%%%%%%%%%%%%%%%%%%%%%%%%%%%%%%
\section{Geometry of the hyperbolic space}
\label{s:geomhyp}

In this section, we review the structure of the hyperbolic space and its geodesic flow and introduce
various objects to be used later, including:
\begin{itemize}
\item the isometry group $G$ of the hyperbolic space and its Lie algebra, including the
horocyclic vector fields $U^\pm_i$ (Section~\ref{s:group});
\item the stable/unstable foliations $E_s,E_u$ (Section~\ref{s:geodesic});
\item the conformal compactification of the hyperbolic space, the maps
$B_\pm$, the coefficients $\Phi_\pm$, and the Poisson kernel (Section~\ref{s:confinf});
\item parallel transport to conformal infinity and the maps $\mathcal A_\pm$ (Section~\ref{s:E}).
\end{itemize}

%%%%%%%%%%%%%%%%%%%%%%%%%%%%%%%%%%%%%%%%%%%%%%%%%%%%%%%%%%%%%%%%%%%%%%%%%%%%%%%%
\subsection{Models of the hyperbolic space}
\label{s:modeles}
 
Consider the Minkowski space $\mathbb R^{1,n+1}$ with the Lorentzian metric
$$
g_M=dx_0^2-\sum_{j=1}^{n+1}dx_j^2.
$$
The corresponding scalar product is denoted $\langle\cdot,\cdot\rangle_M$.
We denote by $e_0,\dots,e_{n+1}$ the canonical basis of $\mathbb R^{1,n+1}$.

The hyperbolic space of dimension $n+1$ is defined to be one sheet of
the two-sheeted hyperboloid
$$
\mathbb H^{n+1}:=\{x\in\mathbb R^{1,n+1}\mid \langle x,x\rangle_M=1,\ x_0>0\}
$$
equipped with the Riemannian metric
$$
g_H:=-g_M|_{T\mathbb H^{n+1}}.
$$
We denote the unit tangent bundle of $\mathbb H^{n+1}$ by
\begin{equation}
  \label{e:shn}
S\mathbb H^{n+1}:=\{(x,\xi)\mid x\in\mathbb H^{n+1},\
\xi\in\mathbb R^{1,n+1},\
\langle \xi,\xi\rangle_M=-1,\
\langle x,\xi\rangle_M=0\}.
\end{equation}
Another model of the hyperbolic space is the unit ball $\mathbb{B}^{n+1}=\{y\in \rr^{n+1}; |y|<1\}$,
which is identified with $\mathbb H^{n+1}\subset \mathbb R^{1,n+1}$ via
the map (here $x=(x_0,x')\in\mathbb R\times\mathbb R^{n+1}$)
\begin{equation}
  \label{defpsi}
\psi : \hh^{n+1} \to \bb^{n+1}, \quad  \psi(x)=\frac{x'}{x_0+1},\quad
\psi^{-1}(y)={1\over 1-|y|^2}(1+|y|^2,2y).
\end{equation}
and the metric $g_H$ pulls back to the following metric on $\mathbb{B}^{n+1}$:
\begin{equation}
  \label{defgH}
(\psi^{-1})^*g_H=\frac{4\, dy^2}{(1-|y|^2)^2}.
\end{equation}
We will also use the upper half-space model $\mathbb U^{n+1}=\mathbb R_{z_0}^+\times\mathbb R_{z}^n$
with the metric
\begin{equation}
  \label{e:umetric}
(\psi^{-1}\psi_1^{-1})^*g_H={dz_0^2+dz^2\over z_0^2},
\end{equation}
where the diffeomorphism $\psi_1:\mathbb B^{n+1}\to\mathbb U^{n+1}$ is given by
(here $y=(y_1,y')\in\mathbb R\times\mathbb R^n$)
\begin{equation}
  \label{e:udiffeo}
\psi_1(y_1,y')={(1-|y|^2,2y')\over 1+|y|^2-2y_1},\quad
\psi_1^{-1}(z_0,z)={(z_0^2+|z|^2-1,2z)\over (1+z_0)^2+|z|^2}.
\end{equation}

%%%%%%%%%%%%%%%%%%%%%%%%%%%%%%%%%%%%%%%%%%%%%%%%%%%%%%%%%%%%%%%%%%%%%%%%%%%%%%%%
\subsection{Isometry group}
\label{s:group}

We consider the group
$$
G=\PSO(1,n+1)\subset \SL(n+2;\mathbb R)
$$
of all linear transformations of $\mathbb R^{1,n+1}$ preserving the Minkowski metric,
the orientation, and the sign of $x_0$ on timelike vectors.
For $x\in\mathbb R^{1,n+1}$ and $\gamma\in G$, denote by $\gamma\cdot x$
the result of multiplying $x$ by the matrix $\gamma$.
The group $G$ is exactly the group of orientation preserving isometries of $\hh^{n+1}$;
under the identification~\eqref{defpsi}, it corresponds to the group of direct M\"obius transformations of $\rr^{n+1}$ preserving the unit ball. 

The Lie algebra of $G$ is spanned by the matrices
\begin{equation}\label{liealgebraofG} 
X= E_{0,1}+E_{1,0}, \quad A_k= E_{0,k}+E_{k,0}, \quad R_{i,j}=E_{i,j}-E_{j,i}
\end{equation} 
for $i,j\geq 1$ and $k\geq 2$, where $E_{i,j}$ is the elementary matrix if $0\leq i,j\leq n+1$
(that is, $E_{i,j}e_k=\delta_{jk}e_i$). Denote for $i=1,\dots, n$
\begin{equation}\label{liealgebraofG2} 
U_i^+:= -A_{i+1}-R_{1,i+1}, \quad U^-_i:=-A_{i+1}+R_{1,i+1}
\end{equation}
and observe that $X,U_i^+,U_i^-,R_{i+1,j+1}$ (for $1\leq i<j\leq n$) also form a basis.
Henceforth we identify elements of the Lie algebra of $G$ with left invariant vector fields
on $G$.

We have the commutator relations (for $1\leq i,j,k\leq n$ and $i\neq j$)
\begin{equation}
  \label{e:comm-rel}
\begin{gathered}\relax
[X,U_i^\pm]=\pm U_i^\pm,\quad
[U_i^\pm,U_j^\pm]=0,\quad
[U_i^+,U_i^-]=2X,\quad
[U_i^\pm,U_j^\mp]=2R_{i+1,j+1},\\
[R_{i+1,j+1},X]=0,\quad [R_{i+1,j+1},U_k^\pm]=\delta_{jk}U_i^\pm-\delta_{ik}U_j^\pm.
\end{gathered}
\end{equation}
The Lie algebra elements $U_i^\pm$ are very important in our argument since they generate
horocylic flows, see Section~\ref{s:horocycl}.
The flows of $U^1_\pm$ in the case $n=1$ are shown in Figure~\ref{f:horo}(a);
for $n>1$, the flows of $U^j_\pm$ do not descend to $S\mathbb H^{n+1}$.

The group $G$ acts on $\mathbb H^{n+1}$ transitively, with the isotropy group
of $e_0\in\mathbb H^{n+1}$ isomorphic to $\SO(n+1)$. It also acts
transitively on the unit tangent bundle $S\mathbb H^{n+1}$,
by the rule $\gamma.(x,\xi)=(\gamma\cdot x,\gamma\cdot\xi)$,
with the isotropy group of $(e_0,e_1)\in S\mathbb H^{n+1}$ being
\begin{equation}
  \label{e:defH}
H=\{\gamma\in G\mid \gamma\cdot e_0=e_0,\ \gamma\cdot e_1=e_1\}\simeq \SO(n).
\end{equation}
Note that $H$
is the connected Lie subgroup of $G$ with Lie algebra spanned by
$R_{i+1,j+1}$ for $1\leq i,j\leq n$.
We can then write $S\mathbb H^{n+1}\simeq G/H$, where the projection
$\pi_S:G\to S\mathbb H^{n+1}$ is given by
\begin{equation}
  \label{e:piS}
\pi_S(\gamma)=(\gamma\cdot e_0,\gamma\cdot e_1)\in S\mathbb H^{n+1},\quad
\gamma\in G.
\end{equation}

%%%%%%%%%%%%%%%%%%%%%%%%%%%%%%%%%%%%%%%%%%%%%%%%%%%%%%%%%%%%%%%%%%%%%%%%%%%%%%%%
\begin{figure}
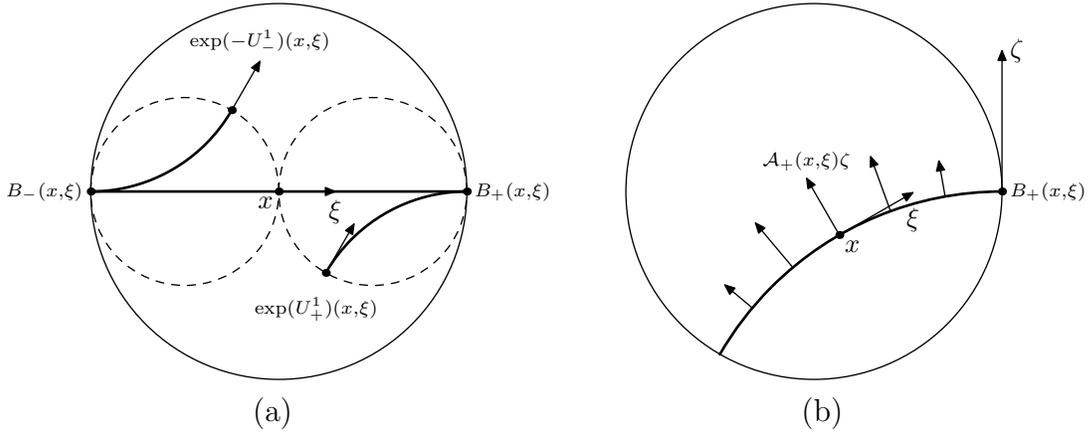

\includegraphics{rrh.5}
\qquad
\includegraphics{rrh.6}
\hbox to\hsize{\hss\quad (a)\hss\hss (b)\quad\hss}
\caption{(a) The horocyclic flows $\exp(\pm U^\pm_1)$ in dimension $n+1=2$,
pulled back to the ball model by the map $\psi$ from~\eqref{defpsi}. The thick lines
are geodesics and the dashed lines are horocycles.
(b) The map $\mathcal A_+$ and the parallel transport of an element of $\mathcal E$
along a geodesic.}
\label{f:horo}
\end{figure}
%%%%%%%%%%%%%%%%%%%%%%%%%%%%%%%%%%%%%%%%%%%%%%%%%%%%%%%%%%%%%%%%%%%%%%%%%%%%%%%%

%%%%%%%%%%%%%%%%%%%%%%%%%%%%%%%%%%%%%%%%%%%%%%%%%%%%%%%%%%%%%%%%%%%%%%%%%%%%%%%%
\subsection{Geodesic flow}
\label{s:geodesic}

The geodesic flow,
$$
\varphi_t:S\mathbb H^{n+1}\to S\mathbb H^{n+1},\quad
t\in\mathbb R,
$$
is given in the parametrization~\eqref{e:shn} by
\begin{equation}\label{formulaflow}
\varphi_t(x,\xi)=(x\cosh t+\xi\sinh t,x\sinh t+\xi\cosh t).
\end{equation} 
We note that, with the projection $\pi_S:G\to S\mathbb H^{n+1}$
defined in~\eqref{e:piS},
$$
\varphi_t(\pi_S(\gamma))=\pi_S(\gamma\exp(tX)),
$$
where $X$ is defined in~\eqref{liealgebraofG}. This means that the generator
of the geodesic flow can be obtained by pushing forward the left invariant field on $G$
generated by $X$ by the map $\pi_S$ (which is possible since $X$ is invariant
under right multiplications by elements of the subgroup $H$ defined in~\eqref{e:defH}). By abuse of notation,
we then denote by $X$ also the generator of the geodesic flow on $S\mathbb H^{n+1}$:
\begin{equation}
  \label{e:defX}
X=\xi\cdot \partial_x+x\cdot \partial_\xi.
\end{equation}
We now provide the stable/unstable decomposition for the geodesic flow, demonstrating that it is hyperbolic
(and thus the flow on a compact quotient by a discrete group will be Anosov).
For $\rho=(x,\xi)\in S\mathbb H^{n+1}$, the tangent space $T_\rho(S\mathbb H^{n+1})$ can be written as
$$
T_\rho(S\mathbb H^{n+1})=\{(v_x,v_\xi)\in(\mathbb R^{1,n+1})^2\mid
\langle x,v_x\rangle_M
=\langle \xi,v_\xi\rangle_M
=\langle x,v_\xi\rangle_M
+\langle \xi,v_x\rangle_M=0
\}.
$$
The differential of the geodesic flow acts by
$$
d\varphi_t(\rho)\cdot (v_x,v_\xi)=(v_x\cosh t+v_\xi\sinh t,v_x\sinh t+v_\xi\cosh t).
$$
We have $T_\rho(S\mathbb H^{n+1})=E^0(\rho)\oplus \widetilde T_\rho(S\mathbb H^{n+1})$,
where $E^0(\rho):=\mathbb R X$ is the flow direction and
$$
\widetilde T_\rho(S\mathbb H^{n+1})
=\{(v_x,v_\xi)\in(\mathbb R^{1,n+1})^2\mid
\langle x,v_x\rangle_M
=\langle x,v_\xi\rangle_M
=\langle \xi,v_x\rangle_M
=\langle \xi,v_\xi\rangle_M=0\},
$$
and this splitting is invariant under $d\varphi_t$.
A natural norm on $\widetilde T_\rho(S\mathbb H^{n+1})$ is given by the formula
\begin{equation}
  \label{e:natnorm}
|(v_x,v_\xi)|^2:=-\langle v_x,v_x\rangle_M-\langle v_\xi,v_\xi\rangle_M,
\end{equation}
using the fact that $v_x,v_\xi$ are Minkowski orthogonal to the timelike vector $x$
and thus must be spacelike or zero. Note that this norm is invariant under the action of $G$.

We now define the \emph{stable/unstable decomposition}
$\widetilde T_\rho (S\mathbb H^{n+1})
=E_s(\rho)\oplus E_u(\rho)$, where
\begin{equation}
  \label{e:stable-unstable}
\begin{gathered}
E_s(\rho):=\{(v,-v)\mid \langle x,v\rangle_M=\langle\xi,v\rangle_M=0\},\\
E_u(\rho):=\{(v,v)\mid \langle x,v\rangle_M=\langle\xi,v\rangle_M=0\}.
\end{gathered}
\end{equation}
Then $T_\rho(S\mathbb H^{n+1})=E_0(\rho)\oplus E_s(\rho)\oplus E_u(\rho)$, this splitting is invariant under
$\varphi_t$ and under the action of $G$, and, using the norm from~\eqref{e:natnorm},
$$
|d\varphi_t(\rho)\cdot w|=e^{-t}|w|,\ w\in E_s(\rho);\quad
|d\varphi_t(\rho)\cdot w|=e^{t}|w|,\ w\in E_u(\rho).
$$
Finally, we remark that the vector subbundles $E_s$ and $E_u$ are spanned by the left-invariant
vector fields $U^+_1,\dots, U^+_n$ and $U^-_1,\dots,U^-_n$ from~\eqref{liealgebraofG2}
in the sense that
$$
\pi_S^* E_s=\Span(U^+_1,\dots,U^+_n)\oplus\mathfrak h,\quad
\pi_S^* E_u=\Span(U^-_1,\dots,U^-_n)\oplus\mathfrak h.
$$
Here $\pi_S^* E_s:=\{(\gamma,w)\in TG\mid (\pi_S(\gamma),d\pi_S(\gamma)\cdot w)\in E_s\}$
and $\pi_S^* E_u$ is defined similarly;
$\mathfrak h$ is the left translation of the Lie algebra of $H$,
or equivalently the kernel of $d\pi_S$.
Note that while the individual vector fields $U^\pm_1,\dots,U^\pm_n$ are not invariant under
right multiplications by elements of $H$ in dimensions $n+1>2$ (and thus do not descend to
vector fields on $S\mathbb H^{n+1}$ by the map $\pi_S$), their spans are invariant under
$H$ by~\eqref{e:comm-rel}.

The dual decomposition, used in the construction of Pollicott--Ruelle resonances, is
\begin{equation}
  \label{e:dual-decomposition}
T^*_\rho (S\mathbb H^{n+1})=E_0^*(\rho)\oplus E_s^*(\rho)\oplus E_u^*(\rho),
\end{equation}
where $E_0^*(\rho),E_s^*(\rho),E_u^*(\rho)$ are dual to
$E_0(\rho),E_u(\rho),E_s(\rho)$ in the original decomposition
(that is, for instance $E_s^*(\rho)$ consists of all covectors annihilating
$E_0(\rho)\oplus E_s(\rho)$).
The switching of the roles of $E_s$ and $E_u$ is due to the fact
that the flow on the cotangent bundle is $(d\varphi_t^{-1})^*$.

%%%%%%%%%%%%%%%%%%%%%%%%%%%%%%%%%%%%%%%%%%%%%%%%%%%%%%%%%%%%%%%%%%%%%%%%%%%%%%%%
\subsection{Conformal infinity}
\label{s:confinf}

The metric~\eqref{defgH} in the ball model $\mathbb B^{n+1}$ is conformally compact;
namely the metric $(1-|y|^2)^2(\psi^{-1})^*g_H$ continues smoothly to the closure
$\overline {\mathbb B^{n+1}}$, which we call the \emph{conformal compactification}
of $\mathbb H^{n+1}$; note that $\mathbb H^{n+1}$ embeds into the interior
of $\overline {\mathbb B^{n+1}}$ by the map~\eqref{defpsi}.
The boundary $\mathbb S^n=\partial\overline{\mathbb B^{n+1}}$, endowed with the standard
metric on the sphere, is called \emph{conformal infinity}.
On the hyperboloid model, it is natural to associate to a point at conformal infinity
$\nu\in\mathbb S^n$ the lightlike ray $\{(s,s\nu)\mid s>0\}\subset\mathbb R^{1,n+1}$;
note that this ray is asymptotic to the curve
$\{(\sqrt{1+s^2},s\nu)\mid s>0\}\subset\mathbb H^{n+1}$, which
converges to $\nu$ in $\overline{\mathbb B^{n+1}}$.

Take $(x,\xi)\in S\mathbb H^{n+1}$. Then $\langle x\pm\xi,x\pm\xi\rangle_M=0$
and $x_0\pm\xi_0>0$, therefore we can write
$$
x\pm \xi=\Phi_\pm(x,\xi)(1,B_\pm(x,\xi)),$$ 
for some maps
\begin{equation}\label{PhiandB}
\Phi_\pm:S\hh^{n+1}\to \mathbb R^+,\quad
B_\pm:S\hh^{n+1}\to\mathbb S^n.
\end{equation}
Then $B_\pm(x,\xi)$ is the limit as $t\to\pm\infty$ of the $x$-projection of the geodesic $\varphi_t(x,\xi)$
in $\overline{\mathbb{B}^{n+1}}$:
$$
B_\pm(x,\xi)=\lim_{t\to \pm\infty}\pi(\varphi_t(x,\xi)),\quad
\pi:S\mathbb H^{n+1}\to\mathbb H^{n+1}.
$$
Note that this implies that for $X$ defined in~\eqref{e:defX},
$dB_\pm\cdot X=0$ since $B_\pm(\varphi_s(x,\xi))=B_\pm(x,\xi)$ for all $s\in\rr$. 
Moreover, since 
$\Phi_\pm(\varphi_t(x,\xi))=e^{\pm t}(x_0+\xi_0)=e^{\pm t}\Phi_\pm(x,\xi)$ from \eqref{formulaflow}, we find
\begin{equation}\label{invarofPhi}
X\Phi_\pm=\pm \Phi_\pm.
\end{equation}
For $(x,\nu)\in\mathbb H^{n+1}\times\mathbb S^n$ (in the hyperboloid model), define the function
\begin{equation}\label{defofP}
P(x,\nu)=(x_0-x'\cdot\nu)^{-1}=(\langle x,(1,\nu)\rangle_M)^{-1}, \quad \textrm{ if }x=(x_0,x')\in \hh^{n+1}.
\end{equation}
Note that $P(x,\nu)>0$ everywhere, and in the Poincar\'e ball model $\mathbb{B}^{n+1}$, we have 
\begin{equation}
\label{e:poisson-y}
P(\psi^{-1}(y),\nu)=\frac{1-|y|^2}{|y-\nu|^2}, \quad y\in \mathbb{B}^{n+1}
\end{equation}
which is the usual Poisson kernel. Here $\psi$ is defined in~\eqref{defpsi}.

For $(x,\nu)\in\mathbb H^{n+1}\times\mathbb S^n$, there exist
unique $\xi_\pm\in S_x\mathbb H^{n+1}$ such that $B_\pm(x,\xi_\pm)=\nu$: these are given by 
\begin{equation}\label{defxi+}
\xi_\pm(x,\nu)=\mp x\pm P(x,\nu)(1,\nu)
\end{equation}
and the following formula holds
\begin{equation}\label{Phiofxi}
\Phi_\pm(x,\xi_\pm(x,\nu))=P(x,\nu).
\end{equation}
Notice that the equation $B_\pm(x,\xi_\pm(x,\nu))=\nu$ implies that $B_\pm$ are submersions. 
The map $\nu\to\xi_\pm(x,\nu)$ is conformal with the standard choice of metrics
on $\mathbb S^n$ and $S_x\mathbb H^{n+1}$; in fact,
for $\zeta_1,\zeta_2\in T_\nu \mathbb S^n$, 
\begin{equation}\label{conformaldxi}
\langle \partial_\nu \xi_\pm(x,\nu)\cdot\zeta_1,\partial_\nu\xi_\pm(x,\nu)\cdot\zeta_2\rangle_{M}
=- P(x,\nu)^{2} \langle \zeta_1,\zeta_2\rangle_{\mathbb R^{n+1}}.
\end{equation}

Using that $\langle x+\xi,x-\xi\rangle_M=2$, we see that
\begin{equation}
  \label{e:relation}
\Phi_+(x,\xi)\Phi_-(x,\xi)(1-B_+(x,\xi)\cdot B_-(x,\xi))=2.
\end{equation}
One can parametrize $S\mathbb H^{n+1}$ by
\begin{equation}
  \label{e:otherpar}
(\nu_-,\nu_+,s)=\bigg(B_-(x,\xi),B_+(x,\xi),{1\over 2}\log{\Phi_+(x,\xi)\over \Phi_-(x,\xi)}\bigg)\in (\mathbb S^n\times\mathbb S^n)_\Delta\times\mathbb R,
\end{equation}
where $(\mathbb S^n\times\mathbb S^n)_\Delta$ is $\mathbb S^n\times\mathbb S^n$ minus the diagonal.
In fact, the geodesic $\gamma(t)=\varphi_t(x,\xi)$
goes from $\nu_-$ to $\nu_+$ in $\overline{\mathbb B^{n+1}}$ and
$\gamma(-s)$ is the point of $\gamma$ closest to $e_0\in\mathbb H^{n+1}$ (corresponding to $0\in\mathbb B^{n+1}$).
In the parametrization~\eqref{e:otherpar}, the geodesic flow $\varphi_t$ is simply 
$$
(\nu_-,\nu_+,s)\mapsto (\nu_-,\nu_+,s+t).
$$
%The inverse to~\eqref{e:otherpar} is
%$$
%(x,\xi)={e^s(1,\nu_+,1,\nu_+)+e^{-s}(1,\nu_-,-1,-\nu_-)\over |\nu_+-\nu_-|},\quad
%|\nu_+-\nu_-|=\sqrt{2(1-\nu_+\cdot\nu_-)}.
%$$
%The Jacobian of the change of variables $(\nu_-,\nu_+,s)\mapsto (x,\xi)$ is
%\begin{equation}
%  \label{e:otherpar-J}
%(1-\nu_+\cdot\nu_-)^{-n}=2^{n}|\nu_+-\nu_-|^{-2n}.
%\end{equation}
%We also define the map $\Psi:\mathbb H^{n+1}\times (\mathbb S^n\times\mathbb S^n)_{\Delta}\to S\mathbb H^{n+1}$ by
%\begin{equation}
%  \label{e:Psi-def}
%\Psi(x,\nu_-,\nu_+)={P(x,\nu_+)(1,\nu_+,1,\nu_+)+P(x,\nu_-)(1,\nu_-,-1,-\nu_-)
%\over |\nu_+-\nu_-|\sqrt{P(x,\nu_-)P(x,\nu_+)}}.
%\end{equation}
%Then $\Psi(x,\nu_-,\nu_+)$ is closest point to $x$ on the geodesic going from $\nu_-$ to $\nu_+$.
%We calculate
%$$
%B_\pm(\Psi(x,\nu_-,\nu_+))=\nu_\pm,\quad
%\Phi_\pm(\Psi(x,\nu_-,\nu_+))=\sqrt{2P(x,\nu_\pm)\over P(x,\nu_\mp)(1-\nu_+\cdot\nu_-)}.
%$$
%Therefore, in the parametrization~\eqref{e:otherpar}, we have
%$$
%\Psi:(x,\nu_-,\nu_+)\mapsto \bigg(\nu_-,\nu_+, {1\over 2}\log{P(x,\nu_+)\over P(x,\nu_-)}\bigg).
%$$

We finally remark that the stable/unstable subspaces of the cotangent bundle
$E_s^*,E_u^*\subset T^*(S\mathbb H^{n+1})$, defined in~\eqref{e:dual-decomposition},
are in fact the conormal bundles of the fibers of the maps $B_\pm$:
\begin{equation}\label{E*}
E_s^*(\rho)=N^*(B_+^{-1}(B_+(\rho))),\quad
E_u^*(\rho)=N^*(B_-^{-1}(B_-(\rho))),\quad
\rho\in S\mathbb H^{n+1}.
\end{equation}
This is equivalent to saying that the fibers of $B_+$ integrate (i.e. are tangent to) the
subbundle $E_0\oplus E_s\subset T(S\mathbb H^{n+1})$,
while the fibers of $B_-$ integrate the subbundle $E_0\oplus E_u$.
To see the latter statement, for say $B_+$, it is enough to note that
$dB_+\cdot X=0$ and differentiation along vectors in $E_s$ annihilates
the function $x+\xi$ and thus the map $B_+$; therefore,
the kernel of $dB_+$ contains $E_0\oplus E_s$, and this containment
is an equality since the dimensions of both spaces are equal to $n+1$.

%%%%%%%%%%%%%%%%%%%%%%%%%%%%%%%%%%%%%%%%%%%%%%%%%%%%%%%%%%%%%%%%%%%%%%%%%%%%%%%%
\subsection{Action of \texorpdfstring{$G$}{G} on the conformal infinity}
\label{s:actionG}

For $\gamma\in G$ and $\nu\in\mathbb S^n$, $\gamma\cdot (1,\nu)$ is a lightlike vector
with positive zeroth component.
We can then define
$N_\gamma(\nu)>0$, $L_\gamma(\nu)\in\mathbb S^n$ by
\begin{equation}
  \label{e:defLgamma}
\gamma\cdot (1,\nu)=N_\gamma(\nu)(1,L_\gamma(\nu)).
\end{equation}
The map $L_\gamma$ gives the action of $G$ on the conformal infinity $\mathbb S^n=\partial\overline{\mathbb B^{n+1}}$.
This action is transitive and the isotropy groups of $\pm e_1\in\mathbb S^n$ are given by
\begin{equation}
  \label{e:H-pm}
H_\pm=\{\gamma\in G\mid\exists s>0: \gamma\cdot (e_0\pm e_1)=s(e_0\pm e_1)\}.
\end{equation}
The isotropy groups $H_\pm$ are the connected subgroups of $G$ with the Lie algebras generated by $R_{i+1,j+1}$ for $1\leq i<j\leq n$, $X$, and $U^\pm_i$ for $1\leq i\leq n$. To see that $H_\pm$ are connected, for $n=1$ we can check
directly that every $\gamma\in H_\pm$ can be written as
a product $e^{tX}e^{sU^\pm_1}$ for some $t,s\in\mathbb R$,
and for $n>1$ we can use the fact that $\mathbb S^n\simeq G/H_\pm$ is simply connected
and $G$ is connected, and the homotopy long exact sequence of a fibration.

The differentials of $N_\gamma$ and $L_\gamma$ (in $\nu$) can be written as
$$
dN_\gamma(\nu)\cdot\zeta=\langle dx_0,\gamma\cdot (0,\zeta)\rangle,\quad
(0,dL_\gamma(\nu)\cdot\zeta)={\gamma\cdot (0,\zeta)-(dN_\gamma(\nu)\cdot\zeta)(1,L_\gamma(\nu))\over N_\gamma(\nu)},
$$
here $\zeta\in T_\nu\mathbb S^n$. We see that the map $\nu\mapsto L_\gamma(\nu)$ is conformal
with respect to the standard metric on $\mathbb S^n$, in fact
for $\zeta_1,\zeta_2\in T_\nu\mathbb S^n$,
$$
\langle dL_\gamma(\nu)\cdot\zeta_1,dL_\gamma(\nu)\cdot\zeta_2\rangle_{\mathbb R^{n+1}}
=N_\gamma(\nu)^{-2} \langle \zeta_1,\zeta_2\rangle_{\mathbb R^{n+1}}.
$$
The maps $B_\pm:S\mathbb H^{n+1}\to \mathbb S^n$ are equivariant under the action of $G$:
$$
B_\pm(\gamma.(x,\xi))=L_\gamma(B_\pm(x,\xi)).
$$
Moreover, the functions $\Phi_\pm(x,\xi)$ and $P(x,\nu)$ enjoy the following properties:
\begin{equation}\label{changeofPhi}
\Phi_\pm(\gamma.(x,\xi))=N_\gamma(B_\pm(x,\xi))\Phi_\pm(x,\xi),\quad
P(\gamma\cdot x,L_\gamma(\nu))=N_\gamma(\nu)P(x,\nu).
\end{equation}

%%%%%%%%%%%%%%%%%%%%%%%%%%%%%%%%%%%%%%%%%%%%%%%%%%%%%%%%%%%%%%%%%%%%%%%%%%%%%%%%
\subsection{The bundle $\mc{E}$ and parallel transport to the conformal infinity}
\label{s:E}

Consider the vector bundle $\mathcal E$ over $S\mathbb H^{n+1}$ defined as follows:
\[
\mc{E}=\{(x,\xi,v)\in S\mathbb H^{n+1}\x T_{x}\hh^{n+1}\mid g_H(\xi,v)=0 \},
\]
i.e. the fibers $\mathcal E(x,\xi)$ consist of all tangent
vectors in $T_x\mathbb H^{n+1}$ orthogonal to $\xi$; equivalently, $\mc{E}(x,\xi)$
consists of all vectors in $\mathbb R^{1,n+1}$ orthogonal to $x$ and $\xi$ with respect to the Minkowski
inner product. Note that $G$ naturally acts on $\mathcal E$, by putting
$\gamma.(x,\xi,v):=(\gamma\cdot x,\gamma\cdot\xi,\gamma\cdot v)$.

The bundle $\mathcal E$ is invariant under parallel transport along geodesics.
Therefore, one can consider the first order differential operator
\begin{equation}
  \label{e:XX-def}
\mathcal X:\CI(S\mathbb H^{n+1};\mathcal E)\to \CI(S\mathbb H^{n+1};\mathcal E)
\end{equation}
which is the generator of parallel transport, namely if $\mathbf v$ is a section of $\mathcal E$
and $(x,\xi)\in S\mathbb H^{n+1}$, then $\mathcal X\mathbf v(x,\xi)$ is the covariant derivative
at $t=0$ of the vector field $\mathbf v(t):=\mathbf v(\varphi_t(x,\xi))$ on the geodesic
$\varphi_t(x,\xi)$. Note that $\mathcal E(\varphi_t(x,\xi))$ is independent of $t$ 
as a subspace of $\mathbb R^{1,n+1}$, and under this embedding, $\mathcal X$ just acts as $X$ on each
coordinate of $v$ in $\mathbb R^{1,n+1}$. The operator ${1\over i}\mathcal X$ is a symmetric operator with respect
to the standard volume form on $S\mathbb H^{n+1}$ and the inner product on $\mathcal E$ inherited
from $T\mathbb H^{n+1}$.

We now consider parallel transport of vectors along geodesics going off to infinity. Let
$(x,\xi)\in S\mathbb H^{n+1}$ and $v\in T_x\mathbb H^{n+1}$. We let
$(x(t),\xi(t))=\varphi_t(x,\xi)$ be the corresponding geodesic
and $v(t)\in T_{x(t)} \mathbb H^{n+1}$ be the parallel transport of $v$ along this
geodesic. We embed $v(t)$ into the unit ball model $\mathbb B^{n+1}$ by defining
$$
w(t)=d\psi(x(t))\cdot v(t)\in \mathbb R^{n+1},
$$
where $\psi$ is defined in~\eqref{defpsi}. Then $w(t)$ converges to 0 as $t\to\pm\infty$,
but the limits $\lim_{t\to\pm \infty}x_0(t) w(t)$ are nonzero for nonzero $v$; we call the transformation
mapping $v$ to these limits the \emph{transport to conformal infinity} as $t\to \pm\infty$.
More precisely, if
$$
v=c\xi+u,\quad
u\in\mathcal E(x,\xi),
$$
then we calculate
\begin{equation}
  \label{e:pmlimit}
\lim_{t\to \pm\infty} x_0(t)w(t)=\pm cB_\pm(x,\xi)+u'-u_0B_\pm(x,\xi),
\end{equation}
where $B_\pm(x,\xi)\in\mathbb S^n$ is defined in Section~\ref{s:confinf}. We will in particular use the inverse of
the map $\mathcal E(x,\xi)\ni u\mapsto u'-u_0B_\pm(x,\xi)\in T_{B_\pm(x,\xi)}\mathbb S^n$: for $(x,\xi)\in S\mathbb H^{n+1}$
and $\zeta\in T_{B_\pm(x,\xi)}\mathbb S^n$, define (see Figure~\ref{f:horo}(b))
\begin{equation}
  \label{e:Adef}
\mathcal A_\pm(x,\xi)\zeta
= (0,\zeta)-\langle (0,\zeta),x\rangle_M (x\pm \xi)
=\pm\frac{\pl_\nu \xi_\pm(x,B_\pm(x,\xi))\cdot\zeta}{P(x,B_\pm(x,\xi))}\in\mathcal E(x,\xi).
\end{equation}
Here $\xi_\pm$ is defined in~\eqref{defxi+}. Note that by~\eqref{conformaldxi}, $\mathcal A_\pm$
is an isometry:
\begin{equation}
\label{Apmiso}
|\mathcal A_\pm(x,\xi)\zeta|_{g_H}
=|\zeta|_{\mathbb R^n},\quad
\zeta\in T_{B_\pm(x,\xi)}\mathbb S^n.
\end{equation}
Also, $\mathcal A_\pm$ is equivariant under the action of $G$:
\begin{equation}\label{equivarianceA}
\mathcal A_\pm(\gamma\cdot x,\gamma\cdot \xi)\cdot dL_\gamma(B_\pm(x,\xi))\cdot \zeta=
N_\gamma(B_\pm(x,\xi))^{-1}\,\gamma\cdot (\mathcal A_\pm(x,\xi)\zeta).
\end{equation}

We now write the limits~\eqref{e:pmlimit} in terms of the $0$-tangent bundle
of Mazzeo--Melrose~\cite{MM}. Consider the boundary defining function
$\rho_0:=2(1-|y|)/(1+|y|)$ on $\overline{\mathbb B^{n+1}}$; note that in the hyperboloid model,
with the map $\psi$ defined in~\eqref{defpsi},
\begin{equation}
  \label{e:rho}
\rho_0(\psi(x))=2{\sqrt{x_0+1}-\sqrt{x_0-1}\over \sqrt{x_0+1}+\sqrt{x_0-1}}=x_0^{-1}+\mathcal O(x_0^{-2})
\quad\text{as }x_0\to\infty.
\end{equation}
The hyperbolic metric can be written near the boundary as $g_H=(d\rho_0^2+h_{\rho_0})/\rho_0^2$ 
with $h_{\rho_0}$ a smooth family of metrics on $\sph^n$ and $h_0=d\theta^2$ is the canonical metric on the sphere (with curvature $1$).

Define the $0$-tangent bundle
${^0}T\overline{\mathbb{B}^{n+1}}$ to be the smooth bundle over $\overline{\mathbb B^{n+1}}$ whose smooth sections are the elements of the Lie algebra $\mc{V}_0(\overline{\mathbb B^{n+1}})$ of smooth vectors fields vanishing at $\sph^n=\overline{\mathbb B^{n+1}}\cap \{\rho_0=0\}$; near the boundary, this algebra
 is locally spanned over $\CI(\overline{\mathbb{B}^{n+1}})$ by the vector fields $\rho_0 \pl_{\rho_0},\rho_0\pl_{\theta_1},\dots,\rho_0\pl_{\theta_n}$ if $\theta_i$ are local coordinates on $\sph^n$.
Note that there is a natural map ${^0}T\bbar{\mathbb{B}^{n+1}}\to T\bbar{\mathbb{B}^{n+1}}$ which is an isomorphism when restricted to the interior $\bb^{n+1}$.
We denote by ${^0}T^*\overline{\bb^{n+1}}$  the dual bundle to ${^0}T\overline{\mathbb{B}^{n+1}}$, generated locally near $\rho_0=0$ by the covectors $d\rho_0/\rho_0,d\theta_1/\rho_0,\dots,d\theta_n/\rho_0$. Note that $T^*\overline{\bb^{n+1}}$ naturally embeds into
${^0}T^*\overline{\bb^{n+1}}$ and this embedding is an isomorphism in the interior. The metric
$g_H$ is a smooth non-degenerate positive definite quadratic form on ${^0}T\overline{\mathbb B^{n+1}}$, that is
$g_H\in \CI(\overline{\mathbb{B}^{n+1}}; \otimes_S^2({^0}T^*\overline{\mathbb{B}^{n+1}}))$, where $\otimes_S^2$ denotes the space of symmetric 2-tensors. We refer the reader to~\cite{MM} for further details (in particular, for an explanation of why
0-bundles are smooth vector bundles); see also~\cite[\S2.2]{melro} for the similar $b$-setting.

We can then interpret~\eqref{e:pmlimit} as follows: for each $(y,\eta)\in S\mathbb B^{n+1}$
and each $w\in T_y\mathbb B^{n+1}$, the parallel transport $w(t)$ of $w$
along the geodesic $\varphi_t(y,\eta)$ (this geodesic extends smoothly to a curve on $\mathbb B^{n+1}$,
as it is part of a line or a circle) has limits as $t\to \pm\infty$ in the 0-tangent bundle
$^0T\overline{\mathbb B^{n+1}}$. In fact (see~\cite[Appendix~A]{GMP}),
the parallel transport 
\[
\tau(y',y): {^0}T_y{\mathbb{B}^{n+1}}\to {^0}T_{y'}{\mathbb{B}^{n+1}}
\] 
from $y$ to $y'\in\mathbb{B}^{n+1}$ along the geodesic starting at $y$ and ending at $y'$ extends smoothly to the boundary 
$(y,y')\in \overline{\mathbb{B}^{n+1}}\x\overline{\mathbb{B}^{n+1}}\setminus {\rm diag}(\sph^n\x\sph^n)$ 
as an endomorphism ${^0}T_y\overline{\mathbb{B}^{n+1}}\to {^0}T_{y'}\overline{\mathbb{B}^{n+1}}$,
where ${\rm diag}(\sph^n\x\sph^n)$ denotes the diagonal in the boundary; this parallel
transport is an isometry with respect to $g_H$.
The same properties hold for parallel transport of covectors in ${^0}T^*\overline{\mathbb{B}^{n+1}}$,
using the duality provided by the metric $g_H$. An explicit relation to the maps $\mathcal A_\pm$ is given by the following formula:
\begin{equation}
  \label{e:ptarsk}
\mathcal A_\pm(x,\xi)\cdot \zeta=d\psi(x)^{-1}\cdot \tau(\psi(x),B_\pm(x,\xi))\cdot (\rho_0\zeta),
\end{equation}
where $\rho_0\zeta\in {}^0T_{B_\pm(x,\xi)}\overline{\mathbb B^{n+1}}$ is tangent to the
conformal boundary $\mathbb S^n$.

%%%%%%%%%%%%%%%%%%%%%%%%%%%%%%%%%%%%%%%%%%%%%%%%%%%%%%%%%%%%%%%%%%%%%%%%%%%%%%%%
%                                  SECTION 4                                   %
%%%%%%%%%%%%%%%%%%%%%%%%%%%%%%%%%%%%%%%%%%%%%%%%%%%%%%%%%%%%%%%%%%%%%%%%%%%%%%%%
\section{Horocyclic operators}
\label{s:horror}

In this section, we build on the results of Section~\ref{s:geomhyp} to construct horocyclic operators
$\mathcal U_\pm:\mathcal D'(S\mathbb H^{n+1};\otimes^j \mathcal E^*)\to\mathcal D'(S\mathbb H^{n+1};\otimes^{j+1}\mathcal E^*)$.

%%%%%%%%%%%%%%%%%%%%%%%%%%%%%%%%%%%%%%%%%%%%%%%%%%%%%%%%%%%%%%%%%%%%%%%%%%%%%%%%
\subsection{Symmetric tensors}
\label{symtens}

In this subsection, we assume that $E$ is a vector space of finite dimension $N$, equipped with an inner product $g_E$, and let 
$E^*$ denote the dual space, which has a scalar product induced by $g_E$ (also denoted $g_E$). 
(In what follows, we shall take either $E=\mc{E}(x,\xi)$ or $E=T_x\hh^{n+1}$ for some $(x,\xi)\in S\hh^{n+1}$, and the scalar product $g_E$ in both case is given by the hyperbolic metric $g_H$ on those vector spaces.) In this section, we will work
with tensor powers of $E^*$, but the constructions apply to tensor powers of $E$ by swapping $E$ with $E^*$.

We introduce some notation for finite sequences to simplify the calculations below.
Denote by $\mathscr A^m$ the space of all sequences $K=k_1\dots k_m$ with $1\leq k_\ell\leq N$.
For $k_1\dots k_m\in\mathscr A^m$, $j_1\dots j_r\in \mathscr A^r$, and
a sequence of distinct numbers
$1\leq \ell_1,\dots,\ell_r\leq m$, denote by
$$
\{\ell_1\to j_1,\dots,\ell_r\to j_r\} K\in \mathscr A^m
$$
the result of replacing the $\ell_p$th element of $K$ by $j_p$,
for all $p$. We can also replace some of $j_p$ by blank space, which means that the corresponding
indices are removed from $K$.
 
For $m\geq 0$ denote by $\otimes^m E^*$ the $m$th tensor power of $E^*$ and 
by $\otimes_S^m E^*$ the subset of those tensors which are symmetric, i.e. $u\in \otimes_S^m E^*$ if
$u(v_{\sigma(1)},\dots v_{\sigma(m)})=u(v_1,\dots,v_m)$ for all  $\sigma\in \Pi_m$ and all 
$v_1,\dots,v_m\in E$, where $\Pi_m$ is the permutation group of $\{1,\dots,m\}$.  
There is a natural linear projection $\mc{S}: \otimes^m E^*\to\otimes^m_SE^*$  defined by
\begin{equation}\label{symmet} 
\mc{S}(\eta^*_{1}\otimes \dots \otimes \eta^*_{m})=\frac{1}{m!}\sum_{\sigma\in\Pi_m}\eta^*_{\sigma(1)}
\otimes \dots \otimes \eta^*_{\sigma(m)}, \quad \eta^*_{k}\in E^*
\end{equation}
The metric $g_E$ induces a scalar product on $\otimes^m E^*$ as follows
\[ \cjg v^*_{1}\otimes \dots \otimes v^*_{m},w^*_{1}\otimes \dots \otimes w^*_{m}\cjd_{g_E}=\prod_{j=1}^m  \cjg v_j^*,w_j^*\cjd_{g_E}, \quad w_{i}^*,v_{i}^*\in E^*. 
\]
The operator $\mc{S}$ is self-adjoint and thus an orthogonal projection with respect to this scalar product. 

Using the metric $g_E$, one can decompose the vector space $\otimes_S^m$ as follows.  
Let $(\mathbf e_i)_{i=1}^{N}$ be an orthonormal basis of $E$ for the metric $g_E$ and 
$(\mathbf{e}_i^*)$ be the dual basis.
First of all, introduce the trace map
$\mathcal T:\otimes^m E^*\to\otimes^{m-2} E^*$ contracting the first two indices by the metric:
for $v_i\in E$, define 
\begin{equation}\label{tracedef}
\mathcal T(u)(v_1,\dots,v_{m-2}):=\sum_{i=1}^{N}u(\mathbf e_i,\mathbf e_i,v_1,\dots,v_{m-2})
\end{equation}
(the result is independent of the choice of the basis). For $m< 2$, we define
$\mathcal T$ to be zero on $\otimes^m E^*$.
Note that $\mc{T}$ maps $\otimes_S^{m+2}E^*$ onto 
$\otimes^m_SE^*$. Set
$$
\mathbf{e}^*_K:=\mathbf e^*_{k_1}\otimes\dots\otimes \mathbf{e}^*_{k_m}\in \otimes^m E^*,\quad
K=k_1\dots k_m\in \mathscr{A}^{m}.
$$
Then
$$
\mathcal T\Big(\sum_{K\in\mathscr A^{m+2}}f_K \mathbf e^*_K\Big)
=\sum_{K\in\mathscr A^m}\sum_{q\in\mathscr A} f_{qqK}\mathbf e_K^*.
$$
 The adjoint of $\mc{T}:\otimes_S^{m+2}E^*\to \otimes_S^mE^*$ with respect to the scalar product $g_E$ 
 is given by the map  $u\mapsto \mc{S}(g_E\otimes u)$. To simplify computations,
we define a scaled version of it: let  
$\mathcal I:\otimes^m_SE^*\to\otimes^{m+2}_SE^*$ be defined by 
\begin{equation}
  \label{e:defI}
\mc{I}(u)=\frac{(m+2)(m+1)}{2}\mc{S}(g_{E}\otimes u)
={(m+2)(m+1)\over 2} \mathcal T^*(u).
\end{equation}
Then
$$
\mathcal I\Big(\sum_{K\in\mathscr A^m} f_K\mathbf e^*_K\Big)=\sum_{K\in\mathscr A^{m+2}}\sum_{\ell,r=1\atop \ell<r}^{m+2}
\delta_{k_\ell k_r}f_{\{\ell\to,r\to\}K}\mathbf e^*_K.
$$
Note that for $u\in \otimes^m_SE^*$,
\begin{equation}
  \label{e:trart}
\mathcal T(\mathcal I u)=(2m+N)u+\mathcal I(\mathcal T u).
\end{equation}

By~\eqref{e:defI} and~\eqref{e:trart},
the homomorphism $\mathcal T\mathcal I:\otimes_S^m E^*\to\otimes_S^mE^*$
is positive definite and thus
an isomorphism. Therefore, for $u\in \otimes^m_S E^*$, we can decompose
$u=u_1+\mathcal I(u_2)$, where $u_1\in\otimes^m_S E^*$ satisfies $\mathcal T(u_1)=0$
and $u_2=(\mathcal T\mathcal I)^{-1}\mathcal Tu\in\otimes^{m-2}_S E^*$. Iterating this process, we can decompose any 
$u\in \otimes_S^mE^*$ into 
\begin{equation}\label{decompositionoftensors}
u=\sum_{r=0}^{\lfloor m/2 \rfloor} \mathcal I^r(u_r),\quad
u_r\in \otimes_S^{m-2r}E^*,\quad
\mathcal T(u_r)=0,
\end{equation}
with $u_r$ determined uniquely by $u$.

Another operation on tensors which will be used is the interior product: if $v\in E$ and $u\in \otimes_S^m E^*$, 
we denote by $\iota_v(u)\in \otimes_S^{m-1} E^*$ the interior product of $u$ by $v$ given by 
\[ \iota_vu(v_1,\dots,v_{m-1}):=u(v,v_1,\dots,v_{m-1}).\] 
If $v^*\in E^*$, we denote $\iota_{v^*}u$ for the tensor $\iota_vu$ by $g_E(v,\cdot)=v^*$.
 
We conclude this section with a correspondence which will be useful in certain calculations later. 
There is a linear
isomorphism between $\otimes_S^m E^*$ and the space $\Pol^m(E)$ of homogeneous polynomials of degree 
$m$ on $E$: to a tensor $u\in \otimes_S^m E^*$ we associate the function on $E$ given by 
$x\to P_u(x):=u(x,\dots,x)$. If we write $x=\sum_{i=1}^Nx_i\mathbf{e}_i$ in a given orthonormal 
basis then
$$
P_{\mathcal S(e^*_K)}(x)=\prod_{j=1}^m x_{k_j},\quad
K=k_1\dots k_m\in\mathscr A^m.
$$
The flat Laplacian associated to $g_E$ is given by $\Delta_E=-\sum_{i=1}^N\pl_{x_i}^2$ in the coordinates induced by the basis $(\mathbf{e}_i)$.
Then it is direct to see that 
\begin{equation}\label{Trace=Lap}
\Delta_E P_u(x)=-m(m-1)P_{\mc{T}(u)}(x),\quad
u\in \otimes^m_SE^*.
\end{equation}
which means that the trace corresponds to applying the Laplacian (see \cite[Lemma 2.4]{DaSh}). In particular, trace-free symmetric tensors of order $m$ correspond to homogeneous harmonic polynomials, and thus restrict to spherical harmonics on the sphere $|x|_{g_E}=1$ of $E$.
We also have
\begin{equation}
  \label{e:morestuff}
P_{\mathcal I(u)}(x)={(m+2)(m+1)\over 2}\, |x|^2 P_u(x),\quad
u\in\otimes^m_SE^*.
\end{equation}

%%%%%%%%%%%%%%%%%%%%%%%%%%%%%%%%%%%%%%%%%%%%%%%%%%%%%%%%%%%%%%%%%%%%%%%%%%%%%%%%
\subsection{Horocyclic operators}
\label{s:horocycl}

We now consider the left-invariant vector fields $X$, $U_\pm^i$, $R_{i+1,j+1}$ on the isometry group $G$,
identified with the elements of the Lie algebra of $G$ introduced in~\eqref{liealgebraofG}, \eqref{liealgebraofG2}.
Recall that $G$ acts on $S\mathbb H^{n+1}$ transitively with the isotropy group $H\simeq \SO(n)$
and this action gives rise to the projection $\pi_S:G\to S\mathbb H^{n+1}$~-- see~\eqref{e:piS}.
Note that, with the maps $\Phi_\pm:S\mathbb H^{n+1}\to \mathbb R^+,B_\pm:S\mathbb H^{n+1}\to\mathbb S^n$ defined in~\eqref{PhiandB}, we have
$$
B_\pm(\pi_S(\gamma))=L_\gamma(\pm e_1),\quad
\Phi_\pm(\pi_S(\gamma))=N_\gamma(\pm e_1),\quad
\quad\gamma\in G,
$$
where $N_\gamma:\mathbb S^n\to\mathbb R^+,L_\gamma:\mathbb S^n\to\mathbb S^n$ are defined in~\eqref{e:defLgamma}.
Since $H_\pm$, the isotropy group of $\pm e_1$ under the action $L_\gamma$, contains $X,U_i^\pm$ in its Lie
algebra (see~\eqref{e:H-pm} and Figure~\ref{f:horo}(a)), we find
\begin{equation}
\label{e:isot}
d(B_\pm\circ\pi_S)\cdot U^\pm_i=0, \quad d(B_\pm\circ\pi_S)\cdot X=0.
\end{equation}
We also calculate
\begin{equation}
  \label{e:isot2}
d(\Phi_\pm\circ\pi_S)\cdot U_i^\pm =0.
\end{equation}
Define the differential operator on $G$
$$
U^\pm_K:=U^\pm_{k_1}\dots U^\pm_{k_m},\quad
K=k_1\dots k_m\in\mathscr A^m.
$$
Note that the order in which $k_1,\dots,k_m$ are listed does not matter by~\eqref{e:comm-rel}.
Moreover, by~\eqref{e:comm-rel}
\begin{equation}
  \label{e:rot-condition-2}
[R_{i+1,j+1},U^\pm_K]=\sum_{\ell=1}^m (\delta_{jk_\ell} U^\pm_{\{\ell\to i\}K}
-\delta_{ik_\ell}U^\pm_{\{\ell\to j\}K}).
\end{equation}
Since $H$ is generated by the vector fields $R_{i+1,j+1}$, we see that
in dimensions $n+1>2$ the horocyclic vector fields $U^\pm_i$,
and more generally the operators $U^\pm_K$, are not invariant
under right multiplication by elements of $H$ and therefore do not descend to differential
operators on $S\mathbb H^{n+1}$~-- in other words, if $u\in \mathcal D'(S\mathbb H^{n+1})$,
then $U^\pm_K (\pi_S^* u)\in \mathcal D'(G)$ is not in the image of $\pi_S^*$.

However, in this section we will show how to differentiate distributions on $S\mathbb H^{n+1}$
along the horocyclic vector fields, resulting in sections of the vector bundle $\mathcal E$
introduced in Section~\ref{s:E} and its tensor powers. First of all, we note that by~\eqref{e:stable-unstable},
the stable and unstable bundles $E_s(x,\xi)$ and $E_u(x,\xi)$ are canonically
isomorphic to $\mathcal E(x,\xi)$ by the maps 
$$
\theta_+:\mathcal E(x,\xi)\to E_s(x,\xi),\
\theta_-:\mathcal E(x,\xi)\to E_u(x,\xi),\quad
\theta_\pm(v)=(-v,\pm v).
$$
For $u\in\mathcal D'(S\mathbb H^{n+1})$, we then define the horocyclic derivatives
$\mathcal U_\pm u\in \mathcal D'(S\mathbb H^{n+1};\mathcal E^*)$ by restricting the
differential $du\in \mathcal D'(S\mathbb H^{n+1};T^*(S\mathbb H^{n+1}))$ to 
the stable/unstable foliations and pulling it back by $\theta_\pm$:
\begin{equation}
  \label{e:Udef1}
\mathcal U_\pm u(x,\xi):=du(x,\xi)\circ \theta_\pm\in\mathcal E^*(x,\xi).
\end{equation}
To relate $\mathcal U_\pm$ to the vector fields $U^\pm_i$ on the group $G$, consider the
orthonormal frame $\mathbf e_1^*,\dots,\mathbf e_n^*$ of the bundle $\pi_S^*\mathcal E^*$ over $G$ defined by
$$
\mathbf e_j^*(\gamma):=\gamma^{-*}(e_{j+1}^*)\in\mathcal E^*(\pi_S(\gamma)).
$$
where the $e_j^*=dx_j$ form the dual basis to the canonical basis $(e_j)_{j=0,\dots,n+1}$ of $\rr^{1,n+1}$, and
$\gamma^{-*}=(\gamma^{-1})^*:(\mathbb R^{1,n+1})^*\to(\mathbb R^{1,n+1})^*$.
More generally, we can define the orthonormal frame $\mathbf e^*_K$ of $\pi_S^*(\otimes^m \mathcal E^*)$ by
$$
\mathbf e^*_K:=\mathbf e_{k_1}^*\otimes\dots\otimes \mathbf e_{k_m}^*,\quad
K=k_1\dots k_m\in \mathscr A^m.
$$
We compute for $u\in\mathcal D'(S\mathbb H^{n+1})$, $du(\pi_S(\gamma))\cdot \theta_\pm (\gamma(e_{j+1}))
=U_j^\pm (\pi_S^*u)(\gamma)$ and thus
\begin{equation}
  \label{e:Udef2}
\pi_S^*(\mathcal U_\pm u)=\sum_{j=1}^n U_j^\pm (\pi_S^* u) \mathbf e_j^*.
\end{equation}
We next use the formula~\eqref{e:Udef2} to define $\mathcal U_\pm$ as an operator
\begin{equation}
  \label{e:Ustate}
\mathcal U_\pm:\mathcal D'(S\mathbb H^{n+1};\otimes^m \mathcal E^*)\to\mathcal D'(S\mathbb H^{n+1};\otimes^{m+1}\mathcal E^*)
\end{equation}
as follows: for $u\in\mathcal D'(S\mathbb H^{n+1};\otimes^m\mathcal E^*)$, define
$\mathcal U_\pm u$ by
\begin{equation}
  \label{e:Udef3}
\pi_S^*(\mathcal U_\pm u)=\sum_{r=1}^n\sum_{K\in\mathscr A^m} (U_r^\pm u_K)\mathbf e^*_{r K},\quad
\pi_S^*u=\sum_{K\in\mathscr A^m}u_K\mathbf e_K^*.
\end{equation}
This definition makes sense (that is, the right-hand side of the first formula in~\eqref{e:Udef3}
lies in the image of $\pi_S^*$) since a section
$$
f=\sum_{K\in\mathscr A^m}f_K\mathbf e^*_K\in \mathcal D'(G;\pi_S^*(\otimes^m\mathcal E^*)),\quad
f_K\in\mathcal D'(G)
$$
lies in the image of $\pi_S^*$ if and only if
$R_{i+1,j+1}f=0$ for $1\leq i<j\leq n$ (the differentiation is well-defined since the
fibers of $\pi_S^*(\otimes^m\mathcal E^*)$ are the same along each integral curve of $R_{i+1,j+1}$),
and this translates to
\begin{equation}
  \label{e:rot-condition}
R_{i+1,j+1}f_K=\sum_{\ell=1}^m (\delta_{j k_\ell}
f_{\{\ell\to i\}K}-\delta_{i k_\ell} f_{\{\ell\to j\}K}),\quad
1\leq i<j\leq n,\
K\in\mathscr A^m;
\end{equation}
to show~\eqref{e:rot-condition} for $f_{r K}=U_r^\pm u_K$, we use~\eqref{e:comm-rel}:
$$
\begin{gathered}
R_{i+1,j+1}f_{rK}=[R_{i+1,j+1},U_r^\pm] u_K+U_r^\pm R_{i+1,j+1}u_K\\
=\delta_{jr}U_i^\pm u_K-\delta_{ir}U_j^\pm u_K+
\sum_{\ell=1}^m \delta_{jk_\ell} U_r^\pm u_{\{\ell\to i\}K}-\delta_{ik_\ell}U_r^\pm u_{\{\ell\to j\}K}.
\end{gathered}
$$

To interpret the operator~\eqref{e:Ustate} in terms of the stable/unstable foliations in a manner
similar to~\eqref{e:Udef1}, consider the connection $\nabla^S$ on the bundle
$\mathcal E$ over $S\mathbb H^{n+1}$ defined as follows: for $(x,\xi)\in S\mathbb H^{n+1}$,
$(v,w)\in T_{(x,\xi)}(S\mathbb H^{n+1})$, and $u\in \mathcal D'(S\mathbb H^{n+1};\mathcal E)$,
let $\nabla^S_{(v,w)}u(x,\xi)$ be the orthogonal projection of $\nabla^{\mathbb R^{1,n+1}}_{(v,w)}u(x,\xi)$
onto $\mathcal E(x,\xi)\subset \mathbb R^{1,n+1}$, where $\nabla^{\mathbb R^{1,n+1}}$ is the
canonical connection on the trivial bundle $S\mathbb H^{n+1}\times \mathbb R^{1,n+1}$
over $S\mathbb H^{n+1}$ (corresponding to differentiating the coordinates
of $u$ in $\mathbb R^{1,n+1}$). Then $\nabla^S$ naturally induces a connection on $\otimes^m\mathcal E^*$,
also denoted $\nabla^S$, and we have for $v,v_1,\dots,v_m\in \mathcal E(x,\xi)$
and $u\in \mathcal D'(S\mathbb H^{n+1};\otimes^m\mathcal E^*)$,
\begin{equation}
  \label{e:Udef4}
\mathcal U_\pm u(x,\xi)(v,v_1,\dots,v_m)=(\nabla^S_{\theta_\pm(v)}u)(v_1,\dots,v_m).
\end{equation}
Indeed, if $\gamma(t)=\gamma(0)e^{tU^\pm_j}$ is an integral curve of $U^\pm_j$ on $G$, then
$\gamma(t) e_2,\dots,\gamma(t)e_{n+1}$ form a parallel frame of $\mathcal E$ over the curve $(x(t),\xi(t))=\pi_S(\gamma(t))$
with respect to $\nabla^S$, since the covariant derivative of $\gamma(t)e_k$ in $t$ with respect
to $\nabla^{\mathbb R^{1,n+1}}$ is simply $\gamma(t)U^\pm_j e_k$; by~\eqref{liealgebraofG2}
this is a linear combination of $x(t)=\gamma(t)e_0$ and $\xi(t)=\gamma(t)e_1$ and thus
$\nabla^S_t (\gamma(t)e_k)=0$.

Note also that the operator $\mathcal X$ defined in~\eqref{e:XX-def} can be interpreted
as the covariant derivative on $\mathcal E$ along the generator $X$ of the geodesic flow
by the connection $\nabla^S$. One can naturally generalize $\mathcal X$ to a first order differential operator
\begin{equation}
  \label{e:XX-def-2}
\mathcal X:\mathcal D'(S\mathbb H^{n+1};\otimes^m\mathcal E^*)\to \mathcal D'(S\mathbb H^{n+1};\otimes^m\mathcal E^*)
\end{equation}
and ${1\over i}\mathcal X$ is still symmetric with respect to the natural measure on $S\mathbb H^{n+1}$ and the
inner product on $\otimes^m\mathcal E^*$ induced by the Minkowski metric. A characterization of $X$ in terms
of the frame $\mathbf e_K^*$ is given by
\begin{equation}
  \label{e:XX-def-3}
\pi_S^*(\mathcal X u)=\sum_{K\in\mathscr A^m} (Xu_K)\mathbf e_K^*,\quad
\pi_S^*u=\sum_{K\in\mathscr A^m} u_K\mathbf e_K^*.
\end{equation}
It follows from~\eqref{e:comm-rel} that
for $u\in\mathcal D'(S\mathbb H^{n+1};\otimes^m \mathcal E^*)$,
\begin{equation}
  \label{e:comm-hor}
\mathcal X\mathcal U_\pm u-\mathcal U_\pm \mathcal Xu=\pm \mathcal U_\pm u.
\end{equation}
We also observe that, since $[U^\pm_i,U^\pm_j]=0$, for each scalar distribution $u\in\mathcal D'(S\mathbb H^{n+1})$
and $m\in\nn$, we have $\mathcal U_\pm^m u\in\mathcal D'(S\mathbb H^{n+1};\otimes_S^m\mathcal E^*)$,
where $\otimes^m_S\mathcal E^*\subset\otimes^m\mathcal E^*$
denotes the space of all symmetric cotensors of order $m$. Inversion of the operator $\mathcal U_\pm^m$
is the topic of the next subsection. We conclude with the following lemma describing how the operator
$\mathcal U_\pm^m$ acts on distributions invariant under the left action of an element of $G$:
%%%%%%%%%%%%%%%%%%%%%%%%%%%%%%%%%%%%%%%%%%%%%%%%%%%%%%%%%%%%%%%%%%%%%%%%%%%%%%%%
\begin{lemm}
  \label{l:u-pm-gamma}
Let $\gamma\in G$ and $u\in\mathcal D'(S\mathbb H^{n+1})$. Assume also that
$u$ is invariant under left multiplications by $\gamma$, namely $u(\gamma.(x,\xi))=u(x,\xi)$
for all\/%
\footnote{Strictly speaking, this statement should be formulated
in terms of the pullback of the distribution $u$ by the map $(x,\xi)\mapsto \gamma.(x,\xi)$.}
$(x,\xi)\in S\mathbb H^{n+1}$.
Then $v=\mathcal U_\pm^m u$ is equivariant under left
multiplication by $\gamma$ in the following sense:
\begin{equation}
  \label{e:v-gamma}
v(\gamma.(x,\xi))=\gamma.v(x,\xi),
\end{equation}
where the action of $\gamma$ on $\otimes_S^m\mathcal E^*$ is naturally induced by its action
on $\mathcal E$ (by taking inverse transposes), which in turn comes from the action of $\gamma$ on $\mathbb R^{1,n+1}$.
\end{lemm}
%%%%%%%%%%%%%%%%%%%%%%%%%%%%%%%%%%%%%%%%%%%%%%%%%%%%%%%%%%%%%%%%%%%%%%%%%%%%%%%%
\begin{proof}
We have for $\gamma'\in G$,
$$
\mathcal U_\pm^m u(\pi_S(\gamma'))=\sum_{K\in\mathscr A^m}
(U_K^\pm (u\circ\pi_S)(\gamma'))\mathbf e_K^*(\gamma').
$$
Therefore, since $U^\pm_j$ are left invariant vector fields on $G$,
$$
\mathcal U_\pm^m u(\gamma.\pi_S(\gamma'))
=\mathcal U_\pm^m u(\pi_S(\gamma\gamma'))
=\sum_{K\in\mathscr A^m} (U_K^\pm (u\circ \pi_S)(\gamma'))\mathbf e_K^*(\gamma\gamma').
$$
It remains to note that $\mathbf e_K^*(\gamma\gamma')=\gamma.\mathbf e_K^*(\gamma')$.
\end{proof}
%%%%%%%%%%%%%%%%%%%%%%%%%%%%%%%%%%%%%%%%%%%%%%%%%%%%%%%%%%%%%%%%%%%%%%%%%%%%%%%%

%%%%%%%%%%%%%%%%%%%%%%%%%%%%%%%%%%%%%%%%%%%%%%%%%%%%%%%%%%%%%%%%%%%%%%%%%%%%%%%%
\subsection{Inverting horocyclic operators}
\label{s:hor-inv}
In this subsection, we will show that distributions $v\in \mathcal D'(S\mathbb H^{n+1};\otimes_S^m \mathcal E^*)$ 
satisfying certain conditions are in fact in the image of $\mc{U}_\pm^m$ acting on 
$\mathcal D'(S\mathbb H^{n+1})$. This is an important step in our construction of Pollicott--Ruelle resonances,
as it will make it possible to recover a scalar resonant state corresponding to a resonance
in the $m$th band. More precisely, we prove
%%%%%%%%%%%%%%%%%%%%%%%%%%%%%%%%%%%%%%%%%%%%%%%%%%%%%%%%%%%%%%%%%%%%%%%%%%%%%%%%
\begin{lemm}
  \label{l:recovery}
Assume that $v\in\mathcal D'(S\mathbb H^{n+1};\otimes^m_S\mathcal E^*)$ satisfies
$\mathcal U_\pm v=0$, and $\mathcal Xv=\pm \lambda v$ for $\lambda\not\in {1\over 2}\mathbb Z$. Then
there exists $u\in\mathcal D'(S\mathbb H^{n+1})$
such that $\mathcal U_\pm^m u=v$ and
$Xu=\pm(\lambda-m) u$. Moreover, if $v$ is equivariant under left multiplication by some $\gamma\in G$ in the sense
of~\eqref{e:v-gamma}, then $u$ is invariant under left multiplication by $\gamma$.
\end{lemm}
%%%%%%%%%%%%%%%%%%%%%%%%%%%%%%%%%%%%%%%%%%%%%%%%%%%%%%%%%%%%%%%%%%%%%%%%%%%%%%%%
The proof of Lemma~\ref{l:recovery} is modeled on the following well-known
formula recovering a homogeneous polynomial of degree $m$
from its coefficients: given constants $a_\alpha$ for each multiindex $\alpha$ of length $m$,
we have
\begin{equation}
  \label{e:recovery-moral}
\partial^\beta_x \sum_{|\alpha|=m} {1\over \alpha!}x^\alpha a_\alpha =a_\beta,\quad
|\beta|=m.
\end{equation}
The formula recovering $u$ from $v$ in Lemma~\ref{l:recovery} is morally similar to~\eqref{e:recovery-moral},
with $U^\pm_j$ taking the role of $\partial_{x_j}$, the condition $\mathcal U_\pm v=0$ corresponding to
$a_\alpha$ being constants, and $U^\mp_j$ taking the role of the multiplication operators $x_j$. However,
the commutation structure of $U^\pm_j$, given by~\eqref{e:comm-rel}, is more involved than that of $\partial_{x_j}$
and $x_j$ and in particular it involves the vector field $X$, explaining the need for the condition $\mathcal Xv=\pm\lambda v$
(which is satisfied by resonant states).

To prove Lemma~\ref{l:recovery}, we define the operator 
$$
\mathcal V_\pm:\mathcal D'(S\mathbb H^{n+1};\otimes^{m+1} \mathcal E^*)\to \mathcal D'(S\mathbb H^{n+1};\otimes^m\mathcal E^*),\quad
\mathcal V_\pm:=\mathcal T\mathcal U_\pm,
$$
where $\mathcal T$ is defined in Section~\ref{symtens}. Then by~\eqref{e:Udef3}
$$
\pi_S^*(\mathcal V_\pm u)=\sum_{K\in\mathscr A^m}\sum_{q\in\mathscr A}(U^\pm_q u_{qK})\mathbf e^*_K,\quad
u=\sum_{K\in\mathscr A^{m+1}} u_K\mathbf e^*_K.
$$
For later use, we record the following fact:
%%%%%%%%%%%%%%%%%%%%%%%%%%%%%%%%%%%%%%%%%%%%%%%%%%%%%%%%%%%%%%%%%%%%%%%%%%%%%%%%
\begin{lemm}
  \label{l:tough-adjoint}
$\mathcal U_\pm^*=-\mathcal V_\pm$,
where the adjoint is understood in the formal sense.
\end{lemm}
%%%%%%%%%%%%%%%%%%%%%%%%%%%%%%%%%%%%%%%%%%%%%%%%%%%%%%%%%%%%%%%%%%%%%%%%%%%%%%%%
\begin{proof}
If
$u\in \CI_0(S\mathbb H^{n+1};\otimes^m \mathcal E^*)$,
$v\in \CI(S\mathbb H^{n+1};\otimes^{m+1} \mathcal E^*)$
and $u_K$, $v_J$ are the coordinates of $\pi_S^*u$ and $\pi_S^*v$ in the bases
$(\mathbf e^*_K)_{K\in\mathscr A^m}$ and $(\mathbf e^*_J)_{J\in\mathscr A^{m+1}}$, then by~\eqref{e:Udef3},
we compute the following pointwise identity on $S\mathbb H^{n+1}$:
$$
\langle \mathcal U_\pm u,\bar v\rangle+\langle u,\overline{\mathcal  V_\pm v}\rangle
=\mathcal V_\pm w,\quad
w\in\CI_0(S\mathbb H^{n+1};\mathcal E^*),\quad
\pi_S^*w=\sum_{K\in\mathscr A^m\atop q\in\mathscr A} u_K\overline{v_{qK}}\,\mathbf e_q^*.
$$
It remains to show that for each $w$, the integral of $\mathcal V_\pm w$ is equal to zero. Since $\mathcal V_\pm$ is
a differential operator of order 1, we must have
$$
\int_{S\mathbb H^{n+1}}\mathcal V_\pm w=\int_{S\mathbb H^{n+1}} \langle w,\eta_\pm\rangle
$$
for all $w$ and some $\eta_\pm\in\CI(S\mathbb H^{n+1};\mathcal E^*)$ independent of $w$. Then $\eta_\pm$ is
equivariant under the action of the isometry group $G$ and in particular, $|\eta_\pm|$ is a constant
function on $S\mathbb H^{n+1}$. Moreover, using that $\int Xf=0$ for all $f\in \CI_0(S\mathbb H^{n+1})$
and $\mathcal V_\pm(\mathcal Xw)=(X\mp 1)\mathcal V_\pm w$, we get for all $w\in\CI_0$,
$$
\mp\int_{S\mathbb H^{n+1}} \langle w,\eta_\pm\rangle=\int_{S\mathbb H^{n+1}} \mathcal V_\pm(\mathcal Xw)
=-\int_{S\mathbb H^{n+1}} \langle w,\mathcal X\eta_\pm\rangle.
$$
This implies that $\mathcal X\eta_\pm=\pm\eta_\pm$ and in particular
$$
X|\eta_\pm|^2=2\langle \mathcal X\eta_\pm,\eta_\pm\rangle=\pm 2|\eta_\pm|^2.
$$
Since $|\eta_\pm|^2$ is a constant function, this implies $\eta_\pm=0$, finishing the proof.
\end{proof}
%%%%%%%%%%%%%%%%%%%%%%%%%%%%%%%%%%%%%%%%%%%%%%%%%%%%%%%%%%%%%%%%%%%%%%%%%%%%%%%%

To construct $u$ from $v$ in Lemma~\ref{l:recovery}, we first handle the case when $\mathcal T(v)=0$; this condition is automatically
satisfied when $m\leq 1$.
%%%%%%%%%%%%%%%%%%%%%%%%%%%%%%%%%%%%%%%%%%%%%%%%%%%%%%%%%%%%%%%%%%%%%%%%%%%%%%%%
\begin{lemm}
  \label{l:long-product}
Assume that $v\in\mathcal D'(S\mathbb H^{n+1};\otimes^m_S\mathcal E^*)$ and
$\mathcal U_\pm v=0$, $\mathcal T(v)=0$.
Define $u=\mathcal V_\mp^m v\in\mathcal D'(S\mathbb H^{n+1})$. Then
\begin{equation}
  \label{e:whale}
\mathcal U_\pm^m u=2^m m!\bigg(\prod_{\ell=n-1}^{n+m-2} (\ell\pm \mathcal X)\bigg )v.
\end{equation}
\end{lemm}
%%%%%%%%%%%%%%%%%%%%%%%%%%%%%%%%%%%%%%%%%%%%%%%%%%%%%%%%%%%%%%%%%%%%%%%%%%%%%%%%
\begin{proof}
Assume that
$$
\pi_S^* v=\sum_{K\in\mathscr A^m} f_K \mathbf e^*_K,\quad
f_K\in\mathcal D'(G).
$$
Then
$$
\pi_S^*u=\sum_{K\in\mathscr A^m}  U^\mp_K f_K,\quad
\pi_S^*(\mathcal U^m_\pm u)=\sum_{K,J\in\mathscr A^m}  U^\pm_J  U^\mp_K f_K \mathbf e^*_J.
$$
For $0\leq r<m$, $J\in\mathscr A^{m-1-r}$, and $p\in\mathscr A$, we have by~\eqref{e:comm-rel}
$$
\sum_{K\in\mathscr A^r\atop q\in\mathscr A} [U^\pm_p,U^\mp_q]U^\mp_K f_{qKJ}
=\pm 2X\sum_{K\in\mathscr A^r}U^\mp_K f_{pKJ}+2\sum_{K\in\mathscr A^r\atop q\in\mathscr A} R_{p+1,q+1}U^\mp_K f_{qKJ}.
$$
To compute the second term on the right-hand side,
we commute $R_{p+1,q+1}$ with $U^\mp_K$ by~\eqref{e:rot-condition-2} and use~\eqref{e:rot-condition}
to get
$$
\begin{gathered}
\sum_{K\in\mathscr A^r\atop q\in\mathscr A} R_{p+1,q+1}U^\mp_K f_{qKJ}
=\sum_{K\in\mathscr A^r\atop q\in\mathscr A}\bigg( \sum_{\ell=1}^{r}
(\delta_{qk_\ell}U^\mp_{\{\ell\to p\}K} f_{qKJ}
-\delta_{pk_\ell}U^\mp_{\{\ell\to q\}K} f_{qKJ})\\
+U^\mp_K f_{pKJ}-\delta_{pq}U^\mp_K f_{qKJ}
+\sum_{\ell=1}^r (\delta_{qk_\ell}U^\mp_Kf_{q(\{\ell\to p\}K)J}
-\delta_{pk_\ell}U^\mp_K f_{q(\{\ell\to q\}K)J})\\
+\sum_{\ell=1}^{m-1-r} (\delta_{qj_\ell}U^\mp_K f_{qK(\{\ell\to p\}J)}
-\delta_{pj_\ell} U^\mp_K f_{qK(\{\ell\to q\}J)})
\bigg).
\end{gathered}
$$
Since $v$ is symmetric and $\mathcal T(v)=0$, the expressions
$\sum_{K\in\mathscr A^r,\,q\in\mathscr A}\delta_{qk_\ell}U^\mp_{\{\ell\to p\}K}f_{qKJ}$,
$\sum_{q\in\mathscr A}f_{q(\{\ell\to q\} K)J}$, and $\sum_{q\in\mathscr A}f_{qK(\{\ell\to q\}J)}$ are zero.
Further using the symmetry of $v$, we find
$$
\sum_{K\in\mathscr A^r\atop q\in\mathscr A}R_{p+1,q+1}U^\mp_K f_{qKJ}
=(n+m-r-2)\sum_{K\in\mathscr A^r} U^\mp_K f_{pKJ}.
$$
and thus
\begin{equation}
  \label{e:gorilla}
\sum_{K\in\mathscr A^r\atop q\in\mathscr A}[U^\pm_p,U^\mp_q]U^\mp_K f_{qKJ}=2
\sum_{K\in\mathscr A^r} U^\mp_K (\pm X+n+m-2r-2) f_{pKJ}.
\end{equation}
Then, using that $\mathcal U_\pm v=0$, we find
\begin{equation}
  \label{e:chimp}
\begin{split}
\sum_{K\in\mathscr A^{r+1}} U^\pm_pU^\mp_K f_{KJ}&= 
\sum_{K\in\mathscr A^r\atop q\in\mathscr A}\sum_{\ell=1}^{r+1} U^\mp_{k_\ell\dots k_r}
[U^\pm_p,U^\mp_q] U^\mp_{k_1\dots k_{\ell-1}} f_{qKJ}\\
&=   2\sum_{K\in\mathscr A^r}\sum_{\ell=1}^{r+1}U^\mp_K(\pm X+n+m-2\ell)f_{pKJ}\\
&=    2(r+1)\sum_{K\in\mathscr A^r}U^\mp_K(\pm X+n+m-r-2)f_{pKJ}.
\end{split}
\end{equation}
By iterating \eqref{e:chimp} we obtain (using also that $v$ is symmetric)
for $J\in\mathscr A^m$,
\[\begin{split}
U^\pm_J\sum_{K\in\mathscr A^m}U^\mp_K f_K&=   2m\,U^\pm_{j_1\dots j_{m-1}}
\sum_{K\in\mathscr A^{m-1}}U_K^\mp(\pm X+n-1)f_{Kj_m}\\
&=  4m(m-1)\,U^\pm_{j_1\dots j_{m-2}}
\sum_{K\in\mathscr{A}^{m-2}}U_K^\mp(\pm X+n)(\pm X+n-1)f_{Kj_{m-1}j_m}\\
&=   \dots \\
&= 2^m m! \prod_{\ell=n-1}^{n+m-2}(\pm X+\ell)f_{J}
\end{split}\]
which achieves the proof.
\end{proof}
%%%%%%%%%%%%%%%%%%%%%%%%%%%%%%%%%%%%%%%%%%%%%%%%%%%%%%%%%%%%%%%%%%%%%%%%%%%%%%%%
To handle the case $\mathcal T (v)\neq 0$, define also the horocyclic Laplacians
$$
\Delta_\pm:=-\mathcal T\mathcal U_\pm^2=-\mathcal V_\pm\mathcal U_\pm:\mathcal D'(S\mathbb H^{n+1})\to\mathcal D'(S\mathbb H^{n+1}),
$$
so that for $u\in\mathcal D'(S\mathbb H^{n+1})$,
$$
\pi_S^*\Delta_\pm u=-\sum_{q=1}^n U^\pm_q U^\pm_q (\pi_S^*u).
$$
Note that, by the commutation relation~\eqref{e:comm-rel}, 
\begin{equation}
  \label{e:comm-lap}
[X,\Delta_\pm]=\pm 2\Delta_\pm.
\end{equation}
Also, by Lemma~\ref{l:tough-adjoint}, $\Delta_\pm$ are symmetric operators.
%%%%%%%%%%%%%%%%%%%%%%%%%%%%%%%%%%%%%%%%%%%%%%%%%%%%%%%%%%%%%%%%%%%%%%%%%%%%%%%%
\begin{lemm}
  \label{l:penible}
Assume that $u\in\mathcal D'(S\mathbb H^{n+1})$ and $\mathcal U_\pm^{m+1} u=0$.
Then
$$
\mathcal U_\pm^{m+2}\Delta_\mp u=-4(\mathcal X\mp m)(2\mathcal X\pm (n-2))\mathcal I(\mathcal U_\pm^m u)-4\mathcal I^2(\mathcal T(\mathcal U_\pm^m u)).
$$
\end{lemm}
%%%%%%%%%%%%%%%%%%%%%%%%%%%%%%%%%%%%%%%%%%%%%%%%%%%%%%%%%%%%%%%%%%%%%%%%%%%%%%%%
\begin{proof}
We have
$$
\pi_S^*(\mathcal U_\pm^{m+2}\Delta_\mp u)=
-\sum_{K\in\mathscr A^{m+2}\atop q\in\mathscr A} U^\pm_K U^\mp_q U^\mp_q u\, \mathbf e^*_{K}.
$$
Using~\eqref{e:comm-rel}, we compute for $K\in\mathscr A^{m+2}$ and $q\in\mathscr A$,
$$
\begin{gathered}\relax
[U^\pm_K,U^\mp_q]=\sum_{\ell=1}^{m+2} U^\pm_{k_1\dots k_{\ell-1}} [U^\pm_{k_\ell}, U^\mp_q] U^\pm_{k_{\ell+1}\dots k_{m+2}}\\
=2\sum_{\ell=1}^{m+2}\big( \delta_{q k_\ell}U^\pm_{\{\ell\to\}K}(\pm X+m-\ell+2)
+U^\pm_{k_1\dots k_{\ell-1}}R_{k_\ell+1,q+1}U^\pm_{k_{\ell+1}\dots k_{m+2}}\big)\\
=2\sum_{\ell=1}^{m+2}\bigg( U^\pm_{\{\ell\to\}K}\big(\delta_{q k_\ell}(\pm X+m-\ell+2)+R_{k_\ell+1,q+1}\big)
+\sum_{r=\ell+1}^{m+2}(\delta_{qk_r}U^\pm_{\{r\to\}K}-
\delta_{k_\ell k_r}U^\pm_{\{\ell\to,r\to q\}K})\bigg)\\
=2\sum_{\ell=1}^{m+2}\bigg(U^\pm_{\{\ell\to\}K}\big(\delta_{qk_\ell}(\pm X+m+1)+R_{k_\ell+1,q+1}\big)
-\sum_{r=\ell+1}^{m+2}\delta_{k_\ell k_r}U^\pm_{\{\ell\to ,r\to q\}K}\bigg).
\end{gathered}
$$
Since $\mathcal U_\pm^{m+1}u=0$, for $K\in\mathscr A^{m+2}$ and $q\in\mathscr A$ we have $U^\pm_Ku=[U_K^\pm,U^\mp_q]u=0$ and thus
$$
U^\pm_K U^\mp_q U^\mp_q u=[[U_K^\pm,U_q^\mp],U^\mp_q]u.
$$
We calculate
$$
\sum_{q\in\mathscr A}[\delta_{qk_\ell}(\pm X+m+1)+R_{k_\ell+1,q+1},U_q^\mp]=(n-2)U_{k_\ell}^\mp
$$
and thus for $K\in\mathscr A^{m+2}$,
$$
\begin{gathered}
\sum_{q\in\mathscr A} U^\pm_K U^\mp_q U^\mp_qu=2\sum_{\ell=1}^{m+2}\bigg(
[U^\pm_{\{\ell\to\}K},U^\mp_{k_\ell}](\pm X+m+n-1)\\
-\sum_{r=\ell+1}^{m+2}\delta_{k_\ell k_r}\sum_{q\in\mathscr A}[U^\pm_{\{\ell\to,r\to q\}K},U^\mp_q]
\bigg)u.
\end{gathered}
$$
Now, for $K\in\mathscr A^{m+2}$,
$$
\begin{gathered}
\sum_{\ell=1}^{m+2}[U^\pm_{\{\ell\to\}K},U^\mp_{k_\ell}](\pm X+m+n-1)u
=2\sum_{\ell,s=1\atop \ell\neq s}^{m+2}\bigg(\delta_{k_\ell k_s}U^\pm_{\{\ell\to,s\to\}K}(\pm X+m)\\
-\sum_{r=s+1\atop r\neq \ell}^{m+2}\delta_{k_sk_r}U^\pm_{\{s\to,r\to\}K}
\bigg)(\pm X+m+n-1)u\\
=2\sum_{\ell,r=1\atop \ell<r}^{m+2} \delta_{k_\ell k_r}U^\pm_{\{\ell\to,r\to\}K}
(\pm 2X+m)(\pm X+m+n-1)u.
\end{gathered}
$$
Furthermore, we have for $K\in\mathscr A^m$,
$$
\begin{gathered}
\sum_{q\in\mathscr A}[U^\pm_{qK},U^\mp_q]u=
2U^\pm_K\big((m+n)(\pm X+m)-m\big)u
-2\sum_{q\in\mathscr A}\sum_{s,p=1\atop s<p}^m \delta_{k_sk_p}U^\pm_{qq\{s\to,p\to\}K}u
\end{gathered}
$$
We finally compute
$$
\begin{gathered}
\sum_{q\in\mathscr A} U^\pm_K U^\mp_q U^\mp_qu=4\sum_{\ell,r=1\atop \ell<r}^{m+2}
\delta_{k_\ell k_r} U^\pm_{\{\ell\to,r\to\}K}X(2X\pm (n+2m-2))u\\
+4\sum_{q\in\mathscr A} \sum_{\ell,r=1\atop \ell<r}^{m+2}
\sum_{s,p=1\atop s<p;\, \{s,p\}\cap \{\ell,r\}=\emptyset}^{m+2}
\delta_{k_\ell k_r}\delta_{k_sk_p}U^\pm_{qq\{\ell\to,r\to,s\to,p\to\}K}u,
\end{gathered}
$$
which finishes the proof.
\end{proof}
%%%%%%%%%%%%%%%%%%%%%%%%%%%%%%%%%%%%%%%%%%%%%%%%%%%%%%%%%%%%%%%%%%%%%%%%%%%%%%%%
Arguing by induction using~\eqref{e:trart} and applying Lemma~\ref{l:penible} to $\Delta_\mp^ru$, we get
%%%%%%%%%%%%%%%%%%%%%%%%%%%%%%%%%%%%%%%%%%%%%%%%%%%%%%%%%%%%%%%%%%%%%%%%%%%%%%%%
\begin{lemm}
  \label{l:moins-penible}
Assume that $u\in\mathcal D'(S\mathbb H^{n+1})$ and $\mathcal U^{m+1}_\pm u=0$, $\mathcal T(\mathcal U^m_\pm u)=0$. Then for
each $r\geq 0$,
$$
\mathcal U_\pm^{m+2r}\Delta_\mp^r u=(-1)^r2^{2r}\bigg(\prod_{j=0}^{r-1} (\mathcal X\mp (m+j))\bigg)\bigg(\prod_{j=1}^r (2\mathcal X\pm (n-2j))\bigg)\mathcal I^r (\mathcal U^m_\pm u).
$$
Moreover, for $r\geq 1$
$$
\begin{gathered}
\mathcal T(\mathcal U_\pm^{m+2r}\Delta_\mp^ru)=(-1)^r2^{2r}r(n+2m+2r-2)\\\cdot\bigg(\prod_{j=0}^{r-1} (\mathcal X\mp (m+j))\bigg)\bigg(\prod_{j=1}^r (2\mathcal X\pm (n-2j))\bigg)
\mathcal I^{r-1}(\mathcal U^m_\pm u).
\end{gathered}
$$
\end{lemm}
%%%%%%%%%%%%%%%%%%%%%%%%%%%%%%%%%%%%%%%%%%%%%%%%%%%%%%%%%%%%%%%%%%%%%%%%%%%%%%%%
We are now ready to finish the proof of Lemma \ref{l:recovery}.
Following \eqref{decompositionoftensors}, we decompose $v$ as $v=\sum_{r=0}^{\lfloor m/2\rfloor}\mc{I}^r(v_r)$  with 
$v_r\in \mc{D'}(S\hh^{n+1}; \otimes_S^{m-2r}\mathcal E^*)$ and $\mathcal T(v_r)=0$.
Since $X$ commutes with $\mathcal T$ and $\mathcal I$, we find
$X v_r=\pm\lambda v_r$. Moreover, since $\mathcal U_\pm v=0$, we have $\mathcal U_\pm v_r=0$.
Put
$$
u_r:=(-\Delta_\mp)^r \mathcal V_\mp^{m-2r} v_r\in\mathcal D'(S\mathbb H^{n+1}).
$$
By Lemma~\ref{l:long-product} (applied to $v_r$) and Lemma~\ref{l:moins-penible} (applied to $\mathcal V_\mp^{m-2r}v_r$
and $m$ replaced by $m-2r$),
$$
\begin{gathered}
\mathcal U_\pm^m u_r
=2^{2r}\bigg(\prod_{j=0}^{r-1}(\lambda-(m-2r+j))\bigg)\bigg(\prod_{j=1}^r(2\lambda+n-2j)\bigg)\mathcal I^r(\mathcal U^{m-2r}_\pm
\mathcal V_\mp^{m-2r}v_r)
\\=2^m(m-2r)!\bigg(\prod_{j=n-1}^{n+m-2r-2}(\lambda+j)\bigg)
\bigg(\prod_{j=m-2r}^{m-r-1}(\lambda- j)\bigg)
\bigg(\prod_{j=1}^r(2\lambda + n-2j)\bigg)\mathcal I^r(v_r).
\end{gathered}
$$
Since $\lambda\not\in{1\over 2}\mathbb Z$, we see that
$v=\mathcal U_\pm^m u$, where $u$ is a linear combination of $u_0,\dots,u_{\lfloor m/2\rfloor}$.
The relation $Xu=\pm(\lambda- m)u$ follows immediately from~\eqref{e:comm-hor} and~\eqref{e:comm-lap}.
Finally, the equivariance property under $G$ follows similarly to Lemma~\ref{l:u-pm-gamma}.

%%%%%%%%%%%%%%%%%%%%%%%%%%%%%%%%%%%%%%%%%%%%%%%%%%%%%%%%%%%%%%%%%%%%%%%%%%%%%%%%
\subsection{Reduction to the conformal boundary}
\label{s:reduction}

We now describe the tensors $v\in\mathcal D'(S\mathbb H^{n+1};\otimes^m_S\mathcal E^*)$
that satisfy $\mathcal U_\pm v=0$ and $Xv=0$ via symmetric tensors on the conformal boundary $\mathbb S^n$.
For that we define the operators
$$
\mathcal Q_\pm:\mathcal D'(\mathbb S^n;\otimes^m (T^*\mathbb S^n))\to
\mathcal D'(S\mathbb H^{n+1};\otimes^m\mathcal E^*)
$$
by the following formula: if $w\in \CI(\mathbb S^{n};\otimes^m(T^*\mathbb S^n))$, 
we set for $\eta_i\in\mc{E}(x,\xi)$
\begin{equation}\label{defofQpm}
\mc{Q}_\pm w(x,\xi)(\eta_1,\dots,\eta_m):=(w\circ B_\pm(x,\xi))(\mc{A}_\pm^{-1}(x,\xi)\eta_1,\dots,\mc{A}_\pm^{-1}(x,\xi)\eta_m)
\end{equation}
where $\mathcal A_\pm(x,\xi):T_{B_\pm(x,\xi)}\mathbb S^n\to \mathcal E(x,\xi)$ is the parallel transport defined 
in~\eqref{e:Adef}, and we see that the operator \eqref{defofQpm}Ê extends continuously to $\mc{D}'(\mathbb S^n;\otimes^m_S (T^*\mathbb S^n))$ since the map $B_\pm:S\mathbb H^{n+1}\to \mathbb S^n$ defined in~\eqref{PhiandB} is a submersion, see~\cite[Theorem~6.1.2]{ho1}; the 
result can be written as $\mc{Q}_\pm w= (\otimes^m (\mc{A}_\pm^{-1})^T).w\circ B_\pm$ where $T$ means transpose. 
%%%%%%%%%%%%%%%%%%%%%%%%%%%%%%%%%%%%%%%%%%%%%%%%%%%%%%%%%%%%%%%%%%%%%%%%%%%%%%%%
\begin{lemm}\label{corresp_avec_bord}
The operator $\mathcal Q_\pm$ is a linear isomorphism from
$\mathcal D'(\mathbb S^n;\otimes^m_S(T^*\mathbb S^n))$ onto
the space
\begin{equation}
  \label{e:cab}
\{v\in\mathcal D'(S\mathbb H^{n+1};\otimes_S^m\mathcal E^*)\mid
\mathcal U_\pm v=0,\ \mathcal Xv=0\}.
\end{equation}
\end{lemm}
%%%%%%%%%%%%%%%%%%%%%%%%%%%%%%%%%%%%%%%%%%%%%%%%%%%%%%%%%%%%%%%%%%%%%%%%%%%%%%%%
\begin{proof}
It is clear that $\mathcal Q_\pm$ is injective. Next, we show that the image of $\mathcal Q_\pm$
is contained in~\eqref{e:cab}. For that it suffices to show that for
$w\in \CI(\mathbb S^n;\otimes_S^m(T^*\mathbb S^n))$, we have $\mathcal U_\pm (\mathcal Q_\pm w)=0$
and $\mathcal X(\mathcal Q_\pm w)=0$. We prove the first statement, the second one is established similarly.
Let $\gamma\in G$, $w_1,\dots,w_m\in \CI(\mathbb S^n;T\mathbb S^n)$, and 
$w_i^*=\cjg w_i, \cdot\cjd_{g_{\mathbb S^n}}$ be the duals through the metric. Then 
\[
\begin{gathered}
\mathcal Q_\pm (w^*_1\otimes\dots\otimes w^*_m)(\pi_S(\gamma))
=\\
\sum_{k_1,\dots,k_m=1}^n \Big(  \prod_{j=1}^m(w^*_j\circ B_\pm\circ \pi_S(\gamma))(\mc{A}^{-1}_\pm(\pi_S(\gamma))\gamma\cdot e_{k_j+1})\Big)\mathbf e_K^*(\gamma)=\\
(-1)^m\sum_{k_1,\dots,k_m=1}^n \Big(  \prod_{j=1}^m\cjg (\mc{A}_\pm.w_j\circ B_\pm)\circ \pi_S(\gamma), 
\gamma\cdot e_{k_j+1}\cjd_{M}\Big)\mathbf e_K^*(\gamma)
\end{gathered}
\]
where we have used \eqref{Apmiso} in the second identity. Now we have from \eqref{e:Adef}
\[\mc{A}_\pm(\pi_S(\gamma))\zeta=(0,\zeta)-\cjg (0,\zeta),\gamma\cdot e_0\cjd_M\gamma(e_0+e_1) 
\]
thus 
\[\mathcal Q_\pm (w^*_1\otimes\dots\otimes w^*_m)(\pi_S(\gamma))
=\sum_{k_1,\dots,k_m=1}^n \Big(  \prod_{j=1}^m\cjg (0,-w_j(B_\pm(\pi_S(\gamma)))), 
\gamma\cdot e_{k_j+1}\cjd_M\Big)\mathbf e_K^*(\gamma).
\]
Since $d(B_\pm\circ \pi_S)\cdot U^\pm_\ell=0$ by~\eqref{e:isot} and
$U^\pm_\ell (\gamma\cdot e_{k_j+1})=\gamma\cdot U^\pm_\ell \cdot e_{k_j+1}$ is a
multiple of $\gamma\cdot(e_0\pm e_1)=\Phi_\pm(\pi_S(\gamma))(1,B_\pm(\pi_S(\gamma)))$, we see that
$\mathcal U_\pm (\mathcal Q_\pm w)=0$ for all $w$.

It remains to show that for $v$ in~\eqref{e:cab}, we have
$v=\mathcal Q_\pm(w)$ for some $w$. For that, define
$$
\tilde v=(\otimes^m\mathcal A_\pm^{T})\cdot v\in \mathcal D'(S\mathbb H^{n+1};B_\pm^*(\otimes_S^m T^*\mathbb S^n))
$$
where $\mc{A}_\pm^T$ denotes the tranpose of $\mc{A}_\pm$.
Then $\mathcal U_\pm v=0$, $\mathcal Xv=0$ imply that $U^\pm_\ell (\pi_S^*\tilde v)=0$
and $X\tilde v=0$ (where to define differentiation we embed $T^*\mathbb S^n$ into $\mathbb R^{n+1}$).
Additionally, $R_{i+1,j+1}(\pi_S^*\tilde v)=0$,
therefore $\pi_S^*\tilde v$ is constant on the right cosets
of the subgroup $H_\pm\subset G$ defined in~\eqref{e:H-pm}. Since
$(B_\pm\circ\pi_S)^{-1}(B_\pm\circ\pi_S(\gamma))=\gamma H_\pm$, we see
that $\tilde v$ is the pull-back under $B_\pm$ of some $w\in\mathcal D'(\mathbb S^n;\otimes_S^m T^*\mathbb S^n)$,
and it follows that $v=\mathcal Q_\pm(w)$.
\end{proof}
In fact, using~\eqref{e:Adef} and
the expression of $\xi_\pm(x,\nu)$  in \eqref{defxi+} in terms of Poisson kernel,
it is not difficult to show that $\mc{Q}_\pm(w)$ belongs to a smaller space of \emph{tempered}
distributions: in the ball model, this can be described as the dual space to the Frechet space
of smooth sections of $\otimes^m ({^0}S\overline{\mathbb{B}^{n+1}})$ over $\overline{\mathbb{B}^{n+1}}$
which vanish to infinite order at the conformal boundary $\sph^n=\pl\overline{\mathbb{B}^{n+1}}$.

We finally give a useful criterion for invariance of $\mathcal Q_\pm(w)$ under the left action
of an element of $G$:
%%%%%%%%%%%%%%%%%%%%%%%%%%%%%%%%%%%%%%%%%%%%%%%%%%%%%%%%%%%%%%%%%%%%%%%%%%%%%%%%
\begin{lemm}
  \label{l:q-pm-gamma}
Take $\gamma\in G$ and let $w\in \mathcal D'(\mathbb S^n;\otimes^m_S(T^*\mathbb S^n))$. 	Take
$s\in\mathbb C$ and define $v=\Phi_\pm^s \mathcal Q_\pm(w)$. Then $v$ is equivariant under left multiplication
by $\gamma$, in the sense of~\eqref{e:v-gamma}, if and only if $w$ satisfies the condition
\begin{equation}
  \label{e:w-gamma}
L_\gamma^* w(\nu)=N_\gamma(\nu)^{-s-m} w(\nu),\quad
\nu\in\mathbb S^n.
\end{equation}
Here $L_\gamma(\nu)\in\mathbb S^n$ and $N_\gamma(\nu)>0$ are defined in~\eqref{e:defLgamma}.
\end{lemm}
%%%%%%%%%%%%%%%%%%%%%%%%%%%%%%%%%%%%%%%%%%%%%%%%%%%%%%%%%%%%%%%%%%%%%%%%%%%%%%%%
\begin{proof}
The lemma follows by a direct calculation from~\eqref{changeofPhi} and~\eqref{equivarianceA}.
\end{proof}
%%%%%%%%%%%%%%%%%%%%%%%%%%%%%%%%%%%%%%%%%%%%%%%%%%%%%%%%%%%%%%%%%%%%%%%%%%%%%%%%

%%%%%%%%%%%%%%%%%%%%%%%%%%%%%%%%%%%%%%%%%%%%%%%%%%%%%%%%%%%%%%%%%%%%%%%%%%%%%%%%
%                                  SECTION 5                                   %
%%%%%%%%%%%%%%%%%%%%%%%%%%%%%%%%%%%%%%%%%%%%%%%%%%%%%%%%%%%%%%%%%%%%%%%%%%%%%%%%
\section{Pollicott--Ruelle resonances}
\label{s:rr}

In this section, we first recall the results of Butterley--Liverani~\cite{BuLi} and Faure--Sj\"ostrand \cite{FaSj} on the Pollicott--Ruelle resonances for Anosov flows.
We next state several useful microlocal properties of these resonances and prove Theorem~\ref{t:main}, modulo properties
of Poisson kernels (Lemma~\ref{l:laplacian} and Theorem~\ref{t:laplacian}) which will be proved in Sections~\ref{s:laplacian} and~\ref{s:poisson-is}. Finally, we prove a pairing formula for resonances
and Theorem~\ref{t:noalg}.

%%%%%%%%%%%%%%%%%%%%%%%%%%%%%%%%%%%%%%%%%%%%%%%%%%%%%%%%%%%%%%%%%%%%%%%%%%%%%%%%
\subsection{Definition and properties}
\label{s:rr-1}

We follow the presentation of \cite{FaSj}; a more recent treatment using different technical tools
is also given in~\cite{DyZw}. We refer the reader to these two papers for the necessary notions of microlocal analysis.

Let $\mathcal M$ be a smooth compact manifold of dimension $2n+1$ and
$\varphi_t=e^{tX}$ be an Anosov flow on $\mathcal M$, generated by a smooth vector field $X$.
(In our case, $\mathcal M=SM$, $M=\Gamma\backslash \mathbb H^{n+1}$, and $\varphi_t$ is the geodesic flow~-- see
Section~\ref{s:rr-2}.) The Anosov property
is defined as follows: there exists a continuous splitting
\begin{equation}
  \label{e:anosov}
T_y \mathcal M=E_0(y)\oplus E_u(y)\oplus E_s(y),\quad
y\in\mathcal M;\quad
E_0(y):=\mathbb RX(y),
\end{equation}
invariant under $d\varphi_t$ and such that the stable/unstable subbundles $E_s,E_u\subset T\mathcal M$ satisfy for some
fixed smooth norm $|\cdot|$ on the fibers of $T\mathcal M$ and some constants $C$ and $\theta>0$,
\begin{equation}
  \label{e:anosov2}
\begin{gathered}
|d\varphi_t(y)v|\leq Ce^{-\theta t}|v|,\quad v\in E_s(y);\\
|d\varphi_{-t}(y)v|\leq Ce^{-\theta t}|v|,\quad v\in E_u(y).
\end{gathered}
\end{equation}
We make an additional assumption that $\mathcal M$ is equipped with a smooth measure $\mu$ which is invariant under
$\varphi_t$, that is, $\mathcal L_X \mu=0$.

We will use the dual decomposition to~\eqref{e:anosov}, given by
\begin{equation}
  \label{e:anosov-star}
T_y^*\mathcal M=E_0^*(y)\oplus E_u^*(y)\oplus E_s^*(y),\quad
y\in\mathcal M,
\end{equation}
where $E_0^*,E_u^*,E_s^*$ are dual to $E_0,E_s,E_u$ respectively (note that
$E_u,E_s$ are switched places), so for example $E_u^*(y)$ consists of covectors annihilating
$E_0(y)\oplus E_u(y)$.

Following~\cite[(1.24)]{FaSj}, we now consider for each $r\geq 0$ an \emph{anisotropic Sobolev space}
$$
\mathcal H^r(\mathcal M),\quad
\CI(\mathcal M)\subset \mathcal H^r(\mathcal M)\subset \mathcal D'(\mathcal M).
$$
Here we put $u:=-r,s:=r$ in~\cite[Lemma~1.2]{FaSj}. Microlocally near $E_u^*$, the space $\mathcal H^r$ is equivalent
to the Sobolev space $H^{-r}$, in the sense that for each pseudodifferential operator $A$ of order 0 whose wavefront set
is contained in a small enough conic neighborhood of $E_u^*$, the operator $A$ is bounded
$\mathcal H^r\to H^{-r}$ and $H^{-r}\to \mathcal H^r$. Similarly, microlocally near $E_s^*$, the space
$\mathcal H^r$ is equivalent to the Sobolev space $H^r$. We also have $\mathcal H^0=L^2$.
The first order differential operator $X$ admits a unique closed unbounded extension from $\CI$
to $\mathcal H^r$, see~\cite[Lemma~A.1]{FaSj}.

The following theorem, defining Pollicott--Ruelle resonances associated to $\varphi_t$, is due to Faure and Sj\"ostrand
\cite[Theorems~1.4 and~1.5]{FaSj}; see also~\cite[Section~3.2]{DyZw}.
\begin{theo}
  \label{t:ruelle}
Fix $r\geq 0$. Then the closed unbounded operator
$$
-X:\mathcal H^r(\mathcal M)\to\mathcal H^r(\mathcal M)
$$
has discrete spectrum in the region $\{\Re \lambda >-r/C_0\}$, for some constant $C_0$ independent
of $r$. The eigenvalues of $-X$ on $\mathcal H^r$, called \emph{Pollicott--Ruelle resonances}, and taken with
multiplicities, do not depend on the choice of $r$ as long as they lie in the appropriate region.
\end{theo}
We have the following criterion for Pollicott--Ruelle resonances which does not use the $\mathcal H^r$ spaces
explicitly:
%%%%%%%%%%%%%%%%%%%%%%%%%%%%%%%%%%%%%%%%%%%%%%%%%%%%%%%%%%%%%%%%%%%%%%%%%%%%%%%%
\begin{lemm}
  \label{l:criterion}
A number $\lambda\in\mathbb C$ is a Pollicott--Ruelle resonance of $X$ if and only the space
\begin{equation}
  \label{e:resonances}
\Res_X(\lambda):=
\{u\in\mathcal D'(\mathcal M)\mid (X+\lambda)u=0,\ \WF(u)\subset E_u^*\}
\end{equation}
is nontrivial. Here $\WF$ denotes the wavefront set, see for instance~\cite[Definition~1.6]{FaSj}.
The elements of~$\Res_X(\lambda)$ are called \textbf{resonant states} associated to $\lambda$ and
the dimension of this space is called \textbf{geometric multiplicity} of $\lambda$.
\end{lemm}
%%%%%%%%%%%%%%%%%%%%%%%%%%%%%%%%%%%%%%%%%%%%%%%%%%%%%%%%%%%%%%%%%%%%%%%%%%%%%%%%
\begin{proof}
Assume first that $\lambda$ is a Pollicott--Ruelle resonance. Take $r>0$ such that $\Re\lambda > - r/C_0$. Then
$\lambda$ is an eigenvalue of $-X$ on $\mathcal H^r$, which implies that there exists
nonzero $u\in \mathcal H^r$ such that $(X+\lambda) u=0$. By~\cite[Theorem~1.7]{FaSj}, we have
$\WF(u)\subset E_u^*$, thus $u$ lies in~\eqref{e:resonances}.

Assume now that $u\in\mathcal D'(\mathcal M)$ is a nonzero element of~\eqref{e:resonances}. For large enough $r$,
we have $\Re\lambda>-r/C_0$ and $u\in H^{-r}(\mathcal M)$. Since $\WF(u)\subset E_u^*$
and $\mathcal H^r$ is equivalent to $H^{-r}$ microlocally near $E_u^*$, we have
$u\in\mathcal H^r$. Together with the identity $(X+\lambda) u$, this shows that
$\lambda$ is an eigenvalue of $-X$ on $\mathcal H^r$ and thus a Pollicott--Ruelle resonance.
\end{proof}
%%%%%%%%%%%%%%%%%%%%%%%%%%%%%%%%%%%%%%%%%%%%%%%%%%%%%%%%%%%%%%%%%%%%%%%%%%%%%%%%
For each $\lambda$ with $\Re\lambda>-r/C_0$, the operator $X+\lambda:\mathcal H^r\to\mathcal H^r$ is Fredholm
of index zero on its domain;
this follows from the proof of Theorem~\ref{t:ruelle}. Therefore, $\dim\Res_X(\lambda)$ is equal to the dimension
of the kernel of the adjoint operator $X^*+\bar\lambda$ on the $L^2$ dual of $\mathcal H^r$, which we denote by
$\mathcal H^{-r}$. Since ${1\over i}X$ is symmetric on $L^2$, we see that $\Res_X(\lambda)$ has the same dimension
as the following space of \emph{coresonant states}
at $\lambda$:
\begin{equation}
  \label{e:resonances*}
\Res_{X^*}(\lambda):=\{u\in\mathcal D'(\mathcal M)\mid (X-\bar\lambda) u=0,\
\WF(u)\subset E_s^*\}.
\end{equation}
The main difference of~\eqref{e:resonances*} from~\eqref{e:resonances} is that the subbundle $E_s^*$ is used instead of $E_u^*$;
this can be justified by applying Lemma~\ref{l:criterion} to the vector field $-X$ instead of $X$, since the roles of the stable/unstable
spaces for the corresponding flow $\varphi_{-t}$ are reversed.

Note also that for any $\lambda,\lambda^*\in\mathbb C$, one can define a pairing
\begin{equation}
  \label{e:inner-product}
\langle u,u^*\rangle\in\mathbb C,\quad
u\in\Res_X(\lambda),\
u^*\in\Res_{X^*}(\lambda^*).
\end{equation}
One way of doing that is using the fact that wavefront sets of $u,u^*$ intersect only at the zero section,
and applying~\cite[Theorem~8.2.10]{ho1}.
An equivalent definition is noting that $u\in \mathcal H^r$ and $u^*\in\mathcal H^{-r}$
for $r>0$ large enough and using the duality of $\mathcal H^r$ and $\mathcal H^{-r}$.
Note that for $\lambda\neq\lambda^*$, we have $\langle u,u^*\rangle=0$; indeed,
$X(u\overline{u^*})=(\lambda^*-\lambda)u\overline {u^*}$ integrates to 0. The question
of computing the product $\langle u,u^*\rangle$ for $\lambda=\lambda^*$ is much more subtle
and related to algebraic multiplicities, see
Section~\ref{s:multiplicity}.

Since ${1\over i}X$ is self-adjoint on $L^2=\mathcal H^0$
(see~\cite[Appendix~A.1]{FaSj}), it has no eigenvalues on this space away from the real line; this implies
that there are no Pollicott--Ruelle resonances in the right half-plane. In other words, we have
%%%%%%%%%%%%%%%%%%%%%%%%%%%%%%%%%%%%%%%%%%%%%%%%%%%%%%%%%%%%%%%%%%%%%%%%%%%%%%%%
\begin{lemm}
  \label{l:upper}
The spaces $\Res_X(\lambda)$ and $\Res_{X^*}(\lambda)$ are trivial for $\Re\lambda>0$.
\end{lemm}
%%%%%%%%%%%%%%%%%%%%%%%%%%%%%%%%%%%%%%%%%%%%%%%%%%%%%%%%%%%%%%%%%%%%%%%%%%%%%%%%

Finally, we note that the results above apply to certain operators on vector bundles. More precisely, let
$\mathscr E$ be a smooth vector bundle over $\mathcal M$ and assume that $\mathcal X$ is a first order differential
operator on $\mathcal D'(\mathcal M;\mathscr E)$ whose principal part is given by $X$, namely
\begin{equation}
  \label{e:lilit}
\mathcal X(f \mathbf u)= f\, \mathcal X(\mathbf u)+(Xf) \mathcal X(\mathbf u),\quad
f\in \mathcal D'(\mathcal M),\
\mathbf u\in \CI(\mathcal M;\mathscr E).
\end{equation}
Assume moreover that $\mathscr E$ is endowed with an inner product $\langle\cdot,\cdot\rangle_{\mathscr E}$
and ${1\over i}\mathcal X$ is symmetric on $L^2$ with respect to this inner product and the measure $\mu$.
By an easy adaptation of the results of \cite{FaSj} (see~\cite{FaTs3} and~\cite{DyZw}),
one can construct
anisotropic Sobolev spaces $\mathcal H^r(\mathcal M;\mathscr E)$ and
Theorem~\ref{t:ruelle} and Lemmas~\ref{l:criterion}, \ref{l:upper}
apply to $\mathcal X$ on these spaces.

%%%%%%%%%%%%%%%%%%%%%%%%%%%%%%%%%%%%%%%%%%%%%%%%%%%%%%%%%%%%%%%%%%%%%%%%%%%%%%%%
\subsection{Proof of the main theorem}
\label{s:rr-2}

We now concentrate on the case
$$
\mathcal M=SM=\Gamma\backslash(S\mathbb H^{n+1}),\quad
M=\Gamma\backslash \mathbb H^{n+1},
$$
with $\varphi_t$ the geodesic flow. Here $\Gamma\subset G=\PSO(1,n+1)$ is a co-compact discrete subgroup with no fixed points,
so that $M$ is a compact smooth manifold. Henceforth we identify functions on the sphere bundle $SM$ with functions
on $S\mathbb H^{n+1}$ invariant under $\Gamma$,
and similar identifications will be used for other geometric objects. It is important to note
that \emph{the constructions of the previous sections, except those involving the conformal infinity, are invariant under left
multiplication by elements of $G$ and thus descend naturally to $SM$.}

The lift of the geodesic flow on $SM$ is the generator of the geodesic flow on $S\mathbb H^{n+1}$ (see Section~\ref{s:geodesic});
both are denoted $X$.
The lifts of the stable/unstable spaces $E_s,E_u$
to $S\mathbb H^{n+1}$ are given in~\eqref{e:stable-unstable}, and we see that~\eqref{e:anosov} holds
with $\theta=1$. The invariant measure $\mu$ on $SM$ is just the product of the volume measure on $M$
and the standard measure on the fibers of $SM$ induced by the metric.

Consider the bundle $\mathcal E$ on $SM$ defined in Section~\ref{s:E}.
Then for each $m$, the operator
$$
\mathcal X:\mathcal D'(SM;\otimes_S^m\mathcal E^*)\to \mathcal D'(SM;\otimes_S^m\mathcal E^*)
$$
defined in~\eqref{e:XX-def-2} satisfies~\eqref{e:lilit}
and ${1\over i}\mathcal X$ is symmetric. The results of Section~\ref{s:rr-1} apply both to $X$
and $\mathcal X$.

Recall the operator $\mathcal U_-$ introduced in Section~\ref{s:horocycl} and its powers,
for $m\geq 0$,
$$
\mathcal U_-^m:\mathcal D'(SM)\to \mathcal D'(SM;\otimes_S^m\mathcal E^*).
$$
The significance of $\mathcal U_-^m$ for Pollicott--Ruelle resonances is explained by the following
%%%%%%%%%%%%%%%%%%%%%%%%%%%%%%%%%%%%%%%%%%%%%%%%%%%%%%%%%%%%%%%%%%%%%%%%%%%%%%%%
\begin{lemm}
  \label{l:u-pm-used}
Assume that $\lambda\in\mathbb C$ is a Pollicott--Ruelle resonance of $X$ and
$u\in \Res_X(\lambda)$ is a corresponding resonant state as defined in~\eqref{e:resonances}. Then
$$
\mathcal U_-^m u=0\quad\text{for }m>-\Re\lambda.
$$
\end{lemm}
%%%%%%%%%%%%%%%%%%%%%%%%%%%%%%%%%%%%%%%%%%%%%%%%%%%%%%%%%%%%%%%%%%%%%%%%%%%%%%%%
\begin{proof}
By~\eqref{e:comm-hor},
$$
(\mathcal X+\lambda+m) \mathcal U_-^m u=0.
$$
Note also that $\WF(\mathcal U_-^mu)\subset E_u^*$ since $\WF(u)\subset E_u^*$
and $\mathcal U_-^m$ is a differential operator.
Since $\lambda+m$ lies in the right half-plane, it remains to apply Lemma~\ref{l:upper}
to $\mathcal U_-^mu$.
\end{proof}
%%%%%%%%%%%%%%%%%%%%%%%%%%%%%%%%%%%%%%%%%%%%%%%%%%%%%%%%%%%%%%%%%%%%%%%%%%%%%%%%
We can then use the operators $\mathcal U_-^m$ to split the resonance spectrum into bands:
%%%%%%%%%%%%%%%%%%%%%%%%%%%%%%%%%%%%%%%%%%%%%%%%%%%%%%%%%%%%%%%%%%%%%%%%%%%%%%%%
\begin{lemm}
  \label{l:bands-exhibited}
Assume that $\lambda\in\mathbb C\setminus {1\over 2}\mathbb Z$. Then
\begin{equation}
  \label{e:bands-1}
\dim\Res_X(\lambda)=\sum_{m\geq 0}\dim\Res_{\mathcal X}^m(\lambda+m),
\end{equation}
where
\begin{equation}
  \label{e:resonances2}
\Res_{\mathcal X}^m(\lambda):=\{v\in\mathcal D'(SM;\otimes_S^m\mathcal E^*)\mid (\mathcal X+\lambda)v=0,\
\mathcal U_-v=0,\ \WF(v)\subset E_u^*\}.
\end{equation}
The space $\Res_{\mathcal X}^m(\lambda)$ is trivial for $\Re\lambda>0$ (by Lemma~\ref{l:upper}).
If $\lambda\in {1\over 2}\mathbb Z$, then we have
\begin{equation}
  \label{e:bands-2}
\dim\Res_X(\lambda)\leq \sum_{m\geq 0}\dim\Res_{\mathcal X}^m(\lambda+m).
\end{equation}
\end{lemm}
%%%%%%%%%%%%%%%%%%%%%%%%%%%%%%%%%%%%%%%%%%%%%%%%%%%%%%%%%%%%%%%%%%%%%%%%%%%%%%%%
\begin{proof}
Denote for $m\geq 1$,
$$
V_m(\lambda):=\{u\in\mathcal D'(SM)\mid (X+\lambda) u=0,\
\mathcal U_-^m u=0,\ \WF(u)\subset E_u^*\}.
$$
Clearly, $V_m(\lambda)\subset V_{m+1}(\lambda)$. Moreover,
by Lemma~\ref{l:u-pm-used} we have $\Res_X(\lambda)=V_m(\lambda)$ for $m$ large enough
depending on $\lambda$. By~\eqref{e:comm-hor}, the operator $\mathcal U_-^m$
acts
\begin{equation}
  \label{e:u-restricted}
\mathcal U_-^m:V_{m+1}(\lambda)\to \Res_{\mathcal X}^m(\lambda+m),
\end{equation}
and the kernel of~\eqref{e:u-restricted} is exactly $V_m(\lambda)$, with the convention
that $V_0(\lambda)=0$. Therefore
$$
\dim V_{m+1}(\lambda)\leq \dim V_m(\lambda)+\dim \Res_{\mathcal X}^m(\lambda+m)
$$
and~\eqref{e:bands-2} follows.

To show~\eqref{e:bands-1}, it remains to prove that the operator~\eqref{e:u-restricted} is onto;
this follows from Lemma~\ref{l:recovery} (which does not enlarge the wavefront set of the resulting
distribution since it only employs differential operators in the proof).
\end{proof}
%%%%%%%%%%%%%%%%%%%%%%%%%%%%%%%%%%%%%%%%%%%%%%%%%%%%%%%%%%%%%%%%%%%%%%%%%%%%%%%%
The space $\Res_{\mathcal X}^m(\lambda+m)$ is called the space of \emph{resonant states at $\lambda$
associated to $m$th band}; later we see that most of the corresponding Pollicott--Ruelle resonances
satisfy $\Re\lambda=-n/2-m$. Similarly, we can describe $\Res_{X^*}(\lambda)$ via
the spaces $\Res_{\mathcal X^*}^m(\lambda+m)$, where
\begin{equation}
  \label{e:resonances2*}
\Res_{\mathcal X^*}^m(\lambda):=\{v\in\mathcal D'(SM;\otimes_S^m\mathcal E^*)\mid (\mathcal X-\bar\lambda) v=0,\
\mathcal U_+v=0,\ \WF(v)\subset E_s^*\};
\end{equation}
note that here $\mathcal U_+$ is used in place of $\mathcal U_-$.

We further decompose $\Res_{\mathcal X}^m(\lambda)$ using \emph{trace free} tensors:
%%%%%%%%%%%%%%%%%%%%%%%%%%%%%%%%%%%%%%%%%%%%%%%%%%%%%%%%%%%%%%%%%%%%%%%%%%%%%%%%
\begin{lemm}
  \label{l:bands2-exhibited}
Recall the homomorphisms $\mathcal T:\otimes_S^m\mathcal E^*\to \otimes_S^{m-2} \mathcal E^*$,
$\mathcal I:\otimes_S^m\mathcal E^*\to \otimes_S^{m-2} \mathcal E^*$
defined in Section~\ref{symtens} (we put $\mathcal T=0$ for $m=0,1$). Define the space
\begin{equation}
  \label{e:bands2}
\Res_{\mathcal X}^{m,0}(\lambda):=\{v\in \Res_{\mathcal X}^m(\lambda)\mid \mathcal T(v)=0\}.
\end{equation}
Then for all $m\geq 0$ and $\lambda$,
\begin{equation}
  \label{e:bands2-e}
\dim \Res_{\mathcal X}^m(\lambda)=\sum_{\ell=0}^{\lfloor {m\over 2}\rfloor} \dim\Res_{\mathcal X}^{m-2\ell,0}(\lambda).
\end{equation}
In fact,
\begin{equation}
  \label{e:bands2-ee}
\Res_{\mathcal X}^{m,0}(\lambda)=\bigoplus_{\ell=0}^{\lfloor {m\over 2}\rfloor} \mathcal I^\ell (\Res_{\mathcal X}^{m-2\ell,0}(\lambda)).
\end{equation}
\end{lemm}
%%%%%%%%%%%%%%%%%%%%%%%%%%%%%%%%%%%%%%%%%%%%%%%%%%%%%%%%%%%%%%%%%%%%%%%%%%%%%%%%
\begin{proof}
The identity~\eqref{e:bands2-ee} follows immediately from~\eqref{decompositionoftensors};
it is straightforward to see that the defining properties of $\Res_{\mathcal X}^m(\lambda)$
are preserved by the canonical tensorial operations involved. The identity~\eqref{e:bands2-e}
then follows since $\mathcal I$ is one to one by the paragraph following~\eqref{e:trart}.
\end{proof}
%%%%%%%%%%%%%%%%%%%%%%%%%%%%%%%%%%%%%%%%%%%%%%%%%%%%%%%%%%%%%%%%%%%%%%%%%%%%%%%%
The elements of $\Res_{\mathcal X}^{m,0}(\lambda)$ can be expressed via distributions on the conformal boundary
$\mathbb S^n$:
%%%%%%%%%%%%%%%%%%%%%%%%%%%%%%%%%%%%%%%%%%%%%%%%%%%%%%%%%%%%%%%%%%%%%%%%%%%%%%%%
\begin{lemm}
  \label{l:reduced}
Let $\mathcal Q_-$ be the operator defined in~\eqref{defofQpm}; recall that it is injective. If
$\pi_\Gamma:S\mathbb H^{n+1}\to SM$ is the natural projection map, then
$$
\pi_\Gamma^*\Res_{\mathcal X}^{m,0}(\lambda)=\Phi_-^{\lambda}\mathcal Q_-(\Bd^{m,0}(\lambda)),
$$
where $\Bd^{m,0}(\lambda)\subset \mathcal D'(\mathbb S^n;\otimes_S^m(T^*\mathbb S^n))$ consists
of all distributions $w$ such that $\mathcal T(w)=0$ and
\begin{equation}
  \label{e:reduced-gamma}
L_\gamma^* w(\nu)=N_{\gamma}(\nu)^{-\lambda-m}w(\nu),\quad
\nu\in\mathbb S^n,\ \gamma\in \Gamma,
\end{equation}
where $L_\gamma,N_\gamma$ are defined in~\eqref{e:defLgamma}. Similarly
$$
\pi_\Gamma^*\Res_{\mathcal X^*}^{m,0}(\lambda)=\Phi_+^{\bar\lambda}\mathcal Q_+(\Bd^{m,0}(\bar\lambda)),\quad
\Bd^{m,0}(\bar\lambda)=\overline{\Bd^{m,0}(\lambda)}.
$$
\end{lemm}
%%%%%%%%%%%%%%%%%%%%%%%%%%%%%%%%%%%%%%%%%%%%%%%%%%%%%%%%%%%%%%%%%%%%%%%%%%%%%%%%
\begin{proof}
Assume first that $w\in\Bd^{m,0}(\lambda)$ and put
$\tilde v=\Phi_-^{\lambda}\mathcal Q_-(w)$. Then by
Lemma~\ref{l:q-pm-gamma} and~\eqref{e:reduced-gamma}, $\tilde v$ is invariant under
$\Gamma$ and thus descends to a distribution
$v\in \mathcal D'(SM;\otimes_S^m\mathcal E^*)$. Since $X\Phi_-^{\lambda}=-\lambda \Phi_-^{\lambda}$
and $U_j^-(\Phi_-^{\lambda}\circ \pi_S)=0$ by~\eqref{invarofPhi} and~\eqref{e:isot}, and
$\mathcal X$ and $\mathcal U_-$ annihilate the image of $\mathcal Q_-$ by Lemma~\ref{corresp_avec_bord},
we have $(\mathcal X+\lambda) v=0$ and $\mathcal U_- v=0$. Moreover, by~\cite[Theorem~8.2.4]{ho1}
the wavefront set of $\tilde v$ is contained in the conormal bundle to the fibers
of the map $B_-$; by~\eqref{E*}, we see that $\WF(v)\subset E_u^*$. Finally,
$\mathcal T(v)=0$ since the map $\mathcal A_-(x,\xi)$ used in the definition of $\mathcal Q_-$ is
an isometry. Therefore,
$v\in\Res_{\mathcal X}^{m,0}(\lambda)$ and we proved the containment
$\pi_\Gamma^*\Res_{\mathcal X}^{m,0}(\lambda)\supset\Phi_-^{\lambda}\mathcal Q_-(\Bd^{m,0}(\lambda))$.
The opposite containment is proved by reversing this argument.
\end{proof}
%%%%%%%%%%%%%%%%%%%%%%%%%%%%%%%%%%%%%%%%%%%%%%%%%%%%%%%%%%%%%%%%%%%%%%%%%%%%%%%%

\noindent\textbf{Remark}.
It follows from the proof of Lemma~\ref{l:reduced} that the condition~$\WF(v)\subset E_u^*$ in~\eqref{e:resonances2}
is unnecessary. This could also be seen by applying~\cite[Theorem~18.1.27]{ho3} to the equations
$(\mathcal X+\lambda)v=0$, $\mathcal U_-v=0$, since $\mathcal X$ differentiates along the direction $E_0$,
$\mathcal U_-$ differentiates along the direction $E_u$ (see~\eqref{e:Udef1} and~\eqref{e:Udef4}),
and the annihilator of $E_0\oplus E_u$ (that is, the joint critical set of $\mathcal X+\lambda,\mathcal U_-$)
is exactly $E_u^*$.

It now remains to relate the space $\Bd^{m,0}(\lambda)$
to an eigenspace of the Laplacian on symmetric tensors.
For that, we introduce the following operator
obtained by integrating the corresponding elements of $\Res^{m,0}_{\mathcal X}(\lambda)$ along the
fibers of $\mathbb S^n$:
%%%%%%%%%%%%%%%%%%%%%%%%%%%%%%%%%%%%%%%%%%%%%%%%%%%%%%%%%%%%%%%%%%%%%%%%%%%%%%%%
\begin{defi}
  \label{d:poisson}
Take $\lambda\in\mathbb C$. The \textbf{Poisson operators}
$$
\mathscr P^\pm_\lambda:\mc{D}'(\sph^n; \otimes^m T^*\sph^n)\to \mc{C}^\infty(\hh^{n+1}; \otimes^m T^*\hh^{n+1})
$$
are defined by the formulas
\begin{equation}
  \label{e:poisson-def}
\begin{gathered}
\mathscr P^-_\lambda w(x)=\int_{S_x\mathbb H^{n+1}}\Phi_-(x,\xi)^{\lambda}\mathcal Q_-(w)(x,\xi)\,dS(\xi),\\ 
\mathscr P^+_\lambda w(x)=\int_{S_x\mathbb H^{n+1}}\Phi_+(x,\xi)^{\bar\lambda}\mathcal Q_+(w)(x,\xi)\,dS(\xi).
\end{gathered}
\end{equation}
Here integration of elements of $\otimes^m\mathcal E^*(x,\xi)$ is performed by embedding them in
$\otimes^m T^*_x \mathbb H^{n+1}$ using composition with the orthogonal projection
$T_x\mathbb H^{n+1}\to \mathcal E(x,\xi)$.
\end{defi}
%%%%%%%%%%%%%%%%%%%%%%%%%%%%%%%%%%%%%%%%%%%%%%%%%%%%%%%%%%%%%%%%%%%%%%%%%%%%%%%%
The operators $\mathscr P^\pm_\lambda$ are related by the identity
\begin{equation}
  \label{e:pid}
\overline{\mathscr P^\pm_\lambda w}=\mathscr P^\mp_\lambda \overline w.
\end{equation}
By Lemma~\ref{l:reduced}, $\mathscr P^-_\lambda$ maps $\Bd^{m,0}(\lambda)$
onto symmetric $\Gamma$-equivariant tensors, which can thus be considered as elements
of $\CI(M;\otimes_S^m T^*M)$. The relation with the Laplacian is given by the following
fact, proved in Section~\ref{s:poisson-basic}:
%%%%%%%%%%%%%%%%%%%%%%%%%%%%%%%%%%%%%%%%%%%%%%%%%%%%%%%%%%%%%%%%%%%%%%%%%%%%%%%%
\begin{lemm}
  \label{l:laplacian}
For each $\lambda$, the image of $\Bd^{m,0}(\lambda)$ under $\mathscr P^-_\lambda$ is contained in the 
eigenspace $\Eig^m(-\lambda(n+\lambda)+m)$, where
\begin{equation}
  \label{e:eigdef}
\Eig^m(\sigma):=\{f\in \CI(M;\otimes_S^m T^*M)\mid \Delta f= \sigma f,\
\nabla^* f=0,\
\mathcal T(f)=0\}.
\end{equation}
Here the trace $\mathcal T$ was defined in Section~\ref{symtens} and the Laplacian~$\Delta$ and the
divergence $\nabla^*$ are introduced in Section~\ref{s:laplacian-def}.
(A similar result for $\mathscr P^+_\lambda$ follows from~\eqref{e:pid}.)
\end{lemm}
%%%%%%%%%%%%%%%%%%%%%%%%%%%%%%%%%%%%%%%%%%%%%%%%%%%%%%%%%%%%%%%%%%%%%%%%%%%%%%%%
Furthermore, in Sections~\ref{s:poisson-basic} and~\ref{s:poisson-is} we show the following crucial
%%%%%%%%%%%%%%%%%%%%%%%%%%%%%%%%%%%%%%%%%%%%%%%%%%%%%%%%%%%%%%%%%%%%%%%%%%%%%%%%
\begin{theo}
  \label{t:laplacian}
Assume that $\lambda\notin\mathcal R_m$, where
\begin{equation}
\label{e:r-m-def0}
\mc{R}_m=\left\{\begin{array}{ll}
-\tfrac{n}{2}-\tfrac{1}{2}\nn_0 & \textrm{ if }n>1 \textrm{ or }m=0\\
-\tfrac{1}{2}\nn_0 & \textrm{ if }n=1 \textrm{ and }m>0  
\end{array}\right.
\end{equation}
Then the map
$\mathscr P^-_\lambda :\Bd^{m,0}(\lambda)\to \Eig^m(-\lambda(n+\lambda)+m)$
is an isomorphism.
\end{theo}
%%%%%%%%%%%%%%%%%%%%%%%%%%%%%%%%%%%%%%%%%%%%%%%%%%%%%%%%%%%%%%%%%%%%%%%%%%%%%%%%
\noindent\textbf{Remark}. In Theorem \ref{t:laplacian}, the set of exceptional points where we do not show isomorphism 
is not optimal but sufficient for our application
(we only need $\mathcal R_m\subset m-{n\over 2}-{1\over 2}\mathbb N_0$); we expect the exceptional set to be contained in 
$-n+1-\nn_0$. 
This result is known for functions, that is for $m=0$, with the exceptional set being $-n-\nn$. 
This was proved by Helgason, Minemura in the case of hyperfunctions on $\mathbb{S}^n$ and by Oshima--Sekiguchi~\cite{OsSe} and Schlichtkrull--Van Den Ban~\cite{VdBSc} for distributions; Grellier--Otal~\cite{GrOt} studied the sharp functional spaces on $\mathbb{S}^n$ of the boundary values of bounded eigenfunctions on $\hh^{n+1}$. The extension to $m>0$ does not seem to be known in the literature and is not trivial, it takes most of Sections \ref{s:laplacian} and \ref{s:poisson-is}.

We finally provide the following refinement of Lemma~\ref{l:bands-exhibited}, needed to handle
the case $\lambda\in (-n/2,\infty)\cap {1\over 2}\mathbb Z$:
%%%%%%%%%%%%%%%%%%%%%%%%%%%%%%%%%%%%%%%%%%%%%%%%%%%%%%%%%%%%%%%%%%%%%%%%%%%%%%%%
\begin{lemm}
  \label{l:better}
Assume that $\lambda\in -{n\over 2}+{1\over 2}\mathbb N$. If $\lambda\in- 2\mathbb N$, then
$$
\dim\Res_X(\lambda)=\sum_{m\geq 0\atop m\neq -\lambda}
\dim\Res^{m}_{\mathcal X}(\lambda+m).
$$
If $\lambda\notin -2\mathbb N$, then~\eqref{e:bands-1} holds.
\end{lemm}
%%%%%%%%%%%%%%%%%%%%%%%%%%%%%%%%%%%%%%%%%%%%%%%%%%%%%%%%%%%%%%%%%%%%%%%%%%%%%%%%
\begin{proof}
We use the proof of Lemma~\ref{l:bands-exhibited}. We first show that for $m$ odd or $\lambda\neq -m$,
\begin{equation}
  \label{e:lop-1}
\mathcal U_-^m(V_{m+1}(\lambda))=\Res^{m}_{\mathcal X}(\lambda+m).
\end{equation}
Using~\eqref{e:bands2-ee}, it suffices to prove that for $0\leq \ell\leq {m\over 2}$, the space
$\mathcal I^\ell(\Res^{m-2\ell,0}_{\mathcal X}(\lambda+m))$ is contained in $\mathcal U_-^m(V_{m+1}(\lambda))$. This follows from the proof
of Lemma~\ref{l:recovery} as long as
$$
\begin{gathered}
\lambda+m\notin \mathbb Z\cap \big( [2\ell+2-n-m,1-n]\cup [m-2\ell,m-\ell-1]\big),\\
\lambda+m+{n\over 2}\notin \mathbb Z\cap [1,\ell];
\end{gathered}
$$
using that $\lambda>-{n\over 2}$, it suffices to prove that
\begin{equation}
  \label{e:aslongas}
\lambda\notin \mathbb Z\cap [-2\ell, -\ell-1].
\end{equation}
On the other hand by Lemma~\ref{l:reduced}, Theorem~\ref{t:laplacian}, and Lemma~\ref{bottomsp}, if $\ell < {m\over 2}$ and
the space $\Res^{m-2\ell,0}_{\mathcal X}(\lambda+m)$ is nontrivial, then
$$
-\Big(\lambda+m+{n\over 2}\Big)^2+{n^2\over 4}+m-2\ell\geq m-2\ell+n-1,
$$
implying
\begin{equation}
  \label{e:bohoho}
\Big| \lambda+m+{n\over 2}\Big|\leq \Big|{n\over 2}-1\Big|
\end{equation}
and~\eqref{e:aslongas} follows. For the case $\ell={m\over 2}$, since $\Delta\geq 0$ on functions, we have
$$
-\Big(\lambda+m+{n\over 2}\Big)^2+{n^2\over 4}\geq 0,
$$
which implies that $\lambda\leq -m$ and thus~\eqref{e:aslongas} holds unless $\lambda=-m$.

It remains to consider the case when $m=2\ell$ is even and $\lambda=-m$. We have
$$
\Res^m_{\mathcal X}(0)=\mathcal I^{\ell}(\Res^{0,0}_{\mathcal X}(0));
$$
that is, $\Res^{m-2\ell',0}_{\mathcal X}(0)$ is trivial for $\ell'<{m\over 2}$. For $n>1$,
this follows immediately from~\eqref{e:bohoho}, and for $n=1$, since the bundle
$\mathcal E^*$ is one-dimensional we get $\Res^{m',0}_{\mathcal X}(\lambda)=0$ for $m'\geq 2$.
Now, $\Res^{0,0}_{\mathcal X}(0)=\Res^0_{\mathcal X}(0)$ corresponds via Lemma~\ref{l:reduced} and Theorem~\ref{t:laplacian}
to the kernel of the scalar Laplacian, that is, to the space of constant functions.
Therefore, $\Res^{0,0}_{\mathcal X}$ is one-dimensional and it is spanned
by the constant function $1$ on $SM$; it follows that $\Res^m_{\mathcal X}(0)$
is spanned by $\mathcal I^{\ell}(1)$.
However, by Lemma~\ref{l:tough-adjoint}, for each $u\in \mathcal D'(SM)$,
$$
\langle \mathcal I^{\ell}(1),\mathcal U_-^m u\rangle_{L^2}=(-1)^m\langle \mathcal V_-^m\mathcal I^{\ell}(1),u\rangle_{L^2}=0.
$$
Since $\mathcal U_-^m(V_{m+1}(\lambda))\subset \Res^m_{\mathcal X}(0)$,
we have $\mathcal U_-^m=0$ on $V_{m+1}(\lambda)$,
which implies that $V_{m+1}(\lambda)=V_m(\lambda)$, finishing the proof.
\end{proof}
%%%%%%%%%%%%%%%%%%%%%%%%%%%%%%%%%%%%%%%%%%%%%%%%%%%%%%%%%%%%%%%%%%%%%%%%%%%%%%%%

To prove Theorem~\ref{t:main}, it now suffices to combine Lemmas~\ref{l:bands-exhibited}--\ref{l:better}
with Theorem~\ref{t:laplacian}.

%%%%%%%%%%%%%%%%%%%%%%%%%%%%%%%%%%%%%%%%%%%%%%%%%%%%%%%%%%%%%%%%%%%%%%%%%%%%%%%%
\subsection{Resonance pairing and algebraic multiplicity}
\label{s:multiplicity}

In this section, we prove Theorem~\ref{t:noalg}. The key component is a pairing formula which states that
the inner product between a resonant and a coresonant state, defined in~\eqref{e:inner-product}, is determined
by the inner product between the corresponding eigenstates of the Laplacian. The nondegeneracy of the resulting inner product
as a bilinear operator on $\Res_X(\lambda)\times \Res_{X^*}(\lambda)$ for $\lambda\not\in {1\over 2}\mathbb Z$
immediately implies the fact that the algebraic and geometric multiplicities of $\lambda$ coincide
(that is, $X+\lambda$ does not have any nontrivial Jordan cells).

To state the pairing formula, we first need a decomposition of the space $\Res_X(\lambda)$, which
is an effective version of the formulas~\eqref{e:bands-1} and~\eqref{e:bands2-e}.
Take $m\geq 0$, $\ell\leq \lfloor m/2\rfloor$,
$w\in \Bd^{m-2\ell,0}(\lambda)$. Let $\mathcal I$ be the operator defined in Section~\ref{symtens}. Then
\eqref{e:bands2-ee} and Lemma~\ref{l:reduced} show that
$$
\Res_{\mathcal X}^m(\lambda)=\bigoplus_{\ell=0}^{\lfloor m/2\rfloor}
\mathcal I^\ell(\Res_{\mathcal X}^{m-2\ell,0}(\lambda))=
\bigoplus_{\ell=0}^{\lfloor m/2\rfloor} \mathcal I^\ell(\Phi_-^{\lambda}\mathcal Q_-(\Bd^{m-2\ell, 0}(\lambda))).
$$
Next, let
$$
\mathcal V_\pm^m:\mathcal D'(SM;\otimes^m_S \mathcal E^*)\to\mathcal D'(SM),\quad
\Delta_\pm:\mathcal D'(SM)\to\mathcal D'(SM)
$$
be the operators introduced in Section~\ref{s:hor-inv}. Then the proofs of Lemma~\ref{l:bands-exhibited} and Lemma~\ref{l:recovery}
show that for $\lambda\not\in {1\over 2}\mathbb Z$,
\begin{equation}
  \label{e:layout}
\begin{gathered}
\Res_X(\lambda)=\bigoplus_{m\geq 0}\bigoplus_{\ell=0}^{\lfloor m/2\rfloor} V_{m\ell}(\lambda),\quad
\Res_{X^*}(\lambda)=\bigoplus_{m\geq 0}\bigoplus_{\ell=0}^{\lfloor m/2\rfloor}V_{m\ell}^*(\lambda);\\
V_{m\ell}(\lambda):=
\Delta_+^\ell\mathcal  V_+^{m-2\ell}(\Phi_-^{\lambda+m}\mathcal Q_-(\Bd^{m-2\ell, 0}(\lambda+m))),\\
V_{m\ell}^*(\lambda):=
\Delta_-^\ell\mathcal  V_-^{m-2\ell}(\Phi_+^{\bar\lambda+m}\mathcal Q_+(\overline{\Bd^{m-2\ell, 0}(\lambda+m)})),
\end{gathered}
\end{equation}
and the operators in the definitions of $V_{m\ell}(\lambda),V^*_{m\ell}(\lambda)$ are one-to-one
on the corresponding spaces. By the proof of Lemma~\ref{l:better}, the decomposition~\eqref{e:layout}
is also valid for $\lambda\in (-n/2,\infty)\setminus (-2\mathbb N)$; for
$\lambda\in (-n/2,\infty)\cap (-2\mathbb N)$, we have
\begin{equation}
  \label{e:layout2}
\begin{gathered}
\Res_X(\lambda)=\bigoplus_{m\geq 0\atop m\neq -\lambda}\bigoplus_{\ell=0}^{\lfloor m/2\rfloor} V_{m\ell}(\lambda),\quad
\Res_{X^*}(\lambda)=\bigoplus_{m\geq 0\atop m\neq -\lambda}\bigoplus_{\ell=0}^{\lfloor m/2\rfloor}V_{m\ell}^*(\lambda).
\end{gathered}
\end{equation}
We can now state the pairing formula:
%%%%%%%%%%%%%%%%%%%%%%%%%%%%%%%%%%%%%%%%%%%%%%%%%%%%%%%%%%%%%%%%%%%%%%%%%%%%%%%%
\begin{lemm}
  \label{l:the-pairing}
Let $\lambda\not\in -{n\over 2}-{1\over 2}\mathbb N_0$ and $u\in\Res_X(\lambda)$, $u^*\in\Res_{X^*}(\lambda)$.
Let
$\langle u,u^*\rangle_{L^2(SM)}$ be defined by~\eqref{e:inner-product}. Then:

1. If $u\in V_{m\ell}(\lambda), u^*\in V^*_{m'\ell'}(\lambda)$, and $(m,\ell)\neq (m',\ell')$, then $\langle u,u^*\rangle_{L^2(SM)}=0$.

2. If $u\in V_{m\ell}(\lambda)$, $u^*\in V^*_{m\ell}(\lambda)$ and $w\in \Bd^{m-2\ell,0}(\lambda+m)$,
$w^*\in\overline{\Bd^{m-2\ell,0}(\lambda+m)}$ are the elements generating $u,u^*$ according to~\eqref{e:layout},
then 
\begin{equation}
  \label{e:the-pairing}
\langle u,u^*\rangle_{L^2(SM)}= c_{m\ell}(\lambda)\langle \mathscr P^-_{\lambda+m}(w), \mathscr P^+_{\lambda+m}(w^*)\rangle_{L^2(M)},
\end{equation}
where
$$
\begin{gathered}
c_{m\ell}(\lambda)=2^{m+2\ell-n}\pi^{-1-{n\over 2}}\ell!(m-2\ell)!\sin\Big(\pi\Big({n\over 2}+\lambda\Big)\Big)\\
\cdot {\Gamma(m+{n\over 2}-\ell)
\Gamma(\lambda+n+2m-2\ell)\Gamma(-\lambda-\ell)\Gamma(-\lambda-m-{n\over 2}+\ell+1)
\over \Gamma(m+{n\over 2}-2\ell)\Gamma(-\lambda-2\ell)}.
\end{gathered}
$$
and under the conditions (i) either $\lambda\notin-2\mathbb N$ or $m\neq-\lambda$
and (ii) $V_{m\ell}(\lambda)$ is nontrivial, we have $c_{m\ell}(\lambda)\neq 0$.
\end{lemm}
%%%%%%%%%%%%%%%%%%%%%%%%%%%%%%%%%%%%%%%%%%%%%%%%%%%%%%%%%%%%%%%%%%%%%%%%%%%%%%%%
\noindent\textbf{Remarks}. (i)
The proofs below are rather technical, and
it is suggested that the reader start with the case of resonances in the first band, $m=\ell=0$, which preserves
the essential analytic difficulties of the proof but considerably reduces the amount of calculations
needed (in particular, one can
go immediately to Lemma~\ref{l:pairing-key}, and
the proof of this lemma for the case $m=\ell=0$ does not involve the operator $\mathscr C_\eta$). We have
$$
c_{00}(\lambda)=(4\pi)^{-n/2}{\Gamma(n+\lambda)\over\Gamma({n\over 2}+\lambda)}.
$$

\noindent (ii)
In the special case of $n=1,m=\ell=0$, Lemma~\ref{l:the-pairing} is a corollary of~\cite[Theorem~1.2]{an-ze},
where the product $uu^*\in\mathcal D'(SM)$ lifts to a Patterson--Sullivan distribution on $S\mathbb H^2$.
In general, if $|\Re\lambda|\leq C$ and $\Im\lambda\to\infty$, then $c_{m\ell}(\lambda)$ grows like $|\lambda|^{{n\over 2}+m}$.

Lemma~\ref{l:the-pairing} immediately gives
%%%%%%%%%%%%%%%%%%%%%%%%%%%%%%%%%%%%%%%%%%%%%%%%%%%%%%%%%%%%%%%%%%%%%%%%%%%%%%%%
\begin{proof}[Proof of Theorem~\ref{t:noalg}]
By Theorem~\ref{t:laplacian}, we know that
$$
\mathscr P^-_\lambda:\Bd^{m-2\ell,0}(\lambda+m)\to\Eig^{m-2\ell}(-(\lambda+m+n/2)^2+n^2/4+m-2\ell)
$$
is an isomorphism. Given~\eqref{e:pid}, we also get the isomorphism
$$
\mathscr P^+_\lambda:\overline{\Bd^{m-2\ell,0}(\lambda+m)}\to \Eig^{m-2\ell}(-(\lambda+m+n/2)^2+n^2/4+m-2\ell).
$$
Here we used that the target space is invariant under complex conjugation. By Lemma~\ref{l:the-pairing},
the bilinear product
\begin{equation}
  \label{e:biprod}
\Res_X(\lambda)\times \Res_{X^*}(\lambda)\to\mathbb C,\quad
(u,u^*)\mapsto\langle u,u^*\rangle_{L^2(SM)}
\end{equation}
is nondenegerate, since the $L^2(M)$ inner product restricted to $\Eig^{m-2\ell}(-(\lambda+m+n/2)^2+n^2/4+m-2\ell)$
is nondegenerate for all $m,\ell$.

Assume now that $\tilde u\in\mathcal D'(SM)$ satisfies $(X+\lambda)^2\tilde u=0$ and $\tilde u\in\mathcal H^r$ for some $r$, $\Re\lambda>-r/C_0$; we need
to show that $(X+\lambda)\tilde u=0$. Put $u:=(X+\lambda)\tilde u$. Then $u\in\Res_X(\lambda)$. However, $u$ also lies
in the image of $X+\lambda$ on $\mathcal H^r$, therefore
we have $\langle u,u^*\rangle=0$ for each $u^*\in\Res_{X^*}(\lambda)$. Since the product~\eqref{e:biprod} is nondegenerate,
we see that $u=0$, finishing the proof.
\end{proof}
%%%%%%%%%%%%%%%%%%%%%%%%%%%%%%%%%%%%%%%%%%%%%%%%%%%%%%%%%%%%%%%%%%%%%%%%%%%%%%%%
In the remaining part of this section, we prove Lemma~\ref{l:the-pairing}. Take some $m,m',\ell,\ell'\geq 0$ such that 
$2\ell\leq m$, $2\ell'\leq m'$, and consider $u\in V_{m\ell}(\lambda)$, $u^*\in V_{m'\ell'}^*(\lambda)$ given by
$$
u=\Delta_+^\ell \mathcal V_+^{m-2\ell} v,\quad
u^*=\Delta_-^{\ell'}\mathcal V_-^{m'-2\ell'} v^*,
$$
where for some $w\in \Bd^{m-2\ell,0}(\lambda+m)$ and $w^*\in \overline{\Bd^{m'-2\ell',0}(\lambda+m')}$,
$$
v=\Phi_-^{\lambda+m}\mathcal Q_-(w)\in\mathcal \Res^{m-2\ell,0}_{\mathcal X}(\lambda+m),\quad
v^*=\Phi_+^{\bar\lambda+m'}\mathcal Q_+(w^*)\in\Res^{m'-2\ell',0}_{\mathcal X^*}(\lambda+m').
$$
Using Lemma~\ref{l:tough-adjoint} and the fact that $\Delta_\pm$ are symmetric, we get
$$
\langle u,u^*\rangle_{L^2(SM)}=(-1)^{m'}\langle \mathcal U_-^{m'-2\ell'}\Delta_-^{\ell'}\Delta_+^\ell
\mathcal V_+^{m-2\ell} v,v^*\rangle_{L^2(SM;\otimes^{m'-2\ell'}\mathcal E^*)}.
$$
By Lemmas~\ref{l:long-product} and~\ref{l:moins-penible}, we have $\mathcal U_-^{m+1}\Delta_+^\ell\mathcal V_+^{m-2\ell} v=0$.
Therefore, if $m'>m$, we derive that $\langle u,u^*\rangle_{L^2(SM)}=0$; by swapping $u$ and $u^*$,
one can similarly handle the case $m'<m$. We therefore assume that $m=m'$. Then by Lemmas~\ref{l:long-product} and~\ref{l:moins-penible}
(see the proof of Lemma~\ref{l:recovery}),
$$
\begin{gathered}
(-1)^{\ell+\ell'}\mathcal U_-^{m-2\ell'}\Delta_-^{\ell'}\Delta_+^\ell \mathcal V_+^{m-2\ell} v
=\mathcal T^{\ell'}\mathcal U_-^m(-\Delta_+)^\ell \mathcal V_+^{m-2\ell} v\\
=2^{m+\ell}(m-2\ell)!{\Gamma(\lambda+n+2m-2\ell-1)\Gamma(-\lambda-\ell)\Gamma(-\lambda-m-{n\over 2}+\ell+1)
\over \Gamma(\lambda+m+n-1)\Gamma(-\lambda-2\ell)\Gamma(-\lambda-m-{n\over 2}+1)
}\mathcal T^{\ell'}\mathcal I^{\ell}v.
\end{gathered}
$$
If $\ell'>\ell$, this implies that $\langle u,u^*\rangle_{L^2(SM)}=0$, and the
case $\ell'<\ell$ is handled similarly. (Recall that $\mathcal T(v)=0$.)
We therefore assume that $m=m',\ell=\ell'$. In this case, by~\eqref{e:trart},
$$
\mathcal T^{\ell}\mathcal I^{\ell}v=2^\ell\ell!{\Gamma(m+{n\over 2}-\ell)\over\Gamma(m+{n\over 2}-2\ell)}v,
$$
which implies that
$$
\begin{gathered}
\langle u,u^*\rangle_{L^2(SM)}=(-2)^{m+2\ell}
\ell!(m-2\ell)!
{\Gamma(m+{n\over 2}-\ell)
\Gamma(\lambda+n+2m-2\ell-1)\over \Gamma(m+{n\over 2}-2\ell)\Gamma(\lambda+n+m-1)}\\
{\Gamma(-\lambda-\ell)\Gamma(-\lambda-m-{n\over 2}+\ell+1)
\over \Gamma(-\lambda-2\ell)\Gamma(-\lambda-m-{n\over 2}+1)}
\langle v,v^*\rangle_{L^2(SM;\otimes^{m-2\ell}\mathcal E^*)}.
\end{gathered}
$$
Note that under assumptions (i) and (ii) of Lemma~\ref{l:the-pairing}, the coefficient in
the formula above is nonzero, see the proof of Lemma~\ref{l:better}.

It then remains to prove the following identity (note that the coefficient there is
nonzero for $\lambda\notin \mathbb Z$ or $\Re\lambda>m-{n\over 2}$):
%%%%%%%%%%%%%%%%%%%%%%%%%%%%%%%%%%%%%%%%%%%%%%%%%%%%%%%%%%%%%%%%%%%%%%%%%%%%%%%%
\begin{lemm}
  \label{l:pairing-key}
Assume that $v\in\Res_{\mathcal X}^{m,0}(\lambda)$ and $v^*\in\Res_{\mathcal X^*}^{m,0}(\lambda)$.
Define
$$
f(x):=\int_{S_xM} v(x,\xi)\,dS(\xi),\quad
f^*(x):=\int_{S_xM} v^*(x,\xi)\,dS(\xi),
$$
where integration of tensors is understood as in Definition~\ref{d:poisson}.
If $\lambda\not \in -({n\over 2}+\mathbb N_0)$, then
$$
\langle f,f^*\rangle_{L^2(M;\otimes^m T^*M)}=2^n\pi^{n\over 2}{\Gamma({n\over 2}+\lambda)
\over (n+\lambda+m-1)\Gamma(n-1+\lambda)}\langle v,v^*\rangle_{L^2(SM;\otimes^m\mathcal E^*)}.
$$
\end{lemm}
%%%%%%%%%%%%%%%%%%%%%%%%%%%%%%%%%%%%%%%%%%%%%%%%%%%%%%%%%%%%%%%%%%%%%%%%%%%%%%%%
\begin{proof}
We write
\begin{equation}
  \label{e:intie-1}
\langle f,f^*\rangle_{L^2(M;\otimes^m T^*M)}
=\int_{S^2 M} \langle v(y,\eta_-),\overline{v^*(y,\eta_+)}\rangle_{\otimes^m T^*_y M}\,dyd\eta_-d\eta_+
\end{equation}
where the bundle $S^2M$ is given by
$$
S^2M=\{(y,\eta_-,\eta_+)\mid y\in M,\ \eta_\pm\in S_yM\}.
$$
Define also
$$
S^2_\Delta M=\{(y,\eta_-,\eta_+)\in S^2M\mid \eta_-+\eta_+\neq 0\}.
$$
On the other hand
\begin{equation}
  \label{e:intie-2}
\langle v,v^*\rangle_{L^2(SM;\otimes^m\mathcal E^*)}
=\int_{SM} \langle v(x,\xi),\overline{v^*(x,\xi)}\rangle_{\otimes^m\mathcal E^*(x,\xi)}\,dxd\xi.
\end{equation}
The main idea of the proof is to reduce~\eqref{e:intie-1} to~\eqref{e:intie-2} by applying the
coarea formula to a correctly chosen map $S^2_\Delta M\to SM$. More precisely, 
consider the following map $\Psi:\mathcal E\to S^2_{\Delta}\mathbb H^{n+1}$:
for $(x,\xi)\in S\mathbb H^{n+1}$ and $\eta\in\mathcal E(x,\xi)$, define
$\Psi(x,\xi,\eta):=(y,\eta_-,\eta_+)$, with
$$
\begin{pmatrix} y\\\eta_-\\\eta_+ \end{pmatrix}
=A\big(|\eta|^2\big)\begin{pmatrix} x\\\xi\\\eta\end{pmatrix},\quad
A(s)=\begin{pmatrix}
\sqrt{s+1}&0&1\\
{s\over\sqrt{s+1}}&{1\over \sqrt{s+1}}&1\\
-{s\over\sqrt{s+1}}&{1\over \sqrt{s+1}}&-1
\end{pmatrix}.
$$
%%%%%
\begin{figure}
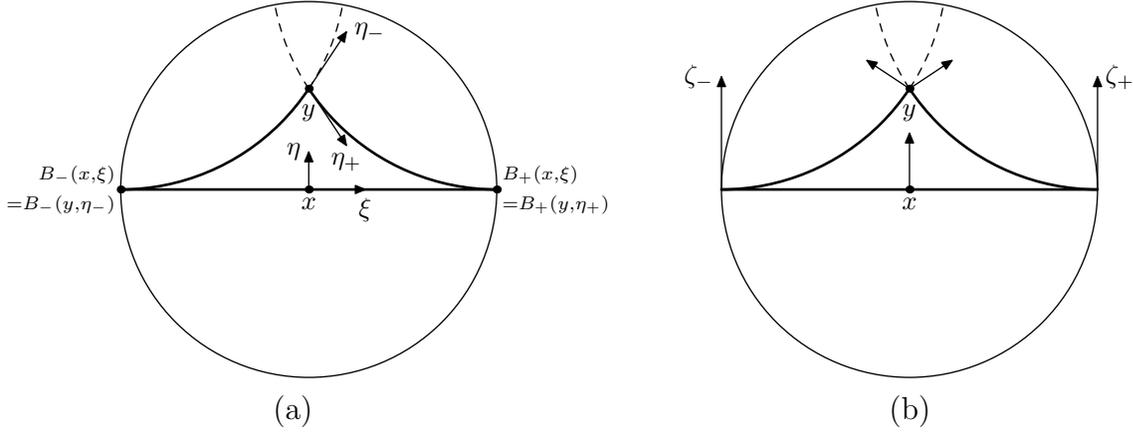

\includegraphics{rrh.1}
\qquad
\includegraphics{rrh.2}
\hbox to\hsize{\qquad\hss (a) \hss\hss \qquad\qquad(b)\hss}
\caption{(a) The map $\Psi:(x,\xi,\eta)\mapsto (y,\eta_-,\eta_+)$.
(b) The vectors $\mathcal A_\pm(x,\xi)\zeta_\pm$ (equal in the case drawn)
and $\mathcal A_\pm(y,\eta_\pm)\zeta_\pm$.}
\label{f:Psi}
\end{figure}
%%%%%
Note that, with $|\eta|$ denoting the Riemannian length of $\eta$ (that is, $|\eta|^2=-\langle \eta,\eta\rangle_M$),
$$
\Phi_\pm(y,\eta_\pm)={\Phi_\pm(x,\xi)\over \sqrt{1+|\eta|^2}},\quad
B_\pm(y,\eta_\pm)=B_\pm(x,\xi),\quad
|\eta_++\eta_-|={2\over\sqrt{1+|\eta|^2}}.
$$
Also,
$$
\det A(s)=-{2\over s+1},\quad
A(s)^{-1}=\begin{pmatrix}
\sqrt{s+1}&-{\sqrt{s+1}\over 2}&{\sqrt{s+1}\over 2}\\
0&{\sqrt{s+1}\over 2}&{\sqrt{s+1}\over 2}\\
-s&{s+1\over 2}&-{s+1\over 2}
\end{pmatrix}.
$$
The map $\Psi$ is a diffeomorphism; the inverse is given by the formulas
$$
x={2y+\eta_+-\eta_-\over |\eta_++\eta_-|},\quad
\xi={\eta_++\eta_-\over |\eta_++\eta_-|},\quad
\eta={2(\eta_--\eta_+)-|\eta_+-\eta_-|^2 y\over |\eta_++\eta_-|^2}.
$$
The map $\Psi^{-1}$ can be visualized as follows (see Figure~\ref{f:Psi}(a)): given $(y,\eta_-,\eta_+)$,
the corresponding tangent vector $(x,\xi)$ is the closest to $y$ point on the geodesic going from $\nu_-=B_-(y,\eta_-)$
to $\nu_+=B_+(y,\eta_+)$ and the vector $\eta$ measures both the distance between $x$ and $y$ and
the direction of the geodesic from $x$ to $y$. The exceptional set $\{\eta_++\eta_-=0\}$
corresponds to $|\eta|=\infty$.

A calculation using~\eqref{e:Adef} shows that for $\zeta_\pm\in T_{B_\pm(x,\xi)} \mathbb S^n$,
$$
\mathcal A_\pm(y,\eta_\pm)\zeta_\pm=\mathcal A_\pm(x,\xi)\zeta_\pm
+{(\mathcal A_\pm(x,\xi)\zeta_\pm)\cdot \eta\over \sqrt{1+|\eta|^2}}(x\pm \xi).
$$
Here $\cdot$ stands for the Riemannian inner product on $\mathcal E$ which is equal
to $-\langle \cdot,\cdot\rangle_M$ restricted to $\mathcal E$.
Then (see Figure~\ref{f:Psi}(b))
$$
\begin{gathered}
(\mathcal A_+(y,\eta_+)\zeta_+)\cdot (\mathcal A_-(y,\eta_-)\zeta_-)
=(\mathcal A_+(x,\xi)\zeta_+)\cdot (\mathcal A_-(x,\xi)\zeta_-)
\\-{2\over 1+|\eta|^2} \big((\mathcal A_+(x,\xi)\zeta_+)\cdot\eta\big)
\big((\mathcal A_-(x,\xi)\zeta_-)\cdot\eta\big)\\
=\big(\mathscr C_\eta(\mathcal A_+(x,\xi)\zeta_+)\big)\cdot (\mathcal A_-(x,\xi)\zeta_-),
\end{gathered}
$$
where $\mathscr C_\eta:\mathcal E(x,\xi)\to\mathcal E(x,\xi)$ is given by
$$
\mathscr C_\eta(\tilde\eta)=\tilde\eta-{2\over 1+|\eta|^2}(\tilde\eta\cdot \eta)\eta.
$$
We can similarly define $\mathscr C_\eta^*:\mathcal E(x,\xi)^*\to\mathcal E(x,\xi)^*$.
Then for $\zeta_\pm\in\otimes^m T^*_{B_\pm(x,\xi)}\mathbb S^n$,
\begin{equation}
  \label{e:tensie}
\begin{gathered}
\langle\otimes^m(\mathcal A_+^{-1}(y,\eta_+)^T)\zeta_+,
\otimes^m(\mathcal A_-^{-1}(y,\eta_-)^T)\zeta_-\rangle_{\otimes^m T^*_y\mathbb H^{n+1}}\\
=\langle\otimes^m\mathscr C_\eta^*\otimes^m(\mathcal A_+^{-1}(x,\xi)^T)\zeta_+,
\otimes^m(\mathcal A_-^{-1}(x,\xi)^T)\zeta_-\rangle_{\otimes^m \mathcal E^*(x,\xi)}.
\end{gathered}
\end{equation}
The Jacobian of $\Psi$ with respect to naturally arising volume forms on $\mathcal E$ and
$S^2_\Delta \mathbb H^{n+1}$ is given by (see Appendix~\ref{s:J-Psi} for the proof)
\begin{equation}
  \label{e:Psi-J}
J_\Psi(x,\xi,\eta)=2^n(1+|\eta|^2)^{-n}.
\end{equation}
Now, $\Psi$ is equivariant under $G$, therefore it descends to a diffeomorphism
$$
\Psi:\mathcal E_M\to S^2_\Delta M,\quad
\mathcal E_M:=\{(x,\xi,\eta)\mid (x,\xi)\in SM,\ \eta\in \mathcal E(x,\xi)\}.
$$
Using Lemma~\ref{l:reduced} and~\eqref{e:tensie}, we calculate
for $(x,\xi,\eta)\in\mathcal E_M$ and $(y,\eta_-,\eta_+)=\Psi(x,\xi,\eta)$,
\begin{equation}
  \label{e:kendall}
\langle v(y,\eta_-),\overline{v^*(y,\eta_+)}\rangle_{\otimes^mT_y^*M}
=(1+|\eta|^2)^{-\lambda}\langle \otimes^m\mathscr C_\eta^* v(x,\xi),\overline{v^*(x,\xi)}\rangle_{\otimes^m\mathcal E^*(x,\xi)}.
\end{equation}
We would now like to plug this expression into~\eqref{e:intie-1}, make the change of variables
from $(y,\eta_-,\eta_+)$ to $(x,\xi,\eta)$, and integrate $\eta$ out, obtaining a multiple of~\eqref{e:intie-2}.
However, this is not directly possible because (i) the integral in $\eta$ typically diverges (ii)
since the expression integrated in~\eqref{e:intie-1} is a distribution, one cannot
simply replace $S^2M$ by $S^2_\Delta M$ in the integral.

We will instead use the asymptotic behavior
of both integrals as one approaches the set $\{\eta_++\eta_-=0\}$,
and Hadamard regularization in $\eta$ in the $(x,\xi,\eta)$ variables.
For that, fix $\chi\in \CI_0(\mathbb R)$ such that $\chi=1$ near 0, and define for $\varepsilon>0$,
$$
\chi_\varepsilon(y,\eta_-,\eta_+)=\chi\big(\varepsilon\, |\eta(y,\eta_-,\eta_+)|\big),
$$
where $\eta(y,\eta_-,\eta_+)$ is the corresponding component of $\Psi^{-1}$; in fact, we can write
$$
\chi_\varepsilon(y,\eta_-,\eta_+)=\chi\Big(\varepsilon\, {|\eta_+-\eta_-|\over |\eta_++\eta_-|}\Big).
$$
Then $\chi_\varepsilon\in\mathcal D'(S^2M)$. In fact, $\chi_\varepsilon$ is supported inside $S^2_\Delta M$;
by making the change of variables $(y,\eta_-,\eta_+)=\Psi(x,\xi,\eta)$ and using~\eqref{e:Psi-J} and~\eqref{e:kendall}, we get
\begin{equation}
  \label{e:abi}
\begin{gathered}
\int_{S^2M} \chi_\varepsilon(y,\eta_-,\eta_+) \langle v(y,\eta_-),\overline{v^*(y,\eta_+)}\rangle_{\otimes^m T_y^*M}\,dyd\eta_- d\eta_+
\\=2^n\int_{\mathcal E_M} \chi(\varepsilon|\eta|)(1+|\eta|^2)^{-\lambda-n}
\langle\otimes^m \mathscr C_\eta^*v(x,\xi),\overline{v^*(x,\xi)}\rangle_{\otimes^m \mathcal E^*(x,\xi)}
\,dxd\xi d\eta.
\end{gathered}
\end{equation}
By Lemma~\ref{l:asysa2}, \eqref{e:abi} has the asymptotic expansion
\begin{equation}
  \label{e:asybsy}
\begin{gathered}
2^n\pi^{n\over 2}{\Gamma({n\over 2}+\lambda)\over (n+\lambda+m-1)\Gamma(n-1+\lambda)}\langle v,v^*\rangle_{L^2(SM;\otimes^m_S\mathcal E^*)}\\
+\sum_{0\leq j\leq -\Re\lambda-{n\over 2}}c_j\varepsilon^{n+2\lambda+2j}+o(1)
\end{gathered}
\end{equation}
for some constants $c_j$.

It remains to prove the following asymptotic expansion as $\varepsilon\to 0$:
\begin{equation}
  \label{e:bbi}
\int_{S^2M} (1-\chi_\varepsilon(y,\eta_-,\eta_+)) \langle v(y,\eta_-),\overline{v^*(y,\eta_+)}\rangle_{\otimes^m T_y^*M}\,dyd\eta_- d\eta_+
\sim\sum_{j=0}^\infty c'_j \varepsilon^{n+2\lambda+2j}
\end{equation}
where $c'_j$ are some constants.
Indeed, $\langle f,f^*\rangle_{L^2(M;\otimes^m T^*M)}$ is equal to the sum of~\eqref{e:abi} and~\eqref{e:bbi};
since~\eqref{e:bbi} does not have a constant term, $\langle f,f^*\rangle$ is equal to
the constant term in the expansion~\eqref{e:asybsy}.

To show~\eqref{e:bbi}, we use the dilation vector field $\eta\cdot\partial_\eta$ on $\mathcal E$, which
under $\Psi$ becomes the following vector field on $S^2_\Delta M$
extending smoothly to $S^2M$:
$$
L_{(y,\eta_-,\eta_+)}=\bigg({\eta_--\eta_+\over 2},
{|\eta_+-\eta_-|^2\over 4}\,y-{\eta_+\over 2}+{\eta_-\cdot\eta_+\over 2}\,\eta_-,
-{|\eta_+-\eta_-|^2\over 4}\,y-{\eta_-\over 2}+{\eta_-\cdot\eta_+\over 2}\,\eta_+
\bigg).
$$
The vector field $L$ is tangent to the submanifold
$\{\eta_++\eta_-=0\}$, in fact
$$
L(|\eta_+-\eta_-|^2)=-L(|\eta_++\eta_-|^2)={|\eta_+-\eta_-|^2\cdot |\eta_++\eta_-|^2\over 2}.
$$
We can then compute (following the identity $L|\eta|=|\eta|$)
$$
L\bigg({|\eta_+-\eta_-|\over |\eta_++\eta_-|}\bigg)={|\eta_+-\eta_-|\over |\eta_++\eta_-|}\quad\text{on }S^2_\Delta M.
$$
Using the $(x,\xi,\eta)$ coordinates and~\eqref{e:Psi-J}, we can compute the divergence of $L$ with respect to the standard volume
form on $S^2M$:
$$
\Div L=n(\eta_+\cdot\eta_-).
$$
Moreover, $B_\pm(y,\eta_\pm)$ are constant along the trajectories
of $L$, and
$$
L(\Phi_\pm(y,\eta_\pm))=-{|\eta_+-\eta_-|^2\over 4}\Phi_\pm(y,\eta_\pm).
$$
We also use~\eqref{e:Adef} to calculate for $\zeta_\pm\in T_{B_\pm(y,\eta_\pm)}\mathbb S^n$,
$$
\begin{gathered}
L\big(
(\mathcal A_+(y,\eta_+)\zeta_+)
\cdot
(\mathcal A_-(y,\eta_-)\zeta_-)
\big)=
\big((\mathcal A_+(y,\eta_+)\zeta_+)\cdot\eta_-\big)
\big((\mathcal A_-(y,\eta_-)\zeta_-)\cdot\eta_+\big),\\
L\big((\mathcal A_\pm(y,\eta_\pm)\zeta_\pm)\cdot\eta_\mp\big)
=(\eta_+\cdot\eta_-)\big((\mathcal A_\pm(y,\eta_\pm)\zeta_\pm)\cdot\eta_\mp\big).
\end{gathered}
$$
Combining these identities and using Lemma~\ref{l:reduced}, we get
\begin{equation}
  \label{e:L-eq}
\begin{gathered}
\Big( L+{\lambda\over 2}|\eta_+-\eta_-|^2\Big)\langle v(y,\eta_-),\overline{v^*(y,\eta_+)}\rangle_{\otimes^m T_y^*M}
\\=m\langle \iota_{\eta_+} v(y,\eta_-), \iota_{\eta_-}\overline{v^*(y,\eta_+)}\rangle_{\otimes^{m-1} T_y^*M}.
\end{gathered}
\end{equation}
Integrating by parts, we find
$$
\begin{gathered}
\varepsilon\partial_\varepsilon \int_{S^2 M} (1-\chi_\varepsilon(y,\eta_-,\eta_+))\langle v(y,\eta_-),\overline{v^*(y,\eta_+)}\rangle_{\otimes^m T_y^*M}\,dyd\eta_-d\eta_+\\
=\int_{S^2M} L\big (1-\chi_\varepsilon(y,\eta_-,\eta_+)\big)\,\langle v(y,\eta_-),\overline{v^*(y,\eta_+)}\rangle_{\otimes^m T_y^*M}\,dyd\eta_-d\eta_+\\
=\int_{S^2M} \bigg({\lambda\over 2}|\eta_+-\eta_-|^2-n(\eta_+\cdot\eta_-)\bigg)
(1-\chi_\varepsilon(y,\eta_-,\eta_+))\langle v(y,\eta_-),\overline{v^*(y,\eta_+)}\rangle_{\otimes^m T_y^*M}\,dyd\eta_-d\eta_+\\
-m\int_{S^2M} (1-\chi_\varepsilon(y,\eta_-,\eta_+))\langle \iota_{\eta_+}v(y,\eta_-),\iota_{\eta_-}\overline{v^*(y,\eta_+)}\rangle_{\otimes^{m-1}T_y^*M}
\,dyd\eta_-d\eta_+.
\end{gathered}
$$
Arguing similarly, we see that if for integers $0\leq r\leq m$, $p\geq 0$, we put
$$
I_{r,p}(\varepsilon):=\int_{S^2 M} |\eta_-+\eta_+|^{2p}(1-\chi_\varepsilon(y,\eta_-,\eta_+))\langle \iota_{\eta_+}^r v(y,\eta_-),\iota_{\eta_-}^r\overline{v^*(y,\eta_+)}\rangle_{\otimes^{m-r} T_y^*M}\,dyd\eta_-d\eta_+
$$
then $(\varepsilon\partial_\varepsilon-2\lambda-n-2(r+p))I_{r,p}(\varepsilon)$ is a finite linear combination of
$I_{r',p'}(\varepsilon)$, where $r'\geq r,p'\geq p$, and $(r',p')\neq (r,p)$. For
example, the calculation above shows that
$$
(\varepsilon\partial_\varepsilon-2\lambda-n)I_{0,0}(\varepsilon)
=-{\lambda+n\over 2}I_{0,1}(\varepsilon)-mI_{1,0}(\varepsilon).
$$
Moreover, if $N$ is fixed and $p$ is large enough depending on $N$,
then $I_{r,p}(\varepsilon)=\mathcal O(\varepsilon^N)$; to see this, note that $I_{r,p}(\varepsilon)$ is bounded by
some fixed $\CI$-seminorm of $|\eta_-+\eta_+|^{2p}(1-\chi_\varepsilon(y,\eta_-,\eta_+))$.
It follows that if $N$ is fixed and $\widetilde N$ is large depending on $N$, then
$$
\bigg(\prod_{j=0}^{\widetilde N}(\varepsilon\partial_\varepsilon-2\lambda-n-2j)
\bigg)I_{0,0}(\varepsilon)=\mathcal O(\varepsilon^N)
$$
which implies the existence of the decomposition~\eqref{e:bbi} and finishes the proof.
\end{proof}
%%%%%%%%%%%%%%%%%%%%%%%%%%%%%%%%%%%%%%%%%%%%%%%%%%%%%%%%%%%%%%%%%%%%%%%%%%%%%%%%

%%%%%%%%%%%%%%%%%%%%%%%%%%%%%%%%%%%%%%%%%%%%%%%%%%%%%%%%%%%%%%%%%%%%%%%%%%%%%%%%
%                                  SECTION 6                                   %
%%%%%%%%%%%%%%%%%%%%%%%%%%%%%%%%%%%%%%%%%%%%%%%%%%%%%%%%%%%%%%%%%%%%%%%%%%%%%%%%
\section{Properties of the Laplacian}
\label{s:laplacian}

In this section, we introduce the Laplacian and study its basic properties (Section~\ref{s:laplacian-def}).
We then give formulas for the Laplacian on symmetric tensors in the half-plane model (Section~\ref{s:laplacian-half})
which will be the basis for the analysis of the following sections. Using these formulas, we study the Poisson kernel
and in particular prove Lemma~\ref{l:laplacian} and the injectivity of the Poisson kernel (Section~\ref{s:poisson-basic}).

%%%%%%%%%%%%%%%%%%%%%%%%%%%%%%%%%%%%%%%%%%%%%%%%%%%%%%%%%%%%%%%%%%%%%%%%%%%%%%%%
\subsection{Definition and Bochner identity}
\label{s:laplacian-def}

The Levi--Civita connection associated to the hyperbolic metric $g_H$ is the operator 
\[
\nabla : \CI(\hh^{n+1},T\hh^{n+1})\to \CI(\hh^{n+1},T^*\hh^{n+1}\otimes T\hh^{n+1})
\] 
which induces a natural covariant derivative, still denoted $\nabla$, on sections of $\otimes^mT^*\hh^{n+1}$. 
We can work in the ball model $\mathbb B^{n+1}$ and use the $0$-tangent structure (see Section~\ref{s:E})
and nabla can be viewed as a differential operator of order $1$   
\[
\nabla: \CI(\overline{\mathbb{B}^{n+1}}; \otimes^m ({^0}T^*\overline{\mathbb{B}^{n+1}}))\to \CI(\overline{\mathbb{B}^{n+1}},\otimes^{m+1} ({^0}T^*\overline{\mathbb B^{n+1}}))
\] 
and we denote by $\nabla^*$ its adjoint with respect to the $L^2$ scalar product, $\nabla^*$ is called the \emph{divergence}: it is given by $\nabla^*u=-\mc{T}(\nabla u)$ where $\mc{T}$ denotes the trace, see Section~\ref{symtens}.
 Define the rough Laplacian acting on $\CI(\overline{\mathbb{B}^{n+1}}; \otimes^m ({^0}T^*\overline{\mathbb{B}^{n+1}}))$ by
\begin{equation}\label{roughlap}
\Delta :=\nabla^*\nabla
\end{equation}  
and this operator maps symmetric tensors to symmetric tensors. 
It also extends to $\mc{D'}(\mathbb{B}^{n+1}; \otimes_S^m ({^0}T^*\overline{\mathbb{B}^{n+1}}))$ by duality.
The operator $\Delta$ commutes with $\mc{T}$ and 
$\mc{I}$: 
\begin{equation}\label{commuteDelta}
\Delta \mc{T}(u)=\mc{T}(\Delta u), \quad \Delta \mc{I}(u)=\mc{I}(\Delta u)
\end{equation}
for all $u\in \mc{D'}(\mathbb{B}^{n+1}; \otimes_S^m ({^0}T^*\overline{\mathbb{B}^{n+1}}))$.

There is another natural operator given by 
\[\Delta_D=D^*D\] 
if
$$
D:  \mc{C}^\infty(\overline{\mathbb{B}^{n+1}}; \otimes_S^{m} ({^0}T^*\overline{\mathbb{B}^{n+1}}))\to
\mc{C}^\infty(\overline{\mathbb{B}^{n+1}}; \otimes_S^{m+1} ({^0}T^*\overline{\mathbb{B}^{n+1}}))
$$
is defined by $D:= \mc{S}\circ \nabla$, where $\mc{S}$ is the symmetrization defined by \eqref{symmet}, and $D^*=\nabla^*$ is the formal adjoint. There is a Bochner--Weitzenb\"ock formula relating $\Delta$ and $\Delta_D$, and using that the curvature is constant, we have on trace-free symmetric tensors of order $m$ by~\cite[Lemma 8.2]{DaSh}
\begin{equation}\label{bochner}
\Delta_D=\frac{1}{m+1}( m\, DD^*+\Delta+m(m+n-1)).
\end{equation}
In particular, since $|\mathcal S\nabla u|^2\leq |\nabla u|^2$ pointwise by the fact that $\mc{S}$ is an orthogonal projection, we see that for 
$u$ smooth and compactly supported, $\|Du\|_{L^2}^2\leq \|\nabla u\|^2_{L^2}$ and thus for $m\geq 1$,
$u\in \CI_0(\mathbb H^{n+1};\otimes_S^m(T^*\mathbb H^{n+1}))$, and $\mathcal Tu=0$,
\begin{equation}\label{bochner2}
\cjg \Delta u,u\cjd_{L^2} \geq (m+n-1)\|u\|^2.
\end{equation}
Since the Bochner identity is local, the same inequality clearly descends to co-compact quotients $\Gamma\backslash \hh^{n+1}$
(where $\Delta$ is self-adjoint and has compact resolvent by standard theory of elliptic operators,
as its principal part is given by the scalar Laplacian), and this implies
%%%%%%%%%%%%%%%%%%%%%%%%%%%%%%%%%%%%%%%%%%%%%%%%%%%%%%%%%%%%%%%%%%%%%%%%%%%%%%%%
\begin{lemm}
  \label{bottomsp} 
The spectrum of $\Delta$ acting on trace-free symmetric tensors of order $m\geq 1$ on hyperbolic compact manifolds 
of dimension $n+1$ is bounded below by $m+n-1$.
\end{lemm}
%%%%%%%%%%%%%%%%%%%%%%%%%%%%%%%%%%%%%%%%%%%%%%%%%%%%%%%%%%%%%%%%%%%%%%%%%%%%%%%%
We finally define
\begin{equation}
  \label{e:e-m-def}
E^{(m)}:=\otimes_S^{m} ({^0}T^*\overline{\mathbb{B}^{n+1}})\cap \ker\mc{T}
\end{equation}
to be the bundle of trace-free symmetric $m$-cotensors over the ball model of hyperbolic space.  

%%%%%%%%%%%%%%%%%%%%%%%%%%%%%%%%%%%%%%%%%%%%%%%%%%%%%%%%%%%%%%%%%%%%%%%%%%%%%%%%
\subsection{Laplacian in the half-plane model}
\label{s:laplacian-half}

We now give concrete formulas concerning the Laplacian on symmetric tensors
in the half-space model $\mathbb U^{n+1}$ (see~\eqref{e:umetric}).
We fix $\nu\in\mathbb S^n$ and map $\mathbb B^{n+1}$ to $\mathbb U^{n+1}$ by a composition
of a rotation of $\mathbb B^{n+1}$ and the map~\eqref{e:udiffeo};
the rotation is chosen so that $\nu$ is mapped to $0\in\overline{\mathbb U^{n+1}}$
and $-\nu$ is mapped to infinity.

The $0$-cotangent and tangent bundles ${^0}T^*\overline{\mathbb B^{n+1}}$ and 
${^0}T\overline{\mathbb B^{n+1}}$ pull back to the half-space, we denote them ${^0}T^*\mathbb{U}^{n+1}$ and 
${^0}T\mathbb{U}^{n+1}$. The coordinates on $\mathbb{U}^{n+1}$ are $(z_0,z)\in\rr^+\x \rr^n$ and $z=(z_1,\dots,z_n)$.
We use the following orthonormal bases of ${^0}T\mathbb{U}^{n+1}$
and ${^0}T^*\mathbb{U}^{n+1}$:
$$
Z_i=z_0\pl_{z_i},\quad
Z_i^*={dz_i\over z_0};
\quad 0\leq i\leq n.
$$
Note that in the compactification $\overline{\mathbb B^{n+1}}$ this basis is smooth only on 
$\overline{\mathbb B^{n+1}}\setminus \{-\nu\}$.

Denote $\mathscr A:=\{1,\dots,n\}$. We can decompose the vector bundle $\otimes_S^m({^0T}^*{\mathbb{U}^{n+1}})$ into an
orthogonal direct sum
\[
\otimes_S^m({^0T}^*{\mathbb{U}^{n+1}})=\bigoplus_{k=0}^{m} E^{(m)}_k, \quad
E^{(m)}_k={\rm span}\big(\mc{S}((Z_0^*)^{\otimes k}\otimes Z_I^*))_{I\in\mathscr{A}^{m-k}}\big)
\]
and we let $\pi_i$ be the orthogonal projection onto $E_i^{(m)}$.
Now, each  tensor $u\in \otimes_S^{m}({^0T}^*{\mathbb{U}^{n+1}})$ can be decomposed as  
$u=\sum_{i=0}^{m}u_i$, with  $u_i=\pi_i(u)\in E^{(m)}_i$ which we can write as
\begin{equation}
  \label{e:uuform}
u=\sum_{i=0}^m u_i,\quad
u_i= \mc{S}((Z_0^*)^{\otimes i}\otimes u'_i),\quad
u'_i\in E^{(m-i)}_0.
\end{equation}
We can therefore identify $E^{(m)}_k$ with $E^{(m-k)}_0$ and view $E^{(m)}$ as a direct sum
$E^{(m)}=\bigoplus_{k=0}^mE^{(m-k)}_0$. The trace-free condition
$\mc{T}(u)=0$ is equivalent to the relations
\begin{equation}
  \label{u'r}
\mc{T}(u'_r)=-\frac{(r+2)(r+1)}{(m-r)(m-r-1)}u'_{r+2},\quad
0\leq r\leq m-2.
\end{equation}
and in particular all $u_i$ are determined by $u_0$ and $u_1$ by iterating the trace map 
$\mc{T}$. The $u_i'$ are related to the elements in the decomposition \eqref{decompositionoftensors} of 
$u_0$ and $u_1$ viewed as a symmetric $m$-cotensor 
on the bundle $(Z_0)^\perp$ using the metric $z_0^{-2}h=\sum_{i}Z_i^*\otimes Z_i^*$. 
We see that a nonzero trace-free tensor $u$ on $\mathbb{U}^{n+1}$ must have a 
nonzero $u_0$ or $u_1$ component. 

Koszul formula gives us for $i,j\geq 1$
\begin{equation}\label{koszul} 
\nabla_{Z_i}Z_j=\delta_{ij}Z_0, \quad \nabla_{Z_0}Z_j=0,\quad \nabla_{Z_i}Z_0=-Z_i,\quad \nabla_{Z_0}Z_0=0
\end{equation}
which implies 
\begin{equation}\label{computenabla} 
\nabla Z_0^*= -\sum_{j=1}^nZ_j^*\otimes Z_j^*=-\frac{h}{z_0^2}, \quad \nabla Z_j^*=Z_j^*\otimes Z_0^*.
\end{equation}

We shall use the following notations: if $\Pi_m$ denotes the set of permutations of $\{1,\dots,m\}$, we 
write $\sigma(I):=(i_{\sigma(1)},\dots,i_{\sigma(m)})$ if $\sigma\in \Pi_m$.
If $S=S_1\otimes \dots\otimes S_\ell$ is a tensor in $\otimes^\ell ({^0}T^*\mathbb{U}^{n+1})$, we denote by $\tau_{i\lra j}(S)$ 
the tensor obtained by permuting $S_i$ with $S_j$ in $S$, and by $\rho_{i\to V}(S)$ the operation of 
replacing $S_i$ by $V\in {^0}T^*\mathbb{U}^{n+1}$ in $S$.\\

%%%%%%%%%%%%%%%%%%%%%%%%%%%%%%%%%%%%%%%%%%%%%%%%%%%%%%%%%%%%%%%%%%%%%%%%%%%%%%%%
\smallsection{The Laplacian and $\nabla^*$ acting on $E_0^{(m)}$ and $E_1^{(m)}$}
We start by computing the action of $\Delta$ on sections of $E^{(m)}_0, E^{(m)}_1$, and we will later
deduce from this computation the action on $E^{(m)}_k$.
Let us consider the tensor $Z^*_I:=Z^*_{i_1}\otimes\dots\otimes Z^*_{i_m}\in E_0^{(m)}$ 
where $I=(i_1,\dots,i_m)\in \mathscr A^{m}$ and
$Z^*_{\sigma(I)}:= Z^*_{i_{\sigma(1)}}\otimes\dots \otimes Z^*_{i_{\sigma(m)}}.$
The symmetrization of $Z^*_I$ is given by $\mc{S}(Z^*_I)=\frac{1}{m!}\sum_{\sigma\in \Pi_m}Z^*_{\sigma(I)}$
and those elements form a basis of the space $E^{(m)}_0$ when $I$ ranges over all combinations of $m$-uplet in 
$\mathscr A=\{1,\dots,n\}$.

%%%%%%%%%%%%%%%%%%%%%%%%%%%%%%%%%%%%%%%%%%%%%%%%%%%%%%%%%%%%%%%%%%%%%%%%%%%%%%%%
\begin{lemm}
  \label{DeltaE_0}
Let $u_0=\sum_{I\in\mathscr{A}^m}f_I\mc{S}(Z^*_I)$ with $f_I\in\mc{C}^\infty(\mathbb{U}^{n+1})$. Then one has 
\begin{equation} \label{calculfinaltracefree}
\begin{split}
\Delta u_0= & \sum_{I\in\mathscr{A}^m}((\Delta+m)f_I)\mc{S}(Z_I^*)+2m\, \mc{S}(\nabla^*u_0\otimes Z_0^*) \\
& +m(m-1)\mc{S}(\mc{T}(u_0)\otimes Z_0^*\otimes Z_0^*) 
\end{split}\end{equation}
while, denoting  $d_{z}f_I=\sum_{i=1}^nZ_i(f_I)Z_i^*$, the divergence is given by 
\begin{equation}\label{tracediv}
\begin{split}
\nabla^*u_0= &-(m-1)\mc{S}(\mc{T}(u_0)\otimes Z_0^*)
 - \sum_{I\in\mathscr{A}^m}\iota_{d_{z}f_I}\mc{S}(Z_I^*).
 \end{split}
\end{equation}
\end{lemm}
%%%%%%%%%%%%%%%%%%%%%%%%%%%%%%%%%%%%%%%%%%%%%%%%%%%%%%%%%%%%%%%%%%%%%%%%%%%%%%%%
\begin{proof}
Using \eqref{computenabla}, we compute 
\[ \nabla(f_I\mc{S}(Z_I^*))= \sum_{i=0}^n(Z_if_I)(z) Z_i^*\otimes \mc{S}(Z_I^*)+ \frac{f_I(z)}{m!} \sum_{k=1}^m \sum_{\sigma\in \Pi_m}
\tau_{1\lra k+1}(Z_0^*\otimes Z^*_{\sigma(I)}). \]
Then taking the trace of $\nabla(f_I\mc{S}(Z_I^*))$ gives 
\begin{equation}
  \label{divergencefT}
\begin{split}
\nabla^*(f_I\mc{S}(Z_I^*))=&  -\frac{f_I}{m!} \sum_{k=2}^m \sum_{\sigma\in \Pi_m}\delta_{i_{\sigma(1)},i_{\sigma(k)}}\rho_{k-1\to Z_0^*}(Z^*_{i_{\sigma(2)}}\otimes \dots \otimes Z^*_{i_{\sigma(m)}})\\
& - \sum_{i=1}^n(Z_if_I) \frac{1}{m!} \sum_{\sigma\in \Pi_m}\delta_{i,i_{\sigma(1)}}
(Z^*_{i_{\sigma(2)}}\otimes \dots \otimes Z^*_{i_{\sigma(m)}})
\end{split}
\end{equation}
We notice that $\mc{S}(\mc{T}(\mc{S}(Z_I^*))\otimes Z_0^*)$ is given by 
\[\mc{S}(\mc{T}(\mc{S}(Z_I^*))\otimes Z_0^*)=\frac{1}{m!(m-1)} \sum_{\sigma\in \Pi_m}\sum_{k=1}^{m-1}\delta_{i_{\sigma(1)},i_{\sigma(2)}}
\tau_{1\lra k}(Z_0^*\otimes Z^*_{i_{\sigma(3)}}\otimes \dots \otimes Z^*_{i_{\sigma(m)}}).\]
which implies \eqref{tracediv}. 
Let us now compute $\nabla^2(f_I\mc{S}(Z_I^*))$: 
\begin{align*}
\nabla^2(f_I\mc{S}(Z_I^*)) &=  \sum_{i,j=0}^nZ_jZ_i(f_I) Z_j^*\otimes Z_i^*\otimes \mc{S}(Z_I^*)- Z_0(f_I) z_0^{-2} h\otimes \mc{S}(Z_I^*) \\
& +\sum_{j=1}^n
Z_j(f_I)Z_j^*\otimes Z_0^*\otimes \mc{S}(Z_I^*) + \frac{Z_0(f_I)}{m!}  \sum_{\sigma\in \Pi_m}\sum_{k=1}^m \tau_{1\lra k+2}(Z_0^*\otimes Z_0^*\otimes Z^*_{\sigma(I)})\\
&+\sum_{i=1}^n \frac{Z_i(f_I)}{m!}\sum_{\sigma\in\Pi_m}\sum_{k=1}^m \tau_{1\lra k+2}(Z_0^*\otimes Z_i^*\otimes Z^*_{\sigma(I)})
\\
& + \sum_{i=1}^n\frac{Z_i(f_I)}{m!} \sum_{\sigma\in \Pi_m}\sum_{k=1}^m \tau_{2\lra k+2}(Z_i^*\otimes Z_0^*\otimes Z^*_{\sigma(I)})\displaybreak[0]\\
 &+ \frac{Z_0(f_I)}{m!} Z_0^*\otimes \sum_{k=1}^m \sum_{\sigma\in \Pi_m}
\tau_{1\lra k+1}(Z_0^*\otimes Z^*_{\sigma(I)})\\
&  - \frac{f_I}{m!} \sum_{j=1}^{n}Z_j^*\otimes \sum_{\sigma\in \Pi_m}\sum_{k=1}^m\tau_{1\lra k+1}(Z_j^*\otimes Z^*_{\sigma(I)})\\
&+ \frac{f_I}{m!} \sum_{k=1}^m \sum_{\substack{\ell=1\\ \ell\not=k+1}}^{m+1}\tau_{1\lra \ell+1}(Z_0^*\otimes \tau_{1\lra k+1}(Z_0^*\otimes Z^*_{\sigma(I)})).
\end{align*}
We then take the trace: the first line has trace $-(\Delta f_I) \mc{S}(Z_I^*)$, the second and fifth lines have vanishing trace,
the sixth line has trace $-mf_I\mc{S}(Z_I^*)$, the last line has trace 
\begin{equation}
  \label{lastline} 
\frac{2f_I}{m!} \sum_{\sigma\in \Pi_m}\sum_{1\leq k<\ell\leq m} \delta_{i_{\sigma(k)},i_{\sigma(\ell)}}
\rho_{k\to Z_0^*}\rho_{\ell \to Z_0^*}(Z^*_{\sigma(I)})
\end{equation}
and the sum of the third and fourth lines has trace
\begin{equation}
  \label{thirdline}
2\sum_{i=1}^n\frac{Z_i(f_I)}{m!}\sum_{\sigma\in \Pi_m}\sum_{k=1}^m \delta_{i,i_{\sigma(k)}}
\rho_{k\to Z_0^*}(Z^*_{\sigma(I)}).
\end{equation}
Computing $\mc{S}(\mc{T}(\mc{S}(Z_I^*))\otimes Z_0^*\otimes Z_0^*)$ gives 
\[
\begin{gathered}
\mc{S}\big(\mc{T}(\mc{S}(Z_I^*))\otimes Z_0^*\otimes Z_0^*\big)= \\
 \frac{2}{m!m(m-1)}
\sum_{1\leq k<\ell\leq m}\sum_{\sigma\in \Pi_m}\delta_{i_{\sigma(1)},i_{\sigma(2)}}\tau_{1\lra k+2}\tau_{2\lra \ell+2}(Z_0^*\otimes Z_0^*\otimes Z^*_{i_{\sigma(3)}}\otimes \dots \otimes Z^*_{i_{\sigma(m)}})
\end{gathered}
\]
therefore the term \eqref{lastline} can be simplified to
\[
m(m-1)f_I\mc{S}\big(\mc{T}(\mc{S}(Z_I^*))\otimes Z_0^*\otimes Z_0^*\big).
\]
Similarly to simplify \eqref{thirdline}, we compute
\[
\begin{gathered}
\mc{S}\big(\nabla^*(f_I\mc{S}(Z_I^*))\otimes Z_0^*\big)= -(m-1)\mc{S}(\mc{T}(f_I\mc{S}(Z_I^*))\otimes Z_0^*\otimes Z_0^*)\\
-\sum_{i=1}^n(Z_if_I) \frac{1}{m!m} \sum_{k=1}^{m}\sum_{\sigma\in \Pi_m}\delta_{i,i_{\sigma(1)}}
\tau_{1\lra k}(Z_0^*\otimes Z^*_{i_{\sigma(2)}}\otimes \dots \otimes Z^*_{i_{\sigma(m)}})
\end{gathered}
\]
so that 
\[
\begin{gathered}
2\sum_{i=1}^n\frac{Z_i(f_I)}{m!}\sum_{\sigma\in \Pi_m}\sum_{k=1}^m \delta_{i,i_{\sigma(k)}}
\rho_{k\to Z_0^*}(Z^*_{\sigma(I)})\\
= -2m\, \mc{S}(\nabla^*(f_I\mc{S}(Z_I^*))\otimes Z_0^*)-2m(m-1)\mc{S}(\mc{T}(f_I\mc{S}(Z_I^*))\otimes Z_0^*\otimes Z_0^*).
\end{gathered}
\]
and this achieves the proof of \eqref{calculfinaltracefree}.Ê
\end{proof}
%%%%%%%%%%%%%%%%%%%%%%%%%%%%%%%%%%%%%%%%%%%%%%%%%%%%%%%%%%%%%%%%%%%%%%%%%%%%%%%%
A similarly tedious calculation, omitted here, yields
%%%%%%%%%%%%%%%%%%%%%%%%%%%%%%%%%%%%%%%%%%%%%%%%%%%%%%%%%%%%%%%%%%%%%%%%%%%%%%%%
\begin{lemm}\label{DeltaE_1}
Let $u_1=\mathcal S(Z_0^*\otimes u'_1)$,
$u'_1=\sum_{J\in\mathscr{A}^{m-1}}g_J \mathcal S(Z^*_J)$ with $g_J\in\mc{C}^\infty(\mathbb{U}^{n+1})$, then
the $E_0^{(m)}\oplus E_1^{(m)}$ components of the Laplacian of $u_1$ are
\begin{equation} \label{Deltau1}
\begin{split}
\Delta u_1=&\sum_{J\in\mathscr{A}^{m-1}} \big((\Delta+n+3(m-1))g_J\big)\mc{S}(Z_0^* \otimes Z_J^*)\\&+
2\sum_{J\in\mathscr{A}^{m-1}} \mc{S}(d_{z}g_J\otimes Z_J^*)
 + \Ker(\pi_0+\pi_1)
\end{split}\end{equation}
and the $E_0^{(m)}\oplus E_1^{(m)}$ components of divergence of $u_1$ are
 \begin{equation}\label{tracediv2}
\begin{split}
\nabla^*u_1=&\,\frac{1}{m}
\sum_{J\in\mathscr{A}^{m-1}}((n+m-1)g_J -Z_0(g_J))\mc{S}(Z_J^*)\\
 &-\frac{(m-1)}{m}
\sum_{J\in\mathscr{A}^{m-1}}\mc{S}( Z_0^*\otimes \iota_{d_{z}g_J}\mc{S}(Z_J^*))+\Ker(\pi_0+\pi_1).
\end{split}
\end{equation}
\end{lemm}
%%%%%%%%%%%%%%%%%%%%%%%%%%%%%%%%%%%%%%%%%%%%%%%%%%%%%%%%%%%%%%%%%%%%%%%%%%%%%%%%

%%%%%%%%%%%%%%%%%%%%%%%%%%%%%%%%%%%%%%%%%%%%%%%%%%%%%%%%%%%%%%%%%%%%%%%%%%%%%%%%
\smallsection{General formulas for Laplacian and divergence}
Armed with Lemmas~\ref{DeltaE_0} and~\ref{DeltaE_1}, we can show the following fact
which, together with~\eqref{u'r}, determines completely the Laplacian on trace-free symmetric tensors.
%%%%%%%%%%%%%%%%%%%%%%%%%%%%%%%%%%%%%%%%%%%%%%%%%%%%%%%%%%%%%%%%%%%%%%%%%%%%%%%%
\begin{lemm}\label{pi0pi1}
Assume that $u\in\mathcal D'(\mathbb U^{n+1};\otimes_S^mT^*\mathbb U^{n+1})$ satisfies $\mathcal T(u)=0$
and is written in the form~\eqref{e:uuform}. Let
$$
u_0=\sum_{I\in\mathscr A^m} f_I\mathcal S(Z_I^*),\quad
u_1=\sum_{J\in\mathscr A^{m-1}}g_J\mathcal S(Z_0\otimes Z_J^*).
$$
Then the projection of $\Delta u$ onto $E^{(m)}_0\oplus E^{(m)}_1$ can be written 
\begin{equation}\label{equationpi0}
\begin{split}
\pi_0(\Delta u)=& \sum_{I\in\mathscr{A}^m}((\Delta+m) f_I)\mc{S}(Z^*_{I})+2\sum_{J\in\mathscr{A}^{m-1}} \mc{S}(d_{z}g_J\otimes Z_J^*)\\
 &+m(m-1)\mc{S}(z_0^{-2}h\otimes \mc{T}(u_0)),
\end{split}
\end{equation} 
\begin{equation}\label{pi1deltaT}
\begin{split}
\pi_1(\Delta u)=& \sum_{J\in\mathscr{A}^{m-1}} ((\Delta+n+3(m-1))g_J)\mc{S}(Z_0^* \otimes Z_J^*)\\
&-2m\sum_{I\in\mathscr{A}^m}\mc{S}(Z_0^*\otimes \iota_{d_{z}f_I}\mc{S}(Z_I^*))\\
&+(m-1)(m-2)\mc{S}(Z_0^*\otimes z_0^{-2}h\otimes \mc{T}(u'_1)) \\
& -2m(m-1)\sum_{I\in\mathscr{A}^m}\mc{S}(Z_0^*\otimes d_{z}f_I \otimes \mc{T}(\mc{S}(Z_I^*))).
\end{split}
\end{equation}
\end{lemm}
%%%%%%%%%%%%%%%%%%%%%%%%%%%%%%%%%%%%%%%%%%%%%%%%%%%%%%%%%%%%%%%%%%%%%%%%%%%%%%%%
\begin{proof}
First, it is easily seen from \eqref{computenabla} that $\Delta u_k$ is a section of $\bigoplus_{j=k-2}^{k+2}E^{(m)}_{j}$.
From Lemmas~\ref{DeltaE_0} and~\ref{DeltaE_1}, we have
\begin{equation}
\label{eq1}
\pi_0(\Delta (u_0+u_1))=\sum_{I\in\mathscr{A}^m}((\Delta+m) f_I)\mc{S}(Z^*_{I})+2\sum_{J\in\mathscr{A}^{m-1}} \mc{S}(d_{z}g_J\otimes Z_J^*).
\end{equation}
Then for $u_2$, using $\mc{S}((Z_0^*)^{\otimes 2}\otimes u_2')=\mc{S}(g_H\otimes u_2')-
\mc{S}(z_0^{-2}h
\otimes u_2')$ and $\Delta\mc{I}=\mc{I}\Delta$,
\[
\pi_0(\Delta u_2)=\pi_0(\mc{S}(z_0^{-2}h\otimes \Delta u'_2))-\pi_0(\Delta(\mc{S}(z_0^{-2}h\otimes u_2')))
\]
and writing $u_2'=-\tfrac{m(m-1)}{2}\mathcal T(u_0)$ by~\eqref{u'r}, we obtain, using \eqref{calculfinaltracefree} 
\begin{equation}\label{eq3}
\pi_0(\Delta u_2)=m(m-1)\mc{S}(z_0^{-2}h\otimes \mc{T}(u_0))
\end{equation}
We therefore obtain \eqref{equationpi0}.

Now we consider the projection on $E^{(m)}_1$ of the equation $(\Delta-s)T=0$. We have from \eqref{calculfinaltracefree}
\[
\pi_1(\Delta u_0)=-2m\sum_{I\in\mathscr{A}^m}\mc{S}(Z_0^*\otimes \iota_{d_{z}f_I}\mc{S}(Z_I^*))
\]
where $\iota_{d_{z}f_I}$ means $\sum_{j=1}^nZ_j(f_I)\iota_{Z_j}$. Then, from \eqref{Deltau1} 
\[
\begin{split}
\pi_1(\Delta u_1)
=&\sum_{J\in\mathscr{A}^{m-1}} \big((\Delta+n+3(m-1))g_J\big)\mc{S}(Z_0^* \otimes Z_J^*).
\end{split}
\]
Using again $\mc{S}((Z_0^*)^{\otimes 2}\otimes u_2')=\mc{S}(g_H\otimes u_2')-
\mc{S}(z_0^{-2}h \otimes u_2')$ and $\Delta\mc{I}=\mc{I}\Delta$,  \eqref{calculfinaltracefree} gives
\[
\begin{split}
\pi_1(\Delta u_2)= -2m(m-1)\sum_{I\in\mathscr{A}^m}\mc{S}(Z_0^*\otimes d_{z}f_I \otimes \mc{T}\mc{S}(Z_I^*)).
\end{split}
\]
Finally, we compute $\pi_1(\Delta u_3)$, using the computation \eqref{Deltau1} we get 
\[
\begin{split}
\pi_1(\Delta u_3)=&\pi_1(\mc{S}(z_0^{-2}h\otimes \Delta \mc{S}(Z_0^*\otimes u'_3))-\pi_1(\Delta \mc{S}(Z_0^*\otimes z_0^{-2}h\otimes u'_3))\\
=& (m-1)(m-2)\mc{S}(Z_0^*\otimes z_0^{-2}h\otimes \mc{T}(u_1')).
\end{split}
\]
We conclude that $\pi_1(\Delta u)$ is given by \eqref{pi1deltaT}.
\end{proof}
%%%%%%%%%%%%%%%%%%%%%%%%%%%%%%%%%%%%%%%%%%%%%%%%%%%%%%%%%%%%%%%%%%%%%%%%%%%%%%%%

Similarly, we also have 
%%%%%%%%%%%%%%%%%%%%%%%%%%%%%%%%%%%%%%%%%%%%%%%%%%%%%%%%%%%%%%%%%%%%%%%%%%%%%%%%
\begin{lemm}
\label{divergenceproj} 
Let $u$ be as in Lemma~\ref{pi0pi1}.
Then the projection onto $E^{(m-1)}_{0}\oplus E^{(m-1)}_1$ of the divergence of $u$ is given by
\begin{equation}\label{divofT}
\pi_0(\nabla^*u)=-\sum_{I\in\mathscr{A}^{m}}\iota_{d_{z} f_I}\mc{S}(Z_I^*)+\frac{1}{m}
\sum_{J\in\mathscr{A}^{m-1}}((n+m-1)g_J -Z_0(g_J))\mc{S}(Z_J^*),
\end{equation}
\begin{equation}\label{divofT2}
\begin{split}
\pi_1(\nabla^*u)=& (m-1)\sum_{I\in\mathscr{A}^m}(Z_0f_I-(m+n-1)f_I)
\mc{S}\big(\mc{T}(\mc{S}(Z_I^*))\otimes Z_0^*\big)\\
& -\frac{(m-1)}{m}
\sum_{J\in\mathscr{A}^{m-1}}\mc{S}( Z_0^*\otimes \iota_{d_{z}g_J}\mc{S}(Z_J^*)).\end{split}
\end{equation}
\end{lemm}
%%%%%%%%%%%%%%%%%%%%%%%%%%%%%%%%%%%%%%%%%%%%%%%%%%%%%%%%%%%%%%%%%%%%%%%%%%%%%%%%
\begin{proof}
The $\pi_0$ part follows from \eqref{tracediv} and \eqref{tracediv2}. For the $\pi_1$ part, we also use \eqref{tracediv} and \eqref{tracediv2} but we need to see the contribution from $\nabla^*u_2$ as well. For that, we write as before 
$u_2'=-\tfrac{m(m-1)}{2}\sum_{I\in\mathscr{A}^m}f_I\mc{T}(\mc{S}(Z_I^*))$ and a direct calculation shows that 
\[
\pi_1(\nabla^*u_2)=(m-1)\sum_{I\in\mathscr{A}^m}(Z_0f_I-(m+n-2)f_I)
\mc{S}(\mc{T}(\mc{S}(Z_I^*))\otimes Z_0^*)
\]
implying the desired result.
\end{proof}
%%%%%%%%%%%%%%%%%%%%%%%%%%%%%%%%%%%%%%%%%%%%%%%%%%%%%%%%%%%%%%%%%%%%%%%%%%%%%%%%

%%%%%%%%%%%%%%%%%%%%%%%%%%%%%%%%%%%%%%%%%%%%%%%%%%%%%%%%%%%%%%%%%%%%%%%%%%%%%%%%
\subsection{Properties of the Poisson kernel}
  \label{s:poisson-basic}
  
In this section, we study the Poisson kernel $\mathscr P^-_\lambda$ defined by~\eqref{e:poisson-def}.

\smallsection{Pairing on the sphere}
We start by proving the following formula:
%%%%%%%%%%%%%%%%%%%%%%%%%%%%%%%%%%%%%%%%%%%%%%%%%%%%%%%%%%%%%%%%%%%%%%%%%%%%%%%%
\begin{lemm}
  \label{l:poisson-1}
Let $\lambda\in\mathbb C$ and $w\in\mathcal D'(\mathbb S^n;\otimes_S^m(T^*\mathbb S^n))$. Then
$$
\mathscr P^-_\lambda w(x)=\int_{\mathbb S^n} P(x,\nu)^{n+\lambda}(\otimes^m(\mathcal A^{-1}_-(x,\xi_-(x,\nu)))^T) w(\nu)\,dS(\nu)
$$
where the map $\xi_-$ is defined in~\eqref{defxi+}.
\end{lemm}
%%%%%%%%%%%%%%%%%%%%%%%%%%%%%%%%%%%%%%%%%%%%%%%%%%%%%%%%%%%%%%%%%%%%%%%%%%%%%%%%
\begin{proof}
Making the change of variables $\xi=\xi_-(x,\nu)$ defined in~\eqref{defxi+}, and using~\eqref{Phiofxi} and~\eqref{conformaldxi}, we have
$$
\begin{gathered}
\mathscr P^-_\lambda w(x)=\int_{S_x\mathbb H^{n+1}}\Phi_-(x,\xi)^{\lambda} (\otimes^m(\mathcal A^{-1}_-(x,\xi))^T) w(B_-(x,\xi))\,dS(\xi)
\\
=\int_{\mathbb S^n} P(x,\nu)^{n+\lambda}(\otimes^m(\mathcal A^{-1}_-(x,\xi_-(x,\nu)))^T) w(\nu)\,dS(\nu)
\end{gathered}
$$
as required.
\end{proof}
%%%%%%%%%%%%%%%%%%%%%%%%%%%%%%%%%%%%%%%%%%%%%%%%%%%%%%%%%%%%%%%%%%%%%%%%%%%%%%%%

%%%%%%%%%%%%%%%%%%%%%%%%%%%%%%%%%%%%%%%%%%%%%%%%%%%%%%%%%%%%%%%%%%%%%%%%%%%%%%%%
\smallsection{Poisson maps to eigenstates}
To show that $\mathscr P^-_\lambda w(x)$ is an eigenstate of the Laplacian, we use the following
%%%%%%%%%%%%%%%%%%%%%%%%%%%%%%%%%%%%%%%%%%%%%%%%%%%%%%%%%%%%%%%%%%%%%%%%%%%%%%%%
\begin{lemm}
  \label{l:poisson-2}
Assume that $w\in\mathcal D'(\mathbb S^n;\otimes^m (T^*\mathbb S^n))$ is the delta function centered at $e_1=\partial_{x_1}\in\mathbb S^n$ with
the value $e^*_{j_1+1}\otimes \dots\otimes e^*_{j_m+1}$,
where $1\leq j_1,\dots,j_m\leq n$.
Then under the identifications~\eqref{defpsi}
and~\eqref{e:udiffeo}, we have
$$
\mathscr P^-_\lambda w(z_0,z)=z_0^{n+\lambda}Z_{j_1}^*\otimes\dots\otimes Z_{j_m}^*.
$$
\end{lemm}
%%%%%%%%%%%%%%%%%%%%%%%%%%%%%%%%%%%%%%%%%%%%%%%%%%%%%%%%%%%%%%%%%%%%%%%%%%%%%%%%
\begin{proof}
We first calculate
$$
P(z,e_1)=z_0.
$$
It remains to show the following identity in the half-space model
\begin{equation}
  \label{e:identitay}
\mathcal A_-^{-T}(z,\xi_-(z,\nu))e^*_{j+1}=Z_j^*,\quad
1\leq j\leq n.
\end{equation}
One can verify~\eqref{e:identitay} by a direct computation: since $\mathcal A_-$ is an isometry, one can instead
calculate the image of $e_{j+1}$ under $\mathcal A_-$, and then apply to it the differentials of the maps $\psi$
and $\psi_1$ defined in~\eqref{defpsi} and~\eqref{e:udiffeo}.

Another way to show~\eqref{e:identitay} is to use the interpretation of $\mathcal A_-$ as parallel transport to conformal
infinity, see~\eqref{e:ptarsk}.
Note that under the diffeomorphism $\psi_1:\mathbb B^{n+1}\to\mathbb U^{n+1}$,
$\nu=e_1$ is sent to infinity and geodesics terminating at $\nu$, to straight lines parallel to the $z_0$ axis.
By~\eqref{computenabla}, the covector field $Z_j^*$ is parallel along these geodesics and orthogonal to their tangent vectors.
It remains to verify that the limit of the field $\rho_0 Z_j^*$ along these geodesics as $z\to \infty$, considered
as a covector in the ball model, is equal to $e_{j+1}^*$.
\end{proof}
%%%%%%%%%%%%%%%%%%%%%%%%%%%%%%%%%%%%%%%%%%%%%%%%%%%%%%%%%%%%%%%%%%%%%%%%%%%%%%%%

%%%%%%%%%%%%%%%%%%%%%%%%%%%%%%%%%%%%%%%%%%%%%%%%%%%%%%%%%%%%%%%%%%%%%%%%%%%%%%%%
\smallsection{Proof of Lemma~\ref{l:laplacian}}
It suffices to show that for each $\nu\in\mathbb S^n$, if $w$ is a delta function centered at $\nu$
with value being some symmetric trace-free tensor in $\otimes_S^m T^*_\nu\mathbb S^n$, then
$$
\big(\Delta+\lambda(n+\lambda)-m\big) \mathscr P^-_\lambda w=0,\quad
\nabla^*\mathscr P^-_\lambda w=0,\quad
\mathcal T(\mathscr P^-_\lambda w)=0.
$$
Since the group of symmetries $G$ of $\mathbb H^{n+1}$ acts transitively on $\mathbb S^n$,
we may assume that $\nu=\partial_1$. Applying Lemma~\ref{l:poisson-2}, we write in the upper half-plane
model,
$$
\mathscr P^-_\lambda w=z_0^{n+\lambda} u_0,\quad
u_0\in E^{(m)}_0,\quad
\mathcal T(u_0)=0.
$$
It immediately follows that $\mathcal T(\mathscr P^-_\lambda w)=0$. To see the other two identities,
it suffices to apply Lemma~\ref{DeltaE_0} together with the formula
$$
\Delta z_0^{n+\lambda}=-\lambda(n+\lambda) z_0^{n+\lambda}.
$$

%%%%%%%%%%%%%%%%%%%%%%%%%%%%%%%%%%%%%%%%%%%%%%%%%%%%%%%%%%%%%%%%%%%%%%%%%%%%%%%%
\smallsection{Injectivity of Poisson}
Notice that $\mathscr P^-_\lambda$ is an analytic family of operators in $\la$. We define the set
\begin{equation}
\label{e:R-m}
\mc{R}_m=\left\{\begin{array}{ll}
-\tfrac{n}{2}-\tfrac{1}{2}\nn_0 & \textrm{ if }n>1 \textrm{ or }m=0\\
-\tfrac{1}{2}\nn_0 & \textrm{ if }n=1 \textrm{ and }m>0  
\end{array}\right.
\end{equation}
and we will prove that if $\la\notin  \mc{R}_m$ and $w\in\mc{D}'(\mathbb{S}^n;\otimes_S^mT^*\mathbb{S}^n)$ is trace-free, 
then $\mathscr P^-_\lambda(w)$ has a weak asymptotic expansion at the conformal infinity with the leading term
given by a multiple of $w$, proving injectivity of $\mathscr P^-_\lambda$.
We shall use the $0$-cotangent bundle approach in the ball model and rewrite $\mc{A}^{-1}_\pm(x,\xi_\pm(x,\nu))$ as the parallel transport $\tau(y',y)$ in ${^0}T\overline{\mathbb{B}^{n+1}}$ with $\psi(x)=y$ and $y'=\nu$, as explained in~\eqref{e:ptarsk}. 
Let $\rho\in \CI(\overline{\mathbb{B}^{n+1}})$ be a smooth boundary defining function which satisfies $\rho>0$ in 
$\mathbb{B}^{n+1}$, $|d\rho|_{\rho^2g_H}=1$ near $\sph^n=\{\rho=0\}$, where $g_H$ is the hyperbolic metric on the ball.
We can for example take the function $\rho=\rho_0$ defined in~\eqref{e:rho} and smooth it near the center $y=0$ of the ball.
Such function is called \emph{geodesic boundary defining function} and induces a diffeomorphism
\begin{equation}\label{collarpsi} 
\theta: [0,\eps)_t\x\sph^n\to \overline{\mathbb{B}^{n+1}}\cap\{\rho< \eps\},
\quad \theta(t,\nu):=\theta_t(\nu)
\end{equation}
where $\theta_t$ is the flow at time $t$ of the gradient $\nabla^{\rho^2g_H}\rho$ of $\rho$ (denoted also $\pl_\rho$) with respect to the metric 
$\rho^2g_H$. For $\rho$ given in~\eqref{e:rho}, we have for $t$ small
$$
\theta(t,\nu)={2-t\over 2+t}\,\nu,\quad
\nu\in\mathbb S^n.
$$
For a fixed geodesic boundary defining function $\rho$, one can identify, over the boundary 
$\sph^n$ of $\overline{\mathbb{B}^{n+1}}$, 
the bundle $T^*\sph^n$ and $T\sph^n$ with the bundles ${^0}T^*\sph^n:={^0}T_{\sph^n}^*\overline{\mathbb{B}^{n+1}}\cap \ker\iota_{\rho\pl_\rho}$ simply by the isomorphism
$v\mapsto \rho^{-1}v$ (and we identify their duals $T\sph^n$ and
${^0}T\sph^n$ as well). Similarly, over $\sph^n$,  $E^{(m)}\cap \ker\iota_{\rho\pl_\rho}$ identifies
with $\otimes^m_ST^*\sph^n\cap \ker\mc{T}$ by the map $v\mapsto \rho^{-m}v$. 
We can then view the Poisson operator as an operator 
\[\mathscr P^-_\lambda: \mc{D}'(\sph^n; E^{(m)}\cap\ker\iota_{\rho\pl_\rho})\to 
\mc{C}^\infty(\mathbb{B}^{n+1}; \otimes_S^m ( {^0}T^*\overline{\mathbb{B}^{n+1}})).\]

%%%%%%%%%%%%%%%%%%%%%%%%%%%%%%%%%%%%%%%%%%%%%%%%%%%%%%%%%%%%%%%%%%%%%%%%%%%%%%%%
\begin{lemm}
  \label{injectiv}
Let $w\in \mc{D}'(\sph^n; E^{(m)}\cap \ker \iota_{{\rho_0}\pl_{\rho_0}})$ and assume that 
$\la\notin \mc{R}_m$. Then $\mathscr P^-_\lambda(w)$ has a weak asymptotic expansion at $\sph^n$ as follows: for each $\nu\in \sph^n$, there exists a neighbourhood $V_\nu\subset \bbar{\mathbb{B}^{n+1}}$ of $\nu$ and a boundary defining function $\rho=\rho_\nu$ such that for any $\varphi \in \mc{C}^\infty(V_\nu\cap\sph^n; \otimes_S^m({^0}T\sph^n))$,  
there exist $F_\pm\in \mc{C}^\infty([0,\eps))$ such that for $t>0$ small
\begin{equation}\label{actionP}
\begin{gathered} 
\int_{\sph^n}  \langle \mathscr P^-_\lambda(w)(\theta(t,\nu)),\otimes^m (\tau(\theta(t,\nu),\nu)).
\varphi(\nu)\rangle dS_{\rho}(\nu)= \\
\left\{\begin{array}{ll}
t^{-\la}F_-(t)+t^{n+\la}F_+(t), &\la\notin -n/2+\nn;\\
t^{-\la}F_-(t)+t^{n+\la}\log(t)F_+(t), & \la\in -n/2+\nn.
\end{array}\right.\end{gathered}\end{equation} 
using the product collar neighbourhood \eqref{collarpsi} associated to $\rho$, and moreover one has 
\begin{equation}\label{Blambda}
F_-(0)=C \frac{\Gamma(\la+\ndemi)}{(\la+n+m-1)\Gamma(\la+n-1)} \cjg e^{\la f}.w,\varphi\cjd 
\end{equation}
for some $f\in \CI(\mathbb S^n)$ satisfying $\rho=\tfrac{1}{4}e^{f}\rho_0+\mc{O}(\rho)$ near $\rho=0$ and $C\not=0$ a constant depending only on $n$.
Here $dS_\rho$ is the Riemannian measure for the metric $(\rho^2g_H)|_{\sph^n}$ and the distributional pairing on $\sph^n$ is with respect to this measure.
\end{lemm}
%%%%%%%%%%%%%%%%%%%%%%%%%%%%%%%%%%%%%%%%%%%%%%%%%%%%%%%%%%%%%%%%%%%%%%%%%%%%%%%%

%%%%%%%%%%%%%%%%%%%%%%%%%%%%%%%%%%%%%%%%%%%%%%%%%%%%%%%%%%%%%%%%%%%%%%%%%%%%%%%%
\begin{figure}
\includegraphics{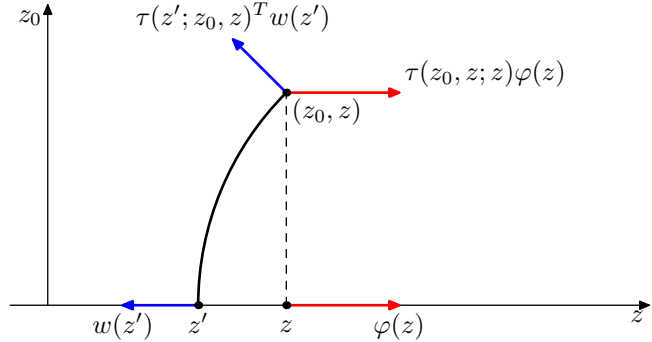}
\caption{The covector $w(z')$, the vector $\varphi(z)$, and their parallel transports
to $(z_0,z)$ viewed in the 0-bundles, for the case $m=1$.}
\label{f:transport}
\end{figure}
%%%%%%%%%%%%%%%%%%%%%%%%%%%%%%%%%%%%%%%%%%%%%%%%%%%%%%%%%%%%%%%%%%%%%%%%%%%%%%%%

%%%%%%%%%%%%%%%%%%%%%%%%%%%%%%%%%%%%%%%%%%%%%%%%%%%%%%%%%%%%%%%%%%%%%%%%%%%%%%%%
\begin{proof}
First we split $w$ into $w_1+w_2$ where $w_1$ is supported near $\nu\in \sph^n$ and $w_2$ is zero near $\nu$.
For the case where $w_2$ has support at positive distance from the support of $\varphi$, we have for any geodesic boundary defining function $\rho$ that 
\[t\mapsto t^{-n-\la}\int_{\sph^n}  \langle\mathscr P^-_\lambda(w_2)(\theta(t,\nu)),\otimes^m (\tau(\theta(t,\nu),\nu)).
\varphi(\nu)\rangle dS_\rho(\nu)
\in \mc{C}^\infty([0,\eps)),\]
this is a direct consequence of Lemma~\ref{l:poisson-1} and the following smoothness properties
\[
\begin{gathered}
(y,\nu)\mapsto \log\big(P(\psi^{-1}(y),\nu)/\rho(y)\big)\in \mc{C}^\infty(\overline{\mathbb{B}^{n+1}}\x\sph^n \setminus {\rm diag}(\sph^n\x\sph^n))\\
\tau(\cdot,\cdot)\in \mc{C}^\infty(\overline{\mathbb{B}^{n+1}}\x \overline{\mathbb{B}^{n+1}}  \setminus {\rm diag}(\sph^n\x\sph^n));
{^0}T^*\bbar{\mathbb{B}^{n+1}}\otimes {^0}T\bbar{\mathbb{B}^{n+1}}).
\end{gathered}
\]
This reduces the consideration of the Lemma to the case where $w$ is $w_1$ supported near $\nu$, 
and to simplify we shall keep the notation $w$ instead of $w_1$. 
We thus consider now $w$ and $\varphi$ to have support near $\nu$.
For convenience of calculations and as we did before, 
we work in the half-space model $\rr^+_{z_0}\x \rr_{z}^n$ by mapping $\nu$ to 
$(z_0,z)=(0,0)$ (using the composition of a rotation on the ball model with
the map defined in~\eqref{e:udiffeo}) and we choose a nighbourhood $V_\nu$ of $\nu$ which is mapped to 
$z_0^2+|z|^2<1$ in $\mathbb{U}^{n+1}$ and choose the geodesic defining function $\rho=z_0$ (and thus
$\theta(z_0,z)=(z_0,z)$). (See Figure~\ref{f:transport}.)
The geodesic boundary defining function $\rho_0=\frac{2(1-|y|)}{1+|y|}$ in the ball 
equals 
\begin{equation}\label{expofrho0}
\rho_0(z_0,z)=4z_0/(1+z_0^2+|z|^2)
\end{equation} 
in the half-space model. The metric $dS_\rho$ becomes the Euclidean metric $dz$ on $\rr^n$ near $0$ and $w$ has compact support in $\rr^n$. 
By~\eqref{e:udiffeo} and~\eqref{e:poisson-y}, the Poisson kernel in these coordinates becomes
$$
\widetilde P(z_0,z;z')=e^{f(z')}P(z_0,z;z')
\textrm{ with } P(z_0,z;z'):={z_0\over z_0^2+|z-z'|^2}, \,\,f(z')=\log(1+|z'|^2)
$$
where $z,z'\in\rr^n$ and $z_0>0$. One has $\rho=\tfrac{1}{4}e^f\rho_0+\mc{O}(\rho)$ near $\rho=0$.

In the Appendix of \cite{GMP}, the parallel transport $\tau(z_0,z; 0,z')$ is computed for $z'\in\rr^n$ is a neighbourhood of $0$: in the local orthonormal basis $Z_0=z_0\pl_{z_0},Z_i=z_0\pl_{z_i}$ of the bundle ${^0}T\mathbb{U}^{n+1}$, near $\nu$,
the matrix of $\tau(z_0,z;z'):=\tau(z_0,z;0,z')$ is given by 
\[
\begin{gathered}
\tau_{00}=1-2P(z_0,z;z')\frac{|z-z'|^2}{z_0}, \quad  \tau_{0i}=-\tau_{i0}=-2z_0(z_i-z_i')\frac{P(z_0,z;z')}{z_0}, \\
\tau_{ij}=\delta_{ij}-2P(z_0,z;z')\frac{(z_i-z_i').(z_j-z_j')}{z_0}.\end{gathered}
\] 
In particular we see that $\tau(z_0,z;z)$ is the identity matrix in the basis $(Z_i)_i$ and 
thus $\tau(\theta(z_0,z),z)$ as well.
We denote $(Z_j^*)_j$ the dual basis to $(Z_j)_j$ as before.

Now, we use the correspondence between symmetric tensors and homogeneous polynomials to facilitate computations, as explained in Section \ref{symtens}. To $\mc{S}(Z^*_I)$, we associate the polynomial on $\rr^n$ given by
$$
P_I(x)=\mc{S}(Z_I^*)\bigg(\sum_{i=1}^nx_iZ_I,\dots,\sum_{i=1}^nx_iZ_I\bigg)=x_I
$$
where $x_I=\prod_{k=1}^m x_{i_k}$ if $I=(i_1,\dots,i_m)$.  
We denote by $\Pol^m(\rr^n)$ the space of homogeneous polynomials of degree $m$ on $\rr^n$ and 
$\Pol_0^m(\rr^n)$ those which are harmonic (thus corresponding to trace free symmetric tensors in $E^{(m)}_0$).
Then we can write $w=\sum_{\alpha}w_\alpha p_\alpha(x)$ for some 
$w_\alpha\in \mc{D'}(\rr^n)$ supported near $0$ and $p_\alpha(x)\in {\rm Pol}^m_0(\rr^n)$. Each 
$p_\alpha(x)$ composed with the linear map $\tau(z';z_0,z)|_{Z_0^\perp}$ becomes the homogeneous polynomial in $x$
\[p_{\alpha}\big(x-2 (z-z')\cjg z-z',x\cjd. \tfrac{P(z_0,z;z')}{z_0}\big)\]
where $\cjg \cdot,\cdot \cjd$ just denotes the Euclidean scalar product.
To prove the desired asymptotic expansion, it suffices to take $\varphi\in \CI_0([0,\infty)_{z_0}\x \rr^n)$ and to  
analyze the following homogeneous polynomial in $x$ as $z_0\to 0$ 
\begin{equation}
  \label{expressiontofourier}
\begin{gathered} 
\int_{\rr^n} \sum_{\alpha} \Big\cjg e^{(n+\la)f}w_\alpha , \varphi(z_0,z)
P(z_0,z;\cdot)^{n+\la}p_\alpha\big(x-2 (z-\cdot)\cjg z-\cdot,x\cjd. \tfrac{P(z_0,z;\cdot)}{z_0}\big)\Big\cjd dz
\end{gathered}
\end{equation}
where the bracket $\cjg w_\alpha, \cdot\cjd$ means the distributional pairing coming from pairing with respect to the canonical measure $dS$ on $\sph^n$, which in $\rr^n$ becomes the measure $4^ne^{-nf}dz$, and so the $e^{nf}$ in \eqref{expressiontofourier} cancels out if one works with the Euclidean measure $dz$, which we do now. 
We remark a convolution kernel in $z$ and thus apply Fourier transform in $z$ (denoted $\mc{F}$): denoting $P(z_0;|z-z'|)$ for $P(z_0,z;z')$, the integral \eqref{expressiontofourier} becomes (up to non-zero multiplicative constant)
\[
I(z_0,x):=\sum_\alpha \Big\cjg \mc{F}^{-1}(e^{\la f}w_\alpha), \mc{F}(\varphi).
\mc{F}_{\zeta\to \cdot}\Big(P(z_0;|\zeta|)^{n+\la}p_\alpha\big(x-2\tfrac{\zeta\cjg \zeta,x\cjd}{z_0} P(z_0;|\zeta|)\big)\Big)\Big\cjd_{\rr^n}
\]
We can expand $p_\alpha\big(x-2\frac{\zeta\cjg \zeta,x\cjd}{z_0} P(z_0;|\zeta|)\big)$ so that 
\[
P(z_0;|\zeta|)^{n+\la}p_\alpha\big(x-2\tfrac{\zeta\cjg \zeta,x\cjd}{z_0} P(z_0;|\zeta|)\big)=\sum_{r=0}^{m}Q_{r,\alpha}(\zeta,x)
z_0^{-r}2^rP(z_0;|\zeta|)^{n+\la+r}
\]
where $Q_{r,\alpha}(\zeta)$ is homogeneous of degree $m$ in $x$ and $2r$ in $\zeta$.
Now we have (for some $C\not=0$ independent of $\la,r,\alpha$)
\[\begin{gathered}
\frac{2^r}{z_0^r}\mc{F}_{\zeta\to \xi}(P^{n+\la+r}(z_0;|\zeta|)Q_{r,\alpha}(\zeta,x))=\\
\frac{C2^{-\la}z_0^{-\la}}{\Gamma(\la+n+r)}
[Q_{r,\alpha}(i\pl_\zeta,x)(|\zeta|^{\la+\ndemi+r}K_{\la+\ndemi+r}(|\zeta|))]_{\zeta=z_0\xi}\\
\end{gathered}\]
where $K_\nu(\cdot)$ is the modified Bessel function (see \cite[Chapter~9]{AbSt}) defined by 
\begin{equation}\label{besseldef}
\begin{gathered}
K_\nu(z):=\frac{\pi}{2}\frac{(I_{-\nu}(z)-I_{\nu}(z))}{\sin(\nu\pi)}\,\, \textrm{ if } I_{\nu}(z):= \sum_{\ell=0}^\infty \frac{1}{\ell!\Gamma(\ell+\nu+1)}\Big(\frac{z}{2}\Big)^{2\ell+\nu}
\end{gathered} 
\end{equation} 
satisfying  that $|K_\nu(z)|=\mc{O}(\frac{e^{-z}}{\sqrt{z}})$ as $z\to \infty $, and for $s\notin\nn_0$
\[
\mc{F}((1+|z|^2)^{-s})(\xi)= \frac{2^{-s+1}(2\pi)^{n/2}}{\Gamma(s)}|\xi|^{s-n/2}K_{s-n/2}(|\xi|).
\]
When $\la\notin (-\ndemi+ \zz) \cup (-n-\demi\nn_0)$, we have 
\begin{equation}
  \label{expansion}
\begin{gathered}
2^{-\la}z_0^{-\la}Q_{r,\alpha}(i\pl_{\zeta},x)(|\zeta|^{\la+\ndemi+r}K_{\la+\ndemi+r}(|\zeta|))|_{\zeta=z_0\xi}= 
\frac{ 2^{r+\ndemi}\pi z_0^{-\la} }{2\sin(\pi(\la+\ndemi+r))}\\
 \x \, \Big( \sum_{\ell=0}^\infty \frac{z_0^{2(\ell-r)}Q_{r,\alpha}(i\pl_\xi,x)(|\tfrac{1}{2}\xi|^{2\ell})}{\ell!\Gamma(\ell-\la-\ndemi-r+1)}
 -z_0^{2\la+n}\sum_{\ell=0}^\infty \frac{z_0^{2\ell}Q_{r,\alpha}(i\pl_\xi,x)(|\tfrac{1}{2}\xi|^{2(\la+r+\ell)+n})}{\ell!\Gamma(\ell+\la+\ndemi+r+1)} \Big).
\end{gathered}
\end{equation}
Here the powers of $|\xi|$ are homogeneous distributions 
(note that for $\lambda\not\in\mathcal R_m$, the exceptional
powers $|\xi|^{-n-j}$, $j\in\mathbb N_0$, do not appear)
and the pairing of \eqref{expansion} with $\mc{F}^{-1}(e^{\la f}w_\alpha)\mc{F}(\varphi)$ makes sense since this distribution is Schwartz as $w_\alpha$ has compact support.
We deduce from this expansion that for any $w_\alpha\in \mc{D}'(\rr^n)$ supported near $0$ and 
$\varphi\in \CI_0(\rr^n)$, when $\la\notin (-\ndemi+\zz)\cup\mathcal (-n-\demi\nn_0)$
\[
I(z_0,x)=z_0^{-\la}F_-(z_0,x)+
z_0^{n+\la}F_+(z_0,x)
\]
for some smooth function $F_\pm\in \mc{C}^\infty([0,\eps)\x \rr^n)$ homogeneous of degree $m$ in $x$. 
We need to analyze $F_-(0,x)$, which is obtained by computing the term of order $0$ in $\xi$ in the expansion \eqref{expansion}
(that is, the terms with $\ell=r$ in the first sum; note that the
terms with $\ell<r$ in this sum are zero): we obtain for some universal constant $C\not=0$ 
\[
F_-(0,x)=C\sum_{\alpha}\cjg e^{\la f}w_\alpha,\varphi\cjd_{\rr^n} \sum_{r=0}^m 
\frac{(-1)^{r} 2^{-r}\Gamma(\la+\ndemi)}{r!\Gamma(\la+n+r)}Q_{r,\alpha}(i\pl_\xi,x)(|\xi|^{2r})
\]
where we have used the inversion formula $\Gamma(1-z)\Gamma(z)=\pi/\sin(\pi z)$
and $Q_{r,\alpha}(i\pl_\xi,x)(|\xi|^{2r})$ is constant in $\xi$. Using Fourier transform, we notice that 
\[
Q_{r,\alpha}(i\pl_\xi,x)(|\xi|^{2r})=\Delta_\zeta^{r}Q_{r,\alpha}(\zeta,x)|_{\zeta=0}= \Delta_\zeta^r( p_\alpha(x- \zeta \cjg \zeta,x\cjd))|_{\zeta=0}\]
We use Lemma \ref{littlecomput} to deduce that 
\[F_-(0,x)=C\sum_{\alpha}\cjg e^{\la f}w_\alpha,\varphi \cjd_{\rr^n} p_{\alpha}(x)m!\frac{\Gamma(\la+\ndemi)}{\Gamma(\la+n+m)}\sum_{r=0}^m 
\frac{(-1)^{r}\Gamma(\la+n+m)}{(m-r)!\Gamma(\la+n+r)}.
\]
The sum over $r$ is a non-zero polynomial of order $m$ in $\la$, and using the binomial formula, we see that its roots are 
$\la=-n-m+2,\dots,-n+1$, therefore we deduce that
\[
F_-(0,x)=C\cjg e^{\la f}w,\varphi\cjd_{\rr^n} \frac{\Gamma(\la+\ndemi)}{(\la+n+m-1)\Gamma(\la+n-1)}.
\]
We obtain the claimed  result except for $\la\in -\ndemi+\nn$ by using that the volume measure 
on $\sph^n$ is $4^{-n}e^{nf}$.

Now assume that $\la=-n/2+j$ with $j\in\nn$. The Bessel function satisfies for $j\in\nn$:  
\[|\xi|^{j}K_{j}(|\xi|)=-\sum_{\ell=0}^{j-1}\frac{(-1)^\ell 2^{j-1-2\ell}(j-\ell-1)!}{\ell !}|\xi|^{2\ell}+|\xi|^{2j}(\log(|\xi|)L_{j}(|\xi|)
+H_j(|\xi|))\] 
for some function $L_{j},H_j\in \CI(\rr^+)\cap L^2(\rr^+)$ with $L_j(0)\not=0$. Then we apply the same arguments as before and this implies the desired statement. 
\end{proof}
%%%%%%%%%%%%%%%%%%%%%%%%%%%%%%%%%%%%%%%%%%%%%%%%%%%%%%%%%%%%%%%%%%%%%%%%%%%%%%%%

We obtain as a corollary:
%%%%%%%%%%%%%%%%%%%%%%%%%%%%%%%%%%%%%%%%%%%%%%%%%%%%%%%%%%%%%%%%%%%%%%%%%%%%%%%%
\begin{corr} 
  \label{l:injdone}
For $m\in\nn_0$ and $\la\notin \mc{R}_m$, the operator 
$\mathscr P^-_\lambda :\mc{D}'(\mathbb S^n;\otimes_S^m(T^*\mathbb S^n)\cap \ker\mc{T})\to 
\mc{C}^\infty(\hh^{n+1}; \otimes_S^m(T^*\hh^{n+1}))$ is injective.
\end{corr}
%%%%%%%%%%%%%%%%%%%%%%%%%%%%%%%%%%%%%%%%%%%%%%%%%%%%%%%%%%%%%%%%%%%%%%%%%%%%%%%%
This corollary immediately implies the injectivity part of Theorem~\ref{t:laplacian} in Section~\ref{s:rr-2}.

%%%%%%%%%%%%%%%%%%%%%%%%%%%%%%%%%%%%%%%%%%%%%%%%%%%%%%%%%%%%%%%%%%%%%%%%%%%%%%%%
%                                  SECTION 7                                   %
%%%%%%%%%%%%%%%%%%%%%%%%%%%%%%%%%%%%%%%%%%%%%%%%%%%%%%%%%%%%%%%%%%%%%%%%%%%%%%%%
\section{Expansions of eigenstates of the Laplacian}
\label{s:poisson-is}

In this section, we show the surjectivity of the Poisson operator $\mathscr P^-_\lambda$
(see Theorem~\ref{t:laplacian} in Section~\ref{s:rr-2}). For that,
we take an eigenstate $u$ of the Laplacian on $M$ and lift it
to $\mathbb H^{n+1}$. The resulting tensor is tempered
and thus expected to have a weak asymptotic expansion at the conformal boundary $\mathbb S^n$;
a precise form of this expansion is obtained by
a careful analysis of both the Laplacian and the divergence-free condition.
We then show that
$u=\mathscr P^-_\lambda w$, where $w$ is some constant times the coefficient of
$\rho^{-\lambda}$ in the expansion of $u$ (compare with Lemma~\ref{injectiv}).

%%%%%%%%%%%%%%%%%%%%%%%%%%%%%%%%%%%%%%%%%%%%%%%%%%%%%%%%%%%%%%%%%%%%%%%%%%%%%%%%
\subsection{Indicial calculus and general weak expansion}

Recall the bundle $E^{(m)}$ defined in~\eqref{e:e-m-def}.
The operator $\Delta$ acting on $\mc{C}^\infty(\overline{\mathbb{B}^{n+1}}; E^{(m)})$ is an elliptic differential operator of order $2$ which lies in the 0-calculus of Mazzeo--Melrose~\cite{MM}, which essentially means that it is an elliptic polynomial in elements of the Lie algebra $\mc{V}_0(\overline{\mathbb{B}^{n+1}})$ of smooth vector fields vanishing at the boundary of the closed unit ball $\overline{\mathbb{B}^{n+1}}$. 
Let $\rho\in \CI(\overline{\mathbb{B}^{n+1}})$ be a smooth geodesic boundary defining function
(see the paragraph preceding~\eqref{collarpsi}).
The theory developped by Mazzeo~\cite{Ma} shows that solutions of  $\Delta u=su$ which are in $\rho^{-N}L^2(\mathbb{B}^{n+1};E^{(m)})$ for some $N$ have weak asymptotic expansions at the boundary $\sph^n=\partial \overline{\mathbb{B}^{n+1}}$ 
where $\rho$ is any geodesic boundary defining function.
To make this more precise, we introduce the \emph{indicial family} of $\Delta$: if $\la\in\cc, \nu\in\sph^n$, then there exists 
a family $I_{\la,\nu}(\Delta)\in {\rm End}(E^{(m)}(\nu))$ depending smoothly on 
$\nu\in \sph^n$ and holomorphically on $\la$ so that for all $u\in \mc{C}^\infty(\overline{\mathbb{B}^{n+1}}; E^{(m)})$, 
\[
t^{-\lambda}\Delta (\rho^{\la} u)(\theta(t,\nu))=I_{\la,\nu}(\Delta)u(\theta(0,\nu))+ \mc{O}(t)
\]
near $\sph^n$, where the remainder is estimated with respect to the metric $g_H$.
Notice that $I_{\la,\nu}(\Delta)$ is independent of the choice of boundary defining function $\rho$.

For $\sigma\in\cc$, 
the \emph{indicial set} ${\rm spec}_b(\Delta-\sigma;\nu)$ at $\nu\in\mathbb S^n$ 
of $\Delta-\sigma$ is the set 
\[
{\rm spec}_b(\Delta-\sigma;\nu):=\{ \la\in\cc\mid I_{\la,\nu}(\Delta)-\sigma\Id\textrm{ is not invertible}\}.
\] 
Then \cite[Theorem~7.3]{Ma} gives the following%
\footnote{The full power of~\cite{Ma} is not needed for this lemma. In fact,
it can be proved in a direct way by viewing the equation $(\Delta-\sigma)u=0$
as an ordinary differential equation in the variable $\log\rho$. The indicial operator
gives the constant coefficient principal part and the remaining terms are exponentially
decaying; an iterative argument shows the needed asymptotics.}
%%%%%%%%%%%%%%%%%%%%%%%%%%%%%%%%%%%%%%%%%%%%%%%%%%%%%%%%%%%%%%%%%%%%%%%%%%%%%%%%
\begin{lemm}
  \label{mazzeo}
Fix $\sigma$ and assume that $\mathrm{spec}_b(\Delta-\sigma;\nu)$ is independent of $\nu\in \sph^n$.
If $u\in \rho^{\delta}L^2(\overline{\mathbb{B}^{n+1}}; E^{(m)})$ with respect to the Euclidean measure for some $\delta\in\rr$, and $(\Delta -\sigma)u=0$, then $u$ has a weak asymptotic expansion at 
$\sph^n=\{\rho=0\}$ of the form 
\[u =\sum_{\la\in {\rm spec}_b(\Delta-\sigma)\atop {\rm Re}(\la)>\delta-1/2}
\sum_{\ell\in\nn_0,\atop 
{\rm Re}(\la)+\ell<\delta-1/2+N} \sum_{p=0}^{k_{\la,\ell}}
\rho^{\la+\ell}(\log\rho)^{p}w_{\la,\ell,p}+\mc{O}(\rho^{\delta+N-\demi-\eps})\]
for all $N\in\nn$ and all $\eps>0$ small, where $k_{\la,\ell}\in\nn_0$, and $w_{\la,\ell,p}$ are in the Sobolev spaces 
\[w_{\la,\ell,p}\in H^{-{\rm Re}(\la)-\ell+\delta-\demi}(\sph^n; E^{(m)}).\] 
Here the weak asymptotic means that for any $\varphi\in \mc{C}^{\infty}(\sph^n)$, as $t\to 0$
\begin{equation}\label{weakas}
\begin{split} 
\int_{\sph^n} u(\theta(t,\nu))\varphi(\nu)dS_\rho(\nu)= & \sum_{\la\in {\rm spec}_b(\Delta-\sigma)\atop {\rm Re}(\la)>\delta-1/2}
\sum_{\ell\in \nn_0\atop 
{\rm Re}(\la)+\ell<\delta-1/2+N} \sum_{p=0}^{k_{\la,\ell}}
t^{\la+\ell}\log(t)^{p}\cjg w_{\la,\ell,p},\varphi\cjd\\
& + \mc{O}(t^{\delta+N-\demi-\eps})\end{split}
\end{equation} 
where $dS_\rho$ is measure on $\sph^n$ induced by the metric $(\rho^2g_H)|_{\sph^n}$ and the distributional pairing is with respect to this measure. Moreover the remainder $\mc{O}(t^{\delta+N-\demi-\eps})$ is conormal in the sense that it remains an $\mc{O}(t^{\delta+N-\demi-\eps})$ after 
applying any finite number of times the operator $t\pl_t$, and it depends on some Sobolev norm of $\varphi$.
\end{lemm}
%%%%%%%%%%%%%%%%%%%%%%%%%%%%%%%%%%%%%%%%%%%%%%%%%%%%%%%%%%%%%%%%%%%%%%%%%%%%%%%%

\noindent\textbf{Remark.} The existence of the expansion \eqref{weakas} proved by Mazzeo in \cite[Theorem~7.3]{Ma} is independent of the choice of $\rho$, but the coefficients in the expansion depend on the choice of $\rho$. 
Let $\la_0\in {\rm spec}_b(\Delta-\sigma)$ with ${\rm Re}(\la_0)>\delta-1/2$ be an element in the indicial set and assume that $k_{\la_0,0}=0$, which means that the exponent $\rho^{\la_0}$ in the weak expansion 
$\eqref{weakas}$ has no log term. Assume also that there is no element $\la\in 
{\rm spec}_b(\Delta-\sigma)$ with ${\rm Re}(\la_0)>{\rm Re}(\la)>\delta-1/2$ such that $\la\in \la_0-\nn$.
Then it is direct to see from the weak expansion that for a fixed function $\chi\in 
\mc{C}^\infty(\mathbb{B}^{n+1})$ equal to $1$ near $\sph^n$ and supported close to $\sph^n$ and for 
each $\varphi\in \mc{C}^\infty(\mathbb{B}^{n+1})$, the Mellin transform
\[
h(\zeta):= \int_{\mathbb{B}^{n+1}}\rho(y)^\zeta\chi(y)\varphi(y)u(y)\, {\rm dvol}_{g_H}(y),\quad
\Re\zeta>n+{1\over 2}-\delta,
\]
(with values in $E^m$) has a meromorphic extension to $\zeta\in\cc$ with a simple pole at $\zeta=n-\la_0$ and residue 
\begin{equation}\label{residue}
{\rm Res}_{\zeta=n-\la_0}h(\zeta)=\cjg w_{\la_0,0,0},\varphi|_{\sph^n}\cjd.
\end{equation}
As an application, if $\rho'$ is another geodesic boundary defining function, one has 
$\rho=e^{f}\rho'+\mc{O}(\rho')$ for some $f\in \mc{C}^\infty(\sph^n)$ and we deduce that if 
$w'_{\la_0,0,0}$ is the coefficient of $(\rho')^{\la_0}$ in the weak expansion of $u$ using $\rho'$, then  as distribution on $\sph^n$
\begin{equation}\label{relationleadingterm} 
w'_{\la_0,0,0}=e^{\la_0f}w_{\la_0,0,0}
\end{equation}
In particular, under the assumption above for $\la_0$
(this assumption can similarly be seen to be independent of the choice of $\rho$), if one knows the exponents of the asymptotic expansion, then proving that the coefficient of $\rho^{\la_0}$ term is nonzero can be done locally near any point of $\sph^n$ and with any choice of geodesic boundary defining function.

Finally, if $w_{\la_0,0,0}$ is the coefficient of $\rho_0^{\la_0}$ in the weak expansion with boundary defining function 
$\rho_0$ defined in~\eqref{e:rho} and if $\gamma^*u=u$ for some 
hyperbolic isometry $\gamma\in G$, we can use that $\rho_0\circ \gamma=N_\gamma^{-1}\cdot \rho_0+\mc{O}(\rho_0^2)$ 
near $\sph^n$, together with \eqref{residue} to get 
\begin{equation}\label{equivarianceu} 
L_{\gamma}^*w_{\la_0,0,0}= N_\gamma^{\la_0} w_{\la_0,0,0} \in \mc{D}'(\sph^n; E^{(m)})
\end{equation}
as distributions on $\sph^n$ (with respect to the canonical measure on $\sph^n$) with values in $E^{(m)}$.
Here $N_\gamma$, $L_\gamma$ are defined in Section~\ref{s:actionG}.
If we view 
$w_{\la_0,0,0}$ as a distribution with values in $\otimes_S^mT^*\sph^n$, the covariance becomes 
\begin{equation}\label{equivarianceubis} 
L_{\gamma}^*w_{\la_0,0,0}= N_\gamma^{\la_0-m} w_{\la_0,0,0} \in \mc{D}'(\sph^n; \otimes_S^mT^*\sph^n).
\end{equation}

Using the calculations of Section~\ref{s:laplacian-half}, we will compute the indicial family of the Laplacian on $E^{(m)}$:
%%%%%%%%%%%%%%%%%%%%%%%%%%%%%%%%%%%%%%%%%%%%%%%%%%%%%%%%%%%%%%%%%%%%%%%%%%%%%%%%
\begin{lemm}
  \label{l:indic-family}
Let $\Delta$ be the Laplacian on sections of $E^{(m)}$. Then the indicial set $\mathrm{spec}_b(\Delta-\sigma,\nu)$ does
not depend on $\nu\in\mathbb S^n$ and is equal to%
\footnote{Our argument in the next section does not actually
use the precise indicial roots, as long as they are independent of $\nu$
and form a discrete set.}
$$
\begin{gathered}
\bigcup_{k=0}^{\lfloor {m\over 2}\rfloor} \{\lambda\mid -\lambda^2+n\lambda+m+2k(2m+n-2k-2)=\sigma\}\\\cup
\bigcup_{k=0}^{\lfloor {m-1\over 2}\rfloor} \{\lambda\mid -\lambda^2+n\lambda+n+3(m-1)+2k(n+2m-2k-4)=\sigma\}.
\end{gathered}
$$  
\end{lemm}
%%%%%%%%%%%%%%%%%%%%%%%%%%%%%%%%%%%%%%%%%%%%%%%%%%%%%%%%%%%%%%%%%%%%%%%%%%%%%%%%
\begin{proof}
We consider an isometry mapping the ball model $\mathbb B^{n+1}$ to the half-plane model $\mathbb U^{n+1}$
which also maps $\nu$ to $0$ and do all the calculations in $\mathbb U^{n+1}$ with the geodesic boundary defining function $z_0$ near $0$.
By~\eqref{u'r},
each tensor $u\in E^{(m)}$ is determined uniquely by its $E^{(m)}_0$ and $E^{(m)}_1$ components,
which are denoted $u_0$ and $u_1$; therefore, it suffices to understand how the corresponding components of
$I_{\lambda,\nu}(\Delta)u$ are determined by $u_0,u_1$.
 We can use the geodesic boundary defining function $\rho=z_0$;
note that $\Delta z_0^\lambda=\lambda(n-\lambda)z_0^\lambda$ for all $\lambda\in\mathbb C$.

Assume first that $u$ satisfies $u_1=0$ and $u_0$ is constant in the frame $\mathcal S(Z_I^*)$. Then by Lemma~\ref{pi0pi1},
$$
\begin{gathered}
\pi_0(z_0^{-\lambda}\Delta(z_0^\lambda u))=R_0u_0=(\lambda(n-\lambda)+m) u_0+m(m-1)\mathcal S(z_0^{-2}h\otimes \mathcal T(u_0)),\\
\pi_1(z_0^{-\lambda}\Delta(z_0^\lambda u))=0.
\end{gathered}
$$
Assume now that $u$ satisfies $u_0=0$ and $u_1$ is constant in the frame $\mathcal S(Z_0^*\otimes Z_J^*)$.
Then by Lemma~\ref{pi0pi1},
$$
\begin{gathered}
\pi_0(z_0^{-\lambda}\Delta(z_0^\lambda u))=0,\\
\pi_1(z_0^{-\lambda}\Delta(z_0^\lambda u))=R_1u_1=(\lambda(n-\lambda)+n+3(m-1))u_1\\
+(m-1)(m-2)\mathcal S(Z_0^*\otimes z_0^{-2}h\otimes \mathcal T(u'_1)).
\end{gathered}
$$
We see that the indicial operator does not intertwine the $u_0$ and $u_1$ components and it remains to
understand for which $\lambda$ the number $s$ is a root of $R_0$ or $R_1$.

Next, we consider the decomposition~\eqref{decompositionoftensors}, where for $u\in E^{(m)}_0$,
we define $\mathcal I(u)={(m+2)(m+1)\over 2}\mathcal S(z_0^{-2}h\otimes u)$:
$$
u_0=\sum_{k=0}^{\lfloor {m\over 2}\rfloor} \mathcal I^k(\otimes u_0^k),\quad
u_1=\sum_{k=0}^{\lfloor {m-1\over 2}\rfloor} \mathcal S(Z_0^*\otimes \mathcal I^k(u_1^k)),
$$
where $u_0^k\in E^{(m-2k)}_0, u_1^k\in E^{(m-2k-1)}_0$ are trace-free tensors.
Using~\eqref{e:trart}, we calculate
$$
\begin{gathered}
R_0 (\mathcal I^k(u_0^k))=(\lambda(n-\lambda)+m)\mathcal I^k(u_0^k)+2\mathcal I(\mathcal T(\mathcal I^k(u_0^k)))\\
=\big(-\lambda^2+n\lambda+m+2k(2m+n-2k-2)\big) \mathcal I^k(u_0^k),\\
R_1(\mathcal S(Z_0^*\otimes \mathcal I^k(u_1^k)))\\
=(\lambda(n-\lambda)+n+3(m-1))\mathcal S(Z_0^*\otimes \mathcal I^k(u_1^k))
+2\mathcal S(Z_0^*\otimes \mathcal I(\mathcal T(\mathcal I^k(u_1^k))))\\
=\big(-\lambda^2+n\lambda+n+3(m-1)+2k(n+2m-2k-4)\big)\mathcal S(Z_0^*\otimes \mathcal I^k(u_1^k)),
\end{gathered}
$$
which finishes the proof of the lemma.
\end{proof}
%%%%%%%%%%%%%%%%%%%%%%%%%%%%%%%%%%%%%%%%%%%%%%%%%%%%%%%%%%%%%%%%%%%%%%%%%%%%%%%%

%%%%%%%%%%%%%%%%%%%%%%%%%%%%%%%%%%%%%%%%%%%%%%%%%%%%%%%%%%%%%%%%%%%%%%%%%%%%%%%%
\subsection{Weak expansions in the divergence-free case}

By Lemma~\ref{mazzeo}, we now know that solutions of $\Delta u=\sigma u$ 
which are trace-free symmetric tensors of order $m$ in some weighted $L^2$ space have weak asymptotic 
expansions at the boundary of $\bbar{\mathbb{B}^{n+1}}$ with exponents obtained from the indicial set of Lemma  \ref{l:indic-family}. 
In fact we can be more precise about the exponents which really appear 
in the weak asymptotic expansion if we ask that $u$ also be divergence-free:
%%%%%%%%%%%%%%%%%%%%%%%%%%%%%%%%%%%%%%%%%%%%%%%%%%%%%%%%%%%%%%%%%%%%%%%%%%%%%%%%
\begin{lemm}
  \label{asympexp}
Let $u\in \rho^{\delta}L^2(\bbar{\mathbb{B}^{n+1}}; E^{(m)})$ be a trace-free symmetric $m$-cotensor
with  $\rho$ a geodesic boundary defining function and $\delta\in(-\infty,\demi)$, where the measure is the Euclidean Lebesgue measure on the ball. Assume that $u$ is a nonzero divergence-free eigentensor for the Laplacian on hyperbolic space: 
\begin{equation}
  \label{e:u-eq}
\Delta u=\sigma u,\quad \nabla^*u=0
\end{equation}
for some $\sigma=m+\frac{n^2}{4}-\mu^2$ with ${\rm Re}(\mu)\in[0,\frac{n+1}{2}-\delta)$ and $\mu\not=0$.
Then the following weak expansion holds: for all $r\in[0, m]$, $N>0$, and $\eps>0$ small
\begin{equation}
  \label{e:asympexp}
\begin{split} 
(\iota_{\rho\partial_\rho})^ru&= \sum_{\ell\in\nn_0\atop {\rm Re}(-\mu)+\ell<N-\eps} 
\rho^{\ndemi-\mu +r+\ell} w^r_{-\mu,\ell} \\&+\sum_{\ell\in\nn_0\atop {\rm Re}(\mu)+\ell<N-\eps} 
\sum_{p=0}^{k_{\mu,\ell}}\rho^{\ndemi+\mu +r+\ell}\log(\rho)^{p}w^r_{\mu,\ell,p}
+\mc{O}(\rho^{\ndemi+N+r-\eps})
\end{split}
\end{equation}
with $w^r_{-\mu,\ell}\in H^{-\frac{n}{2}+ {\rm Re}(\mu)-r-\ell+\delta-\frac{1}{2}}(\sph^n; E^{(m-r)})$,
$w^r_{\mu,\ell,p}\in H^{-\frac{n}{2}- {\rm Re}(\mu)-r-\ell+\delta-\frac{1}{2}}(\sph^n; E^{(m-r)})$.
Moreover, if $\mu\notin{1\over 2}\mathbb N_0$, then $k_{\mu,\ell}=0$.
\end{lemm}
%%%%%%%%%%%%%%%%%%%%%%%%%%%%%%%%%%%%%%%%%%%%%%%%%%%%%%%%%%%%%%%%%%%%%%%%%%%%%%%%

\noindent\textbf{Remarks}. (i)
If $u$ is the lift to $\hh^{n+1}$ of an eigentensor on a compact quotient $M=\Gamma\backslash\hh^{n+1}$, then $u\in L^\infty(\mathbb{B}^{n+1};E^{(m)})$ and so for all $\eps>0$ the following regularity holds
\[
w_{-\mu,0}\in H^{-\frac{n}{2}+ {\rm Re}(\mu)-\eps}(\sph^n; E^{(m)}) , \quad w_{\mu,0,0}\in 
H^{-\frac{n}{2}-{\rm Re}(\mu)-\eps}(\sph^n; E^{(m)}).
\]

(ii) The existence of the expansion~\eqref{e:asympexp} does not depend on the choice of $\rho$.
For $r=0$, this follows from analysing the Mellin transform of $u$ as in the remark following Lemma~\ref{mazzeo}.
For $r>0$, we additionally use that if $\rho'$ is another geodesic boundary defining function,
then $\rho\partial_\rho-\rho'\partial_{\rho'}\in \rho \cdot {}^0T\overline{\mathbb B^{n+1}}$
(indeed, the dual covector by the metric
is $\rho^{-1}d\rho-(\rho')^{-1}d\rho'$ and we have
$\rho'=e^f\rho$ for some smooth function $f$ on $\overline{\mathbb B^{n+1}}$). Therefore, $(\iota_{\rho'\partial_{\rho'}})^r u$ is a linear combination
of contractions with $0$-vector fields of $\rho^{r-r'}(\iota_{\rho\partial_\rho})^{r'} u$
for $0\leq r'\leq r$, which have the desired asymptotic expansion.
Moreover, as follows from~\eqref{relationleadingterm}, for each $r\in [0,m]$,
the condition that $w^{r'}_{-\mu,0}=0$ for all $r'\in [0,r]$ also does not depend
on the choice of $\rho$, and same can be said about $w^{r'}_{\mu,0,0}$
when $\mu\notin{1\over 2}\mathbb N_0$.

%%%%%%%%%%%%%%%%%%%%%%%%%%%%%%%%%%%%%%%%%%%%%%%%%%%%%%%%%%%%%%%%%%%%%%%%%%%%%%%%
\begin{proof}
It suffices to describe the weak asymptotic expansion of $u$ near any point 
$\nu\in \sph^n$. For that, we work in the half-space model $\mathbb{U}^{n+1}$ by sending $-\nu$ to $\infty$ and $\nu$ to $0$ as we did before
(composing a rotation of the ball model with
the map~\eqref{e:udiffeo}). Since the choice of geodesic boundary defining function does not change the nature of the weak asymptotic expansion (but only the coefficients), we can take the geodesic boundary defining function $\rho$ to be equal to $\rho(z_0,z)=z_0$ inside $|z|+z_0<1$ (which corresponds to a neighbourhood of 
$\nu$ in the ball model).
Considering the weak asymptotic \eqref{weakas} of $u$ near $0$ amounts to taking $\varphi$ supported near $\nu$ in $\sph^n$ in \eqref{weakas}: for instance, if we work in the half-space model we shall consider $\varphi(z)$ 
supported in $|z|<1$ in the boundary of $\mathbb{U}^{n+1}$. 

We decompose $u=\sum_{k=0}^mu_k$ with 
$u_k\in \rho^{\delta}L^2(\mathbb{U}^{n+1}; E^{(m)}_k)$ and we write $u_k=\mc{S}((Z_0^*)^{\otimes k}\otimes u'_k)$ for some 
$u'_k\in \rho^{\delta}L^2(\mathbb{U}^{n+1}; E^{(m-k)}_0)$ following what we did in \eqref{e:uuform}.
Now, since $u\in\rho^{\delta}L^2(\bbar{\mathbb{B}^{n+1}})=\rho_0^{\delta}L^2(\bbar{\mathbb{B}^{n+1}})$ 
satisfies $\Delta u=\sigma u$, we deduce from the form of the Laplacian near $\rho=0$ that 
$u\in \rho_0^{\delta-2k}H^{2k}(\bbar{\mathbb{B}^{n+1}}; E^{(m)})$ for all $k\in\nn$ where $H^k$ denotes the 
Sobolev space of order $k$ associated to the Euclidean Laplacian on the closed unit ball. 
Then by Sobolev embedding one has that for each  $t>0$, 
$u|_{z_0=t}$ belongs to $(1+|z|)^{N}L^2(\rr_z^n; E^{(m)})$ for some $N\in\nn$ and we can consider its Fourier transform in $z$, as a tempered distribution.%
\footnote{Unlike Lemma~\ref{injectiv}, we only use Fourier analysis here for convenience
of notation~-- all the calculations below
could be done with differential operators in $z$ instead.}
Then Fourier transforming the equation $(\pi_0+\pi_1)(\Delta u-\sigma u)=0$ in the $z$-variable (recall that $\pi_i$ is the orthogonal projection on $E^{(m)}_i$), 
and writing the Fourier variable $\xi$ as $\xi=\sum_{i=1}^n\xi_i dz_i=\sum_{i=1}^nz_0\xi_i Z_i^*$, 
with the notations of Lemma~\ref{pi0pi1}, we get 
\begin{equation}\label{fourier1}
\begin{gathered}
 \sum_{I\in\mathscr{A}^m}((-(Z_0)^2+nZ_0+z_0^2|\xi|^2+m-\sigma) \hat{f}_I)\mc{S}(Z^*_{I})
+2i\sum_{J\in\mathscr{A}^{m-1}} \hat{g}_J\mc{S}(\xi \otimes Z_J^*)\\
+m(m-1)\sum_{I}\hat{f}_I \mc{S}(z_0^{-2}h\otimes \mc{T}(\mc{S}(Z_I^*)))=0.\end{gathered}
\end{equation}
and
\begin{equation}\label{pi1deltaTF}
\begin{gathered}
\sum_{J\in\mathscr A^{m-1}}((-(Z_0)^2+nZ_0+z_0^2|\xi|^2+n+3(m-1)-\sigma) \hat g_J)\mathcal S(Z^*_J)\\
-2im\sum_{I\in\mathscr A^m} \hat f_I \iota_\xi \mathcal S(Z_I^*)-2im(m-1)\sum_{I\in\mathscr A^m}\hat f_I
\mathcal S(\xi\otimes \mathcal T(\mathcal S(Z_I^*)))\\
+(m-1)(m-2)\sum_{J\in\mathscr A^{m-1}}\hat g_J \mathcal S(z_0^{-2}h\otimes \mathcal T(\mathcal S(Z_J^*)))=0.
\end{gathered}
\end{equation}
where hat denotes Fourier transform in $z$ and $\iota_\xi$ means $\sum_{j=1}^nz_0\xi_j\iota_{Z_j}$.
Similarly we Fourier transform in $z$ the equation $(\pi_0+\pi_1)(\nabla^*u)=0$ using Lemma~\ref{divergenceproj}
to obtain 
\begin{equation}\label{fourier2}
\begin{gathered}
\sum_{I\in\mathscr{A}^{m}}i\hat{f}_I\iota_{\xi}\mc{S}(Z_I^*)
=\frac{1}{m}\sum_{J\in\mathscr{A}^{m-1}}((n+m-1)\hat{g}_J -Z_0(\hat{g}_J))\mc{S}(Z_J^*),\\
\sum_{I\in\mathscr{A}^{m}} (Z_0\hat{f}_I-(n+m-1)\hat{f}_I)\mc{T}(\mc{S}(Z_I^*))=\frac{1}{m}
\sum_{J\in\mathscr{A}^{m-1}}i\hat{g}_J\iota_{\xi}\mc{S}(Z_J^*).
\end{gathered}
\end{equation}

Now, we use the correspondence between symmetric tensors and homogeneous polynomials to facilitate computations, as explained in Section \ref{symtens} and in the proof of Lemma \ref{injectiv}; that is,
to $\mc{S}(Z^*_I)$, we associate the polynomial $x_I$ on $\rr^n$.
If $\xi\in\rr^n$ is a fixed element and $u\in {\rm Pol}^m(\rr^n)$, 
we write $\pl_\xi u=du.\xi \in   {\rm Pol}^{m-1}(\rr^n)$ for the derivative of $u$ in the direction of $\xi$ and 
$\xi^* u$ for the element $\cjg \xi,\cdot\cjd_{\rr^n} u\in  {\rm Pol}^{m+1}(\rr^n)$.
The trace map $\mc{T}$ becomes  $-\tfrac{1}{(m(m-1))}\Delta_x$. 
We define $\hat{u}_0 :=\sum_{I\in \mathscr{A}^m}\hat{f}_Ix_I$ and 
$\hat{u}_1=\sum_{J\in \mathscr{A}^{m-1}}\hat{g}_Jx_J$. The elements $\hat{f}_I(z_0,\xi),
\hat{g}_I(z_0,\xi)$ belong to the space $\mc{C}^\infty(\rr^+_{z_0}; \mathscr{S}'(\rr^n_\xi))$. We 
decompose them as
\begin{equation}
\label{e:hatdec}
\hat{u}_0 =\sum_{j=0}^{\lfloor {m\over 2}\rfloor}|x|^{2j}\hat{u}_0^{2j},\quad \hat{u}_1=\sum_{j=0}^{\lfloor{m-1\over 2}\rfloor}|x|^{2j}
\hat{u}_1^{2j}
\end{equation}
for some $\hat{u}_i^{2j}\in {\rm Pol}_0^{m-i-2j}(\rr^n)$ (harmonic in $x$, that is trace-free).\\

Using the homogeneous polynomial description of $u_0$, equation \eqref{fourier1} becomes 
\begin{equation}\label{T0T1}
\begin{gathered}
(-(Z_0)^2+nZ_0+z_0^2|\xi|^2+m-\sigma) \hat{u}_0
+2iz_0\xi^* \hat{u}_1
-|x|^2\Delta_x\hat{u}_0=0.
\end{gathered}\end{equation}
First, if $W$ is a harmonic homogeneous polynomial in $x$ of degree $j$, one has 
$\Delta_x(\xi^* W)=-2\pl_{\xi}W$ and $\Delta^2_x(\xi^* W)=0$, thus one can write 
\begin{equation}
\label{zetaW} 
\xi^* W=\Big(\xi^* W-\frac{\pl_{\xi}W}{n+2(j-1)}|x|^2\Big)+\frac{\pl_{\xi}W}{n+2(j-1)}|x|^2
\end{equation}
for the decomposition \eqref{decompositionoftensors} of $\xi^* W$.
In particular, one can write the decomposition  \eqref{decompositionoftensors} of 
$\xi^* \hat{u}_1$ as
\[
\xi^* \hat{u}_1 =\sum_{j=0}^{\lfloor{m-1\over 2}\rfloor}|x|^{2j}\Big(\xi^* \hat{u}^{2j}_1-\frac{\pl_{\xi} \hat{u}^{2j}_1}
{n+2(m-2-2j)}|x|^2+\frac{\pl_{\xi} \hat{u}^{2(j-1)}_1}{n+2(m-2j)}\Big)
\]
We can write $\Delta_x \hat{u}_0=\sum_{j=0}^{\lfloor m/2\rfloor}\la_j |x|^{2j-2}\hat{u}^{2j}_0$ for  
$\la_j=-2j(n+2(m-j-1))$.
Thus \eqref{T0T1} gives for $j\leq \lfloor m/2\rfloor$
\begin{equation}
\label{relation1}
\begin{gathered}
(-(Z_0)^2+nZ_0+z_0^2|\xi|^2+m-\sigma-\la_j) \hat{u}^{2j}_0\\
+2iz_0 \Big(\xi^* \hat{u}^{2j}_1-\frac{|x|^2\pl_{\xi} \hat{u}^{2j}_1}
{n+2(m-2-2j)}+\frac{\pl_{\xi} \hat{u}^{2(j-1)}_1}{n+2(m-2j)}\Big)=0.
\end{gathered}
\end{equation}
Notice that $\iota_{\xi}(\mc{S}(Z^*_I))$ corresponds to the polynomial $\frac{z_0}{m}dx_I.\xi=
\frac{z_0}{m}\pl_\xi.x_I$
if $I\in \mathscr{A}^m$. From \eqref{fourier2} we thus have for $c_m:=n+m-1$
\begin{equation}\label{relationdediv}
\begin{split}
-iz_0\pl_\xi \hat{u}_0=&  (Z_0-c_m)\hat{u}_1, \\
-iz_0\pl_\xi \hat{u}_1= & (Z_0-c_m)\Delta_x\hat{u}_0.
\end{split}\end{equation}
Next, \eqref{pi1deltaTF} implies
$$
\begin{gathered}
(-(Z_0)^2+nZ_0+z_0^2|\xi|^2+n+3(m-1)-\sigma)\hat u_1-2iz_0\partial_\xi\hat u_0\\
+2iz_0\xi^* \Delta_x\hat u_0-|x|^2\Delta_x \hat u_1=0.
\end{gathered}
$$
Using \eqref{relationdediv}, this can be rewritten as 
\begin{equation}\label{pi1deltaTFbis}
\begin{gathered}
(-(Z_0)^2+(n+2)Z_0+z_0^2|\xi|^2-n+m-1-\sigma) \hat{u}_1\\
+2iz_0\xi^* \Delta_x\hat{u}_0
-|x|^2\Delta_x\hat{u}_1=0.
\end{gathered}
\end{equation}
We can write $\Delta_x \hat{u}_1=\sum_{j=0}^{[(m-1)/2]}\la'_j |x|^{2j-2}\hat{u}^{2j}_1$ for  
$\la'_j=-2j(n+2(m-j-2))$. We get from \eqref{pi1deltaTFbis}
\begin{equation}\label{relation2}
\begin{gathered}
\Big(-(Z_0)^2+(n+2)Z_0+z_0^2|\xi|^2-n+m-1-\sigma-\la'_j\Big) \hat{u}^{2j}_1\\
+2iz_0\Big(\la_{j+1}\xi^* \hat{u}^{2(j+1)}_0-\frac{\la_{j+1}\pl_{\xi} \hat{u}^{2(j+1)}_0}
{n+2(m-3-2j)}|x|^2+\frac{\la_j\pl_{\xi} \hat{u}^{2j}_0}{n+2(m-1-2j)}\Big)=0.
\end{gathered}
\end{equation}
\smallskip

We shall now partially uncouple the system of equations for $\hat{u}_0^{2j}$ and $\hat{u}_1^{2j}$.
Using \eqref{zetaW} and applying the  decomposition \eqref{decompositionoftensors}, we have
\[
\pl_\xi (|x|^{2j}\hat{u}_0^{2j})=|x|^{2j}\pl_\xi \hat{u}^{2j}_0\frac{n+2(m-j-1)}{n+2(m-2j-1)}+
2j|x|^{2j-2}\Big(\xi^* \hat{u}_0^{2j}-\frac{|x|^2\pl_{\xi}\hat{u}_0^{2j}}{n+2(m-2j-1)}\Big)
\]
\[
\pl_\xi (|x|^{2j}\hat{u}_1^{2j})=|x|^{2j}\pl_\xi \hat{u}^{2j}_1\frac{n+2(m-j-2)}{n+2(m-2j-2)}+
2j|x|^{2j-2}\Big(\xi^* \hat{u}_1^{2j}-\frac{|x|^2\pl_{\xi}\hat{u}_1^{2j}}{n+2(m-2j-2)}\Big)
\]
and from \eqref{relationdediv}, this implies that for $j\geq 0$
\begin{equation}\label{eqdivnul}
\begin{gathered}
(Z_0-c_m)\hat{u}_1^{2j}=\\
-iz_0\Big(\pl_\xi \hat{u}_0^{2j}\frac{n+2(m-j-1)}{n+2(m-2j-1)}+2(j+1)
\big(\xi^* \hat{u}_0^{2(j+1)}-\frac{|x|^2\pl_{\xi}\hat{u}_0^{2(j+1)}}{n+2(m-2j-3)}\big)\Big),\end{gathered}
\end{equation}
and for $j>0$
\begin{equation}
\label{eqdivnul2}
\begin{gathered}
(Z_0-c_m)\hat{u}_0^{2j}=\\
iz_0\Big(\frac{\pl_\xi \hat{u}_1^{2(j-1)}}{2j(n+2(m-2j))}+\frac{1}{n+2(m-j-1)}
\big(\xi^* \hat{u}_1^{2j}-\frac{|x|^2\pl_{\xi}\hat{u}_1^{2j}}{n+2(m-2j-2)}\big)\Big).
\end{gathered}
\end{equation}
Combining with \eqref{relation1} and~\eqref{relation2} we get for $j\geq 0$
\begin{equation}\label{relation1bis}
\begin{gathered}
(-(Z_0)^2+(n+4j)Z_0+z_0^2|\xi|^2+m-\sigma-\la_j -4jc_m) \hat{u}^{2j}_0\\
+2iz_0\frac{n+2(m-2j-1)}{n+2(m-j-1)} 
\Big(\xi^* \hat{u}^{2j}_1-\frac{|x|^2\pl_{\xi} \hat{u}^{2j}_1}
{n+2(m-2-2j)}\Big)=0,
\end{gathered}\end{equation}
\begin{equation}\label{relation2bis}
\begin{gathered}
(-(Z_0)^2+(n+2+4j)Z_0+z_0^2|\xi|^2-n+m-1-\sigma-\la_j'-4jc_m) \hat{u}^{2j}_1\\
+2iz_0(\la_{j+1}+4j(j+1))\Big(\xi^* \hat{u}^{2(j+1)}_0-\frac{|x|^2\pl_{\xi} \hat{u}^{2(j+1)}_0}
{n+2(m-3-2j)}\Big)=0,
\end{gathered}
\end{equation}
\begin{equation}\label{relation3}
\begin{gathered}
(-(Z_0)^2+(n+2-\tfrac{\la_{j+1}}{j+1})Z_0+z_0^2|\xi|^2-n+m-1-\sigma+\tfrac{\la_{j+1}}{j+1}(c_m-j)) \hat{u}^{2j}_1\\
+2iz_0\frac{(n+2(m-j-1))(n+2(m-2j-2))}{n+2(m-2j-1)}\pl_\xi \hat{u}_0^{2j}=0
\end{gathered}
\end{equation}
and for $j>0$
\begin{equation}
\label{relation4}
\begin{gathered}
(-Z_0^2+(n-\tfrac{\la_j}{j})Z_0+z_0^2|\xi|^2+m-\sigma+\tfrac{\la_j}{j}(c_m-j)) \hat{u}^{2j}_0\\
-iz_0\frac{2(m-1-2j)+n}{j(n+2(m-2j))}\pl_\xi \hat{u}_1^{2(j-1)}=0.
\end{gathered}
\end{equation}
\smallskip

To prove the lemma, we will show the following weak asymptotic expansion for $i=0,1$:
\begin{equation}
  \label{expansionhatu1}
\begin{gathered}
\cjg \hat{u}_i^{2j}(z_0,\cdot), \hat{\varphi}\cjd= %& 
\sum_{\ell\in\nn_0,\atop {\rm Re}(-\mu)+\ell<N-\eps} 
z_0^{\ndemi-\mu +2j+i+\ell}\cjg \widetilde w^{2j}_{i;-\mu,\ell},\varphi\cjd
\\
%&
+\sum_{\ell\in\nn_0,\atop {\rm Re}(\mu)+\ell<N-\eps} 
\sum_{p=0}^{k_{\mu,\ell}}z_0^{\ndemi+\mu+2j +i+\ell}\log(z_0)^{p}\cjg \widetilde w^{2j}_{i;\mu,\ell,p},\varphi\cjd
+\mc{O}(z_0^{\ndemi+2j+i+N-\eps}), 
\end{gathered}
\end{equation}
where $\widetilde w^{2j}_{i;-\mu,\ell}$ and $\widetilde w^{2j}_{i;\mu;\ell,p}$ are distributions in some Sobolev spaces in $\{|z|<1\}\subset \rr^n$
and for $\mu\notin {1\over 2}\mathbb N_0$, we have $k_{\mu,\ell}=0$.

Define for $0\leq r\leq m$ and $\varphi\in \mc{C}_0^\infty(\rr^n)$ supported in $\{|z|<1\}$,
$$
F^r(\varphi)(z_0):=
\begin{cases}
\cjg \hat{u}_0^{r}(z_0,\cdot),\hat{\varphi}\cjd,&\text{$r$ is even};\\
\cjg \hat{u}_1^{r-1}(z_0,\cdot),\hat{\varphi}\cjd,&\text{$r$ is odd}.\\
\end{cases}
$$
Since $\hat{u}_i^{r-i}$ is the Fourier transform in $z$ of iterated traces 
of $u_i$, Lemma~\ref{mazzeo} gives that the function 
$F^r(\varphi)(z_0)$ satisfies  for all $N\in\nn$, $\eps>0$
\begin{equation}
\label{expansionum} 
F^r(\varphi)(z_0) = \sum_{\la\in {\rm spec}_b(\Delta-\sigma)\atop {\rm Re}(\la)>\delta-1/2}
\sum_{\ell\in\nn_0,\atop 
{\rm Re}(\la)+\ell<N-\eps} \sum_{p=0}^{k_{\la,\ell}^r}
z_0^{\la+\ell}\log(z_0)^{p}\cjg w_{\la,\ell,p}^r,\varphi\cjd
+\mc{O}(z_0^{N-\eps})
\end{equation}
as $z_0\to 0$, and some $w_{\la,\ell,p}^r$ in some Sobolev space on $\{|z|<1\}$. 
We pair~\eqref{relation1bis}, \eqref{relation2bis} with 
$\hat{\varphi}$, and it is direct to see that we obtain a differential equation in $z_0$ of the form 
\begin{equation}
  \label{e:pila}
P^r(Z_0)F^r(\varphi)(z_0)=-z_0^2 F^r(\Delta \varphi)(z_0)+z_0F^{r+1}(Q^r\varphi)(z_0)
\end{equation}
for $Z_0=z_0\partial_{z_0}$,
$$
P^r(\lambda):=-\lambda^2+(n+2r)\lambda-r(n+r)-{n^2\over 4}+\mu^2
=-\Big(\lambda-{n\over 2}-r\Big)^2+\mu^2,
$$
and $Q^r$ some differential operator of order 1 with values in homomorphisms on the space
of polynomials in $x$. Here we denote $F^{m+1}=0$.

We now show the expansion~\eqref{expansionhatu1} by induction on $r=2j+i=m,m-1,\dots, 0$.
By plugging the expansion \eqref{expansionum} in the equation~\eqref{e:pila} and using
\begin{equation}\label{eqind}
\begin{gathered}
P^r(Z_0)z_0^{\lambda}\log(z_0)^p=z_0^\lambda\big(P_0^r(\lambda)(\log z_0)^p+
p\partial_\lambda P_0^r(\lambda)(\log z_0)^{p-1}\\+\mc{O}((\log z_0)^{p-2})\big)
\end{gathered} \end{equation}
we see that if for some $p$, $z_0^{\lambda}(\log z_0)^p$ is featured in the asymptotic expansion of
$F^r(\varphi)(z_0)$, then
either $\lambda\in n/2+r-\mu+\mathbb N_0$,
or $\lambda\in n/2+r+\mu+\mathbb N_0$,
or $z_0^{\lambda-2}(\log z_0)^p$ is featured in the expansion of $F^r(\Delta\varphi)(z_0)$.
Moreover, if $p>0$ and $\lambda\notin \{n/2+r\pm \mu\}$, then either $z_0^{\lambda}(\log z_0)^{p'}$ is featured in $F^r(\varphi)(z_0)$ for some $p'>p$,
or $z_0^{\lambda-2}(\log z_0)^p$ is featured in $F^r(\Delta\varphi)(z_0)$, or $z_0^{\lambda-1}(\log z_0)^p$
is featured in $F^{r+1}(Q^r\varphi)(z_0)$.
If $p>0$ and $\lambda= n/2+r\pm \mu$, then
(since $\mu\neq 0$ and thus $\partial_\lambda P_0^r(\lambda)\neq 0$)
either $z_0^{\lambda}(\log z_0)^{p'}$ is featured in $F^r(\varphi)(z_0)$ for some $p'>p$,
or $z_0^{\lambda-2}(\log z_0)^{p-1}$ is featured in $F^r(\Delta\varphi)(z_0)$, or $z_0^{\lambda-1}(\log z_0)^{p-1}$
is featured in $F^{r+1}(Q^r\varphi)(z_0)$, however the latter two cases are only possible 
when $\lambda=n/2+r+\mu$ and $\mu\in{1\over 2}\mathbb N_0$.
Together these facts (applied to $\varphi$
as well as its images under combinations of $\Delta$ and $Q^r$) imply that the weak expansion
of $u^{2j}_i$ has the form~\eqref{expansionhatu1}.

The asymptotic expansions~\eqref{e:asympexp} now follow from~\eqref{expansionhatu1}
since $\rho\partial_\rho=Z_0$ for our choice of $\rho$
and for each $r\in [0,m]$, by~\eqref{u'r} and \eqref{e:hatdec} we see that
(identifying symmetric tensors with homogeneous polynomials in $(x_0,x)$)
\begin{equation}
  \label{e:polly}
(\iota_{Z_0})^r u(x_0,x)=\sum_{r'=r}^m \sum_{s\geq 0\atop
r'+2s\leq m}
c_{m,r,r',s}\,x_0^{r'-r}|x|^{2s} u_{r'-2\lfloor r'/2\rfloor}^{2\lfloor r'/2\rfloor+2s}(x)
\end{equation}
for some constants $c_{m,r,r',s}$; for later use, we also note that $c_{m,r,r,0}\neq 0$.
\end{proof}

%%%%%%%%%%%%%%%%%%%%%%%%%%%%%%%%%%%%%%%%%%%%%%%%%%%%%%%%%%%%%%%%%%%%%%%%%%%%%%%%
\subsection{Surjectivity of the Poisson operator}

In this section, we prove the surjectivity part of Theorem~\ref{t:laplacian} in Section~\ref{s:rr-2}
(together with the injectivity part established in Corollary~\ref{l:injdone}, this
finishes the proof of that theorem). The remaining essential component of the proof
is showing that unless $u\equiv 0$,
a certain term in the asymptotic expansion of Lemma~\ref{asympexp} is nonzero
(in particular we will see
that $u$ cannot be vanishing to infinite order on $\mathbb S^n$ in the weak sense).
We start with
%%%%%%%%%%%%%%%%%%%%%%%%%%%%%%%%%%%%%%%%%%%%%%%%%%%%%%%%%%%%%%%%%%%%%%%%%%%%%%%%
\begin{lemm}
\label{l:vanish1}
Take some $u$ satisfying~\eqref{e:u-eq}.
Assume that for all $r\in [0,m]$, the coefficient $w^r_{-\mu,0}$
of the weak expansion~\eqref{e:asympexp} is zero.
(By Remark (ii) following Lemma~\ref{asympexp}, this condition is independent of the choice of $\rho$.)
Then $u\equiv 0$.
If $\mu\notin {1\over 2}\mathbb N_0$, then we can replace $w^r_{-\mu,0}$
by $w^r_{\mu,0,0}$ in the assumption above.
\end{lemm}
%%%%%%%%%%%%%%%%%%%%%%%%%%%%%%%%%%%%%%%%%%%%%%%%%%%%%%%%%%%%%%%%%%%%%%%%%%%%%%%%
\begin{proof}
We choose some $\nu\in\mathbb S^n$ and transform $\mathbb B^{n+1}$ to the half-space model
as explained in the proof of Lemma~\ref{asympexp}, and use the notation of that proof.
Define the function $f\in \CI(\mathbb B^{n+1})$ in the half-space model as follows:
$$
f=\begin{cases}
z_0^{-m}u^{2m}_0&\quad\text{if $m$ is even};\\
z_0^{-m}u^{2m-1}_1&\quad\text{if $m$ is odd}.
\end{cases}
$$
Here $u^{2j}_0$, $u^{2j}_1$ are obtained by taking the inverse Fourier transform of
$\hat u^{2j}_0,\hat u^{2j}_1$.
By~\eqref{relation1bis}, \eqref{relation2bis} (see also~\eqref{e:pila})
we have
\begin{equation}
  \label{e:Feq}
(\Delta_{\hh^{n+1}}-n^2/4+\mu^2)f= 0.
\end{equation}
Denote by $\mc{C}_{\rm temp}^\infty(\mathbb{B}^{n+1})$ the set of smooth functions $f$ in $\mathbb{B}^{n+1}$ which are tempered in the sense that there exists $N\in \rr$ such that $\rho_0^{N}f\in L^2(\mathbb{B}^{n+1})$.
Set $\la:=-n/2+\mu$; it is proved in \cite{VdBSc, OsSe} (see also \cite{GrOt} for a simpler presentation in the case $|{\rm Re}(\la)+n/2|<n/2$) that the Poisson operator acting on distributions on hyperbolic space is an isomorphism 
\[
\mathscr P^-_\lambda: \mc{D}'(\sph^n)\to \ker(\Delta_{\hh^{n+1}}+\la(n+\la))\cap \mc{C}_{\rm temp}^\infty(\mathbb{B}^{n+1})
\]
for $\la\notin -n-\nn_0$, and if ${\rm Re}(\la)\geq -n/2$ with $\la\not=0$ any element 
$v\in \mc{C}_{\rm temp}^\infty(\mathbb{B}^{n+1})$ with $(\Delta_{\hh^{n+1}}+\la(n+\la))v=0$ and $v\not\equiv 0$
satisfies a weak expansion for any $N\in\nn$
\[
v=\mathscr P^-_\lambda(v_{-\mu,\ell})=\sum_{\ell=0}^N\Big(\rho_0^{n/2-\mu+\ell}v_{-\mu,\ell}+\sum_{p=1}^{k_{\mu,\ell}}\rho_0^{n/2+\mu+\ell}\log(\rho_0)^pv_{\mu,\ell,p}\Big)+\mc{O}(\rho_0^{n/2-\mu+N})
\]
with $v_{-\mu,0}\not\equiv0$; moreover $k_{\mu,\ell}=0$  if $\la\notin-\ndemi+\demi\nn_0$, and 
$v_{\mu,0,0}\not=0$ for such $\la$ (here $v_{-\mu,\ell},v_{\mu,\ell,p}$ are distributions on $\sph^n$ as before).%
\footnote{The existence of the weak expansion with known coefficients for elements in the image of $\mathscr P^-_\lambda$ is directly
related to the special case $m=0$ of Lemma~\ref{injectiv} and
the existence of a weak expansion for scalar eigenfunctions of the Laplacian
follows from the $m=0$ case of Lemma~\ref{asympexp}.
However, neither the surjectivity of the scalar Poisson operator nor
the fact that eigenfunctions have nontrivial terms in their weak expansions
follows from these statements.}

Next, by~\eqref{e:polly}, for some nonzero constant $c$ we have
$$
f=c (z_0^{-1}\iota_{Z_0})^m u=c\langle u,\otimes^m\partial_{z_0}\rangle.
$$
A calculation using~\eqref{e:udiffeo} shows that in the ball model, using
the geodesic boundary defining function $\rho_0$ from~\eqref{e:rho},
\begin{equation}
  \label{e:lili}
\partial_{z_0}=-\bigg({1-|y|^2\over 2} \,\nu+(1+y\cdot\nu)\,y\bigg)\partial_y
\end{equation}
is a $\CI(\overline{\mathbb B^{n+1}})$-linear combination
of $\partial_{\rho_0}$ and a 0-vector field. It follows from the form
of the expansion~\eqref{e:asympexp} and the assumption of this lemma
that the coefficient of $\rho_0^{{n\over 2}-\mu}$ of the weak expansion of $f$ is zero.
(If $\mu\notin {1\over 2}\mathbb N_0$, then we can also consider instead the coefficient of $\rho_0^{{n\over 2}+\mu}$.)

By~\eqref{e:Feq} and the surjectivity of the scalar Poisson kernel discussed above, we now see that $f\equiv 0$.
Now, for each fixed $y\in\mathbb B^{n+1}$ and each
$\eta\in T_y\mathbb B^{n+1}$, we can choose
$\nu$ such that $\eta$ is a multiple of~\eqref{e:lili} at $y$;
in fact, it suffices to take $\nu$ so that the geodesic $\varphi_t(y,\eta)$ converges
to $-\nu$ as $t\to +\infty$. Therefore, for each $y,\eta$, we have
$\langle u,\otimes^m \eta\rangle=0$ at $y$. Since $u$ is a symmetric tensor, this implies
$u\equiv 0$.
\end{proof}
%%%%%%%%%%%%%%%%%%%%%%%%%%%%%%%%%%%%%%%%%%%%%%%%%%%%%%%%%%%%%%%%%%%%%%%%%%%%%%%%
We now relax the assumptions of Lemma~\ref{l:vanish1} to only include the term with $r=0$:
%%%%%%%%%%%%%%%%%%%%%%%%%%%%%%%%%%%%%%%%%%%%%%%%%%%%%%%%%%%%%%%%%%%%%%%%%%%%%%%%
\begin{lemm}
\label{l:vanish2}
Take some $u$ satisfying~\eqref{e:u-eq}.
If $n=1$ and $m>0$, then we additionally assume that $\mu\neq {1\over 2}$.
Assume that the coefficient $w^0_{-\mu,0}$
of the weak expansion~\eqref{e:asympexp} is zero.
(By Remark (ii) following Lemma~\ref{asympexp}, this condition is independent of the choice of $\rho$.)
Then $u\equiv 0$.
If $\mu\notin {1\over 2}\mathbb N_0$, then we can replace $w^0_{-\mu,0}$
by $w^0_{\mu,0,0}$ in our assumption.
\end{lemm}
%%%%%%%%%%%%%%%%%%%%%%%%%%%%%%%%%%%%%%%%%%%%%%%%%%%%%%%%%%%%%%%%%%%%%%%%%%%%%%%%
\begin{proof}
Assume that $w^0_{\pm\mu,0}=0$; here we consider the case of
$w^0_{\mu,0}:=w^0_{\mu,0,0}$ only when $\mu\notin{1\over 2}\mathbb N_0$.
By Lemma~\ref{l:vanish1},
it suffices to prove that $w^r_{\pm\mu,0}=0$ for $r=0,\dots,m$.
This is a local statement and we use the half-plane model and the notation
of the proof of Lemma~\ref{asympexp}. By~\eqref{e:polly}, it then suffices to show that
if $\widetilde w^0_{0;\pm\mu,0}=0$ in the expansion~\eqref{expansionhatu1},
then $\widetilde w^{2j}_{i;\pm\mu,0}=0$ for all $i,j$.

We argue by induction on $r=2j+i=0,\dots,m$. Assume first that $i=0$, $j>0$, and $\widetilde w^{2(j-1)}_{1;\pm\mu,0}=0$. Then
we plug~\eqref{expansionhatu1} into~\eqref{relation4} and consider the coefficient next to~$z_0^{{n\over 2}\pm\mu+2j}$;
this gives $\widetilde w^{2j}_{0;\pm\mu,0}=0$ if
for $\lambda={n\over 2}\pm\mu+2j$, the following constant is nonzero:
\begin{equation}
  \label{e:coffee1}
\begin{gathered}
-\lambda^2+\Big(n-{\lambda_j\over j}\Big)\lambda+m-\sigma+{\lambda_j\over j}(c_m-j)\\
=(n+2m-2-4j)(\pm 2\mu -n-2m+2+4j).
\end{gathered}
\end{equation}
We see immediately that~\eqref{e:coffee1} is nonzero unless $m=2j$.
For the case $m=2j$, we can use~\eqref{eqdivnul2} directly; taking the coefficient next to~$z_0^{{n\over 2}\pm \mu+m}$,
we get $\widetilde w^{2j}_{0;\pm\mu,0}=0$ as long as ${n\over 2}\pm \mu+m\neq c_m$, or equivalently
$\pm\mu\neq {n\over 2}-1$; the latter inequality is immediately true unless $n=1$,
and it is explicitely excluded by the statement of the present lemma when $n=1$.

Similarly, assume that $i=1$, $0\leq 2j<m$, and $\widetilde w^{2j}_{0;\pm\mu,0}=0$. Then we plus~\eqref{expansionhatu1}
into~\eqref{relation3} and consider the coefficient next to~$z_0^{{n\over 2}\pm \mu+2j+1}$;
this gives $\widetilde w^{2j}_{1;\pm\mu,0}=0$ if for $\lambda={n\over 2}\pm \mu+2j+1$, the following constant is nonzero:
\begin{equation}
  \label{e:coffee2}
\begin{gathered}
-\lambda^2+(n+2-\tfrac{\lambda_{j+1}}{j+1})\lambda-n+m-1-\sigma+\tfrac{\lambda_{j+1}}{j+1}(c_m-j)\\
=(n+2m-4-4j)(\pm 2\mu-n-2m+4+4j).
\end{gathered}
\end{equation}
We see immediately that~\eqref{e:coffee2} is nonzero unless $m=2j+1$. For the case $m=2j+1$,
we can use~\eqref{eqdivnul} directly; taking the coefficient next to~$z_0^{{n\over 2}\pm \mu+m}$,
we get $\widetilde w^{2j}_{1;\pm\mu,0}=0$ as long as ${n\over 2}\pm \mu+m\neq c_m$,
which we have already established is true.
\end{proof}
%%%%%%%%%%%%%%%%%%%%%%%%%%%%%%%%%%%%%%%%%%%%%%%%%%%%%%%%%%%%%%%%%%%%%%%%%%%%%%%%
We finish the section by the following statement, which immediately
implies the surjectivity part of Theorem~\ref{t:laplacian}. Note that for the lifts
of elements of $\Eig^m(-\la(n+\la)+m)$, we can take any $\delta<1/2$ below.
The condition $\Re\lambda<{1\over 2}-\delta$ for $m>0$ follows from Lemma~\ref{bottomsp}.
%%%%%%%%%%%%%%%%%%%%%%%%%%%%%%%%%%%%%%%%%%%%%%%%%%%%%%%%%%%%%%%%%%%%%%%%%%%%%%%%
\begin{corr}
  \label{corsurj}
Let $u\in \rho^{\delta}L^2(\bbar{\mathbb{B}^{n+1}}; E^{(m)})$ be a trace-free symmetric $m$-cotensor
with  $\rho$ a geodesic boundary defining function and $\delta\in(-\infty,\demi)$, where the measure is the Euclidean Lebesgue measure on the ball. Assume that $u$ is a nonzero divergence-free eigentensor for the Laplacian on hyperbolic space: 
\begin{equation}
  \label{e:ella}
\Delta u=(-\la(n+\la)+m) u,\quad \nabla^*u=0
\end{equation}
with ${\rm Re}(\la)<\frac{1}{2}-\delta$ and $\la\notin\mathcal R_m$,
with $\mathcal R_m$ defined in~\eqref{e:r-m-def0}.
Then there exists  $w \in H^{{\rm Re}(\la)+\delta-\frac{1}{2}}(\sph^n; 
\otimes_S^mT^*\sph^n)$ such that $u=\mathscr P^-_\lambda (w)$. Moreover if $\gamma^*u=u$ for some $\gamma\in G$,
then $L_\gamma^*w=N_\gamma^{-\la-m}w$. 
\end{corr}
%%%%%%%%%%%%%%%%%%%%%%%%%%%%%%%%%%%%%%%%%%%%%%%%%%%%%%%%%%%%%%%%%%%%%%%%%%%%%%%%
\begin{proof} For the case ${\rm Re}(\la)\geq -n/2$ we set $\mu=n/2+\la$ and apply Lemma~\ref{asympexp}: the distribution $w$ will be given by $C(\la)w_{-\mu,0}$ for some constant $C(\la)$ to be chosen, and this has the desired covariance with respect to elements of $G$ by using \eqref{equivarianceubis} from the Remark after Lemma~\ref{mazzeo}.

To see that $u=\mathscr P^-_\lambda (w)$ for a certain $C(\la)$, it suffices to use the weak expansion in 
Lemma~\ref{injectiv} and the identity \eqref{relationleadingterm} from the Remark following Lemma \ref{mazzeo}, 
to deduce that $C(\la)B(\la)w_{-\mu,0}$ appears as the leading coefficient 
of the power $\rho_0^{-\la}$ in the expansion of $u$, where $B(\la)$ is a non-zero constant times 
the factor appearing in \eqref{Blambda}; here $\rho_0$ is defined in~\eqref{e:rho}.
(The factor $B(\lambda)$ does not depend on the point $\nu\in\mathbb S^n$ since
the Poisson operator is equivariant under rotations of $\overline{\mathbb B^{n+1}}$.)
Then choosing $C(\lambda):=B(\la)^{-1}$, we observe that 
$u$ and $\mathscr P^-_\lambda (w)$ both satisfy~\eqref{e:ella}
and have the same asymptotic coefficient of $\rho_0^{-\la}$ in their weak expansion~\eqref{e:asympexp};
thus from Lemma~\ref{l:vanish2} we have $u=\mathscr P^-_\lambda (w)$. Finally, for ${\rm Re}(\la)<-n/2$ 
with $\la\notin -\ndemi-\demi\nn_0$ we do the same thing but setting $\mu:=-n/2-\la$ in Proposition \ref{asympexp}.
\end{proof}
%%%%%%%%%%%%%%%%%%%%%%%%%%%%%%%%%%%%%%%%%%%%%%%%%%%%%%%%%%%%%%%%%%%%%%%%%%%%%%%%

%%%%%%%%%%%%%%%%%%%%%%%%%%%%%%%%%%%%%%%%%%%%%%%%%%%%%%%%%%%%%%%%%%%%%%%%%%%%%%%%
\appendix
%%%%%%%%%%%%%%%%%%%%%%%%%%%%%%%%%%%%%%%%%%%%%%%%%%%%%%%%%%%%%%%%%%%%%%%%%%%%%%%%

%%%%%%%%%%%%%%%%%%%%%%%%%%%%%%%%%%%%%%%%%%%%%%%%%%%%%%%%%%%%%%%%%%%%%%%%%%%%%%%%
%                                  APPENDIX A                                  %
%%%%%%%%%%%%%%%%%%%%%%%%%%%%%%%%%%%%%%%%%%%%%%%%%%%%%%%%%%%%%%%%%%%%%%%%%%%%%%%%
\section{Some technical calculations}
\label{s:technical}

%%%%%%%%%%%%%%%%%%%%%%%%%%%%%%%%%%%%%%%%%%%%%%%%%%%%%%%%%%%%%%%%%%%%%%%%%%%%%%%%
\subsection{Asymptotic expansions for certain integrals}

In this subsection, we prove the following version of Hadamard regularization:
%%%%%%%%%%%%%%%%%%%%%%%%%%%%%%%%%%%%%%%%%%%%%%%%%%%%%%%%%%%%%%%%%%%%%%%%%%%%%%%%
\begin{lemm}
  \label{l:asysa}
Fix $\chi\in \CI_0(\mathbb R)$ and define for $\Re\alpha>0$, $\beta\in\mathbb C$, and $\varepsilon>0$,
$$
F_{\alpha\beta}(\varepsilon):=\int_0^\infty t^{\alpha-1}(1+t)^{-\beta}\chi(\varepsilon t)\,dt.
$$
If $\alpha-\beta\not\in\mathbb N_0$, then $F_{\alpha\beta}(\varepsilon)$ has the following asymptotic expansion
as $\varepsilon\to +0$:
\begin{equation}
  \label{e:asysa}
F_{\alpha\beta}(\varepsilon)= {\Gamma(\alpha)\Gamma(\beta-\alpha)\over\Gamma(\beta)}\chi(0)
+\sum_{0\leq j\leq \Re(\alpha-\beta)} c_j \varepsilon^{\beta-\alpha+j}+o(1),
\end{equation}
for some constants $c_j$ depending on $\chi$.
\end{lemm}
%%%%%%%%%%%%%%%%%%%%%%%%%%%%%%%%%%%%%%%%%%%%%%%%%%%%%%%%%%%%%%%%%%%%%%%%%%%%%%%%
\begin{proof}
We use the following identity obtained by integrating by parts:
\begin{equation}
  \label{e:asysa-1}
\begin{gathered}
\varepsilon\partial_\varepsilon F_{\alpha\beta}(\varepsilon)
=\int_0^\infty t^\alpha(1+t)^{-\beta}\partial_t(\chi(\varepsilon t))\,dt\\
=(\beta-\alpha)F_{\alpha\beta}(\varepsilon)-\beta F_{\alpha,\beta+1}(\varepsilon).
\end{gathered}
\end{equation}
By using the Taylor expansion of $\chi$ at zero, we also see that
$$
\chi(\varepsilon t)=\chi(0)+\mathcal O(\varepsilon t);
$$
given the following formula obtained by the change of variables $s=(1+t)^{-1}$ and using the beta function,
$$
\int_0^\infty t^{\alpha-1} (1+t)^{-\beta}\,dt={\Gamma(\alpha)\Gamma(\beta-\alpha)\over \Gamma(\beta)},\quad\text{if }\Re\beta>\Re\alpha>0,
$$
we see that
$$
F_{\alpha\beta}(\varepsilon)={\Gamma(\alpha)\Gamma(\beta-\alpha)\over \Gamma(\beta)}\chi(0) +\mathcal O(\varepsilon)\quad\text{if }\Re(\beta-\alpha)>1.
$$
By applying this asymptotic expansion to $F_{\alpha,\beta+M}$ for large integer $M$ and iterating~\eqref{e:asysa-1},
we derive the expansion~\eqref{e:asysa}.
\end{proof}
%%%%%%%%%%%%%%%%%%%%%%%%%%%%%%%%%%%%%%%%%%%%%%%%%%%%%%%%%%%%%%%%%%%%%%%%%%%%%%%%

For the next result, we need the following two calculations (see Section~\ref{symtens} for some of the notation used):
%%%%%%%%%%%%%%%%%%%%%%%%%%%%%%%%%%%%%%%%%%%%%%%%%%%%%%%%%%%%%%%%%%%%%%%%%%%%%%%%
\begin{lemm}
  \label{l:tec}
For each $\ell\geq 0$,
$$
\int_{\mathbb S^{n-1}} (\otimes^{2\ell}\eta)\, dS(\eta)=
{2\pi^{n-1\over 2}\Gamma(\ell+{1\over 2})\over
\Gamma(\ell+{n\over 2})}\mathcal S(\otimes^\ell I),
$$
where $I=\sum_{j=1}^n \partial_j\otimes\partial_j$.
\end{lemm}
%%%%%%%%%%%%%%%%%%%%%%%%%%%%%%%%%%%%%%%%%%%%%%%%%%%%%%%%%%%%%%%%%%%%%%%%%%%%%%%%
\begin{proof}
Since both sides are symmetric tensors, it suffices to show that for each $x\in\mathbb R^n$,
$$
\int_{\mathbb S^{n-1}}(x\cdot\eta)^{2\ell}\,dS(\eta)={2\pi^{n-1\over 2}\Gamma(\ell+{1\over 2})\over
\Gamma(\ell+{n\over 2})}|x|^{2\ell}.
$$
Without loss of generality (using homogeneity and rotational invariance), we may assume that $x=\partial_1$. Then
using polar coordinates and Fubini's theorem, we have
$$
{\Gamma(\ell+{n\over 2})\over 2}\int_{\mathbb S^{n-1}}\eta_1^{2\ell}\,dS(\eta)=
\int_{\mathbb R^n} e^{-|\eta|^2}\eta_1^{2\ell}\,d\eta=\pi^{n-1\over 2}\Gamma\Big(\ell+{1\over 2}\Big)
$$
finishing the proof.
\end{proof}
%%%%%%%%%%%%%%%%%%%%%%%%%%%%%%%%%%%%%%%%%%%%%%%%%%%%%%%%%%%%%%%%%%%%%%%%%%%%%%%%

%%%%%%%%%%%%%%%%%%%%%%%%%%%%%%%%%%%%%%%%%%%%%%%%%%%%%%%%%%%%%%%%%%%%%%%%%%%%%%%%
\begin{lemm}
  \label{l:tec2}
For each $\eta\in\mathbb R^n$, define the linear map
$\mathscr C_\eta:\mathbb R^n\to\mathbb R^n$ by
$$
\mathscr C_\eta(\tilde\eta)=\tilde\eta-{2\over 1+|\eta|^2}(\tilde\eta\cdot\eta)\eta.
$$
Then for each $A_1,A_2\in \otimes^m_S\mathbb R^n$ with $\mathcal T(A_1)=\mathcal T(A_2)=0$, and each $r\geq 0$, we have
$$
\int_{\mathbb S^{n-1}} \langle (\otimes^m \mathscr C_{r\eta}) A_1,A_2\rangle\,dS(\eta)=2\pi^{n\over 2}
\sum_{\ell=0}^m {m!\over (m-\ell)!\Gamma({n\over 2}+\ell)}\bigg(-{r^2\over 1+r^2}\bigg)^\ell\langle A_1,A_2\rangle.
$$
\end{lemm}
%%%%%%%%%%%%%%%%%%%%%%%%%%%%%%%%%%%%%%%%%%%%%%%%%%%%%%%%%%%%%%%%%%%%%%%%%%%%%%%%
\begin{proof}
We have
$$
\mathscr C_{r\eta}=\Id-{2r^2\over 1+r^2}\, \eta^*\otimes\eta,
$$
where $\eta^*\in(\mathbb R^n)^*$ is the dual to $\eta$ by the standard metric.
Then
$$
\int_{\mathbb S^{n-1}}\langle (\otimes^m\mathscr C_{r\eta})A_1,A_2\rangle\,dS(\eta)
=\int_{\mathbb S^{n-1}}\Big\langle \otimes^m\Big(I-{2r^2\over 1+r^2}\,\eta\otimes\eta\Big),\sigma(A_1\otimes A_2)\Big\rangle\,dS(\eta).
$$
where $\sigma$ is the operator defined by
$$
\sigma(\eta_1\otimes \dots\otimes \eta_m\otimes \eta'_1\otimes\dots\otimes \eta'_m)
=\eta_1\otimes \eta'_1\otimes\dots\otimes \eta_m\otimes \eta'_m.
$$
We use Lemma~\ref{l:tec}, a binomial expansion, and the fact that $A_j$ are symmetric, to calculate
$$
\begin{gathered}
\int_{\mathbb S^{n-1}}\Big\langle\otimes^m\Big(I-{2r^2\over 1+r^2}\,\eta\otimes\eta\Big),\sigma(A_1\otimes A_2)\Big\rangle\,dS(\eta)\\
=\sum_{\ell=0}^m {m!\over \ell!(m-\ell)!}\Big(-{2r^2\over 1+r^2}\Big)^\ell\int_{\mathbb S^{n-1}} 
\langle(\otimes^{2\ell}\eta)\otimes(\otimes^{m-\ell}I),\sigma(A_1\otimes A_2)\rangle\,dS(\eta)\\
=2\pi^{n-1\over 2}\sum_{\ell=0}^m {m!\over \ell!(m-\ell)!}\cdot {\Gamma(\ell+{1\over 2})\over\Gamma(\ell+{n\over 2})}
\Big(-{2r^2\over 1+r^2}\Big)^\ell \langle \mathcal S(\otimes^\ell I)\otimes(\otimes^{m-\ell} I),\sigma(A_1\otimes A_2)\rangle.
\end{gathered}
$$
Since $\mathcal T(A_1)=\mathcal T(A_2)=0$, we can compute
$$
\langle \mathcal S(\otimes^\ell I)\otimes(\otimes^{m-\ell} I),\sigma(A_1\otimes A_2)\rangle={2^\ell(\ell!)^2\over (2\ell)!}\langle A_1,A_2\rangle.
$$
Here $2^\ell(\ell!)^2/ (2\ell)!$ is the proportion of all permutations $\tau$ of $2\ell$ elements such that
for each $j$, $\tau(2j-1)+\tau(2j)$ is odd. It remains to calculate
$$
\begin{gathered}
\sum_{\ell=0}^m {m!\over \ell! (m-\ell)!}\cdot {\Gamma(\ell+{1\over 2})\over\Gamma(\ell+{n\over 2})}
\cdot {2^\ell (\ell!)^2\over (2\ell)!}\,t^\ell
=\sum_{\ell=0}^m {\sqrt{\pi}\,m!\over (m-\ell)!\Gamma(\ell+{n\over 2})}(t/2)^\ell.
\end{gathered}
$$
\end{proof}
%%%%%%%%%%%%%%%%%%%%%%%%%%%%%%%%%%%%%%%%%%%%%%%%%%%%%%%%%%%%%%%%%%%%%%%%%%%%%%%%
We can now state the following asymptotic formula, used in the proof of Lemma~\ref{l:pairing-key}:
%%%%%%%%%%%%%%%%%%%%%%%%%%%%%%%%%%%%%%%%%%%%%%%%%%%%%%%%%%%%%%%%%%%%%%%%%%%%%%%%
\begin{lemm}
  \label{l:asysa2}
Let $\chi\in \CI_0(\mathbb R)$ be equal to 1 near 0, and take $A_1,A_2\in \otimes_S^m \mathbb R^n$
satisfying $\mathcal T(A_1)=\mathcal T(A_2)=0$. Then for $\lambda\in \mathbb C$,
$\lambda\not\in -({n\over 2}+\mathbb N_0)$, we have as $\varepsilon\to +0$,
$$
\begin{gathered}
\int_{\mathbb R^n}\chi(\varepsilon|\eta|)(1+|\eta|^2)^{-\lambda-n}\langle(\otimes^m\mathscr C_\eta)A_1,A_2\rangle\,d\eta\\
=\pi^{n\over 2}{\Gamma({n\over 2}+\lambda)\over (n+\lambda+m-1)\Gamma(n-1+\lambda)}\langle A_1,A_2\rangle+\sum_{0\leq j\leq -\Re\lambda-{n\over 2}}c_j\varepsilon^{n+2\lambda+2j}+o(1),
\end{gathered}
$$
for some constants $c_j$.
\end{lemm}
%%%%%%%%%%%%%%%%%%%%%%%%%%%%%%%%%%%%%%%%%%%%%%%%%%%%%%%%%%%%%%%%%%%%%%%%%%%%%%%%
\begin{proof}
We write, using the change of variables $\eta=\sqrt{t}\theta$, $\theta\in\mathbb S^n$,
and $\chi(s)=\tilde\chi(s^2)$, and by Lemma~\ref{l:tec2}
$$
\begin{gathered}
\int_{\mathbb R^n}\chi(\varepsilon|\eta|)(1+|\eta|^2)^{-\lambda-n}\langle(\otimes^m\mathscr C_\eta)A_1,A_2\rangle\,d\eta\\
={1\over 2}\int_0^\infty \tilde \chi(\varepsilon^2 t)t^{{n\over 2}-1} (1+t)^{-\lambda-n} \int_{\mathbb S^{n-1}}
\langle (\otimes^m\mathscr C_{\sqrt{t}\theta})A_1,A_2\rangle\,dS(\theta) dt\\
=\pi^{n\over 2}\sum_{\ell=0}^m {(-1)^\ell m!\over (m-\ell)!\Gamma({n\over 2}+\ell)}\langle A_1,A_2\rangle
\int_0^\infty \tilde\chi(\varepsilon^2 t)t^{{n\over 2}+\ell-1} (1+t)^{-\lambda-n-\ell} \,dt.
\end{gathered}
$$
We now apply Lemma~\ref{l:asysa} to get the required asymptotic expansion. The constant term in the expansion is $\langle A_1,A_2\rangle$
times
\begin{equation}
  \label{e:lassie}
\begin{gathered}
\pi^{n\over 2}\Gamma\Big({n\over 2}+\lambda\Big)\sum_{\ell=0}^m {(-1)^\ell m!\over (m-\ell)!\Gamma(n+\lambda+\ell)}\\
=\pi^{n\over 2}(-1)^m m!\Gamma\Big({n\over 2}+\lambda\Big)\sum_{\ell=0}^m {(-1)^\ell \over \ell!\Gamma(n+\lambda+m-\ell)}.
\end{gathered}
\end{equation}
We now use the binomial expansion
$$
{(1-t)^{n+\lambda+m-1}\over\Gamma(n+\lambda+m)}=\sum_{\ell=0}^\infty {(-1)^\ell \over \ell!\Gamma(n+\lambda+m-\ell)}\,t^\ell
$$
and the sum in the last line of~\eqref{e:lassie} is the $t^m$ coefficient of
$$
\begin{gathered}
(1-t)^{-1}\cdot {(1-t)^{n+\lambda+m-1}\over\Gamma(n+\lambda+m)}={(1-t)^{n+\lambda+m-2}\over\Gamma(n+\lambda+m)}\\
={1\over n+\lambda+m-1}\sum_{j=0}^\infty {(-1)^j \over j!\Gamma(n+\lambda+m-j-1)}\,t^j;
\end{gathered}
$$
this finishes the proof.
\end{proof}
%%%%%%%%%%%%%%%%%%%%%%%%%%%%%%%%%%%%%%%%%%%%%%%%%%%%%%%%%%%%%%%%%%%%%%%%%%%%%%%%

%%%%%%%%%%%%%%%%%%%%%%%%%%%%%%%%%%%%%%%%%%%%%%%%%%%%%%%%%%%%%%%%%%%%%%%%%%%%%%%%
\subsection{The Jacobian of \texorpdfstring{$\Psi$}{Psi}}
\label{s:J-Psi}

Here we compute the Jacobian of the map $\Psi:\mathcal E\to S^2_\Delta \mathbb H^{n+1}$ appearing in the proof of Lemma~\ref{l:pairing-key},
proving~\eqref{e:Psi-J}.
By the $G$-equivariance of $\Psi$ we may assume that
$x=\partial_0,\xi=\partial_1,\eta=\sqrt{s}\,\partial_2$ for some $s\geq 0$. We then consider the following
volume 1 basis of $T_{(x,\xi,\eta)}\mathcal E$:
$$
\begin{gathered}
X_1=(\partial_1,\partial_0,0),\
X_2=(\partial_2,0,\sqrt{s}\,\partial_0),\
X_3=(0,\partial_2,-\sqrt{s}\,\partial_1),\
X_4=(0,0,\partial_2);
\\
\partial_{x_j},\partial_{\xi_j},\partial_{\eta_j},\quad
3\leq j\leq n+1.
\end{gathered}
$$
We have $\Psi(x,\xi,\eta)=(y,\eta_-,\eta_+)$, where
$$
y=(\sqrt{s+1},0,\sqrt{s},0,\dots,0),\quad
\eta_\pm=\Big(\mp{s\over \sqrt{s+1}},{1\over\sqrt{s+1}},\mp\sqrt{s},0,\dots,0\Big).
$$
Then we can consider the following volume 1 basis for $T_{(y,\eta_-,\eta_+)}S^2_\Delta\mathbb H^{n+1}$:
$$
\begin{gathered}
Y_1=\Big(\partial_1,{y\over\sqrt{s+1}},{y\over\sqrt{s+1}}\Big),\
Y_2=\Big(\sqrt{s}\,\partial_0+\sqrt{s+1}\,\partial_2,{\sqrt{s}\over\sqrt{s+1}}y,-{\sqrt{s}\over\sqrt{s+1}}y\Big),\\
Y_3={(0,\sqrt{s}\,\partial_0-\sqrt{s}\,\partial_1+\sqrt{s+1}\,\partial_2,0)\over \sqrt{s+1}},\
Y_4={(0,0,\sqrt{s}\,\partial_0+\sqrt{s}\,\partial_1+\sqrt{s+1}\,\partial_2)\over \sqrt{s+1}};\\
\partial_{y_j},\partial_{\nu_{-j}},\partial_{\nu_{+j}},\quad
3\leq j\leq n+1.
\end{gathered}
$$
Then the differential $d\Psi{(x,\xi,\eta)}$ maps
$$
\begin{gathered}
X_1\mapsto \sqrt{s+1}\,Y_1-\sqrt s\,Y_3
-\sqrt s\,Y_4,\\
X_2\mapsto Y_2,\\
X_3\mapsto -\sqrt{s}\,Y_1+\sqrt{s+1}\,Y_3+\sqrt{s+1}\,Y_4,\\
X_4\mapsto {1\over\sqrt{s+1}}\,Y_2+{1\over s+1}Y_3-{1\over s+1}Y_4.
\end{gathered}
$$
Moreover, for $3\leq j\leq n+1$, $d\Psi{(x,\xi,\eta)}$ maps linear combinations
of $\partial_{x_j},\partial_{\xi_j},\partial_{\eta_j}$ to linear combinations
of $\partial_{y_j},\partial_{\nu_{-j}},\partial_{\nu_{+j}}$ by the
matrix $A(s)$. The identity~\eqref{e:Psi-J} now follows by a direct calculation.

%%%%%%%%%%%%%%%%%%%%%%%%%%%%%%%%%%%%%%%%%%%%%%%%%%%%%%%%%%%%%%%%%%%%%%%%%%%%%%%%
\subsection{An identity for harmonic polynomials}

We give a technical lemma which is used in the proof of Lemma~\ref{injectiv}
(injectivity of the Poisson kernel).
%%%%%%%%%%%%%%%%%%%%%%%%%%%%%%%%%%%%%%%%%%%%%%%%%%%%%%%%%%%%%%%%%%%%%%%%%%%%%%%%
\begin{lemm}
\label{littlecomput}
Let $P$ be a harmonic homogeneous polynomial of order $m$ in $\rr^n$, then for $r\leq m$, we have for all $x\in\rr^n$
\[\Delta_\zeta^rP(x-\zeta\cjg \zeta,x\cjd)|_{\zeta=0}= 2^r\frac{m! r!}{(m-r)!}P(x).\] 
\end{lemm}
%%%%%%%%%%%%%%%%%%%%%%%%%%%%%%%%%%%%%%%%%%%%%%%%%%%%%%%%%%%%%%%%%%%%%%%%%%%%%%%%
\begin{proof}
By homogeneity, it suffices to choose $|x|=1$. We set $t=\cjg\zeta,x\cjd$ and $u=\zeta-tx$ and $P(x-\zeta\cjg \zeta,x\cjd)$ viewed in the $(t,u)$ coordinates is the homogeneous polynomial $(t,u)\mapsto P((1-t^2)x-tu)$. 
Now, we write for all $u\in (\rr x)^\perp$ and $t>0$
\[ P(tx-u)=\sum_{j=0}^m t^{m-j}P_j(u)\] 
where $P_j$ is a homogeneous polynomial of degree $j$ in $u\in (\rr x)^{\perp}$, and since the Laplacian $\Delta_\zeta$ written in the $t,u$ coordinates is $-\pl_t^2+\Delta_u$, the condition $\Delta_xP=0$ can be rewritten 
\[
\Delta_uP_j(u)=(m-j+2)(m-j+1)P_{j-2}(u), \quad \Delta_uP_1(u)=\Delta_u P_0=0,
\]
which gives for all $j$ and $\ell\geq 1$
\[
\Delta_u^\ell P_{2\ell}(u)=m(m-1)\cdots (m-2\ell+1)P_0, \quad \Delta^j P_{2\ell-1}(u)|_{u=0}=0.
\]
We write $\Delta_\zeta^r=\sum_{k=0}^r \tfrac{r!}{k!(r-k)!}(-1)^k\pl_t^{2k}\Delta_u^{r-k}$ and using parity and homogeneity considerations, we have
\[
\begin{split}
\Delta_\zeta^rP(x-\zeta\cjg \zeta,x\cjd)|_{\zeta=0} & = 
\sum_{k=0}^r\frac{(-1)^k r!}{k!(r-k)!}\sum_{2j\leq m}[\pl_t^{2k}((1-t^2)^{m-2j}t^{2j})\Delta_u^{r-k}P_{2j}(u)]|_{(t,u)=0}\\
& =  \sum_{\max(0,r-m/2)\leq k\leq r}\frac{(-1)^k r!}{k!(r-k)!}(\pl_t^{2k}((1-t^2)^{m-2(r-k)}t^{2(r-k)}))|_{t=0}\,
\Delta_u^{r-k}P_{2(r-k)}\\
& =   P_0\cdot  \frac{m! r!}{(m-r)!} \sum_{r/2\leq k\leq r}\frac{(-1)^{k+r} (2k)!}{k!(r-k)! (2k-r)!}=2^r\frac{m! r!}{(m-r)!}P_0
\end{split}
\]
and $P_0$ is the constant given by $P(x)$. Here we used the identity
$$
\begin{gathered}
\sum_{r/2\leq k\leq r}\frac{(-1)^{k+r} (2k)!}{k!(r-k)! (2k-r)!}
=\sum_{0\leq k\leq r/2} (-1)^k{ r!\over k!(r-k)!}\cdot {(2r-2k)!\over r!(r-2k)!}=2^r
\end{gathered}
$$
which holds since both sides are equal to the $t^r$ coefficient of the product
$$
\begin{gathered}
(1-t^2)^r\cdot(1-t)^{-1-r}={(1+t)^r\over 1-t},\\
(1-t)^{-1-r}={1\over r!}d_t^{r}(1-t)^{-1}=
\sum_{j=0}^\infty {(j+r)!\over j!r!}\, t^j;
\end{gathered} 
$$
the $t^r$ coefficient of $(1+t)^r/(1-t)$ equals the sum of the $t^0,t^1,\dots, t^r$ coefficients
of $(1+t)^r$, or simply $(1+1)^r=2^r$.
\end{proof}
%%%%%%%%%%%%%%%%%%%%%%%%%%%%%%%%%%%%%%%%%%%%%%%%%%%%%%%%%%%%%%%%%%%%%%%%%%%%%%%%

%%%%%%%%%%%%%%%%%%%%%%%%%%%%%%%%%%%%%%%%%%%%%%%%%%%%%%%%%%%%%%%%%%%%%%%%%%%%%%%%
%                                  APPENDIX B                                  %
%%%%%%%%%%%%%%%%%%%%%%%%%%%%%%%%%%%%%%%%%%%%%%%%%%%%%%%%%%%%%%%%%%%%%%%%%%%%%%%%
\section{The special case of dimension 2}
\label{s:dim2}

We explain how the argument of Section~\ref{s:o-dim2} fits into the framework
of Sections~\ref{s:geomhyp} and~\ref{s:horror}.
In dimension $2$ it is more standard to use the upper half-plane model
$$
\mathbf H^2:=\{w\in\mathbb C\mid \Im w>0\},
$$
which is related to the half-space model of Section~\ref{s:modeles} by the formula
$w=-z_1+iz_0$.

The group of all isometries of $\mathbf H^2$
is $\PSL(2;\mathbb R)$, the quotient of $\SL(2;\mathbb R)$ by the group generated
by the matrix $-\Id$, and the action of $\PSL(2;\mathbb R)$ on $\mathbf H^2$ is by M\"obius transformations:
$$
\begin{pmatrix} a & b \\ c & d \end{pmatrix} . z
= {az+b\over cz+d},\quad
z\in\mathbf H^2\subset\mathbb C.
$$
Under the identifications~\eqref{defpsi} and~\eqref{e:udiffeo}, this
action corresponds to the action of $\PSO(1,2)$ on $\mathbb H^2\subset\mathbb R^{1,2}$
by the group isomorphism $\PSL(2;\mathbb R)\to\PSO(1,2)$ defined by
\begin{equation}
  \label{e:psl-pso}
\begin{pmatrix} a & b \\ c & d \end{pmatrix} \mapsto
\begin{pmatrix}
{a^2+b^2+c^2+d^2\over 2} & {a^2-b^2+c^2-d^2\over 2} & -ab - cd\\
{a^2+b^2-c^2-d^2\over 2} & {a^2-b^2-c^2+d^2\over 2} & cd - ab\\
-ac - bd & bd - ac & ad + bc
\end{pmatrix}.
\end{equation}
The induced Lie algebra isomorphism maps the vector fields $X,U_-,U_+$
of~\eqref{e:u-pm-dim-2} to the fields $X,U^-_1,U^+_1$ of~\eqref{liealgebraofG}, \eqref{liealgebraofG2}.

The horocyclic operators $\mathcal U_\pm:\mathcal D'(S\mathbb H^2)\to\mathcal D'(S\mathbb H^2;\mathcal E^*)$
of Section~\ref{s:horocycl} (and analogously horocyclic operators of higher orders)
then take the following form:
$$
\mathcal U_\pm u=(U_\pm u)\eta^*,
$$
where $\eta^*$ is the dual to the section $\eta\in \CI(S\mathbb H^2;\mathcal E)$ defined
as follows: for $(x,\xi)\in S\mathbb H^2$, $\eta(x,\xi)$ is the
unique vector in $T_x \mathbb H^2$ such that $(\xi,\eta)$ is a positively
oriented orthonormal frame. Note also that $\eta(x,\xi)=\pm\mathcal A_\pm(x,\xi)\cdot\zeta(B_\pm(x,\xi))$,
where $\mathcal A_\pm(x,\xi)$ is defined in Section~\ref{s:E} and
$\zeta(\nu)\in T_\nu\mathbb S^1$, $\nu\in\mathbb S^1$, is the result of rotating
$\nu$ counterclockwise by $\pi/2$; therefore, if we use $\eta$ and $\zeta$ to trivialize
the relevant vector bundles, then the operators $\mathcal Q_\pm$ of~\eqref{defofQpm}
are simply the pullback operators by $B_\pm$, up to multiplication by $\pm 1$.

%%%%%%%%%%%%%%%%%%%%%%%%%%%%%%%%%%%%%%%%%%%%%%%%%%%%%%%%%%%%%%%%%%%%%%%%%%%%%%%%
%                                  APPENDIX C                                  %
%%%%%%%%%%%%%%%%%%%%%%%%%%%%%%%%%%%%%%%%%%%%%%%%%%%%%%%%%%%%%%%%%%%%%%%%%%%%%%%%
\section{Eigenvalue asymptotics for symmetric tensors}
\label{s:weylie}

%%%%%%%%%%%%%%%%%%%%%%%%%%%%%%%%%%%%%%%%%%%%%%%%%%%%%%%%%%%%%%%%%%%%%%%%%%%%%%%%
\subsection{Weyl law}
\label{s:weyl}

In this section, we prove the following asymptotic of the counting function for trace free divergence free tensors
(see Sections~\ref{symtens} and~\ref{s:laplacian-def} for the notation):
%%%%%%%%%%%%%%%%%%%%%%%%%%%%%%%%%%%%%%%%%%%%%%%%%%%%%%%%%%%%%%%%%%%%%%%%%%%%%%%%
\begin{prop}\label{Weyllaw}
If $(M,g)$ is a compact Riemannian manifold of dimension $n+1$ and constant sectional curvature $-1$, and if 
$$
\Eig^m(\sigma)=\{u\in \CI(M;\otimes^m_S T^*M)\mid \Delta u = \sigma u,\
\nabla^*u = 0,\
\mathcal T(u)=0\},
$$
then the following Weyl law holds as $R\to \infty$
$$
\sum_{\sigma\leq R^2}\dim\Eig^m(\sigma)= c_0(n)(c_1(n,m)-c_1(n,m-2))\Vol(M)
R^{n+1}+\mathcal O(R^n),
$$
where $c_0(n)={(2\sqrt\pi)^{-n-1}\over\Gamma({n+3\over 2})}$
and $c_1(n,m)={(m+n-1)!\over m!(n-1)!}$ is the dimension of the space
of homogeneous polynomials of order $m$ in $n$ variables. (We put
$c_1(n,m):=0$ for $m<0$.)
\end{prop}
%%%%%%%%%%%%%%%%%%%%%%%%%%%%%%%%%%%%%%%%%%%%%%%%%%%%%%%%%%%%%%%%%%%%%%%%%%%%%%%%
\noindent\textbf{Remark}. The constant $c_2(n,m):=c_1(n,m)-c_1(n,m-2)$ is the dimension
of the space of harmonic homogeneous polynomials of order $m$ in $n$ variables. We have
$$
c_2(n,0)=1,\quad
c_2(n,1)=n.
$$
For $m\geq 2$, we have $c_2(n,m)>0$ if and only if $n>1$.

The proof of Proposition~\ref{Weyllaw} uses the following two technical lemmas:
%%%%%%%%%%%%%%%%%%%%%%%%%%%%%%%%%%%%%%%%%%%%%%%%%%%%%%%%%%%%%%%%%%%%%%%%%%%%%%%%
\begin{lemm}
Take $u\in\mathcal D'(M;\otimes^m_S T^*M)$. Then, denoting $D=\mathcal S\circ \nabla$ as in Section~\ref{s:laplacian-def},
\begin{align}
  \label{e:ididi-1}
[\Delta,\nabla^*]u&=(2-2m-n)\nabla^*u-2(m-1)D (\mathcal T(u)),\\
  \label{e:ididi-2}
[\Delta,D]u&=(2m+n) Du+2m\mathcal S(g\otimes\nabla^* u).
\end{align}
\end{lemm}
%%%%%%%%%%%%%%%%%%%%%%%%%%%%%%%%%%%%%%%%%%%%%%%%%%%%%%%%%%%%%%%%%%%%%%%%%%%%%%%%
\begin{proof}
We have
$$
\Delta\nabla^*u=\mathcal T^2(\nabla^3 u),\quad
\nabla^*\Delta u=\mathcal T^2(\tau_{1\lra 3}\nabla^3 u).
$$
where $\tau_{j\lra k} v$ denotes the result of swapping $j$th and $k$th indices in a cotensor $v$.
We have
$$
\Id-\tau_{1\lra 3}=
(\Id-\tau_{1\lra 2})
+\tau_{1\lra 2}(\Id-\tau_{2\lra 3})
+\tau_{1\lra 2}\tau_{2\lra 3}(\Id-\tau_{1\lra 2}),
$$
therefore (using that $\mathcal T\tau_{1\lra 2}=\mathcal T$)
$$
[\Delta,\nabla^*]u=\mathcal T^2\big(\nabla(\Id-\tau_{1\lra 2})\nabla^2 u
+\tau_{2\lra 3}(\Id-\tau_{1\lra 2})\nabla^3 u\big)
$$
Since $M$ has sectional curvature $-1$, we have for any
cotensor $v$ of rank $m$,
$$
(\Id-\tau_{1\lra 2})\nabla^2 v=\sum_{\ell=1}^m(\tau_{1\lra \ell+2}-\tau_{2\lra \ell+2})(g\otimes v).
$$
Then we compute  (using that $\mathcal T(\tau_{2\lra 3}\tau_{1\lra 3})=\mathcal T(\tau_{2\lra 3})$)
$$
\begin{gathered}\relax
[\Delta,\nabla^*]u=\mathcal T^2\bigg(
\tau_{2\lra 3}-\Id+
\sum_{\ell=1}^m ((\tau_{2\lra \ell+3}-\tau_{3\lra \ell+3})\tau_{1\lra 3}
+\tau_{2\lra 3}(\tau_{1\lra \ell+3}-\tau_{2\lra \ell+3}))
\bigg)(g\otimes\nabla u).
\end{gathered}
$$
Now,
$$
\begin{gathered}
\mathcal T^2(g\otimes\nabla u)=\mathcal T^2(\tau_{2\lra 4}\tau_{1\lra 3}(g\otimes\nabla u))=
\mathcal T^2(\tau_{2\lra 3}\tau_{1\lra 4}(g\otimes \nabla u))=-(n+1)\nabla^*u,\\
\mathcal T^2(\tau_{2\lra 3}(g\otimes\nabla u))
=\mathcal T^2(\tau_{3\lra 4}\tau_{1\lra 3}(g\otimes\nabla u))=
\mathcal T^2(\tau_{2\lra 3}\tau_{2\lra 4}(g\otimes\nabla u))=
-\nabla^*u,
\end{gathered}
$$
and since $u$ is symmetric, for $1<\ell\leq m$,
$$
\begin{gathered}
\mathcal T^2(\tau_{2\lra \ell+3}\tau_{1\lra 3}(g\otimes\nabla u))=
\mathcal T^2(\tau_{2\lra 3}\tau_{1\lra \ell+3}(g\otimes \nabla u))=-\nabla^*u,\\
\mathcal T^2(\tau_{3\lra \ell+3}\tau_{1\lra 3}(g\otimes\nabla u))=
\mathcal T^2(\tau_{2\lra 3}\tau_{2\lra \ell+3}(g\otimes \nabla u))=\tau_{1\lra \ell-1}\nabla(\mathcal T(u)).
\end{gathered}
$$
We then compute
$$
[\Delta,\nabla^*]u=(2-2m-n)\nabla^*u-2\sum_{\ell=1}^{m-1}\tau_{1\lra \ell}\nabla(\mathcal T(u)),
$$
finishing the proof of~\eqref{e:ididi-1}. The identity~\eqref{e:ididi-2}
follows from~\eqref{e:ididi-1} by taking the adjoint on the space of symmetric tensors.
\end{proof}
%%%%%%%%%%%%%%%%%%%%%%%%%%%%%%%%%%%%%%%%%%%%%%%%%%%%%%%%%%%%%%%%%%%%%%%%%%%%%%%%

%%%%%%%%%%%%%%%%%%%%%%%%%%%%%%%%%%%%%%%%%%%%%%%%%%%%%%%%%%%%%%%%%%%%%%%%%%%%%%%%
\begin{lemm}
Denote by $\tilde\pi_m:\otimes^m_S T^*M\to \otimes^m_S T^*M$ the orthogonal projection onto
the space $\ker \mathcal T$ of trace free tensors.
Then for each $m$,
the space
\begin{equation}
  \label{e:F-m}
F^m:=\{v\in C^\infty(M;\otimes^m_S T^*M)\mid
\mathcal T(v)=0,\
\tilde\pi_{m+1}(Dv)=0\}
\end{equation}
is finite dimensional.
\end{lemm}
%%%%%%%%%%%%%%%%%%%%%%%%%%%%%%%%%%%%%%%%%%%%%%%%%%%%%%%%%%%%%%%%%%%%%%%%%%%%%%%%
\begin{proof}
The space $F^m$ is contained in the kernel of the operator
$$
P_m:=\nabla^* \tilde\pi_{m+1} D
$$
acting on trace free sections of $\otimes^m_ST^*M$. By~\cite[Lemma 5.2]{DaSh}, the operator
$P_m$ is elliptic; therefore, its kernel is finite dimensional.
\end{proof}
We now prove Proposition~\ref{Weyllaw}.
For each $m\geq 0$ and $s\in\mathbb R$, denote
$$
W^m(\sigma):=\{u\in \mathcal D'(M;\otimes^m_S T^*M)\mid \Delta u=\sigma u,\ \mathcal T(u)=0\}.
$$
The operator $\Delta$ acting on trace free symmetric tensors is elliptic and in fact, its principal
symbol coincides with that of the scalar Laplacian: $p(x,\xi)=|\xi|_g^2$. It follows
that $W^m(\sigma)$ are finite dimensional and consist of smooth sections.
By the general argument of H\"ormander~\cite[Section~17.5]{ho3} (see also \cite[Theorem~10.1]{di-sj}
and \cite[Theorem~6.8]{e-z}; all of these arguments adapt straightforwardly to the case of operators
with diagonal principal symbols acting on vector bundles),
we have the following Weyl law:
\begin{equation}
  \label{e:weyl-base}
\sum_{\sigma\leq R^2} \dim W^m(\sigma)=c_0(n)(c_1(n+1,m)-c_1(n+1,m-2))\Vol(M) R^{n+1}+\mathcal O(R^n);
\end{equation}
here $c_1(n+1,m)-c_1(n+1,m+2)$ is the dimension of the
vector bundle on which we consider the operator $\Delta$.

By~\eqref{e:ididi-1}, for $m\geq 1$ the divergence operator acts
\begin{equation}
  \label{e:divact}
\nabla^*:W^m(\sigma)\to W^{m-1}(\sigma+2-2m-n).
\end{equation}
This operator is surjective except at finitely many points $\sigma$:
%%%%%%%%%%%%%%%%%%%%%%%%%%%%%%%%%%%%%%%%%%%%%%%%%%%%%%%%%%%%%%%%%%%%%%%%%%%%%%%%
\begin{lemm}
Let $C_1=\dim F^{m-1}$, where $F^{m-1}$ is defined in~\eqref{e:F-m}. Then the number of values $\sigma$
such that~\eqref{e:divact} is not surjective does not exceed $C_1$. 
\end{lemm}
%%%%%%%%%%%%%%%%%%%%%%%%%%%%%%%%%%%%%%%%%%%%%%%%%%%%%%%%%%%%%%%%%%%%%%%%%%%%%%%%
\begin{proof}
Assume that~\eqref{e:divact} is not surjective for some $\sigma$. Then there exists nonzero
$v\in W^{m-1}(\sigma+2-2m-n)$ which is orthogonal to $\nabla^*(W^m(\sigma))$. Since the spaces
$W^{m-1}(\sigma)$ are mutually orthogonal, we see from~\eqref{e:divact} that $v$ is also orthogonal to
$\nabla^*(W^m(\sigma))$ for all $\sigma\neq\sigma$. It follows that for each $\sigma$
and each $u\in W^m(\sigma)$, we have $\langle Dv,u\rangle_{L^2}=0$.
Since $\bigoplus_\sigma W^m(\sigma)$ is dense in the space of trace free tensors,
we see that for each $u\in C^\infty(M;\otimes^m_ST^*M)$ with $\mathcal T(u)=0$, we have
$\langle Dv, u\rangle_{L^2}=0$, which implies that $v\in F^{m-1}$. It remains
to note that $F^{m-1}$ can have a nontrivial intersection with at most $C_1$ of
the spaces $W^{m-1}(\sigma+2-2m-n)$.
\end{proof}
%%%%%%%%%%%%%%%%%%%%%%%%%%%%%%%%%%%%%%%%%%%%%%%%%%%%%%%%%%%%%%%%%%%%%%%%%%%%%%%%
Since $\Eig^m(\sigma)$ is the kernel of~\eqref{e:divact}, we have
$$
\dim \Eig^m(\sigma)\geq \dim W^m(\sigma)-\dim W^{m-1}(\sigma+2-2m-n),
$$
and this inequality is an equality if~\eqref{e:divact} is surjective. We then see that
for some constant $C_2$ independent of $R$,
$$
\begin{gathered}
\sum_{\sigma\leq R^2}\dim W^m(\sigma)
-\sum_{\sigma\leq R^2+2-2m-n}\dim W^{m-1}(\sigma)\leq 
\sum_{\sigma\leq R^2}\dim\Eig^m(\sigma)\\
\leq C_2+\sum_{\sigma\leq R^2}\dim W^m(\sigma)
-\sum_{\sigma\leq R^2+2-2m-n}\dim W^{m-1}(\sigma)
\end{gathered}
$$
and Proposition~\ref{Weyllaw} now follows from~\eqref{e:weyl-base} and the identity
$c_1(n+1,m)-c_1(n+1,m-1)=c_1(n,m)$.

%%%%%%%%%%%%%%%%%%%%%%%%%%%%%%%%%%%%%%%%%%%%%%%%%%%%%%%%%%%%%%%%%%%%%%%%%%%%%%%%
\subsection{The case $m=1$}
\label{s:m1}

In this section, we describe space $\Eig^1(\sigma)$ in terms of Hodge theory;
see for instance \cite[Section~7.2]{petersen} for the notation used.
Note that symmetric cotensors of order $1$
are exactly differential 1-forms on $M$. Since the operator
$\nabla:C^\infty(M)\to C^\infty(M;T^*M)$ is equal to the operator $d$ on $0$-forms,
we have
$$
\Eig^1(\sigma)=\{u\in\Omega^1(M)\mid \Delta u=\sigma u,\ \delta u=0\}.
$$
Here $\Delta=\nabla^*\nabla$; using that $M$ has sectional curvature $-1$, we write
$\Delta$ in terms of the Hodge Laplacian $\Delta_\Omega:=d\delta+\delta d$
on 1-forms using the following Weitzenb\"ock formula~\cite[Corollary~7.21]{petersen}:
$$
\Delta u=(\Delta_\Omega+n)u,\quad
u\in\Omega^1(M).
$$
We then see that
\begin{equation}
  \label{e:hodge-relation}
\Eig^1(\sigma)=\{u\in\Omega^1(M)\mid \Delta_\Omega u=(\sigma-n)u,\ \delta u=0\}.
\end{equation}
Finally, let us consider the case $n=1$. The Hodge star operator
acts from $\Omega^1(M)$ to itself, and we see that for $\sigma\neq 1$,
\begin{equation}
\begin{gathered}
\Eig^1(\sigma)=\{*u\mid u\in\Omega^1(M),\ \Delta_\Omega u =(\sigma-1)u,\  du = 0\}\\
=\{*(df)\mid  f\in \CI(M),\ \Delta f = (\sigma-1)f \}.
\end{gathered}
\end{equation}
Note that $*(df)$ can be viewed as the Hamiltonian field of $f$ with respect
to the naturally induced symplectic form (that is, volume form) on $M$.

%%%%%%%%%%%%%%%%%%%%%%%%%%%%%%%%%%%%%%%%%%%%%%%%%%%%%%%%%%%%%%%%%%%%%%%%%%%%%%%%
%                               ACKNOWLEDGEMENTS                               %
%%%%%%%%%%%%%%%%%%%%%%%%%%%%%%%%%%%%%%%%%%%%%%%%%%%%%%%%%%%%%%%%%%%%%%%%%%%%%%%%
\smallsection{Acknowledgements} 
We would like to thank Maciej Zworski, Richard Melrose, Steve Zelditch, Rafe Mazzeo, and Kiril Datchev for
many useful discussions and suggestions regarding this project,
an anonymous referee for many useful suggestions,
and Tobias Weich for corrections and references.
The first author would like to thank MSRI (NSF grant 0932078 000,
Fall 2013) where part of this work was done, another part
was supported by the NSF grant DMS-1201417,
and part of this project was completed during the period SD served as a Clay Research Fellow.
CG and FF have been supported by ``Agence Nationale
de la Recherche'' under the grant ANR-13-BS01-0007-01.

%%%%%%%%%%%%%%%%%%%%%%%%%%%%%%%%%%%%%%%%%%%%%%%%%%%%%%%%%%%%%%%%%%%%%%%%%%%%%%%%
%                                 BIBLIOGRAPHY                                 %
%%%%%%%%%%%%%%%%%%%%%%%%%%%%%%%%%%%%%%%%%%%%%%%%%%%%%%%%%%%%%%%%%%%%%%%%%%%%%%%%

% arXiv bibliography macro
\def\arXiv#1{\href{http://arxiv.org/abs/#1}{arXiv:#1}}


\begin{thebibliography}{0}

\bibitem[AbSt]{AbSt} Milton Abramowitz and Irene Stegun,
	\emph{Handbook of mathematical functions with formulas, graphs, and mathematical tables,\/} 
	National Bureau of Standards Applied Mathematics Series, 55.

\bibitem[AnZe07]{an-ze} Nalini Anantharaman and Steve Zelditch,
	\emph{Patterson--Sullivan distributions and quantum ergodicity,\/}
	Ann. Henri Poincar\'e \textbf{8}(2007), no.~2, 361--426.

\bibitem[AnZe12]{an-ze2} Nalini Anantharaman and Steve Zelditch,
	\emph{Intertwining the geodesic flow and the Schr\"odinger group on hyperbolic surfaces,\/}
	Math. Ann. \textbf{353}(2012), no.~4, 1103--1156.

\bibitem[VdBSc]{VdBSc} Erik Van Den Ban and Henrik Schlichtkrull,
	\emph{Asymptotic expansions and boundary values of eigenfunctions on Riemannian symmetric spaces,\/}
	J. reine angew. Math. \textbf{380}(1987), 108--165.

\bibitem[BoH\"a]{bo-ha} Jean-Fran\c cois Bony and Dietrich H\"afner,
	\emph{Decay and non-decay of the local energy for the wave equation on the
	de Sitter--Schwarzschild metric,\/}
	Comm. Math. Phys. \textbf{282}(2008), 697--719.

\bibitem[BuOl95]{BuOlBook} Ulrich Bunke and Martin Olbrich,
	\emph{Selberg zeta and theta functions. A differential operator approach,\/}
	Mathematical Research \textbf{83}, Akademie-Verlag, Berlin, 1995.

\bibitem[BuOl96]{BuOl1} Ulrich Bunke and Martin Olbrich,
	\emph{Fuchsian groups of the second kind and representations carried by the limit set,\/}
	Invent. Math. \textbf{127}(1996), no.~1, 127--154.

\bibitem[BuOl99]{BuOl2} Ulrich Bunke and Martin Olbrich,
	\emph{Group cohomology and the singularities of the Selberg zeta function associated to a Kleinian group,\/}
	Ann. of Math. (2) \textbf{149}(1999), no.~2, 627--689.

\bibitem[BuOl01]{BuOl3} Ulrich Bunke and Martin Olbrich,
	\emph{Regularity of invariant distributions,\/}
	preprint, \arXiv{math.DG/0103144}.

\bibitem[BuLi]{BuLi} Oliver Butterley and Carlangelo Liverani,
	\emph{Smooth Anosov flows: correlation spectra and stability,\/} 
	J. Mod. Dyn. 1(2)(2007), 301--322.

\bibitem[DaSh]{DaSh} Nurlan Dairbekov and Vladimir Sharafutdinov,
	\emph{On conformal Killing symmetric tensor fields on Riemannian manifolds,\/}
	Sib. Adv. Math. \textbf{21}(2011), no.~1, 1--41.

\bibitem[DDZ]{ddz} Kiril Datchev, Semyon Dyatlov, and Maciej Zworski,
	\emph{Sharp polynomial bounds on the number of Pollicott--Ruelle resonances,\/}
	Erg. Theory Dyn. Syst. \textbf{34}(2014), 1168--1183.
	
\bibitem[DiSj]{di-sj} Mouez Dimassi and Johannes Sj\"ostrand,
	\emph{Spectral asymptotics in the semi-classical limit,\/}
	Cambridge University Press, 1999.

\bibitem[Do]{dolgopyat} Dmitry Dolgopyat,
	\emph{On decay of correlations in Anosov flows,\/}
	Ann. of Math. (2), \textbf{147}(1998), 357--390.

\bibitem[Dy12]{zeeman} Semyon Dyatlov,
	\emph{Asymptotic distribution of quasi-normal modes for Kerr--de Sitter black holes,\/}
	Ann. Henri Poincar\'e \textbf{13}(2012), 1101--1166.
	
\bibitem[Dy15]{kdsu} Semyon Dyatlov,
	\emph{Asymptotics of linear waves and resonances with applications to black holes,\/}
	Comm. Math. Phys. \textbf{335}(2015), 1445--1485.

\bibitem[DyZw]{DyZw} Semyon Dyatlov and Maciej Zworski,
	\emph{Dynamical zeta functions for Anosov flows via microlocal analysis,\/}
	to appear in Annales de l'ENS, \arXiv{1306.4203}.

\bibitem[FaSj]{FaSj} Fr\'ederic Faure and Johannes Sj\"ostrand,
	\emph{Upper bound on the density of Ruelle resonances for Anosov flows,\/}
	Comm. Math. Phys. \textbf{308}(2011), no.~2, 325--364.

\bibitem[FaTs12]{FaTs1} Fr\'ed\'eric Faure and Masato Tsujii,
	\emph{Prequantum transfer operator for symplectic Anosov diffeomorphism,\/}
	preprint, \arXiv{1206.0282}.
	
\bibitem[FaTs13a]{FaTs2} Fr\'ed\'eric Faure and Masato Tsujii,
	\emph{Band structure of the Ruelle spectrum of contact Anosov flows,\/}
	Comptes Rendus Math. \textbf{351}(2013), no.~9--10, 385--391.

\bibitem[FaTs13b]{FaTs3} Fr\'ederic Faure and Masato Tsujii,
	\emph{The semiclassical zeta function for geodesic flows on negatively curved manifolds,\/}
	preprint, \arXiv{1311.4932}.

\bibitem[FlFo]{FlFo} Livio Flaminio and Giovanni Forni,
	\emph{Invariant distributions and time averages for horocycle flows,\/}
	Duke Math. J. \textbf{119}(2003), no.~3, 393--588.

\bibitem[FoGe]{FoGe} Israel Gelfand and Sergey Fomin,
	\emph{Geodesic flows on manifolds of constant negative curvature,\/}
	Amer. Math. Soc. Transl. (2) 1 (1955), 49--65. 

\bibitem[Fr86]{fried} David Fried,
	\emph{Analytic torsion and closed geodesics on hyperbolic manifolds,\/}
	Invent. Math. \textbf{84}(1986), no.~3, 523--540.

\bibitem[Fr95]{fried2} David Fried,
	\emph{Meromorphic zeta functions for analytic flows,\/}
	Comm. Math. Phys. \textbf{174}(1995), 161--190.

\bibitem[GLP]{GLP} Paolo Giulietti, Carlangelo Liverani, and Mark Pollicott,
	\emph{Anosov Flows and Dynamical Zeta Functions,\/} 
	Ann. of Math. (2) \textbf{178}(2013), 687--773.	

\bibitem[GrOt]{GrOt} Sandrine Grellier and Jean-Pierre Otal,
	\emph{Bounded eigenfunctions in the real hyperbolic space,\/} 
	Int. Math Res. Not. \textbf{62}(2005), 3867--3897.

\bibitem[GMP]{GMP} Colin Guillarmou, Sergiu Moroianu, and Jinsung Park,
	\emph{Eta invariant and Selberg zeta function of odd type over convex co-compact hyperbolic manifolds,\/}
	Adv. Math. \textbf{225}(2010), no.~5, 2464--2516.
	
\bibitem[Gu]{guillemin} Victor Guillemin,
	\emph{Lectures on spectral theory of elliptic operators,\/}
	Duke Math. J. \textbf{44}(1977), no.~3, 485--517.

\bibitem[HHS]{HHS} S\"onke Hansen, Joachim Hilgert, and Michael Schr\"oder,
	\emph{Patterson--Sullivan distributions in higher rank,\/}
	preprint, \arXiv{1105.5788}.

\bibitem[He70]{He1} Sigurdur Helgason,
 \emph{A duality for symmetric spaces with applications to group representations,\/} 
 Adv. Math. \textbf{5}(1970) 1--154.

\bibitem[He74]{He} Sigurdur Helgason,
	\emph{Eigenspaces of the Laplacian; integral representations and irreducibility,\/}
	J. Funct. Anal. \textbf{17}(1974), no.~3, 328--353.

\bibitem[H\"oI]{ho1} Lars H\"ormander,
	\emph{The Analysis of Linear Partial Differential Operators I. Distribution Theory and Fourier Analysis,\/}
	Springer, 1983.

\bibitem[H\"oIII]{ho3} Lars H\"ormander,
	\emph{The Analysis of Linear Partial Differential Operators III. Pseudo-Differential Operators,\/}
	Springer, 1994.
	
\bibitem[Ju]{Ju} Andreas Juhl,
	\emph{Cohomological theory of dynamical zeta functions,\/}
	Progress in Mathematics, 194. Birkhauser Verlag, Basel, 2001.

\bibitem[KaHa]{KaHa} Anatole Katok and Boris Hasselblatt,
	\emph{Introduction to the modern theory of dynamical systems,\/}
	Cambridge Univ. Press, 1995.

\bibitem[Le]{leboeuf} Patricio Leboeuf,
	\emph{Periodic orbit spectrum in terms of Pollicott--Ruelle resonances,\/}
	Phys. Rev. E \textbf{69}(2004), 026204.
	
\bibitem[Li]{liverani} Carlangelo Liverani,
	\emph{On contact Anosov flows,\/}
	Ann. of Math. (2), \textbf{159}(2004), 1275--1312.

\bibitem[Ma]{Ma} Rafe Mazzeo,
	\emph{Elliptic theory of differential edge operators. I,\/} 
	Comm. PDE \textbf{16}(1991), no.~10, 1615--1664. 

\bibitem[MaMe]{MM} Rafe Mazzeo and Richard Melrose,
	\emph{Meromorphic extension of the resolvent on complete spaces with
	asymptotically constant negative curvature,\/}
	J.\ Funct.\ Anal.\ \textbf{75}(1987), 260--310.

\bibitem[MaVe]{MaVe} Rafe Mazzeo and Boris Vertman,
	\emph{Elliptic theory of differential edge operators, II: boundary value problems,\/} 
	preprint, \arXiv{1307.2266}

\bibitem[Me]{melro} Richard Melrose,
	\emph{The Atiyah--Patodi--Singer index theorem,\/}
	Research Notes in Mathematics \textbf{4}, A.K. Peters, Ltd., Wellesley, MA, 1993.

\bibitem[Mi]{Mi} Katsuhiro Minemura,
	\emph{Eigenfunctions of the Laplacian on a real hyperbolic space,\/}
	J. Math. Soc. Japan \textbf{27}(1975), no.~1, 82--105. 

\bibitem[Mo]{Mo} Calvin Moore,
	\emph{Exponential decay of correlation coefficients for geodesic flows,\/}
	in Group Representation Ergodic Theory, Operator Algebra and Mathematical Physics, 
	Springer, Berlin, 1987.

\bibitem[NoZw09]{NoZw} St\'ephane Nonnenmacher and Maciej Zworski,
	\emph{Quantum decay rates in chaotic scattering,\/}
	Acta Math. \textbf{203}(2009), no.~2, 149--233. 

\bibitem[NoZw13]{NoZw2} St\'ephane Nonnenmacher and Maciej Zworski,
	\emph{Decay of correlations for normally hyperbolic trapping,\/}
	to appear in Invent. Math., \arXiv{1302.4483}.

\bibitem[Ol]{olbrich} Martin Olbrich,
	\emph{Die Poisson-Transformation f\"ur homogene Vektorb\"undel,\/}
	Ph.D. dissertation,
	\url{http://www.uni-math.gwdg.de/olbrich/Poisson.dvi}

\bibitem[OsSe]{OsSe} Toshio Oshima and Jiro Sekiguchi,
	\emph{Eigenspaces of invariant differential operators on an affine symmetric space,\/}
	Invent. Math. \textbf{57}(1980), no.~1, 1--81.

\bibitem[Ot]{Ot} Jean-Pierre Otal,
	\emph{Sur les fonctions propres du laplacien du disque hyperbolique,\/}
	Comptes Rendus de l'Acad\'emie des Sciences, Paris \textbf{327}(1998), 161--166.

\bibitem[Pe]{petersen} Peter Petersen,
	\emph{Riemannian geometry,\/}
	second edition, Springer, 2006.
	
\bibitem[Po86a]{pollicott} Mark Pollicott,
	\emph{Meromorphic extensions of generalised zeta functions,\/}
	Invent. Math. \textbf{85}(1986), 147--164.

\bibitem[Po86b]{Po} Mark Pollicott,
	\emph{Distributions at infinity for Riemann surfaces,\/}
	Dynamical systems and ergodic theory (Warsaw, 1986), 91--100, Banach Center Publ., 23, PWN, Warsaw, 1989. 
	
\bibitem[Ra]{Ra} Marina Ratner,
	\emph{The rate of mixing for geodesic and horocycle flows,\/}
	Erg. Theory Dyn. Syst. \textbf{7}(1987), 267--288.
	
\bibitem[Ru]{ruelle} David Ruelle,
	\emph{Resonances of chaotic dynamical systems,\/}
	Phys. Rev. Lett. \textbf{56}(1986), no.~5, 405--407.
	
\bibitem[Ts10]{Ts0} Masato Tsujii,
	\emph{Quasi-compactness of transfer operators for contact Anosov flows,\/}
	Nonlinearity \textbf{23}(2010), 1495--1545.
	
\bibitem[Ts12]{tsujii} Masato Tsujii,
	\emph{Contact Anosov flows and the FBI transform,\/}
	Erg. Theory Dyn. Syst. \textbf{32}(2012), no.~6, 2083--2118.

\bibitem[Ze]{zelditch} Steven Zelditch,
	\emph{Uniform distribution of eigenfunctions on compact hyperbolic surfaces,\/}
	Duke Math. J. \textbf{55}(1987), no.~4, 919--941.

\bibitem[Zw]{e-z} Maciej Zworski,
	\emph{Semiclassical analysis,\/}
	Graduate Studies in Mathematics \textbf{138}, AMS, 2012.

\end{thebibliography}
\end{document}